\documentclass[12pt,CJK]{article}
\usepackage{amssymb}
\usepackage{}
\usepackage{stmaryrd}
\usepackage{amsfonts,amsmath,amssymb,amsthm,mathrsfs}
\usepackage{latexsym}
\usepackage{graphicx}
\usepackage{indentfirst}
\usepackage{color}
\usepackage{fancyhdr}
\usepackage[colorlinks,linktocpage,linkcolor=blue]{hyperref}

\theoremstyle{plain}
\newtheorem{thm}{Theorem}[section]
\newtheorem{definition}{Definition}
\newtheorem{lemma}[thm]{Lemma}
\newtheorem{corollary}[thm]{Corollary}
\newtheorem{remark}[thm]{Remark}
\numberwithin{equation}{section}

\theoremstyle{remark}

\def\Xint#1{\mathchoice
  {\XXint\displaystyle\textstyle{#1}}%
  {\XXint\textstyle\scriptstyle{#1}}%
  {\XXint\scriptstyle\scriptscriptstyle{#1}}%
  {\XXint\scriptscriptstyle\scriptscriptstyle{#1}}%
  \!\int}
\def\XXint#1#2#3{{\setbox0=\hbox{$#1{#2#3}{\int}$}
  \vcenter{\hbox{$#2#3$}}\kern-.5\wd0}}

\def\dashint{\Xint-}

\title{\textbf{The Methods of Layer Potentials for General Elliptic Homogenization Problems \\
in Lipschitz Domains}}

\author{ Qiang Xu
\thanks{Corresponding author.}
\thanks{Email: xuqiang@math.pku.edu.cn.}\\
School of Mathematical Sciences, Peking University, \\
Beijing, 100871, PR China. \vspace{0.2cm}\\
Peihao Zhao
\thanks{Email: zhaoph@lzu.edu.cn}\\
School of Mathematics and Statistics, Lanzhou University, \\
Lanzhou, 710000, PR China. \vspace{0.2cm}\\
Shulin Zhou
\thanks{Email: szhou@math.pku.edu.cn.}\\
School of Mathematical Sciences, Peking University, \\
Beijing, 100871, PR China.
}

\begin{document}
\allowdisplaybreaks
\maketitle
\begin{abstract}
In terms of layer potential methods,
this paper is devoted to study the $L^2$ boundary value problems for nonhomogeneous elliptic
operators with rapidly oscillating coefficients in a periodic setting.
Under a low regularity assumption on the coefficients,
we establish the solvability for
Dirichlet, regular and Neumann problems in a bounded Lipschitz domain, as well as,
the uniform nontangential maximal function estimates and square function estimates.
The main difficulty is reflected in two aspects:
(i) we can not treat the lower order terms
as a compact perturbation to the leading term due to the low regularity assumption;
(ii) the nonhomogeneous systems do not possess a scaling-invariant property in general.
Although this work may be regarded as a follow-up to C. Kenig and Z. Shen's in \cite{SZW24},
we make an effort to find a clear way of how to handle the nonhomogeneous operators
by using the known results of the homogenous ones. Also, we mention that the periodicity condition
plays a key role in the scaling-invariant estimates.
\\
\textbf{Key words:} Layer potential; elliptic homogenization; Lipschitz domains.
\end{abstract}

\tableofcontents

\section{Introduction and main results}

The quantitative results for the general elliptic systems with Dirichlet or Neumann boundary conditions,
arising in the periodic homogenization theory, have been systematically studied.
The papers \cite{QXS,QXS1} mainly concerned the optimal uniform regularity estimates,
which were derived from the related results of M. Avellaneda,
F. Lin \cite{MAFHL,MAFHL5,MAFHL3}
and, of C. Kenig, F. Lin, Z. Shen \cite{SZW12}, respectively.
In terms of convergence rates,
the paper \cite{QX2} received an almost sharp error estimate $O(\varepsilon\ln(1/\varepsilon))$
in Lipschitz domains under no regularity assumption on the coefficients, inspired by
C. Kenig, F. Lin, Z. Shen \cite{SZW12} and by T. Suslina \cite{TS}.

In this paper, we turn to study the well-posedness of the $L^2$ Dirichlet, regularity, and Neumann problems
for nonhomogeneous elliptic systems with rapidly oscillating periodic coefficients.
More precisely, we continue to consider the following operators depending on a parameter $\varepsilon > 0$,
\begin{eqnarray*}
\mathcal{L}_{\varepsilon} =
-\text{div}\big[A(x/\varepsilon)\nabla +V(x/\varepsilon)\big]
+ B(x/\varepsilon)\nabla +c(x/\varepsilon) + \lambda I,
\end{eqnarray*}
where $\lambda>0$ is a constant, and $I$ denotes the identity matrix.

\subsection{Basic assumptions}

Let $d\geq 3$, $m\geq 1$, and $1 \leq i,j \leq d$.
Suppose that $A = (a_{ij}^{\alpha\beta})$, $V=(V_i^{\alpha\beta})$,
$B=(B_i^{\alpha\beta})$ and $c=(c^{\alpha\beta})$ are real measurable functions,
satisfying the following conditions:
\begin{itemize}
\item the uniform ellipticity condition
\begin{equation}\label{a:1}
 \mu |\xi|^2 \leq a_{ij}^{\alpha\beta}(y)\xi_i^\alpha\xi_j^\beta\leq \mu^{-1} |\xi|^2
 \quad \text{for}~y\in\mathbb{R}^d,~\xi\in\mathbb{R}^{md},
 ~\text{where}~ \mu>0;
\end{equation}
 (The summation convention for repeated indices is used throughout.)
\item the periodicity condition
\begin{equation}\label{a:2}
A(y+z) = A(y),~~ V(y+z) = V(y),
~~B(y+z) = B(y),~~ c(y+z) = c(y)
\end{equation}
for $y\in\mathbb{R}^d$ and $z\in \mathbb{Z}^d$;
\item the boundedness condition
\begin{equation}\label{a:3}
\max\big\{\|V\|_{L^{\infty}(\mathbb{R}^d)},
 ~\|B\|_{L^{\infty}(\mathbb{R}^d)},~\|c\|_{L^{\infty}(\mathbb{R}^d)}\big\}
 \leq \kappa;
\end{equation}
\item the regularity condition
\begin{equation}\label{a:4}
 \max\big\{ \|A\|_{C^{0,\tau}(\mathbb{R}^d)},~ \|V\|_{C^{0,\tau}(\mathbb{R}^d)},
 ~\|B\|_{C^{0,\tau}(\mathbb{R}^d)}\big\} \leq \kappa,
 \qquad \text{where}~\tau\in(0,1)~\text{and}~\kappa > 0.
\end{equation}
\end{itemize}


Throughout the paper, we always assume $\Omega\subset\mathbb{R}^d$ is a bounded Lipschitz domain, and $R_0$ denotes
the diameter of $\Omega$, unless otherwise stated. To establish the existence,
the following constant is crucial,
\begin{equation}\label{KEY:1}
\lambda_0 = \frac{c(m,d)}{\mu}\Big\{\|V\|_{L^\infty(\mathbb{R}^d)}^2 +
\|B\|_{L^\infty(\mathbb{R}^d)}^2+ \|c\|_{L^\infty(\mathbb{R}^d)}\Big\}.
\end{equation}

\subsection{Main results}

\begin{thm}[scaling-invariant estimates]\label{thm:0.1}
Let $B=B(0,1)$ be a unite ball in $\mathbb{R}^d$.
Suppose that the coefficients of $\mathcal{L}_\varepsilon$
satisfy $\eqref{a:1}$ and $\eqref{a:3}$ with $\lambda\geq\lambda_0$.
Let $u_\varepsilon\in H^1(4B;\mathbb{R}^m)$ be a weak solution to
$\mathcal{L}_\varepsilon(u_\varepsilon) = 0$ in $4B$, satisfying
\begin{equation*}
 \bigg(\dashint_{4B} |u_\varepsilon|^2\bigg)^{1/2}
 =  \bigg(\frac{1}{|4B|}\int_{4B} |u_\varepsilon|^2dx\bigg)^{1/2}\leq 1.
\end{equation*}
\begin{itemize}
  \item If $A\in\emph{VMO}(\mathbb{R}^d)$ satisfies $\eqref{a:2}$,
  then for any $2\leq p<\infty$, there exists a constant
  \begin{equation}\label{KEY:2}
  C_p = C_p\Big(\mu,m,d,p,\|A\|_{\emph{VMO}},
  \|V/\sqrt{\lambda}\|_{L^\infty(\mathbb{R}^d)},
  \|B/\sqrt{\lambda}\|_{L^\infty(\mathbb{R}^d)},
  \|c/\lambda\|_{L^\infty(\mathbb{R}^d)}\Big)
  \end{equation}
  such that
\begin{equation}\label{pri:0.1}
\|u_\varepsilon\|_{W^{1,p}(B)}
\leq C_p.
\end{equation}
 Moreover, for any $\sigma\in(0,1)$, there exists a constant
 $C_\sigma = C_p$ with $p=d/(1-\sigma)$ such that
 \begin{equation}\label{pri:0.2}
 (\emph{i})~~[u_\varepsilon]_{C^{0,\sigma}(B)} \leq C_\sigma,
 \qquad\quad (\emph{ii})~~\|u_\varepsilon\|_{L^{\infty}(B)} \leq C_\sigma.
 \end{equation}
\item  If the coefficients $A$ and $V$ satisfy $\eqref{a:2}$, $\eqref{a:4}$, then
there exists a constant
  \begin{equation}\label{KEY:3}
  C_\tau =
  C_\tau\Big(\mu,m,d,\tau,\big[A\big]_{C^{0,\tau}(\mathbb{R}^d)},
  \big[V/\sqrt{\lambda}\big]_{C^{0,\tau}(\mathbb{R}^d)},
  C_{\bar{p}}\Big) \text{~with~~} \bar{p}>d,
  \end{equation}
  such that
\begin{equation}\label{pri:0.3}
\begin{aligned}
\|\nabla u_\varepsilon\|_{L^\infty(B)}\leq C_\tau.
\end{aligned}
\end{equation}
\end{itemize}
\end{thm}

Here the definition of the space $\text{VMO}(\mathbb{R}^d)$ may be found in \cite[pp.43]{S4}.

\begin{definition}
We say that the operator
$\mathcal{L}_\varepsilon$ satisfies the properties $(\emph{H}_1)$ or $(\emph{H}_2)$,
if there exists
a constant $C_0>0$ such that the fundamental
solution $\mathbf{\Gamma}_\varepsilon(x,y)$ associated with $\mathcal{L}_\varepsilon$
has the size estimate
$$
 \big|\mathbf{\Gamma}_\varepsilon(x,y)\big|\leq C_0|x-y|^{2-d}
\leqno{(\emph{H}_1)}
$$
for any $x,y\in\mathbb{R}^d$ with $x\not= y$,
or there exists a constant $C_{00}>0$ such that we have the decay estimates
$$
\begin{aligned}
\big|\nabla_x\mathbf{\Gamma}_\varepsilon(x,y)\big|
+ \big|\nabla_y\mathbf{\Gamma}_\varepsilon(x,y)\big|
&\leq C_{00}|x-y|^{1-d}, \\
\big|\nabla_x\nabla_y\mathbf{\Gamma}_\varepsilon(x,y)\big|
&\leq C_{00}|x-y|^{-d}.
\end{aligned} \leqno{(\emph{H}_2)}
$$
\end{definition}

\begin{thm}[local boundedness properties]\label{thm:0.2}
Let $B=B(x,R)$ for some $x\in{\mathbb{R}^d}$ and $R>0$.
Suppose that the coefficients of $\mathcal{L}_\varepsilon$
satisfy $\eqref{a:1}$ and $\eqref{a:3}$ with $\lambda\geq\lambda_0$.
Let $u_\varepsilon\in H^1(4B;\mathbb{R}^m)$ be a weak solution to
$\mathcal{L}_\varepsilon(u_\varepsilon) = \emph{div}(f) + F$ in $4B$,
where $f\in L^p(4B;\mathbb{R}^{md})$ with $p>d$ and,
$F\in L^q(4B;\mathbb{R}^m)$ with $q>(d/2)$.
\begin{itemize}
  \item If the operator $\mathcal{L}_\varepsilon$ satisfies the property $(\emph{H}_1)$,
  then there exists a constant $C_1$, independent of $R$ and $\varepsilon$, such that
\begin{equation}\label{pri:0.4}
\|u_\varepsilon\|_{L^\infty(B)}
\leq C_1\Bigg\{\Big(\dashint_{2B} |u_\varepsilon|^2 \Big)^{1/2}
+ R\Big(\dashint_{2B}|f|^p \Big)^{1/p}
+R^2\Big(\dashint_{2B}|F|^q \Big)^{1/q}\Bigg\}
\end{equation}
where $C_1$ depends on $\mu,d,m,p,q,
\|V/\sqrt{\lambda}\|_{L^\infty(\mathbb{R}^d)},
\|B/\sqrt{\lambda}\|_{L^\infty(\mathbb{R}^d)}$ and $C_0$.
\item If the operator $\mathcal{L}_\varepsilon$
additionally satisfies $(\emph{H}_2)$,
then one may derive
\begin{equation}\label{pri:0.5}
\begin{aligned}
\|\nabla u_\varepsilon\|_{L^\infty(B)}
\leq C_2\Bigg\{\frac{1}{R}\Big(\dashint_{2B} |u_\varepsilon|^2 \Big)^{1/2}
+ \|f\|_{L^\infty(2B)} + R^\sigma[f]_{C^\sigma(2B)}
+R\Big(\dashint_{2B}|F|^p \Big)^{1/p}\Bigg\},
\end{aligned}
\end{equation}
where $\sigma\in(0,1)$, and $C_2$ depends on
$m,d,p,q,\sigma,\|V/\sqrt{\lambda}\|_{L^\infty(\mathbb{R}^d)},
\|B/\sqrt{\lambda}\|_{L^\infty(\mathbb{R}^d)},C_0$ and $C_{00}$.
\end{itemize}
\end{thm}

\begin{thm}\label{thm:2.1}
Suppose that the coefficients of $\mathcal{L}_\varepsilon$
satisfy $\eqref{a:1}$ and $\eqref{a:3}$ with $\lambda\geq\max\{\lambda_0,\mu\}$.
\begin{itemize}
  \item If $A\in\emph{VMO}(\mathbb{R}^d)$ satisfies $\eqref{a:2}$, then there
  exists a unique fundamental solution $\mathbf{\Gamma}_\varepsilon(x,y)$  being
  H\"older continuous in
$\{(x,y)\in\mathbb{R}^d\times\mathbb{R}^d:x\not=y\}$, such that
  \begin{equation}\label{relation:1}
  \text{$(\emph{H}_1)$}
  \Longleftrightarrow
  \emph{(ii)~in~}\eqref{pri:0.2}
  \Longleftrightarrow    \eqref{pri:0.4}.
  \end{equation}
  \item If the coefficients $A$ and $V$ satisfy $\eqref{a:2}$, $\eqref{a:4}$, then
  we have
  \begin{equation}\label{relation:2}
  \text{$(\emph{H}_2)$} \Longleftrightarrow \eqref{pri:0.3}
   \Longleftrightarrow    \eqref{pri:0.5}.
  \end{equation}
\end{itemize}
\end{thm}

\begin{thm}[asymptotic expansions]\label{thm:0.4}
Suppose that the coefficients of $\mathcal{L}_\varepsilon$ satisfy
$\eqref{a:1}$, $\eqref{a:2}$ and $\eqref{a:3}$ with $\lambda\geq\max\{\lambda_0,\mu\}$.
Assume
$A\in \emph{VMO}(\mathbb{R}^d)$ and let
$\mathbf{\Gamma}_{\varepsilon}(x,y),\mathbf{\Gamma}_{0}(x,y)$ be two fundamental solutions of
$\mathcal{L}_\varepsilon$ and $\mathcal{L}_0$, respectively.
Then there holds
\begin{equation}\label{pri:0.6}
\big|\mathbf{\Gamma}_{\varepsilon}(x,y)-\mathbf{\Gamma}_0(x,y)\big|
\leq \frac{C\varepsilon}{|x-y|^{d-1}}
\end{equation}
for any $x,y\in\mathbb{R}^d$ with $x\not=y$. Moreover,
if the coefficients $A,V,B$ additionally satisfy $\eqref{a:4}$,
then we have
\begin{equation}\label{pri:0.7}
\begin{aligned}
\big|\nabla_x\mathbf{\Gamma}_{\varepsilon}(x,y)-\nabla_x\mathbf{\Gamma}_0(x,y)
-\nabla\chi_0(x/\varepsilon)\mathbf{\Gamma}_0(x,y)
- \nabla\chi_k(x/\varepsilon)\nabla_{x_k}\mathbf{\Gamma}_0(x,y)\big|
&\leq \frac{C\varepsilon^{\rho}}{|x-y|^{d-1+\rho}},\\
\big|\nabla_y\mathbf{\Gamma}_{\varepsilon}(x,y)-\nabla_y\mathbf{\Gamma}_0(x,y)
-\nabla\chi_0^*(y/\varepsilon)\mathbf{\Gamma}_0(x,y)
- \nabla\chi_k^*(y/\varepsilon)\nabla_{y_k}\mathbf{\Gamma}_0(x,y)\big|
&\leq \frac{C\varepsilon^{\rho}}{|x-y|^{d-1+\rho}}
\end{aligned}
\end{equation}
for any $x,y\in\mathbb{R}^d$ with $x\not=y$ and $0<\rho<1$,
where $C$ depends only on $\mu,\kappa,\lambda,\tau,d,m$ and $\rho$.
\end{thm}

In the estimate $\eqref{pri:0.7}$, the notation $\chi_i$ and $\chi_i^*$ with $i=0,\cdots,d$ are correctors associated with
$\mathcal{L}_\varepsilon$, $\mathcal{L}_\varepsilon^*$, respectively,
in which $\mathcal{L}_\varepsilon^*$ is the adjoint operator of $\mathcal{L}_\varepsilon$ (see
Definition $\ref{def:2.1}$ and Subsection $\ref{subsec:2.1}$). Although we do not require
that the operator $\mathcal{L}_\varepsilon$ is self-adjoint, the symmetry condition on the leading term is still necessary
to be imposed, i.e.,
\begin{equation*}
 A^* = A, \qquad\quad \big( a_{ji}^{\beta\alpha} = a_{ij}^{\alpha\beta}\big).
\end{equation*}

\begin{thm}[$L^2$ Dirichlet problems]\label{thm:1.1}
Suppose that the coefficients of $\mathcal{L}_\varepsilon$
satisfy $\eqref{a:1}$, $\eqref{a:2}$, $\eqref{a:3}$ and $\eqref{a:4}$
with $A^*= A$ and $\lambda\geq\max\{\mu,\lambda_0\}$. Also,
the coefficients $V,B$ additionally satisfy $\|V-B\|_{L^\infty(\mathbb{R}^d)}\leq \epsilon_0$,
where $\epsilon_0>0$ is sufficiently small.
Then for any $g\in L^2(\partial\Omega;\mathbb{R}^m)$,
there exists a unique solution $u_\varepsilon\in H^{1/2}(\Omega;\mathbb{R}^m)$
satisfying
\begin{equation}\label{pde:1.1}
(\mathbf{DH_\varepsilon})\left\{
\begin{aligned}
\mathcal{L}_\varepsilon(u_\varepsilon) &= 0 &\quad &\emph{in}~~\Omega, \\
 u_\varepsilon &= g &\emph{n.t.~}&\emph{on} ~\partial\Omega, \\
 (u_\varepsilon)^* &\in L^2(\partial\Omega). &\quad &
\end{aligned}\right.
\end{equation}
Moreover, the solution satisfies the uniform estimates
\begin{equation}\label{pri:1.1}
\big\|(u_\varepsilon)^*\big\|_{L^2(\partial\Omega)}
\leq C\|g\|_{L^2(\partial\Omega)},
\end{equation}
and the square function estimate
\begin{equation}\label{pri:1.4}
\int_{\Omega}|\nabla u_\varepsilon|^2 \emph{dist}(x,\partial\Omega)dx
\leq C\int_{\partial\Omega} |g|^2 dS
\end{equation}
where $C$ depends only on $\mu,\kappa,\lambda,m,d$ and $\Omega$.
\end{thm}

Note that the second line of $(\mathbf{DH_\varepsilon})$
means that the solution $u_\varepsilon$ converges to $f$ in a nontangenial way
instead of in the sense of trace, and using the abbreviation ``n.t.'' depicts this difference.
The notation $(u_\varepsilon)^*$ in the estimate $\eqref{pri:1.1}$
represents the nontangential maximal function of
$u_\varepsilon$ on $\partial\Omega$, defined by
\begin{equation*}
(v)^*(P) = \sup_{y\in\Lambda_{N_0}^{+}(P)}|v(y)|,
\qquad \Lambda_{N_0}^{+}(P) = \big\{y\in\Omega:|y-P|\leq N_0
\text{dist}(y,\partial\Omega)\big\},
\end{equation*}
where $x\in\partial\Omega$, and $N_0$ depending on $d$ and the character of $\Omega$ is sufficiently large.
The quantity
in the left-hand side of $\eqref{pri:1.4}$ is referred to as the square function of $u_\varepsilon$.

In order to state the Neumann boundary value problem,
the conormal derivatives related to $\mathcal{L}_\varepsilon$ is defined as
\begin{equation}\label{op:2}
\frac{\partial}{\partial\nu_\varepsilon}
 = n\cdot\big[A(\cdot/\varepsilon)\nabla + V(\cdot/\varepsilon)\big]
 \qquad \text{on}~\partial\Omega,
\end{equation}
where $n=(n_1,\cdots,n_d)$ is the outward unit normal
vector to $\partial\Omega$.

\begin{thm}[$L^2$ Neumann problems]\label{thm:1.2}
Suppose that the coefficients of $\mathcal{L}_\varepsilon$
satisfy $\eqref{a:1}$, $\eqref{a:2}$, $\eqref{a:3}$ and $\eqref{a:4}$ with $A=A^*$ and
$\lambda\geq\max\{\mu,\lambda_0\}$.
Then for any $f\in L^2(\partial\Omega;\mathbb{R}^m)$,
there exists a unique solution $u_\varepsilon$ in $C^{1}(\Omega;\mathbb{R}^m)$
such that
\begin{equation}\label{pde:1.2}
(\mathbf{NH_\varepsilon})\left\{
\begin{aligned}
\mathcal{L}_\varepsilon(u_\varepsilon) &= 0 &\quad &\emph{in}~~\Omega, \\
 \frac{\partial u_\varepsilon}{\partial\nu_\varepsilon} &= f &\emph{n.t.~}&\emph{on} ~\partial\Omega, \\
 (\nabla u_\varepsilon)^* &\in L^2(\partial\Omega), &\quad &
\end{aligned}\right.
\end{equation}
and it satisfies the uniform estimate
\begin{equation}\label{pri:1.2}
\big\|(\nabla u_\varepsilon)^*\big\|_{L^2(\partial\Omega)}
\leq C\|f\|_{L^2(\partial\Omega)},
\end{equation}
where $C$ depends only on $\mu,\kappa,\lambda,m,d$ and $\Omega$.
\end{thm}

\begin{thm}[$L^2$ regularity problems]\label{thm:1.3}
Assume the same conditions as in Theorem $\ref{thm:1.2}$.
Then for any $g\in H^1(\partial\Omega;\mathbb{R}^m)$, the following equations
\begin{equation}\label{pde:1.3}
(\mathbf{RH_\varepsilon})\left\{
\begin{aligned}
\mathcal{L}_\varepsilon(u_\varepsilon) &= 0 &\quad &\emph{in}~~\Omega, \\
 u_\varepsilon &= g &\emph{n.t.~}&\emph{on} ~\partial\Omega, \\
 (\nabla u_\varepsilon)^* &\in L^2(\partial\Omega) &\quad &
\end{aligned}\right.
\end{equation}
has a unique solution $u_\varepsilon$ in $H^{3/2}(\Omega;\mathbb{R}^m)$,
satisfying the uniform estimate
\begin{equation}\label{pri:1.3}
\big\|(\nabla u_\varepsilon)^*\big\|_{L^2(\partial\Omega)}
+ \big\|(u_\varepsilon)^*\big\|_{L^2(\partial\Omega)}
\leq C\|g\|_{H^1(\partial\Omega)},
\end{equation}
where $C$ depends only on $\mu,\kappa,\lambda,m,d$ and $\Omega$.
\end{thm}

\subsection{Motivation and informal summary of results}

The results in Theorems $\ref{thm:1.1},\ref{thm:1.2},\ref{thm:1.3}$ are quite similar to those obtained
by Kenig and Shen in \cite{SZW24}
for the homogeneous operator $L_\varepsilon = -\text{div}(A(x/\varepsilon)\nabla)$.
Frankly speaking, the main ideas in the paper have been developed in \cite{SZW24,GZS1},
and we plan to simply introduce their strategies and then state ours.

Step 1. Approximate the fundamental solution $\Gamma_{1}(x,y)$ by
freezing the coefficients in the case of $|x-y|\leq 1$ and
using its asymptotic expansion for $|x-y|> 1$,
where $\Gamma_{1}(x,y)$ is the fundamental solution of $L_1 = -\text{div}(A(x)\nabla)$.
The purpose is to give $L^p$ bounds with $1<p<\infty$ for singular integral operators on
Lipschitz surfaces, such as Theorems $\ref{lemma:3.3},\ref{thm:5.2},\ref{thm:6.10}$.

Step 2. Establish the following Rellich estimates:
\begin{equation}\label{pri:0.8}
 \|\nabla u_\varepsilon\|_{L^2(\partial\Omega)}
 \leq C\big\|\frac{\partial u_\varepsilon}{\partial n_\varepsilon}\big\|_{L^2(\partial\Omega)},
 \qquad\text{and}\quad
 \|\nabla u_\varepsilon\|_{L^2(\partial\Omega)}
 \leq C\|\nabla_{\text{tan}}u_\varepsilon\|_{L^2(\partial\Omega)},
\end{equation}
where the notation $\partial/\partial n_\varepsilon$ denotes the normal derivative related to
$L_\varepsilon$ and the notation $\nabla_{\text{tan}}$ represents the tangential derivatives.
The estimate $\eqref{pri:0.8}$ is a crucial ingredient in the use of layer potentials to solve
$L^2$ boundary value problems in Lipschitz domains. This step in fact includes two aspects:
\begin{itemize}
\item in the case of small scales, the estimate $\eqref{pri:0.8}$ is
based upon the so-called Rellich identity;
\item for the large cases, it is due to a localization technique coupled with the convergence rate
$O(\varepsilon^{1/2})$ in the $H^1(\Omega)$-norm.
\end{itemize}

Step 3. The continuity arguments together with the Fredholm operator theory
were applied to show the invertibility of the trace operators.

\begin{remark}
\emph{In fact, there are two ways to derive the estimate $\eqref{pri:0.8}$ for the large scale.
The method stated in Step 2 has been fully developed in \cite{GZS1}, originally suggested in \cite{SZW12},
and we will also employ it to develop a similar type estimate (see Section 5). Another way is due to a basic insight that
the difference $Q(u)(x^\prime,x_d) = u(x^\prime,x_d+1)-u(x^\prime,x_d)$ is a solution whenever
$u$ is a solution on account of the periodicity of $A$, which may be found in \cite{SZW24,SZW25}.
Also, we mention that this observation is meaningful in a nonperiodic setting.}
\end{remark}

\begin{remark}
\emph{In terms of the homogeneous property of $L_\varepsilon$, one may obtain
$\Gamma_\varepsilon(x,y) = \varepsilon^{2-d}\Gamma_1(x/\varepsilon,y/\varepsilon)$ for any
$x,y\in\mathbb{R}^d$ with $x\not=y$. Due to this property,
the $L^p$ bounds on singular
integral operators related to $\Gamma_\varepsilon(x,y)$ may be reduced
to the scale $\varepsilon = 1$ by rescaling arguments. In this sense, the large scale means $r>1$ while
the small scale means $0<r\leq 1$ in the paper \cite{SZW24,SZW25}, and it is a little different from ours.}
\end{remark}

Compared to Kenig and Shen's work in \cite{SZW24}, we confront with the following main difficulties:
\begin{itemize}
  \item although a series of uniform interior estimates has been developed in \cite{QXS1},
  we have not established the scaling-invariant estimates for $\mathcal{L}_\varepsilon$ yet,
  which means there is no evidence that the hypotheses $(\text{H}_1)$ and $(\text{H}_1)$ are clearly true under
  some suitable conditions. However, the size estimates for the fundamental solution
  $\mathbf{\Gamma}_\varepsilon(x,y)$ play an essential role in the layer potential method;
  \item due to the nonhomogeneous property of $\mathcal{L}_\varepsilon$, we can not account on the property
  \begin{equation}\label{eq:0.3}
  \mathbf{\Gamma}_\varepsilon(x,y) = r^{2-d}\mathbf{\Gamma}_{\varepsilon/r}(x/r,y/r)
  \qquad \forall x,y\in\mathbb{R}^d\text{~with~}x\not=y \text{~and~}r>0,
  \end{equation}
  which will bring some difficulties at least in a technical point of view;
  \item the lower order terms in homogenized operator $\mathcal{L}_0$ may be regarded as a
  compact perturbation, while the counterparts of $\mathcal{L}_\varepsilon$
  can not be handled in the same way. Meanwhile, how to approximate the fundamental
  solution $\mathbf{\Gamma}_\varepsilon(x,y)$ in an effective way is unknown, at least, to us.
\end{itemize}

According to the difficulties above,
we try to introduce our strategies and the source of the ideas.

The first difficulty has been overcome by Theorems $\ref{thm:0.1}$,$\ref{thm:0.2}$ and $\ref{thm:2.1}$,
in which the scaling-invariant estimates are most crucial, and we will outline the main ideas therein.

An important finding is that the constant $\lambda_0$ should be given in the form of $\eqref{a:5}$,
which guarantees that the operator $\mathcal{L}_\varepsilon$ will be ``positive'' in the case of $\lambda\geq\lambda_0$,
and this property will benefit the so-called Caccioppoli's inequality
\begin{equation*}
\mu \int_{B(0,R)} |\nabla u_\varepsilon|^2 dx
+\lambda\int_{B(0,R)} |u_\varepsilon|^2 dx
\leq \frac{C_\mu}{R^2}\int_{B(0,2R)} |u_\varepsilon|^2 dx,
\end{equation*}
where $\mathcal{L}_\varepsilon(u_\varepsilon) = 0$ in $B(0,2R)$ with any $R>0$.
We remark that the constant $C_\mu$ depends only on $\mu,m,d$. This is the first step
to derive the scaling-invariant estimate. Then, another key observation is
that the bootstrap'' process is based upon a limited iteration, and the following quantities
\begin{equation*}
 \|V\|_{L^\infty(\mathbb{R}^d)}/\sqrt{\lambda},
 \quad \|B\|_{L^\infty(\mathbb{R}^d)}/\sqrt{\lambda}, \quad\text{and}\quad
\|c\|_{L^\infty(\mathbb{R}^d)}/\lambda
\end{equation*}
are scaling-invariant. Hence, the task is reduced to track the constant,
and we repeat using Caccioppoli's inequality
to send the factor ``$\lambda^{-\frac{1}{2}}$'' to the coefficients $V,B$ and $c$.
The details may be found in the proof Theorem $\ref{thm:0.1}$.
Generally speaking, the main ideas in Theorem $\ref{thm:0.1}$ is using
Caccioppoli's inequality stated above (or see Lemma $\ref{lemma:2.2}$) to improve
the related estimates in \cite{QXS1} such that
\begin{equation*}
\begin{aligned}
C(\mu,\kappa,\lambda,m,d,p\|A\|_{\text{VMO}}) &\rightsquigarrow \eqref{KEY:2}, \\
C(\mu,\kappa,\lambda,m,d,\tau) &\rightsquigarrow \eqref{KEY:3}.
\end{aligned}
\end{equation*}
We end this paragraph by mention that the ideas used here are inspired by Shen's work in \cite{SZW26} and its references therein.

The second difficult mainly appears in Theorem $\ref{thm:0.4}$ and in layer potential methods,
since the property $\eqref{eq:0.3}$ is not true for $\mathcal{L}_\varepsilon$, it is impossible to
rescale the related estimates to the case $\varepsilon = 1$ as in \cite{MAFHL3,SZW24}.
Thus we have to consider the estimates such as $L^p$ bounds for singular
integral operators on Lipschitz surfaces in terms of small and
large scales, separately. Here the small scale means $0<r\leq \varepsilon$ while the large scale means
$r>\varepsilon$. In fact, the results in Theorem $\ref{thm:0.4}$ gave us the asymptotic expansions of
fundamental solutions in large scales, which will benefit the later discussion
in the case of $r>\varepsilon$. The rescaling arguments merely worked for small scales, which means
showing the related estimates in small scales is equivalent to proving it in the case of $\varepsilon=1$.
The reader may clearly find this point from the relationship between Lemma $\ref{lemma:4.7}$ and Lemma
$\ref{lemma:5.2}$. We comment that although no tough difficulty
arises from lacking the property $\eqref{eq:0.3}$, it obviously increases the cases we have to handle,
and the main ideas in the proof of Theorem $\ref{thm:0.4}$ are still from \cite{MAFHL3}.

To dissolve the third one, motivated by Shen's work in \cite{SZW23},
the main ideas are to make a comparison with the homogeneous operator
$L_\varepsilon$, which have also been applied to the homogenized problems in Section $\ref{sec:3}$. Thus, we establish
three comparing lemmas (see Lemmas $\ref{lemma:3.2}$, $\ref{lemma:4.7}$ and $\ref{lemma:5.2}$).
To do so, an equality on the difference between two fundamental
solutions has been shown in Lemma $\ref{lemma:2.3.1}$. Observing the equality, it is not hard to see
that the difficulties produced by the lower order terms and lower regularity assumptions.
In fact, the comparing lemmas can not be directly achieved by
using the equality $\eqref{eq:5.1}$ in Lemma $\ref{lemma:2.3.1}$. We have to seek
the ``freezing coefficient'' arguments as an transitional stage. To our surprising,
not only leading term but also the lower order terms are inevitably frozen
at the same point to make the proof reasonable. It is very different from the classical cases such as
the well known Schauder estimates, in which it suffices to freeze the leading term.
Meanwhile, a similar case happen to using the continuity method to show the invertibility
of the trace operators. That is one reason why we use a section to discuss homogenized systems in detail.

Concerning homogenized systems, we point out that the nontangential maximal function estimates do not
depend on layer potentials. Instead, by using the results obtained in \cite{GaoW},
the radial maximal function coupled with a interior estimate will
lead to the desired estimates (see Theorem $\ref{thm:3.2}$, $\ref{thm:3.4}$), which releases us from
the Rellich identity compared to the case of variable coefficients. In this sense, there provides a new
way to derive this kind of estimates, and its original thinking belongs to \cite{SZW2}.
On the other hand, let $\mathcal{K}_{\widehat{A}}, \mathcal{K}_{0}$ be the trace operators related to
$L_0$ and $\mathcal{L}_0$, respectively (see Lemma $\ref{lemma:3.6}$). One may verify that
$\mathcal{K}_{0} - \mathcal{K}_{\widehat{A}}$ will be a compact operator on
$L^p(\partial\Omega;\mathbb{R}^m)$ by the Sobolev embedding theorem,
and this together with nontangential maximal function estimates will show
the well-posedness of $L^2$ Dirichlet, Neumann and regular problems for the
elliptic systems with constant coefficients. We mention that the methods used here may be extended to
$L^p$ boundary value problems without any real difficulty.

\begin{remark}
\emph{In terms of pseudodifferential operators, the research on boundary
value problems of nonhomogeneous elliptic systems is not new.
We refer the reader to \cite{HMT,DMMMMT,MMMT} and their references therein for more details.
The method of freezing coefficients used here was originally developed in \cite{MMMT} for Laplace-Beltrami
operator on Lipschitz subdomains of Riemannian manifolds.
Compared to theirs, obvious differences are that the argument approached here
requires lower regularity assumption on coefficients, and the expression of $\mathcal{L}_\varepsilon$
appears more complicated.}
\end{remark}

\begin{remark}
\emph{In the case of $m=1$, Kenig and Shen established the solvability for $L^p$ Dirichlet
problem for $L_\varepsilon = -\text{div}(A(x/\varepsilon)\nabla)$ with $2-\delta<p<\infty$, originally
obtained by Dahlberg, in \cite{SZW25}. Also, the $L^p$ Neumann and regularity problems were solved
for the sharp range $1<p<2+\delta$, in which the coefficient $A$ satisfies elliptic, symmetric, periodic
and a certain square-Dini conditions. We refer the reader to \cite{K,KP,KP1} for the related fields on
$L^p$ boundary value problems with minimal smoothness assumptions.
For $m>1$ and $\partial\Omega\in C^{1,\tau}$, the well-posedness of $L^p$ Dirichlet, Neumann and regular
problems obtained by Kenig, Lin and Shen in \cite{SZW3}
due to a real method coupled with reverse H\"older's
inequality and Neumann functions (see for example \cite{SZW27}). We will study this topic in a separate work, and
how to extend the related reverse H\"older's
inequality to the case of $\partial\Omega\in C^1$ is still open.}
\end{remark}

\begin{remark}
\emph{As we mentioned before, the size estimates $(\text{H}_1)$ and $(\text{H}_2)$ for fundamental solutions
play an crucial role in layer potential methods, and Theorem $\ref{thm:2.1}$ proved that they are
equivalent to the scaling-invariant estimates $\eqref{pri:0.2}$ and $\eqref{pri:0.3}$, respectively.
In fact, this result has partially been pointed out by Hofmann and Kim in \cite{HS}, and
we made a few improvements to their related proofs, inspired by Avellaneda, Lin \cite{MAFHL}. Also,
to obtain a sharp estimate,
the decay estimates of Green functions for $\mathcal{L}_\varepsilon$ have already been developed
in \cite{QXS}, in which the constant additionally depends on $\Omega$. We finally mention that
the paper \cite{DHM} recently studied a very similar model in the classical case, in which the authors
gave a scaling-invariant estimate through the Giorgi-Nash-Moser theory.}
\end{remark}

\begin{remark}
\emph{The asymptotic expansions $\eqref{pri:0.7}$ is not sharp, but enough to approximate
$\nabla\mathbf{\Gamma}_\varepsilon(x,y)$ in the large scales. Besides,
following the ideas
in \cite{MAFHL3}, it is not hard to derive the $L^p$ boundedness of operator
$(\partial/\partial x_i)(-\mathcal{L})^{-1}(\partial/\partial x_j)$,
where the notation $\mathcal{L}$ denotes $\mathcal{L}_1$  by ignoring the the subscript.
We also mention that the results in Theorems $\ref{thm:0.1},\ref{thm:0.2},\ref{thm:2.1}$ and $\ref{thm:0.4}$ are
independent of the symmetry condition $A=A^*$.
From the point of view of convergence rates, the estimates
$\eqref{pri:1.1},\eqref{pri:1.4}, \eqref{pri:1.2}$ and $\eqref{pri:1.3}$
are much harder than the Lipschitz estimate $\eqref{pri:0.4}$ in the quantitative estimates.
Finally,
we remark that the well-posedness of $L^p$ boundary value problems are active topics and, without
attempting to be exhaustive, we refer the reader to
\cite{MAFHL5,BMSHSTA,CMM,DKV,Fabes,GaoW,GZS1,GX,HMT,JK,K,KP,KP1,DMMMMT,MMMT,S4,SZW23,SZW27,SZW26,V}
and the references therein for more details.}
\end{remark}

\subsection{Outline of the paper}
In the next section, we mainly present the proofs for Theorems $\ref{thm:0.1}$, $\ref{thm:0.2}$ and $\ref{thm:0.4}$, and
the existence and size estimates of fundamental solutions of $\mathcal{L}_\varepsilon$, as well as
some basic definitions, lemmas, theorems. Section $\ref{sec:3}$ is devoted to show the well-posedness
of $L^2$ boundary value problems for the nonhomogeneous elliptic systems with constant coefficients.
Although Section $\ref{sec:4}$ merely hand the related problems in the small scales,
it is the core of the paper, and the last section will give the whole proofs for Theorems
$\ref{thm:1.1}$, $\ref{thm:1.2}$ and $\ref{thm:1.3}$.

\section{Preliminaries}

\begin{definition}\label{def:2.1}
Let $\mathcal{L}_\varepsilon^*$ be the adjoint operator of $\mathcal{L}_\varepsilon$,
which is given by
\begin{equation*}
\mathcal{L}_\varepsilon^* =
-\text{div}\big[A^*(x/\varepsilon)\nabla +V^*(x/\varepsilon)\big]
+ B^*(x/\varepsilon)\nabla +c^*(x/\varepsilon) + \lambda I
\end{equation*}
where
\begin{equation*}
A^* = (A_{ji})^t = (a_{ji}^{\beta\alpha}),
\quad V^* = B^t = (B_i^{\beta\alpha}),
\quad B^* = V^t = (V_i^{\beta\alpha}),
\quad c^* = c^t =(c^{\beta\alpha}).
\end{equation*}
Then we define
the bilinear forms $\mathrm{B}_{\mathcal{L}_\varepsilon;\mathbb{R}^d}[\cdot,\cdot]$ and
$\mathrm{B}_{\mathcal{L}_\varepsilon^*;\mathbb{R}^d}[\cdot,\cdot]$,
associated with $\mathcal{L}_\varepsilon$ and $\mathcal{L}_\varepsilon^*$ as
\begin{equation*}
\begin{aligned}
&\mathrm{B}_{\mathcal{L}_\varepsilon;\mathbb{R}^d}[u,v]
=\int_{\mathbb{R}^d}\bigg\{\Big[a_{ij}^{\alpha\beta}\Big(\frac{x}{\varepsilon}\Big)
\frac{\partial u^\beta}{\partial x_j}
+ V_i^{\alpha\beta}\Big(\frac{x}{\varepsilon}\Big)u^\beta\Big]
\frac{\partial v^\alpha}{\partial x_i}
+\Big[B_i^{\alpha\beta}\Big(\frac{x}{\varepsilon}\Big)
\frac{\partial u^\beta}{\partial x_i} +
c^{\alpha\beta}\Big(\frac{x}{\varepsilon}\Big)u^\beta + \lambda u^\alpha\Big]v^\alpha\bigg\} dx,\\
& \mathrm{B}_{\mathcal{L}_\varepsilon^*;\mathbb{R}^d}[v,u]
=\int_{\mathbb{R}^d}\bigg\{\Big[a_{ji}^{\beta\alpha}\Big(\frac{x}{\varepsilon}\Big)
\frac{\partial v^\alpha}{\partial x_j}
+ B_i^{\beta\alpha}\Big(\frac{x}{\varepsilon}\Big)v^\alpha\Big]
\frac{\partial u^\beta}{\partial x_i}
+\Big[V_i^{\beta\alpha}\Big(\frac{x}{\varepsilon}\Big)
\frac{\partial v^\alpha}{\partial x_i} +
c^{\beta\alpha}\Big(\frac{x}{\varepsilon}\Big)v^\alpha + \lambda v^\beta\Big]u^\beta\bigg\} dx
\end{aligned}
\end{equation*}
for any $u,v\in H^1(\mathbb{R}^d;\mathbb{R}^m)$, respectively.
\end{definition}

Let $X,Y:\mathbb{R}^d\to \mathbb{R}^m$ be vector-valued functions, and
$E:\mathbb{R}^d\to \mathbb{R}^{m\times m}$ be a matrix-valued function,
and due to the fact that
\begin{equation*}
 \int_{\mathbb{R}^d} X E Y^t dx =
 \int_{\mathbb{R}^d} Y E^t X^t dx,
\end{equation*}
it is not hard to see that
\begin{equation}
\mathrm{B}_{\mathcal{L}_\varepsilon;\mathbb{R}^d}[u,v]
= \mathrm{B}_{\mathcal{L}_\varepsilon^*;\mathbb{R}^d}[v,u]
\end{equation}
for any $u,v\in H^1(\mathbb{R}^d;\mathbb{R}^m)$.

\begin{lemma}[Energy estimates]\label{lemma:2.1}
Suppose that the coefficients of $\mathcal{L}_\varepsilon$ satisfy
the conditions $\eqref{a:1}$ and $\eqref{a:3}$. Then we have the boundedness property
\begin{equation}
\Big|\mathrm{B}_{\mathcal{L}_\varepsilon;\mathbb{R}^d}
[u,v]\Big|
\leq C\|u\|_{H^1(\mathbb{R}^d)}\|v\|_{H^1(\mathbb{R}^d)}
\qquad \forall u,v\in H^1(\mathbb{R}^d;\mathbb{R}^m),
\end{equation}
where $C$ depends on $\mu,\kappa,\lambda,m,d$. Moreover,
there exists a constant $\lambda_0$ given in the form of $\eqref{KEY:1}$
such that for any $\lambda\geq\lambda_0$ there holds the coercivity property
\begin{equation}\label{pri:2.0.2}
\mu\|\nabla u\|_{L^2(\mathbb{R}^d)}^2
+\lambda\|u\|_{L^2(\mathbb{R}^d)}^2
\leq 2\mathrm{B}_{\mathcal{L}_\varepsilon;\mathbb{R}^d}[u,u]
\end{equation}
for any $u\in H^1(\mathbb{R}^d;\mathbb{R}^m)$.
\end{lemma}

\begin{thm}[Weak solutions]\label{thm:2.0}
Suppose that the coefficients of $\mathcal{L}_\varepsilon$ satisfy
the conditions $\eqref{a:1}$, $\eqref{a:3}$.
If $\lambda\geq\max\{\lambda_0,\mu\}$,
then for any $F\in H^{-1}(\mathbb{R}^d;\mathbb{R}^m)$ there exists
a unique weak solution $u_\varepsilon\in H^{1}(\mathbb{R}^d;\mathbb{R}^m)$ to
$\mathcal{L}_\varepsilon(u_\varepsilon) = F$ in $\mathbb{R}^d$, and
we have the uniform estimate
\begin{equation}\label{pri:2.0.1}
\|u_\varepsilon\|_{H^{1}(\mathbb{R}^d)}
\leq C\|F\|_{H^{-1}(\mathbb{R}^d)},
\end{equation}
where $C$ depends only on $\mu,d,m$.
\end{thm}

\begin{remark}
\emph{The proof of Lemma $\ref{lemma:2.1}$ is standard and quite similar to \cite[Lemma 3.1]{QXS1},
so we do not reproduce here, and
it is well known that Theorem $\ref{thm:2.0}$ follows from Lemma $\ref{lemma:2.1}$ straightforwardly.}
\end{remark}

\subsection{Correctors}\label{subsec:2.1}
Define the correctors $\chi_k = (\chi_{k}^{\alpha\beta})$ with $0\leq k\leq d$, related to $\mathcal{L}_\varepsilon$ as follows:
\begin{equation}
\left\{ \begin{aligned}
 &L_1(\chi_k) = \text{div}(V)  \quad \text{in}~ \mathbb{R}^d, \\
 &\chi_k\in H^1_{per}(Y;\mathbb{R}^{m^2})~~\text{and}~\int_Y\chi_k dy = 0
\end{aligned}
\right.
\end{equation}
for $k=0$, and
\begin{equation}
 \left\{ \begin{aligned}
  &L_1(\chi_k^\beta + P_k^\beta) = 0 \quad \text{in}~ \mathbb{R}^d, \\
  &\chi_k^\beta \in H^1_{per}(Y;\mathbb{R}^m)~~\text{and}~\int_Y\chi_k^\beta dy = 0
 \end{aligned}
 \right.
\end{equation}
for $1\leq k\leq d$, where $P_k^\beta = x_k(0,\cdots,1,\cdots,0)$ with 1 in the
$\beta^{\text{th}}$ position, $Y = (0,1]^d \cong \mathbb{R}^d/\mathbb{Z}^d$, and $H^1_{per}(Y;\mathbb{R}^m)$ denotes the closure
of $C^\infty_{per}(Y;\mathbb{R}^m)$ in $H^1(Y;\mathbb{R}^m)$.
Note that $C^\infty_{per}(Y;\mathbb{R}^m)$ is the subset of $C^\infty(Y;\mathbb{R}^m)$, which collects all $Y$-periodic vector-valued functions. By asymptotic expansion arguments (see \cite[pp.103]{ABJLGP} or \cite[pp.31]{VSO}), we obtain the homogenized operator
\begin{equation}\label{eq:2.1.2}
 \mathcal{L}_0 = -\text{div}(\widehat{A}\nabla+ \widehat{V}) + \widehat{B}\nabla + \widehat{c} + \lambda I,
\end{equation}
where $\widehat{A} = (\widehat{a}_{ij}^{\alpha\beta})$, $\widehat{V}=(\widehat{V}_i^{\alpha\beta})$,
$\widehat{B} = (\widehat{B}_i^{\alpha\beta})$ and $\widehat{c}= (\widehat{c}^{\alpha\beta})$ are given by
\begin{equation}\label{eq:2.1.1}
\begin{aligned}
\widehat{a}_{ij}^{\alpha\beta} = \int_Y \big[a_{ij}^{\alpha\beta} + a_{ik}^{\alpha\gamma}\frac{\partial\chi_j^{\gamma\beta}}{\partial y_k}\big] dy, \qquad
\widehat{V}_i^{\alpha\beta} = \int_Y \big[V_i^{\alpha\beta} + a_{ij}^{\alpha\gamma}\frac{\partial\chi_0^{\gamma\beta}}{\partial y_j}\big] dy, \\
\widehat{B}_i^{\alpha\beta} = \int_Y \big[B_i^{\alpha\beta} + B_j^{\alpha\gamma}\frac{\partial\chi_i^{\gamma\beta}}{\partial y_j}\big] dy, \qquad
\widehat{c}^{\alpha\beta} = \int_Y \big[c^{\alpha\beta} + B_i^{\alpha\gamma}\frac{\partial\chi_0^{\gamma\beta}}{\partial y_i}\big] dy.
\end{aligned}
\end{equation}

\begin{lemma}
Suppose that the coefficients of $\mathcal{L}_\varepsilon$ satisfy the conditions
$\eqref{a:1}$, $\eqref{a:2}$ and $\eqref{a:3}$. Then there holds
\begin{equation}\label{pri:2.1.1}
 \max_{0\leq k\leq d}\|\nabla\chi_k\|_{L^2(Y)}\leq C(\mu,\kappa,d,m).
\end{equation}
Moreover, if we further assume $A=A^*$, then we have
\begin{equation}\label{a:5}
\left\{\begin{aligned}
&\widehat{A}=\widehat{A}^*,\\
&\mu |\xi|^2 \leq \xi^T\widehat{A}\xi\leq \mu^{-1} |\xi|^2 \quad\forall\xi\in\mathbb{R}^{md}, \\
&\max\big\{|\widehat{V}|,~|\widehat{B}|,~|\widehat{c}|\big\}\leq C(\mu,\kappa,d,m).
\end{aligned}\right.
\end{equation}
\end{lemma}

\begin{proof}
It is not hard to see that the estimate $\eqref{pri:2.1.1}$ is based upon
\begin{equation*}
\frac{\mu}{2}\int_Y|\nabla\chi_0|^2 dy \leq \frac{1}{\mu}\int_Y|V|^2dy
\quad \text{and} \quad
\frac{\mu}{2}\int_Y|\nabla\chi_k|^2 dy \leq \frac{1}{\mu}\int_Y|A|^2dy
\end{equation*}
for $k=1,\cdots,d$.
The first and second lines in $\eqref{a:5}$ may be found in \cite[pp.23-24]{ABJLGP}, and
the last one follows from $\eqref{pri:2.1.1}$ and \eqref{eq:2.1.1}.
\end{proof}

In view of Theorem $\ref{thm:2.0}$ we similarly obtain the following theorem.

\begin{thm}
Assume the coefficients of $\mathcal{L}_0$ satisfy $\eqref{a:5}$, then there exists
$\widehat{\lambda}$ depending on $\mu,\kappa,m,d$ such that the equation
$\mathcal{L}_0(u_0) = F$ in $\mathbb{R}^d$ has an unique weak solution $u_0$ in
$H^1(\mathbb{R}^d;\mathbb{R}^m)$,
whenever $\lambda\geq\max\{\widehat{\lambda},\mu\}$.
Moreover, we may have
\begin{equation}\label{pri:2.2.2}
\lambda\|u_0\|_{H^1(\mathbb{R}^d)}^2
\leq 2\mathrm{B}_{\mathcal{L}_0;\mathbb{R}^d}[u_0,u_0].
\end{equation}
\end{thm}

\begin{remark}
\emph{Let $\Omega\subset\mathbb{R}^d$ be a bounded Lipschitz domain.
In fact, the estimate $\eqref{pri:2.2.2}$ will still hold for $\Omega$ and $\Omega_{-}$,
in which the notation $\Omega_{-}$ represents the exterior of $\Omega$.}
\end{remark}

\subsection{Scaling-invariant estimates}

\begin{lemma}[Caccioppoli's inequality]\label{lemma:2.2}
Let $B=B(x_0,R)\subset\mathbb{R}^d$ with any $R>0$.
Suppose that the coefficients of $\mathcal{L}_\varepsilon$
satisfy $\eqref{a:1}$ and $\eqref{a:3}$.
Assume that $u_\varepsilon\in H^1(4B;\mathbb{R}^m)$ is
a weak solution to $\mathcal{L}_\varepsilon(u_\varepsilon) = \emph{div}(f)+F$
in $4B$, where $f\in L^2(4B;\mathbb{R}^{md})$ and
$F\in L^q(4B;\mathbb{R}^m)$ with $q=2d/(d+2)$.
Then there exists a positive constant $\lambda_0$
such that for any $\lambda\geq \lambda_0$ there holds the following estimate
\begin{equation}\label{pri:2.6}
\sqrt{\mu}\Big(\dashint_B |\nabla u_\varepsilon|^2\Big)^{\frac{1}{2}}
+ \sqrt{\lambda} \Big(\dashint_B |u_\varepsilon|^2 \Big)^{\frac{1}{2}}
\leq \frac{C_\mu}{R}\Big(\dashint_{2B} |u_\varepsilon|^2 \Big)^{\frac{1}{2}}
+ C\Big(\dashint_{2B} |f|^2 \Big)^{\frac{1}{2}}
+CR\Big(\dashint_{2B} |F|^q \Big)^{\frac{1}{q}},
\end{equation}
where $C_\mu$ depends only on $\mu,m,d$. In particular, if the source terms $f$ and $F$ vanish, then
for any integer $k>0$, there exists $C_k$ depending only on $C_\mu$ and $k$ such that
\begin{equation}\label{pri:2.1}
\mu \int_{B} |\nabla u_\varepsilon|^2 dx
+\lambda\int_{B} |u_\varepsilon|^2 dx
\leq \frac{C_k}{(1+\lambda R^2)^kR^2}\int_{2B} |u_\varepsilon|^2 dx.
\end{equation}
\end{lemma}

\begin{proof}
In fact, the stated estimate $\eqref{pri:2.6}$ has already been established in
\cite[Lemma 2.7]{QXS}, in which we proved
\begin{equation*}
\begin{aligned}
\frac{\mu}{2}\int_{\mathbb{R}^d} \phi^2|\nabla u_\varepsilon|^2 dx
&+  (\lambda - C^\prime)\int_{\mathbb{R}^d} \phi^2 |u_\varepsilon|^2 dx \\
&\leq  C\int_{\mathbb{R}^d} |\nabla\phi|^2|u_\varepsilon|^2 dx
+C\int_{\mathbb{R}^d} \phi^2|f|^2 dx + \int_{\mathbb{R}^d} \phi^2|F||u_\varepsilon| dx,
\end{aligned}
\end{equation*}
and it is not hard to verify that $C^\prime$ may be given in the form of $\eqref{KEY:1}$,
where $\phi\in C^1_0(\Omega)$ is a cut-off function satisfying
$\phi = 1$ in $B$, $\phi = 0$ outside $2B$, and $|\nabla\phi| \leq 2/r$.
Here we may choose $\lambda_0 = C^\prime$, and the reminder of the proof
is standard.

We now offer a proof for the estimate $\eqref{pri:2.1}$.
By translation we may assume $x_0=0$. For any integer $k>0$ and $x\in B(0,R)$ let $r=R/2^{k}$.
Then it follows from $\eqref{pri:2.6}$ that
\begin{equation*}
\mu\int_{B(x,r)} |\nabla u_\varepsilon|^2 dx
+\lambda\int_{B(x,r)} |u_\varepsilon|^2 dx
\leq \frac{C_\mu}{r^2}\int_{B(x,2r)} |u_\varepsilon|^2 dx.
\end{equation*}
By iteration, this implies that
\begin{equation*}
 \int_{B(x,r)}|u_\varepsilon|^2 dx
 \leq \frac{C_k}{\lambda^k r^{2k}} \int_{B(x,2^kr)} |u_\varepsilon|^2 dx,
\end{equation*}
where $C_k$ depends only on $C_\mu$ and $k$.
Recalling $r=R/2^k$ we have
\begin{equation*}
\int_{B(x,R/2^{k})}|u_\varepsilon|^2 dx
\leq \frac{C_k}{(\lambda R^2)^k}\int_{B(x,R)}|u_\varepsilon|^2dx
\leq \frac{C_k}{(\lambda R^2)^k}\int_{B(0,2R)}|u_\varepsilon|^2dx
\end{equation*}
for any $x\in B(0,R)$. Thus a covering argument leads to
\begin{equation*}
\int_{B(0,R)}|u_\varepsilon|^2 dx
\leq \frac{C_k}{(\lambda R^2)^k}\int_{B(0,2R)}|u_\varepsilon|^2dx.
\end{equation*}
which gives a family of inequalities
\begin{equation*}
\int_{B(x^\prime,R^\prime)}|u_\varepsilon|^2 dx
\leq \frac{C_k}{(\lambda R^2)^k}\int_{B(x^\prime,2R^\prime)}|u_\varepsilon|^2dx.
\end{equation*}
with $R^\prime = R/2$ and $x^\prime\in B(0,(3/2)R)$. Repeat the covering argument and
we consequently have
\begin{equation}\label{f:2.1}
\int_{B(0,(3/2)R)}|u_\varepsilon|^2 dx
\leq \frac{C_k}{(\lambda R^2)^k}\int_{B(0,2R)}|u_\varepsilon|^2dx.
\end{equation}

Using the estimate $\eqref{pri:2.6}$ again, we may obtain
\begin{equation*}
\begin{aligned}
\mu\int_{B(0,R)} |\nabla u_\varepsilon|^2 dx
+\lambda\int_{B(0,R)} |u_\varepsilon|^2 dx
&\leq \frac{C_\mu}{R^2}\int_{B(0,(3/2)R)}|u_\varepsilon|^2dx,\\
\mu\int_{B(0,R)} |\nabla u_\varepsilon|^2 dx
+\lambda\int_{B(0,R)} |u_\varepsilon|^2 dx
&\leq \frac{C_k}{(\lambda R^2)^kR^2}\int_{B(0,2R)}|u_\varepsilon|^2dx,
\end{aligned}
\end{equation*}
where we use the estimate $\eqref{f:2.1}$ in the last inequality, and this gives the stated estimate
$\eqref{pri:2.1}$. We have completed the proof.
\end{proof}


\noindent\textbf{Proof of Theorem $\ref{thm:0.1}$.}
In fact, the stated estimates $\eqref{pri:0.1}$, $\eqref{pri:0.2}$ and $\eqref{pri:0.3}$ have
been established in \cite[Theorem 3.3,~Corollary 3.5,~Theorem 4.4]{QXS}, respectively. However,
we do not seek the constants there to be scaling-invariant since it is sufficient to
establish the corresponding global ones as the interior parts. So,
the main purpose of the theorem is to figure out the scaling-invariant constants
$\eqref{KEY:2}$ and $\eqref{KEY:3}$.

Step 1. Show $C_p$ to be the form of $\eqref{KEY:2}$ in $W^{1,p}$ estimates $\eqref{pri:0.1}$.
Recall the result of \cite[Theorem 3.3]{QXS}, in which we proved
\begin{equation}\label{f:0.1}
\|u\|_{W^{1,p}(B/2)} \leq C\|u\|_{L^{2}(B)},\qquad p\geq 2,
\end{equation}
and the constant $C$ may be given by
\begin{equation}\label{f:0.2}
C=C_L(\mu,m,d,p,\|A\|_{\text{VMO}})\Big(\underbrace{\|A\|_{L^\infty(\mathbb{R}^d)}}_X
+\underbrace{\|V\|_{L^\infty(\mathbb{R}^d)}+\|B\|_{L^\infty(\mathbb{R}^d)}}_Y
+\underbrace{\|c\|_{L^\infty(\mathbb{R}^d)}+\lambda}_Z\Big)^{k_0}
\end{equation}
where $k_0 = [\frac{d}{2}]+1$ and the notation $[\frac{d}{2}]$ denotes the integer part of
$\frac{d}{2}$, which is due the so-called ``bootstrap'' process.
Also, $C_L$ came from the counterpart of $W^{1,p}$ estimates related to
the leading term
$L = -\text{div}(A(x/\varepsilon)\nabla)$.
In view of $\eqref{pri:2.1}$ we have
\begin{equation*}
 \|u_\varepsilon\|_{L^2(B)}\leq C_k\lambda^{-k/2}.
\end{equation*}
with an integer $k\geq 0$ and, it will be given in the later computation. This together with
$\eqref{f:0.1}$ and $\eqref{f:0.2}$ leads to
\begin{equation*}
\|u_\varepsilon\|_{W^{1,p}(B/2)}
\leq C_L(X+Y+Z)^{k_0}\lambda^{-k/2}
\leq C_L(X+Y/\sqrt{\lambda}+Z/\lambda)^{k_0}:=C_p
\end{equation*}
(by mention that we do not distinguish the constants
if the difference depends only on $\mu,m,d$),
where the integer $k$ may be chosen as $0,k_0,2k_0$ in the second inequality,
according to the following fact
\begin{equation*}
X^{k_0}+Y^{k_0}+Z^{k_0}\leq (X+Y+Z)^{k_0}\leq 2^{k_0-1}(X^{k_0}+Y^{k_0}+Z^{k_0}).
\end{equation*}
Then the estimate $\eqref{pri:0.2}$ immediately follows from the Sobolev embedding theorem with
$\sigma = 1-d/p$.

To show $\eqref{KEY:3}$ will be much involved. We have to first establish
$\eqref{KEY:3}$ for the classical Schauder estimates,
which means that in the case of $\varepsilon=1$ without periodicity condition
$\eqref{a:2}$, there holds
\begin{equation}\label{f:0.3}
\|u_1\|_{C^{1,\tau}(B/2)}\leq C_{\tau}
\end{equation}
and the constant $C_\tau$ in the form of $\eqref{KEY:3}$ is valid.
Note that in such the case the estimate $\eqref{f:0.3}$ is no longer
scaling-invariant for $R>1$, and in the next step we will handle
the estimate $\eqref{f:0.3}$.

Step 2. Recall the notation $\mathcal{L} = \mathcal{L}_1$ and $u$
denotes the corresponding $u_1$. Thus we rewrite $\mathcal{L}(u) = 0$ in $2B$ as
\begin{equation*}
 -\text{div}(A(x)\nabla u) = \text{div}(V(x)u) - B\nabla u - cu-\lambda I u
 \qquad \text{in}~ 2B.
\end{equation*}
It follows from the well-known interior Schauder theory
(see for example \cite[Theorem 5.19]{MGLM}) that
\begin{equation*}
\|u\|_{C^{1,\tau}(B/2)}
\leq C_{L,\tau}
\Big\{\|V\|_{C^{0,\tau}(B)}\|u\|_{C^{0,\tau}(B)}
+\|V+B+c+\lambda\|_{L^\infty(B)}\|u\|_{W^{1,p}(B)}\Big\}
\end{equation*}
where $p>d$, and $C_{L,\tau}=C_{L,\tau}\big(\mu,m,d,[A]_{C^{0,\tau}(\mathbb{R}^d)}\big)$ came from
the related homogeneous system. In terms of the estimates $\eqref{pri:0.1}$ and
$\eqref{pri:0.2}$ we then obtain
\begin{equation*}
\begin{aligned}
\|u\|_{C^{1,\tau}(B/2)}
&\leq C_{L,\tau}C_p
\Big\{\|V\|_{C^{0,\tau}(\mathbb{R}^d)}
+\|V+B+c+\lambda\|_{L^\infty(\mathbb{R}^d)}\Big\}\|u\|_{L^2(2B)} \\
&\leq C_{L,\tau}C_p
\Big\{\|V/\sqrt{\lambda}\|_{C^{0,\tau}(\mathbb{R}^d)}
+\|(V+B)/\sqrt{\lambda}
\|_{L^\infty(\mathbb{R}^d)}
+\|c/\lambda\|_{L^\infty(\mathbb{R}^d)}+1\Big\} :=C_{\tau},
\end{aligned}
\end{equation*}
in which we also employ Caccioppoli's inequality $\eqref{pri:2.1}$ in the last inequality.

Step 3. Before proceeding further, it is better illustrating the difficulties
and the source of ideas in the proof. Note that the following transformation as defined in
\cite[Theorem 4.4]{QXS}
\begin{equation}\label{T:1}
 u_\varepsilon = T(x,\varepsilon)v_\varepsilon
 = [I+\varepsilon\chi_0(x/\varepsilon)]v_\varepsilon
\end{equation}
is not scaling-invariant for $R>1$, since $\chi_0^R = R\chi_0$ in $RY=(0,R]^d$,
and it satisfies
\begin{equation}
 -\text{div}(A\nabla\chi_0^R) = \text{div}(RV) \quad\text{in}~ \mathbb{R}^d
 \qquad\text{and}\qquad\dashint_{RY} \chi_0^R = 0.
\end{equation}
Let $\widetilde{A}(x/\varepsilon) = A(x/\varepsilon)[I+\varepsilon\chi_0(x/\varepsilon)]$,
and $\varepsilon^\prime = \varepsilon/R$.
Although the scaled coefficients
\begin{equation}
\widetilde{A}(x/\varepsilon^\prime) = A(x/\varepsilon^\prime)\big[I+\varepsilon^\prime\chi_0^R(x/\varepsilon^\prime)\big]
\end{equation}
keep a similar pattern compared to $\widetilde{A}(x/\varepsilon)$,
the operators determined by them
are definitely not in the same type class of operators and,
the new one has changed both the periodicity and the bound of
$[\widetilde{A}]_{C^{0,\tau}(\mathbb{R}^d)}$.
Thus, the proof of \cite[Theorem 4.4]{QXS} can not be simply improved
by a rescaling argument when $R>1$.

Fortunately, the methods developed for the uniform global Lipschitz estimate
\cite[Theorem 1.3]{QXS} are still useful here, by which
we overcome the difficulties arising from the corresponding Dirichlet correctors without
the periodicity near a boundary.
By the transformation $\eqref{T:1}$, it is not hard
to have the following equation
\begin{equation}\label{eq:0.1}
-\text{div}(A(x/\varepsilon)\nabla v_\varepsilon)
=\text{div}(\widetilde{f}+\varepsilon A(x/\varepsilon)\chi_0(x/\varepsilon)\nabla v_\varepsilon
) + \widetilde{F} \qquad \text{in}\quad 4B,
\end{equation}
with
\begin{equation}\label{eq:0.2}
\begin{aligned}
\widetilde{f} &= \varepsilon V(x/\varepsilon)\chi_0(x/\varepsilon)v_\varepsilon \\
\widetilde{F}&=  A(x/\varepsilon)\nabla_y\chi_0\nabla v_\varepsilon
+ V(x/\varepsilon)\nabla v_\varepsilon - B(x/\varepsilon)\nabla u_\varepsilon
- (c(x/\varepsilon)+\lambda I)u_\varepsilon
\end{aligned}
\end{equation}
where $y=x/\varepsilon$. Thus, the problem is reduced to estimate the quantities
$\|v_\varepsilon\|_{C^{1,\sigma_1}(B)}$, $\|\widetilde{f}\|_{C^{0,\sigma_1}(B)}$,
$\|\widetilde{F}\|_{L^{\bar{p}}(B)}$ with $\bar{p}=d/(1-\sigma_1)$, as well as
$\|T^{-1}\|_{C^{0,\sigma_1}(B)}$,
where $\sigma_1\in(0,\tau]$ will be fixed later.

Step 4.
We now proceed to prove $\eqref{pri:0.3}$ with the constant
in the form of $\eqref{KEY:3}$. Applying the interior Lipschitz estimates
\cite[Lemma 4.3]{QXS} (originally due to \cite[Lemma 16]{MAFHL}) to
$\eqref{eq:0.1}$, we have
\begin{equation}\label{f:0.4}
\|\nabla v_\varepsilon\|_{L^\infty(B/2)}
\leq C_{L_\varepsilon,\tau}
\bigg\{
\varepsilon\|A(x/\varepsilon)\chi_0(x/\varepsilon)
\nabla v_\varepsilon\|_{C^{0,\sigma_1}(B)} + \|\widetilde{f}\|_{C^{0,\sigma_1}(B)}
+\|\widetilde{F}\|_{L^{\bar{p}}(B)}\bigg\}
\end{equation}
where we mention that $C_{L_\varepsilon,\tau} = C_{L,\tau}$ which
reveals that the constant is independent of $\varepsilon$.

To estimate the right-hand side
of $\eqref{f:0.4}$, we need to derive some properties for $T(x,\varepsilon)$,
and we have that $\|T\|_{L^\infty(B)}\in (1/2,3/2)$ and
$\|T^{-1}\|_{L^\infty(B)}\in (3/2,2)$ whenever $0<\varepsilon<\varepsilon_0$ with
$\varepsilon_0\in(0,1/2]$ being sufficiently small, as well as,
\begin{equation}\label{f:0.5}
\begin{aligned}
\|T\|_{C^{0,\sigma_1}(B)} +
\|T^{-1}\|_{C^{0,\sigma_1}(B)}
&\leq C_L(\mu,m,d,\bar{p},[A]_{\text{VMO}})\|V\|_{L^\infty(\mathbb{R}^d)}, \\
\|\chi_0\|_{C^{1,\sigma_1}(B)}
&\leq C_{L,\tau}(\mu,m,d,[A]_{C^{0,\tau}})\|V\|_{C^{0,\tau}(\mathbb{R}^d)}
\qquad \text{with}~\sigma_1\in(0,\tau],
\end{aligned}
\end{equation}
which are based upon the classical interior Schauder estimate
(see \cite[Theorem 5.19]{MGLM}) and $H^1$ theory.

We also need the following quantitative estimates
\begin{equation}\label{f:0.6}
\begin{aligned}
\|u_\varepsilon\|_{C^{0,\sigma_1}(B)}
+\|u_\varepsilon\|_{W^{1,\bar{p}}(B)}
&\leq C_{\bar{p}}C_k\lambda^{-k/2}\\
\|\nabla u_\varepsilon\|_{L^{\infty}(B)}
&\leq \varepsilon^{\sigma-1}C_{\tau}C_{\sigma}C_k\lambda^{-k/2} \\
\|u_\varepsilon\|_{C^{1,\sigma_1}(B)}
&\leq \varepsilon^{\sigma-\sigma_1-1}C_{\tau}C_{\sigma}C_k\lambda^{-k/2}
\qquad \text{with}~\sigma\in(0,1).
\end{aligned}
\end{equation}
We mention that the constant $C_\tau$ in the second, third lines of
$\eqref{f:0.6}$ actually comes from the classical Schauder estimate $\eqref{f:0.3}$,
and $C_{\sigma},C_k$ are from the uniform estimates $\eqref{pri:0.2}$ and $\eqref{pri:2.1}$,
respectively. This type of these two estimates is referred to as ``a nonuniform estimate'',
which has been originally developed in \cite[Lemma 4.10]{QXS}, and also been used to
avoid the Rellich type estimate applied to deriving
a sharp one (see \cite[Theorem 1.2]{GX}). We finally remark that the order of the constants
in fact have shown the outline of the arguments.

From the transformation $\eqref{T:1}$, it follows that
\begin{equation}\label{T:2}
\begin{aligned}
 \nabla u_\varepsilon
 &= T\nabla v_\varepsilon
 + \nabla_y\chi_0 v_\varepsilon \\
 \nabla v_\varepsilon
 &= T^{-1}\Big(\nabla u_\varepsilon
 - \nabla_y\chi_0 T^{-1}u_\varepsilon\Big),
\end{aligned}
\end{equation}
and this together with $\eqref{f:0.5}$ and $\eqref{f:0.6}$ gives
\begin{equation}\label{f:0.7}
\begin{aligned}
\|\nabla v_\varepsilon\|_{C^{0,\sigma_1}(B)}
&\leq \|T^{-1}\|_{C^{0,\sigma_1}(B)}\|\nabla u_\varepsilon\|_{C^{0,\sigma_1}(B)}
+ \|T^{-1}\|_{C^{0,\sigma_1}(B)}^2\|\nabla_y \chi_0\|_{C^{0,\sigma_1}(B)}
\|u_\varepsilon\|_{C^{0,\sigma_1}(B)} \\
&\leq \varepsilon^{\sigma-\sigma_1-1}C_LC_{\tau}C_{\sigma}C_k\lambda^{-k/2}
\|V\|_{L^\infty(\mathbb{R}^d)}
+ C_{L,\tau} C_{\bar{p}}C_k\lambda^{-k/2} \|V\|_{C^{0,\tau}(\mathbb{R}^d)}^3\\
&\leq \varepsilon^{\sigma-\sigma_1-1} C_{\tau}C_{\bar{p}}\lambda^{-(k-3)/2}
\|V/\sqrt{\lambda}\|_{C^{0,\tau}(\mathbb{R}^d)}^3
\end{aligned}
\end{equation}
The last inequality obeys the following conventions:
(1) We say $C_1 \leq C_2$ if $C_2$ partially depends on $C_1$;
(2) $C_1+C_2\leq C_1C_2$; (3) $C^k = C$ for any integer $k>0$.

Hence, by $\eqref{f:0.5}$ and $\eqref{f:0.7}$ we have
\begin{equation}\label{f:0.8}
\begin{aligned}
\|A(x/\varepsilon)\chi_0(x/\varepsilon)
\nabla v_\varepsilon\|_{C^{0,\sigma_1}(B)}
&\leq \|A(x/\varepsilon)\chi_0(x/\varepsilon)\|_{C^{0,\sigma_1}(B)}
\|\nabla v_\varepsilon\|_{C^{0,\sigma_1}(B)}\\
&\leq \varepsilon^{\sigma-2\sigma_1-1} C_{\tau}C_{\bar{p}}
\|V/\sqrt{\lambda}\|_{C^{0,\tau}(\mathbb{R}^d)}^5,
\end{aligned}
\end{equation}
and a routine computation will lead to
\begin{equation}\label{f:0.9}
\begin{aligned}
&\|\widetilde{f}\|_{C^{0,\sigma_1}(B)}
\leq C_{\bar{p}}\|V\sqrt{\lambda}\|^3_{C^{0,\tau}(\mathbb{R}^d)},\\
&\|\widetilde{F}\|_{L^{\bar{p}}(B)} \leq
 C_\tau C_{\bar{p}}
 \Big\{\|V/\sqrt{\lambda}\|_{C^{0,\tau}(\mathbb{R}^d)}^2
 +\|B/\sqrt{\lambda}\|_{L^\infty(\mathbb{R}^d)}
 + \|c/\sqrt{\lambda}\|_{L^\infty(\mathbb{R}^d)}+1\Big\}
\end{aligned}
\end{equation}
which are based upon $\eqref{T:1}$, $\eqref{T:2}$, $\eqref{f:0.5}$ and $\eqref{f:0.6}$
and we omit the details here. Then plugging the estimates $\eqref{f:0.8}$ and
$\eqref{f:0.9}$ back into $\eqref{f:0.4}$ we obtain
\begin{equation*}
\begin{aligned}
\|\nabla v_\varepsilon\|_{L^\infty(B/2)}
\leq \varepsilon^{\sigma-2\sigma_1}C_\tau C_{\bar{p}}
\leq C_\tau C_{\bar{p}}
\end{aligned}
\end{equation*}
for any $0<\varepsilon<\varepsilon_0$
by choosing $\sigma_1\leq \min\{\tau,\sigma/2\}$. This consequently implies
the stated estimate $\eqref{pri:0.3}$ with the constant $C_\tau C_{\bar{p}}$
being the form of $\eqref{KEY:3}$. We finally remark
the case $\varepsilon_0\leq \varepsilon\leq 1$ is trivial
due to the estimate $\eqref{f:0.3}$.
\qed

\begin{corollary}
Let $B=B(x,R)\subset\mathbb{R}^d$ with $R>0$, and
$C_\sigma$ be given in Theorem $\ref{thm:0.1}$. Suppose that
the coefficients of $\mathcal{L}_\varepsilon$ satisfy
$\eqref{a:1}$ and $\eqref{a:3}$ with $\lambda\geq\lambda_0$, and
$A\in \emph{VMO}(\mathbb{R}^d)$ satisfies $\eqref{a:2}$.
If $u_\varepsilon\in H^1(2B)$ is a weak solution to
$\mathcal{L}_\varepsilon(u_\varepsilon)=0$ in $2B$.
then for any $p>0$, there exists a constant
$C$ depending on $C_\sigma$ and $p$ such that
\begin{equation}\label{pri:2.2.1}
\|u_\varepsilon\|_{L^\infty(B)}
\leq C\Big(\dashint_{2B}|u_\varepsilon|^p\Big)^{1/p}.
\end{equation}
\end{corollary}

\noindent\textbf{Proof of Theorem $\ref{thm:0.2}$.}
We first to handle the estimate $\eqref{pri:0.4}$, and the other one
$\eqref{pri:0.5}$ could
be derived by a similar argument.
Let $\varphi\in C_0^1(\mathbb{R}^d)$ be a cut-off function
such that $\varphi = 1$ on $B(0,R)$ and $\varphi =0$ outside $B(0,2R)$ with
$|\nabla\varphi|\leq C/R$. Then we have
\begin{equation*}
\mathcal{L}_\varepsilon(\varphi u_\varepsilon)
= \text{div}(\varphi f - A(x/\varepsilon)\nabla\varphi u_\varepsilon)
- f\nabla\varphi + F\varphi - A(x/\varepsilon)\nabla\varphi\nabla u_\varepsilon
+\big(B(x/\varepsilon)-V(x/\varepsilon)\big)\nabla\varphi u_\varepsilon.
\end{equation*}

For any $x\in B(0,R/2)$,
it follows the expression $\eqref{pri:2.3.4}$ that
\begin{equation*}
\begin{aligned}
&u_\varepsilon(x)
= \underbrace{\int_{2B}\nabla_y\mathbf{\Gamma}_\varepsilon(x,y)
A(y/\varepsilon)\nabla\varphi(y) u_\varepsilon(y) dy}_{I_1}
- \underbrace{\int_{2B}\nabla_y\mathbf{\Gamma}_\varepsilon(x,y)
 f(y)\varphi(y) dy}_{I_2} \\
&+ \underbrace{\int_{2B}
\mathbf{\Gamma}_\varepsilon(x,y)\Big[F(y)\varphi(y) - f(y)\nabla\varphi(y)
-A(y/\varepsilon)\nabla\varphi(y)\nabla u_\varepsilon(y)
+\big(B(y/\varepsilon)-V(y/\varepsilon)\big)\nabla\varphi(y) u_\varepsilon(y)\Big]dy}_{I_3},
\end{aligned}
\end{equation*}
where we employ the integration by part in the first term of the right-hand side above.
Due to the decay estimate $(\text{H}_1)$, it is not hard to see that
\begin{equation*}
\begin{aligned}
|I_3|&\leq C_0\bigg(\int_{B(0,2R)}\frac{dy}{|x-y|^{(d-2)q^\prime}}\bigg)^{\frac{1}{q^\prime}}
\bigg(\int_{B(0,2R)}|F(y)|^qdy\bigg)^{\frac{1}{q}} + C_0R\dashint_{B(0,2R)}
\big(|f|+|\nabla u_\varepsilon|+|u_\varepsilon|\big)\\
&\leq C_0\bigg\{R^2 \Big(\dashint_{B(0,2R)}|F|^q\Big)^{\frac{1}{q}}
+  R\Big(\dashint_{B(0,2R)}|f|^p\Big)^{\frac{1}{p}}
+ R\Big(\dashint_{B(0,2R)}|\nabla u_\varepsilon|^2\Big)^{\frac{1}{2}}\bigg\}
+ CR\Big(\dashint_{B(0,2R)}|u_\varepsilon|^2\Big)^{\frac{1}{2}}\\
&\leq C\Big(\dashint_{B(0,4R)}|u_\varepsilon|^2\Big)^{\frac{1}{2}}
+C_0\bigg\{R\Big(\dashint_{B(0,4R)}|f|^p\Big)^{\frac{1}{p}}
+R^2 \Big(\dashint_{B(0,4R)}|F|^q\Big)^{\frac{1}{q}}\bigg\}
\end{aligned}
\end{equation*}
where $C$ is dependent of $\mu,m,d, C_0$ and
$\|(B-V)/\sqrt{\lambda}\|_{L^\infty(\mathbb{R}^d)}$, and we use
Caccioppoli's inequality $\eqref{pri:2.6}$ and H\"older's inequality in the last step.

Then we show the estimate for $I_1$,
\begin{equation*}
\begin{aligned}
|I_1|&\leq \frac{C}{R}\int_{B(0,2R)\setminus B(0,R/2)}
|\nabla\mathbf{\Gamma}_\varepsilon(x,y)||u_\varepsilon(y)|dy\\
&\leq\frac{C}{R}\bigg(\int_{B(0,2R)\setminus B(0,R/2)}
|\nabla\mathbf{\Gamma}_\varepsilon(x,y)|^2dy\bigg)^{\frac{1}{2}}
\bigg(\int_{B(0,2R)\setminus B(0,R/2)}
|u_\varepsilon|^2dy\bigg)^{\frac{1}{2}} \\
&\leq \frac{C}{R^2}\bigg(\int_{B(0,2R)\setminus B(0,R/2)}
|\mathbf{\Gamma}_\varepsilon(x,y)|^2dy\bigg)^{\frac{1}{2}}
\bigg(\int_{B(0,2R)}|u_\varepsilon|^2dy\bigg)^{\frac{1}{2}}
\leq C\bigg(\dashint_{B(0,2R)}|u_\varepsilon|^2\bigg)^{\frac{1}{2}}
\end{aligned}
\end{equation*}
where we use Caccioppoli's inequality $\eqref{pri:2.6}$
in the third step and the last one follows from the decay estimate $(\text{H}_1)$.
Finally, the estimate for $I_2$ is base upon the following computation:
\begin{equation*}
\begin{aligned}
\int_{B(0,R)}|\nabla_y\mathbf{\Gamma}_\varepsilon(x,y)f(y)|dy
&\leq C \sum_{k=0}^\infty(2^{-k}R)^d
\dashint_{2^{-k-1}R\leq |x-y|\leq 2^{-k}R}\Big|\nabla_y\mathbf{\Gamma}_\varepsilon(x,y)f(y)
\Big|dy \\
&\leq C\sum_{k=0}^\infty(2^{-k}R)^{d}
\bigg(\dashint_{2^{-k-1}R\leq |x-y|\leq 2^{-k}R}|\nabla_y\mathbf{\Gamma}_\varepsilon(x,y)|^2
dy\bigg)^{\frac{1}{2}}
\bigg(\dashint_{B(x,2^{-k}R)}|f|^p
\bigg)^{\frac{1}{p}}\\
&\leq C\sum_{k=0}^\infty(2^{-k}R)^{d-1-\frac{d}{p}}
\bigg(\dashint_{2^{-k-2}R\leq |x-y|\leq 2^{-k+1}R}|\mathbf{\Gamma}_\varepsilon(x,y)|^2
dy\bigg)^{\frac{1}{2}}\|f\|_{L^p(B(0,2R))} \\
& \leq CR\sum_{k=0}^\infty(2^{-k})^{1-\frac{d}{p}}
\bigg(\dashint_{B(0,2R)}|f|^p\bigg)^{\frac{1}{p}}
\end{aligned}
\end{equation*}
and this together the fact $p>d$ implies
\begin{equation*}
|I_2| \leq CR
\bigg(\dashint_{B(0,2R)}|f|^p\bigg)^{\frac{1}{p}}
\end{equation*}
Combining the estimates for $I_1,I_2,I_3$ will finally lead to the stated estimate
$\eqref{pri:0.4}$.

Concerning the estimate $\eqref{pri:0.5}$,
the tough term in the computations is to estimate
\begin{equation*}
\bigg|\int_{2B}
\nabla_x\nabla_y\mathbf{\Gamma}_\varepsilon(x,y)f(y)\varphi(y)dy\bigg|,
\end{equation*}
which may be controlled by
\begin{equation*}
\begin{aligned}
\int_{B(x,3R)}
&\big|\nabla_x\nabla_y\mathbf{\Gamma}_\varepsilon(x,y)\big[f(y)-f(x)\big]\big|dy
\quad + \\
&\qquad\|f\|_{L^\infty(2B)}
\bigg\{\int_{\partial(2B)}|\nabla_x\mathbf{\Gamma}_\varepsilon(x,y)|dy
+ \int_{2B\setminus B}|\nabla_x\nabla_y\mathbf{\Gamma}_\varepsilon(x,y)|dy\bigg\}.
\end{aligned}
\end{equation*}
This will lead to
\begin{equation*}
\bigg|\int_{2B}
\nabla_x\nabla_y\mathbf{\Gamma}_\varepsilon(x,y)f(y)\varphi(y)dy\bigg|
\leq C\Big\{R^{\sigma}[f]_{C^{0,\sigma}(4B)}+ \|f\|_{L^\infty(2B)}\Big\}
\end{equation*}
on account of $(\text{H}_2)$,
where $C$ depends only on $m,d,\sigma$ and $C_{00}$.
The remainder of the argument is standard
and left to the reader, and we ends the proof here.
\qed

\noindent\textbf{Proof of Theorem $\ref{thm:2.1}$.}
We mention that Theorem $\ref{thm:0.2}$ actually shows that
\begin{equation*}
 (\text{H}_1) \Rightarrow \eqref{pri:0.4}  \Rightarrow (\text{ii}) ~\text{in}~\eqref{pri:0.2}.
\end{equation*}
In the next section, Theorem $\ref{thm:2.3.1}$ will give that
$(\text{ii}) ~\text{in}~\eqref{pri:0.2} \Rightarrow (\text{H}_1)$, and this implies
the equivalence $\eqref{relation:1}$. Concerning $\eqref{relation:2}$, we just mention that
$\eqref{pri:0.5} \Rightarrow (\text{H}_2)$ has been included in \cite[Lemma 4.11]{QXS}
(or see \cite{MAFHL}). We end the proof here.
\qed

\subsection{Fundamental solutions}

\begin{thm}[Fundamental solutions]\label{thm:2.3.1}
Suppose that the coefficients of $\mathcal{L}_\varepsilon$ satisfy
$\eqref{a:1}$ and $\eqref{a:3}$ with $\lambda\geq\max\{\lambda_0,\mu\}$.
Assume the coefficient $A\in\emph{VMO}(\mathbb{R}^d)$ satisfies
$\eqref{a:2}$.
Then there exists a unique fundamental matrix
$\mathbf{{\Gamma}}_\varepsilon(\cdot,x)$ in
$H^{1}(\mathbb{R}^d\setminus B(x,r);\mathbb{R}^{m^2})
\cap W^{1,s}_{loc}(\mathbb{R}^d;\mathbb{R}^{m^2})$
with $s\in[1,\frac{d}{d-1})$ for any $x\in \mathbb{R}^d$ and $r>0$, such that
\begin{equation}\label{eq:2.3.1}
\mathrm{B}_{\mathcal{L}_\varepsilon;\mathbb{R}^d}
\big[\mathbf{{\Gamma}}_\varepsilon^\gamma(\cdot,x),\phi\big]
= \phi^\gamma(x) \qquad \forall \phi\in C_0^{\infty}(\mathbb{R}^d;\mathbb{R}^m).
\end{equation}
In fact, $\mathbf{{\Gamma}}_\varepsilon(y,x)$ is H\"older continuous in
$\{(x,y)\in\mathbb{R}^d\times\mathbb{R}^d:x\not=y\}$, and
if ${^*\mathbf{\Gamma}}_\varepsilon(x,y)$ denotes the fundamental matrix related to
$\mathcal{L}_\varepsilon^*$, then we have
$\mathbf{^*{\Gamma}}_\varepsilon(x,y) = [\mathbf{{\Gamma}}_\varepsilon(y,x)]^t$,
and the following estimate
\begin{eqnarray}\label{pri:2.3.3}
 |\mathbf{\Gamma}_\varepsilon(y,x)| \leq \frac{C}{|x-y|^{d-2}}
\end{eqnarray}
for any $x,y\in \mathbb{R}^d$ with $x\not=y$,
where $C$ depends only on $\mu,\kappa,\lambda,d,m,\|A\|_{\emph{VMO}}$.
Moreover, for any $F\in H^{-1}(\mathbb{R}^d;\mathbb{R}^m)\cap
L^q_{loc}(\mathbb{R}^d;\mathbb{R}^m)$ with $q>(d/2)$, the weak solution $u_\varepsilon$
to $\mathcal{L}_\varepsilon(u_\varepsilon) = F$ in $\mathbb{R}^d$ is continuous and
has the following representation
\begin{equation}\label{pri:2.3.4}
 u^\gamma_\varepsilon(x) =
 \int_{\mathbb{R}^d} \mathbf{{\Gamma}}_\varepsilon^{\gamma\alpha}(x,y)F^\alpha(y)dy.
\end{equation}
\end{thm}

\begin{lemma}[Approximating fundamental matrices]
Assume the same conditions as in Theorem $\ref{thm:2.1}$. Define
the approximating fundamental matrices
${\Gamma}_{\rho,\varepsilon}(\cdot,x)$ satisfying
\begin{equation}\label{def:2.0}
 \mathrm{B}_{\mathcal{L}_\varepsilon;\mathbb{R}^d}
 [{\Gamma}_{\rho,\varepsilon}^{\gamma}(\cdot,x),u]
 = \dashint_{B(x,\rho)} u^\gamma \qquad \forall u
 \in C_0^\infty(\mathbb{R}^d;\mathbb{R}^m),
\end{equation}
where $1\leq \gamma\leq m$.
Then we have the uniform estimate
\begin{eqnarray}\label{pri:2.3.1}
 \big|{\Gamma}_{\rho,\varepsilon}^\gamma(y,x)\big| \leq \frac{C}{|x-y|^{d-2}}
\end{eqnarray}
for any $x,y\in\mathbb{R}^d$ with $x\not=y$, and any $\rho<|x-y|/8$,
where $C$ depends only on $\mu,\kappa,\lambda,d,m$ and $\|A\|_{\emph{VMO}}$.
Moreover, for any $s\in [1,\frac{d}{d-1})$ and $R>0$, we have
\begin{equation}\label{pri:2.3.2}
 \sup_{\rho>0}\big\|{\Gamma}_{\rho,\varepsilon}^\gamma(\cdot,x)
 \big\|_{W^{1,s}(B(x,R))} \leq C(\mu,\kappa,\lambda,d,m,s,\|A\|_{\emph{VMO}},R).
\end{equation}
\end{lemma}

\begin{proof}
First of all, we show ${\Gamma}_{\rho,\varepsilon}(y,x)=
\big[{\Gamma}_{\rho,\varepsilon}^{\beta\gamma}(y,x)\big]$ is well defined.
Let $I(u) = \dashint_{B(x,\rho)} u dy$, then $I\in H^{-1}(\Omega;\mathbb{R}^m)$ and
$|I(u)|\leq C\rho^{1-d/2}\|\nabla u\|_{L^{2}(\mathbb{R}^d)}$.
It follows from Theorem $\ref{thm:2.0}$ that there exists a unique weak solution
${\Gamma}_{\rho,\varepsilon}^{\gamma}(\cdot,x)\in H^{1}(\mathbb{R}^d;\mathbb{R}^m)$
satisfying $\eqref{def:2.0}$ and
\begin{equation}\label{f:2.3.1}
 \|\nabla {\Gamma}_{\rho,\varepsilon}(\cdot,x)\|_{L^2(\mathbb{R}^d)}
 \leq C\rho^{1-d/2}.
\end{equation}
For any $F\in C^\infty_0(\mathbb{R}^d;\mathbb{R}^m)$, consider the equation
$\mathcal{L}_\varepsilon^*(v_\varepsilon) = F$ in $\mathbb{R}^d$.
There exists the unique solution
$v_\varepsilon\in H^{1}(\mathbb{R}^d;\mathbb{R}^m)$ such that
\begin{equation}
 \int_{\mathbb{R}^d}  {\Gamma}_{\rho,\varepsilon}^{\beta\gamma}(z,x) F^\beta(z)dz
 = \mathrm{B}_{\mathcal{L}_\varepsilon^*;\mathbb{R}^d}
 [v_\varepsilon, {\Gamma}_{\rho,\varepsilon}^\gamma(\cdot,x)]
 = \mathrm{B}_{\mathcal{L}_\varepsilon;\mathbb{R}^d}
 [{\Gamma}_{\rho,\varepsilon}^\gamma(\cdot,x),v_\varepsilon]
 = \dashint_{B(x,\rho)} v^\gamma_\varepsilon.
\end{equation}

Let $R=|x-y|/4$ and $\mathrm{supp}(F) \subsetneqq B(y,R)$.
Then $\mathcal{L}_\varepsilon^*(v_\varepsilon) = 0$ in $B(x,R)$ and it follows
from a local boundedness estimate $\eqref{pri:2.2.1}$ that
\begin{eqnarray*}
\Big|\int_{B(y,R)}  {\Gamma}_{\rho,\varepsilon}^\gamma(z,y) F(z)dz\Big|
\leq  \|v_\varepsilon\|_{L^\infty(B(x,R/2))}
\leq C\bigg(\dashint_{B(x,R)}|v_\varepsilon|^2\bigg)^{\frac{1}{2}}
\end{eqnarray*}
for any $\rho < R/2$. This together with
\begin{equation*}
\bigg(\dashint_{B(x,R)}|v_\varepsilon|^2\bigg)^{\frac{1}{2}}
\leq \bigg(\dashint_{B(x,R)}|v_\varepsilon|^{\frac{2d}{d-2}}\bigg)^{\frac{d-2}{2d}}
\leq CR^{2}
\bigg(\dashint_{B(y,R)}|F|^2\bigg)^{\frac{1}{2}}
\end{equation*}
where we also use the estimate $\eqref{pri:2.0.1}$,
implies
\begin{equation*}
 \bigg(\dashint_{B(y,R)}\big|{\Gamma}_{\rho,\varepsilon}^\gamma(\cdot,y)
 \big|^2\bigg)^{\frac{1}{2}} \leq C R^{2-d}.
\end{equation*}

In view of $\eqref{def:2.0}$, it is not hard to see that
${\Gamma}_{\rho,\varepsilon}(y,x)$ actually satisfies
$\mathcal{L}_\varepsilon\big({\Gamma}_{\rho,\varepsilon}^\gamma
(\cdot,x)\big) =0$ in $B(y,R)$.
By using the local boundedness estimate $\eqref{pri:2.2.1}$ again, we obtain
\begin{equation*}
 \big|{\Gamma}_{\rho,\varepsilon}^\gamma(y,x)\big|
 \leq C \dashint_{B(y,R)}\bigg(
 \big|{\Gamma}_{\rho,\varepsilon}^\gamma(\cdot,x)\big|^2\bigg)^{1/2}
 \leq C |x-y|^{2-d}
\end{equation*}
for any $\rho < |x-y|/8$, where $C$ may be given by $C_\sigma$.

Then we will prove $\eqref{pri:2.3.2}$, and it will be established by three steps.

Step 1. We verify the following estimates
\begin{eqnarray}\label{f:2.3.2}
\int_{\mathbb{R}^d\setminus B(x,R)}
|\nabla {\Gamma}_{\rho,\varepsilon}(z,x)|^2dz
\leq CR^{2-d}, \quad
\int_{\mathbb{R}^d\setminus B(x,R)}
|{\Gamma}_{\rho,\varepsilon}(z,x)|^{\frac{2d}{d-2}}dz
\leq CR^{-d}
\end{eqnarray}
for any $\rho>0$ and $R>0$ (note: this $R$ is a new one and will be used below).

On the one hand, let $\varphi\in C^1_0(\mathbb{R}^d)$
be a cut-off function satisfying
$\varphi\equiv 0$ on $B(x,R/2)$ and $\varphi \equiv 1$ outside $B(x,R)$ with
$|\nabla\varphi| \leq C/R$.
Set $u = {\Gamma}_{\rho,\varepsilon}^\gamma(\cdot,x)\varphi^2 $
in $\eqref{def:2.0}$ and $\lambda\geq \lambda_0$.
It follows from Caccioppoli's inequality $\eqref{pri:2.6}$ and the estimate $\eqref{pri:2.3.1}$ that
\begin{equation}\label{f:2.3.3}
\begin{aligned}
 \int_{\mathbb{R}^d} \varphi^2 |\nabla {\Gamma}_{\rho,\varepsilon}^\gamma(z,x)|^2 dz
 &\leq C\int_{\mathbb{R}^d}
 |\nabla\varphi|^2|{\Gamma}_{\rho,\varepsilon}^\gamma(z,x)|^2 dz\\
 &\leq \frac{C}{R^2}\int_{B(x,R)\setminus B(x,R/2)} |z-y|^{4-2d} dz
 \leq CR^{2-d}
\end{aligned}
\end{equation}
for any $\rho<R/2$.
On the other hand, in the case of $\rho>(R/2)$ it follows from $\eqref{f:2.3.1}$ that
\begin{eqnarray*}
\int_{\mathbb{R}^d\setminus B(x,R)}|\nabla {\Gamma}_{\rho,\varepsilon}^\gamma(z,x)|^2dz
\leq  \int_{\mathbb{R}^d}|\nabla {\Gamma}_{\rho,\varepsilon}^\gamma(z,x)|^2 dz
\leq CR^{2-d}.
\end{eqnarray*}
Thus we have the first inequality of $\eqref{f:2.3.2}$ for all $\rho>0$.

For the second estimate in $\eqref{f:2.3.2}$,
we observe
\begin{equation}\label{f:2.3.4}
\begin{aligned}
 \int_{\mathbb{R}^d}\big|{\Gamma}_{\rho,\varepsilon}^\gamma(z,x)\varphi\big|^{\frac{2d}{d-2}} dz
 &\leq C\Big(\int_{\mathbb{R}^d}
 \big|\nabla\big({\Gamma}_{\rho,\varepsilon}^\gamma(z,x) \varphi\big)
 \big|^2 dz\Big)^{\frac{d}{d-2}} \\
 &\leq C\bigg(\int_{\mathbb{R}^d}
 \Big(|\nabla\varphi|^2|{\Gamma}_{\rho,\varepsilon}^\gamma(z,x)|^2
 + \varphi^2|\nabla {\Gamma}_{\rho,\varepsilon}^\gamma(z,x)|^2\Big)dz
 \bigg)^{\frac{d}{d-2}} \leq CR^{-d}
\end{aligned}
\end{equation}
for any $\rho < R/2$, where we use Sobolev's inequality in the first
inequality and $\eqref{f:2.3.3}$ in the last inequality. We remark that
the constant $C$ does not involve $R$. For $\rho \geq R/2$,
we have
\begin{equation*}
\int_{\mathbb{R}^d\setminus B(x,R)}
|{\Gamma}_{\rho,\varepsilon}^\gamma(z,x)|^{\frac{2d}{d-2}} dz
\leq C\Big(\int_{\mathbb{R}^d}
|\nabla {\Gamma}_{\rho,\varepsilon}^\gamma(z,x)|^{2}dz\Big)^{\frac{d}{d-2}}
\leq C R^{-d},
\end{equation*}
where we use Sobolev's inequality in the first inequality
and $\eqref{f:2.3.1}$ in the last inequality.
This together with $\eqref{f:2.3.4}$ proved the second inequality in
the stated estimate $\eqref{f:2.3.2}$ for all $\rho>0$.

Step 2. In term of the parameter $\rho$,
we now show the uniform weak-$L^{\frac{d}{d-2}}$ estimate
for ${\Gamma}_{\rho,\varepsilon}^\gamma(\cdot,x)$ and weak-$L^{\frac{d}{d-1}}$
for $\nabla {\Gamma}_{\rho,\varepsilon}^\gamma(\cdot,x)$.
In the case of $t=R^{1-d}$, we obtain
\begin{eqnarray}\label{f:2.3.5}
\big|\big\{y\in\mathbb{R}^d: |\nabla {\Gamma}_{\rho,\varepsilon}^\gamma(\cdot,x)|>t\big\}\big|
\leq CR^{d} + t^{-2}\int_{\mathbb{R}^d\setminus B(x,R)}
|\nabla {\Gamma}_{\rho,\varepsilon}^\gamma(\cdot,x)|^2
\leq C t^{-\frac{d}{d-1}}\quad \forall \rho>0,
\end{eqnarray}
where we use the first inequality in $\eqref{f:2.3.2}$. Similarly,
for $t=R^{2-d}$ we have
\begin{eqnarray}\label{f:2.3.6}
\big|\big\{y\in\mathbb{R}^d: |{\Gamma}_{\rho,\varepsilon}^\gamma(\cdot,x)|>t\big\}\big|
\leq C t^{-\frac{d}{d-2}}\qquad \forall \rho>0.
\end{eqnarray}

Step 3. In view of $\eqref{f:2.3.5}$ and $\eqref{f:2.3.6}$ we finally have
\begin{eqnarray}\label{f:2.3.7}
 \int_{B(x,R)}|{\Gamma}_{\rho,\varepsilon}^\gamma(\cdot,x)|^s
 \leq C \theta^{s}R^{d} + C\int_{\theta}^\infty t^{s-1}\cdot t^{-\frac{d}{d-2}} dt
 \leq CR^{(2-d)s+d}
\end{eqnarray}
for $s\in[1,\frac{d}{d-2})$, where we choose $\theta = R^{2-d}$. By the same computation,
there holds
\begin{eqnarray*}
 \int_{B(x,R)}|\nabla {\Gamma}_{\rho,\varepsilon}^\gamma(\cdot,x)|^s
 \leq CR^{(1-d)s+d}
\end{eqnarray*}
for $s\in [1,\frac{d}{d-1})$,
where $C$ depends only on $\mu,\kappa,\lambda,d,m,s,\|A\|_{\text{VMO}}$.
Combining the two inequalities above leads to
the stated estimate $\eqref{pri:2.3.2}$, and the proof is complete.
\end{proof}

\noindent\textbf{Proof of Theorem \ref{thm:2.3.1}.}
Fix $s\in(1,\frac{d}{d-1})$.
From the uniform estimate $\eqref{pri:2.3.2}$ and
a diagonalization process,
it follows that there exist a subsequence of
$\{{\Gamma}_{\rho_n,\varepsilon}^\gamma(\cdot,x)\}_{n=1}^\infty$
and $\mathbf{{\Gamma}}_{\varepsilon}^\gamma(\cdot,x)$ in
$W^{1,s}_{loc}(\mathbb{R}^d;\mathbb{R}^{m})$,
such that
\begin{eqnarray}\label{f:2.3.8}
 {\Gamma}_{\rho_n,\varepsilon}^\gamma(\cdot,x)~\rightharpoonup~
 \mathbf{{\Gamma}}_{\varepsilon}^\gamma(\cdot,x)
 \quad \text{weakly in}~ W^{1,s}(B(x,R);\mathbb{R}^m)
\end{eqnarray}
for any $R>0$, as $\lim_{n\to\infty}\rho_n =0$.
Hence, in terms of $\eqref{def:2.0}$ we may immediately derive
\begin{equation*}
\mathrm{B}_{\mathcal{L}_\varepsilon;\mathbb{R}^d}
[\mathbf{{\Gamma}}_\varepsilon^\gamma(\cdot,x),\phi]
= \lim_{n\to\infty}  \mathrm{B}_{\mathcal{L}_\varepsilon;\mathbb{R}^d}
[{\Gamma}_{\rho_n,\varepsilon}^\gamma(\cdot,x),\phi]
= \lim_{n\to\infty} \dashint_{B(x,\rho_n)} \phi^\gamma
= \phi^\gamma(x)
\end{equation*}
for any $\phi\in C^{\infty}_0(\mathbb{R}^d;\mathbb{R}^m)$.
In other words, we have
$\mathcal{L}_\varepsilon(\mathbf{{\Gamma}}_\varepsilon(\cdot,x))
= 0$ in $\mathbb{R}^d\setminus B(x,R)$ with any $R>0$.

Let $r = \text{dist}(K,x)>0$ for any compact set $K\subset\mathbb{R}^d\setminus\{x\}$.
For any $B=B(y,r/4)$ with $y\in K$, we have
\begin{equation}\label{f:2.3.10}
\begin{aligned}
\big[{\Gamma}_{\rho,\varepsilon}(\cdot,x)\big]_{C^{0,\sigma}(B)}
&\leq Cr^{-\sigma}\Big(\dashint_{2B}|{\Gamma}_{\rho,\varepsilon}(\cdot,x)|^2\Big)^{1/2} \\
&\leq Cr^{1-\sigma-\frac{d}{2}}
\Big(\int_{\mathbb{R}^d\setminus B(x,r/4)}|{\Gamma}_{\rho,\varepsilon}(\cdot,x)|^{2^*}\Big)^{1/2^*}
\leq Cr^{2-\sigma-d}
\end{aligned}
\end{equation}
and
\begin{equation}\label{f:2.3.11}
\begin{aligned}
\big\|{\Gamma}_{\rho,\varepsilon}(\cdot,x)\big\|_{L^{\infty}(B)}
\leq Cr^{2-d},
\end{aligned}
\end{equation}
where $\sigma\in(0,1)$ and we use the estimate $\eqref{pri:0.2}$ and $\eqref{f:2.3.2}$.
Combining the estimates $\eqref{f:2.3.10}$ and $\eqref{f:2.3.11}$ implies that the sequence
$\big\{{\Gamma}_{\rho_n,\varepsilon}(\cdot,x)\big\}_{n=1}^\infty$ is
equicontinuous on $K$. Therefore, one may conclude that ${\mathbf{\Gamma}}_\varepsilon(\cdot,x)$
is locally H\"older continuous in $\mathbb{R}^d\setminus\{x\}$.

Also, it follows from
$\eqref{f:2.3.7}$ and $\eqref{f:2.3.8}$ that
\begin{equation*}
 \Big(\dashint_{B(x,R)} |\mathbf{{\Gamma}}_\varepsilon^\gamma(\cdot,x)|^s\Big)^{1/s}
 \leq \lim_{n\to\infty}\Big(\dashint_{B(x,R)} |{\Gamma}_{\rho_n,\varepsilon}^\gamma(\cdot,x)|^s\Big)^{1/s}
 \leq CR^{2-d}
\end{equation*}
for any $s\in (1,\frac{d}{d-2})$.
Let $R = |x-y|/4$ again,
and then the above estimate together with $\eqref{pri:2.2.1}$ gives
\begin{equation*}
|\mathbf{{\Gamma}}_\varepsilon(y,x)|
\leq C\dashint_{B(y,R)}|\mathbf{{\Gamma}}_\varepsilon(\cdot,x)|
\leq C\Big(\dashint_{B(x,5R)}|\mathbf{{\Gamma}}_\varepsilon(\cdot,x)|^s\Big)^{1/s}
\leq CR^{2-d}.
\end{equation*}
which is the stated estimate $\eqref{pri:2.3.3}$.

Let ${^*\mathbf{\Gamma}}_\varepsilon(\cdot,y)$ be the fundamental matrix associated with
$\mathcal{L}_\varepsilon^*$, which
may be similarly determined by a family of
$\{{^*\Gamma}_{\varrho,\varepsilon}(\cdot,y)\}$ with $\varrho>0$, satisfying
\begin{equation*}
\mathrm{B}_{\mathcal{L}_\varepsilon^*;\mathbb{R}^d}
\big[{^*\Gamma}_{\varrho,\varepsilon}^\theta(\cdot,y),\phi\big]
= \dashint_{B(y,\varrho)} \phi^\theta \qquad \forall \phi\in C_0^{\infty}(\mathbb{R}^d;\mathbb{R}^m).
\end{equation*}
Thus for any $\rho,\varrho>0$, we obtain
\begin{eqnarray*}
 \dashint_{B(x,\rho)}{^*\Gamma}^{\gamma\theta}_{\varrho,\varepsilon}(z,y) dz
 = \mathrm{B}_{\mathcal{L}_\varepsilon;\mathbb{R}^d}
 [\Gamma_{\rho,\varepsilon}^\gamma(\cdot,x),{^*\Gamma}^{\theta}_{\varrho,\varepsilon}(\cdot,y)]
 = \mathrm{B}_{\mathcal{L}_\varepsilon^*;\mathbb{R}^d}
 [{^*\Gamma}^{\theta}_{\varrho,\varepsilon}(\cdot,y),\Gamma_{\rho,\varepsilon}^\gamma(\cdot,x)]
 =  \dashint_{B(y,\varrho)}{\Gamma}^{\theta\gamma}_{\rho,\varepsilon}(z,x) dz .
\end{eqnarray*}
Note that $\mathcal{L}_\varepsilon^*[{^*\Gamma}_{\varrho,\varepsilon}^{\theta}(\cdot,y)] = 0$
in $\mathbb{R}^d\setminus B(y,\rho)$ and
$\mathcal{L}_\varepsilon[{\Gamma}_{\rho,\varepsilon}^{\gamma}(\cdot,x)] = 0$
in $\mathbb{R}^d\setminus B(x,\varrho)$.
In view of the estimates $\eqref{pri:0.2}$, $\eqref{f:2.3.10}$ and $\eqref{f:2.3.11}$,
it has been known that
${^*\Gamma}_{\varrho,\varepsilon}^{\theta}(\cdot,y)$
and ${\Gamma}_{\rho,\varepsilon}^{\gamma}(\cdot,x)$ are locally H\"{o}lder continuous, respectively.
Therefore, we conclude
${^*\mathbf{\Gamma}}_{\varepsilon}^{\gamma\theta}(x,y)
= {\mathbf{\Gamma}}_{\varepsilon}^{\theta\gamma}(y,x)$ as $\rho\to 0$ and $\varrho\to 0$,
which gives
${^*\mathbf{\Gamma}}_\varepsilon(x,y) = [\mathbf{\Gamma}_\varepsilon(y,x)]^t$
for any $x,y\in\mathbb{R}^d$ and $x\not=y$.

For any $F\in L^{p}_{loc}(\mathbb{R}^d;\mathbb{R}^m)\cap H^{-1}(\mathbb{R}^d;\mathbb{R}^m)$
with $p>(d/2)$,
the weak solution $u_\varepsilon\in H^{1}(\mathbb{R}^d;\mathbb{R}^m)$ to
$\mathcal{L}_\varepsilon(u_\varepsilon) = F$ in $\mathbb{R}^d$
is locally H\"older continuous, according to \cite[Corollary 3.5]{QXS}.
Then we have
\begin{equation*}
 u^\theta_\varepsilon(y)
 = \mathrm{B}_{\mathcal{L}_\varepsilon^*;\mathbb{R}^d}
 [{^*\mathbf{\Gamma}}_\varepsilon^\theta(\cdot,y),u_\varepsilon]
 = \mathrm{B}_{\mathcal{L}_\varepsilon;\mathbb{R}^d}
 [u_\varepsilon,{^*\mathbf{\Gamma}}_\varepsilon^\theta(\cdot,y)]
 = \int_{\mathbb{R}^d} {^*\mathbf{\Gamma}}_\varepsilon^{\alpha\theta}(z,y)F^\alpha(z) dz.
\end{equation*}
which together with ${^*\mathbf{\Gamma}}_\varepsilon(x,y) = [\mathbf{\Gamma}_\varepsilon(y,x)]^t$
implies the stated formula $\eqref{pri:2.3.4}$.

Finally, we verify the uniqueness. If
$\widetilde{{\mathbf{\Gamma}}}_\varepsilon$ is another
fundamental matrix of $\mathcal{L}_\varepsilon$, we then
have $\tilde{u}_\varepsilon(x) = \int_{\mathbb{R}^d}
\widetilde{{\mathbf{\Gamma}}}_\varepsilon(x,z) F(z)dz$.
It follows from the uniqueness of the weak solution that
$ \int_{\mathbb{R}^d}
\big[\widetilde{{\mathbf{\Gamma}}}_\varepsilon(x,z)
- {\mathbf{\Gamma}}_\varepsilon(x,z)\big] F(z)dz = 0$
for any $F\in C^{\infty}_{0}(\mathbb{R}^d;\mathbb{R}^m)$, and thus
$\widetilde{{\mathbf{\Gamma}}}_\varepsilon(x,y) =
{\mathbf{\Gamma}}_\varepsilon(x,y)$ in
$\{(x,y)\in\mathbb{R}^d\times\mathbb{R}^d:x\not=y\}$.
\qed

\begin{lemma}\label{lemma:2.3.1}
Suppose that the coefficients of $\mathcal{L}$ and $\widetilde{\mathcal{L}}$ satisfy
the same conditions as in Theorem $\ref{thm:2.3.1}$.
Let $\mathbf{\Gamma}_{\mathcal{L}}$ and
$\mathbf{\Gamma}_{\widetilde{\mathcal{L}}}$ be the fundamental solutions of
$\mathcal{L}$ and $\widetilde{\mathcal{L}}$, respectively. Then there holds
\begin{equation}\label{eq:5.1}
\begin{aligned}
\mathbf{\Gamma}^{\alpha\delta}_{\mathcal{L}}(x,y)
-{\mathbf{\Gamma}}^{\alpha\delta}_{\widetilde{\mathcal{L}}}(x,y)
&= \int_{\mathbb{R}^d}\Big[\widetilde{a}_{ij}^{\beta\gamma}(z)-a_{ij}^{\beta\gamma}(z)\Big]
\frac{\partial {\mathbf{\Gamma}}^{\gamma\delta}_{\widetilde{\mathcal{L}}}}{\partial z_j}(z,y)
\frac{\partial \mathbf{\Gamma}^{\alpha\beta}_{\mathcal{L}}}{\partial z_i}(x,z)dz\\
&+\int_{\mathbb{R}^d}\Big[\widetilde{V}_{i}^{\beta\gamma}(z)-V_{i}^{\beta\gamma}(z)\Big]
\mathbf{\Gamma}^{\gamma\delta}_{\widetilde{\mathcal{L}}}(z,y)
\frac{\partial\mathbf{\Gamma}^{\alpha\beta}_{\mathcal{L}}}{\partial z_i}(x,z)dz\\
&+\int_{\mathbb{R}^d}\Big[\widetilde{B}_{i}^{\beta\gamma}(z)-B_{i}^{\beta\gamma}(z)\Big]
\frac{\partial \mathbf{\Gamma}^{\gamma\delta}_{\widetilde{\mathcal{L}}}}{\partial z_i}(z,y)
\mathbf{\Gamma}^{\alpha\beta}_{\mathcal{L}}(x,z) dz \\
&+\int_{\mathbb{R}^d}\Big[\widetilde{\mathbf{c}}^{\beta\gamma}(z)-\mathbf{c}^{\beta\gamma}(z)\Big]
\mathbf{\Gamma}^{\gamma\delta}_{\widetilde{\mathcal{L}}}(z,y)
\mathbf{\Gamma}^{\alpha\beta}_{\mathcal{L}}(x,z)dz
\end{aligned}
\end{equation}
for any $x,y\in\mathbb{R}^d$ with $x\not=y$,
where $\mathbf{c} = c+\lambda I$ and $\widetilde{\mathbf{c}}=\widetilde{c}+\widetilde{\lambda}I$.
\end{lemma}

\begin{proof}
Before approaching the proof, we identify the notation
for the reader's convenience, which are
${^*\mathbf{\Gamma}}(\cdot,x) = \mathbf{\Gamma}_{\mathcal{L}^*}(\cdot,x)$ and
${\mathbf{\Gamma}}(x,\cdot) = \mathbf{\Gamma}_{\mathcal{L}}(x,\cdot)$. Thus
one may have $\mathbf{\Gamma}_{\mathcal{L}^*}(x,y) = [\mathbf{\Gamma}_{\mathcal{L}}(y,x)]^t$
from ${^*\mathbf{\Gamma}}_\varepsilon(x,y) = [\mathbf{\Gamma}_\varepsilon(y,x)]^t$.

In view of $\eqref{eq:2.3.1}$ we have
\begin{equation*}
\begin{aligned}
\mathrm{B}_{\mathcal{L}^*;\mathbb{R}^d}[{^*\mathbf{\Gamma}}^{\cdot\alpha}(\cdot,x),
\mathbf{\Gamma}_{\widetilde{\mathcal{L}}}^{\cdot\delta}(\cdot,y)]
&= \mathbf{\Gamma}_{\widetilde{\mathcal{L}}}^{\alpha\delta}(x,y), \\
\mathrm{B}_{\widetilde{\mathcal{L}};\mathbb{R}^d}
[\mathbf{\Gamma}^{\cdot\delta}_{\widetilde{\mathcal{L}}}(\cdot,y),
{^*\mathbf{\Gamma}}^{\cdot\alpha}(\cdot,x)]
&= {^*\mathbf{\Gamma}}^{\delta\alpha}(y,x)
= \mathbf{\Gamma}^{\alpha\delta}(x,y),
\end{aligned}
\end{equation*}
where we use the fact ${^*\mathbf{\Gamma}}_\varepsilon(x,y) = [\mathbf{\Gamma}_\varepsilon(y,x)]^t$
in the last equality. This will give the stated equality $\eqref{eq:5.1}$, and we are done.
\end{proof}

\subsection{Asymptotic behavior of fundamental solutions}

\noindent\textbf{Proof of Theorem $\ref{thm:0.4}$.}
On account of the results in Theorem $\ref{thm:0.2}$,
the main idea is similar to that given in \cite[pp.901-903]{MAFHL3},
and we provide a proof for the sake of the completeness.
A notable difference is that the fundamental solutions
$\mathbf{\Gamma}_{\varepsilon}(x,y)$ and $\mathbf{\Gamma}_0(x,y)$ studied here have no
homogeneous properties.
We also mention that
the case of $|x-y|\leq \varepsilon$
is trivial since it follows from the estimates
$\eqref{pri:2.3.3}$ and $\eqref{pri:3.0}$ that
\begin{equation*}
\big|\mathbf{\Gamma}_{\varepsilon}(x,y)-\mathbf{\Gamma}_0(x,y)\big|
\leq C|x-y|^{2-d}\leq C\varepsilon|x-y|^{1-d}.
\end{equation*}

Thus, it suffices to study the case $|x-y|>\varepsilon$.
Let $r=|x-y|/4$, and $f\in C^{1}_0(B(y,r))$.
Suppose that $u_\varepsilon,u_0\in H^1(\mathbb{R}^d)$ satisfy
\begin{equation*}
 \mathcal{L}_\varepsilon(u_\varepsilon) = f = \mathcal{L}_0(u_0) \quad\text{in}\quad \mathbb{R}^d.
\end{equation*}
To achieve our goal, we introduce the first order approximating corrector
\begin{equation}
w_\varepsilon
= u_\varepsilon - u_0 - \varepsilon\chi_0(x/\varepsilon)u_0
-\varepsilon\chi_k(x/\varepsilon)\nabla_k u_0,
\end{equation}
and by a similar calculation given earlier in \cite[Lemma 5.1]{QXS1} it verifies
\begin{equation}
 \mathcal{L}_\varepsilon(w_\varepsilon) = \varepsilon\text{div}(\tilde{f})
 + \varepsilon \tilde{F} \qquad\text{in}\quad\mathbb{R}^d,
\end{equation}
where the notation $\tilde{f}$ and
$\tilde{F}$ are related to $\nabla^2 u_0$,
$\nabla u_0$ and $u_0$, as well as the correctors together with dual correctors.
Furthermore, we may derive $\|\tilde{f}\|_{L^2(\mathbb{R}^d)}+\|\tilde{F}\|_{L^2(\mathbb{R}^d)}
\leq C\|u_0\|_{H^2(\mathbb{R}^d)}$ due to the $L^\infty$ bounds of correctors and
dual correctors (see \cite[Lemma 2.9]{QXS1} or \cite[Remark 2.9]{QXS}).

Thus on the one hand, in view of the estimate $\eqref{pri:2.0.1}$ we have
\begin{equation}\label{f:5.1}
\|\nabla w_\varepsilon\|_{L^2(\mathbb{R}^d)}
\leq C\varepsilon\Big\{\|\tilde{f}\|_{L^2(\mathbb{R}^d)} +
\|\tilde{F}\|_{L^2(\mathbb{R}^d)}\Big\}
\leq C\varepsilon\|u_0\|_{H^2(\mathbb{R}^d)},
\end{equation}
where $C$ depends on $\mu,\kappa,\lambda,m,d$.

One the other hand, it follows from the estimate $\eqref{pri:0.4}$ that
\begin{equation*}
|w_\varepsilon(x)| \leq C\Big(\dashint_{B(x,r)}|w_\varepsilon|^2\Big)^{\frac{1}{2}}
+ C\varepsilon\bigg\{
r\Big(\dashint_{B(x,r)}|\tilde{f}|^p\Big)^{\frac{1}{p}}
+r^2\Big(\dashint_{B(x,r)}|\tilde{F}|^q\Big)^{\frac{1}{q}}\bigg\}
\end{equation*}
where $p>d$ and $q>(d/2)$, and this implies
\begin{equation}\label{f:5.2}
\begin{aligned}
\big|u_\varepsilon(x)-u_0(x)\big|
&\leq C\bigg(\dashint_{B(x,r)}|w_\varepsilon|^2dz\bigg)^{\frac{1}{2}}
+C\varepsilon\big\|u_0\big\|_{W^{1,\infty}(B(x,r))}
+C\varepsilon(r+r^2)\big\|u_0\big\|_{W^{2,\infty}(B(x,r))}\\
&\leq Cr^{1-\frac{d}{2}}\|\nabla w_\varepsilon\|_{L^2(\mathbb{R}^d)}
+C\varepsilon r^{1-\frac{d}{2}}\|\nabla u_0\|_{H^1(\mathbb{R}^d)}
+C\varepsilon r^{1-\frac{d}{2}}\|u_0\|_{H^2(\mathbb{R}^d)}
\end{aligned}
\end{equation}
where we employ the H\"older's inequality and Sobolev's inequality, as well as,
the following interior estimates
\begin{equation*}
\|\nabla^l u_0\|_{L^\infty(B(x,r))}
\leq C\Big(\dashint_{B(x,2r)}|\nabla^l u_0|^2 dz\Big)^{1/2}
\leq \frac{C}{1+\sqrt{\lambda}r}\Big(\dashint_{B(x,3r)}|\nabla^l u_0|^2 dz\Big)^{1/2}
\end{equation*}
for $l=0,1,2$, since $\mathcal{L}_0(u_0) = 0$ in $B(x,3r)$, in which
$\nabla^0 u_0$ denotes $u_0$ by abusing the notation. The last inequality in
the above estimate follows from Caccioppoli's inequality $\eqref{pri:2.1}$, coupled
with the fact that $\nabla^l u_0$ will still be a solution of
$\mathcal{L}_0(u_0) = 0$ in $B(x,3r)$.

Plugging the estimate $\eqref{f:5.1}$ back into $\eqref{f:5.2}$, we have
\begin{equation}
\big|u_\varepsilon(x)-u_0(x)\big|
\leq C\varepsilon r^{1-\frac{d}{2}}\|u_0\|_{H^2(\mathbb{R}^d)},
\end{equation}
where $C$ is independent of $r$.

Noting that $u_\varepsilon$ and $u_0$ may be represented by the related fundamental solutions
\begin{equation*}
 u_\varepsilon(x) = \int_{\mathbb{R}^d}
 \mathbf{\Gamma}_\varepsilon(x,z)f(z) dz
 \quad \text{and}
 \quad
 u_0(x) = \int_{\mathbb{R}^d}\mathbf{\Gamma}_0(x,z)f(z) dz
\end{equation*}
(see $\eqref{pri:2.3.4}$),
we are able to obtain
\begin{equation*}
\Big|\int_{B(y,r)}\big[\mathbf{\Gamma}_\varepsilon(x,z)- \mathbf{\Gamma}_0(x,z)\big]f(z)dz\Big|
\leq C\varepsilon r^{1-\frac{d}{2}} \|u_0\|_{H^2(\mathbb{R}^d)}
\leq C\varepsilon r^{1-\frac{d}{2}} \|f\|_{L^2(B(y,r))},
\end{equation*}
where we use the $H^2$ estimates  in the last step, and this implies
\begin{equation}\label{f:5.5}
\bigg(\dashint_{B(y,r)}\big|\mathbf{\Gamma}_\varepsilon(x,z)- \mathbf{\Gamma}_0(x,z)\big|^2dz\bigg)^{1/2}
\leq C\varepsilon r^{1-d}.
\end{equation}

Recall that the fundamental solutions $\mathbf{\Gamma}_\varepsilon(x,\cdot)$ and
$\mathbf{\Gamma}_0(x,\cdot)$ satisfy
\begin{equation*}
\mathcal{L}_\varepsilon^*\big[{^*\mathbf{\Gamma}}_\varepsilon(\cdot,x)\big] =
\mathcal{L}_0^*\big[{^*\mathbf{\Gamma}}_0(\cdot,x)\big]=0 \qquad\text{in}\quad B(y,2r)
\end{equation*}
(by neglecting the transport notation),
and let
\begin{equation*}
\Phi_\varepsilon^x(z)= {^*\mathbf{\Gamma}}_\varepsilon(z,x)
- {^*\mathbf{\Gamma}}_0(z,x) - \varepsilon\chi_0^*(z/\varepsilon){^*\mathbf{\Gamma}}_0(z,x)
-\varepsilon\chi_k^*(z/\varepsilon)\nabla_k{^*\mathbf{\Gamma}}_0(z,x),
\end{equation*}
where $\chi_k^*$ with $k=0,1,\cdots,d$ are the correctors associated
with $\mathcal{L}_\varepsilon^*$,
and it admits that
\begin{equation*}
\mathcal{L}_\varepsilon^*(\Phi_\varepsilon^x) =
\varepsilon \text{div}(\breve{f}) + \varepsilon \breve{F}
\qquad\text{in}\quad B(y,2r),
\end{equation*}
where $\breve{f}$ and $\breve{F}$ can be archived by setting $u_0 = \mathbf{\Gamma}_0(x,\cdot)$ in
the expressions in $\tilde{f}$ and $\tilde{F}$, respectively.
We refer the reader to \cite[Remark 2.23]{QX2} or \cite[Lemma 5.1]{QXS}
for the concrete expressions of $\breve{f}$ and $\breve{F}$.

we may use the estimate $\eqref{pri:0.4}$ again, similar to the first line of $\eqref{f:5.2}$, and
\begin{equation}
\begin{aligned}
\big|\mathbf{\Gamma}_\varepsilon(x,y)&-\mathbf{\Gamma}_0(x,y)\big|
\leq
C\bigg(\dashint_{B(y,r)}\big|\mathbf{\Gamma}_\varepsilon(x,z)- \mathbf{\Gamma}_0(x,z)\big|^2dz\bigg)^{1/2} \\
&+ C\varepsilon \big\|\mathbf{\Gamma}_0(x,\cdot)\big\|_{W^{1,\infty}(B(y,r))}
+ C\varepsilon\big(r+r^2\big)\big\|\mathbf{\Gamma}_0(x,\cdot)\big\|_{W^{2,\infty}(B(y,r))}
\leq C\varepsilon r^{1-d},
\end{aligned}
\end{equation}
where we use the estimate $\eqref{pri:3.0}$ in the last inequality. This is exactly
the stated estimate $\eqref{pri:0.6}$.

We now turn to prove the estimate $\eqref{pri:0.7}$, and we shall adopt the same procedure
as in the proof of $\eqref{pri:0.6}$. Also, it suffices to prove the case of
$r=|x-y|>\varepsilon$, and
it follows from the interior
Lipschitz estimate $\eqref{pri:0.5}$ that
\begin{equation}\label{f:5.3}
|\nabla \Phi_\varepsilon^x(y)|
\leq \frac{C}{r}\Big(\dashint_{B(y,r)}|\Phi_\varepsilon^x|^2 dz\Big)^{\frac{1}{2}}
+ C\varepsilon\bigg\{\|\breve{f}\|_{L^\infty(B(y,r))}
+  r^\sigma[\breve{f}]_{C^{0,\sigma}(B(y,r))}
+  r \|\breve{F}\|_{L^\infty(B(y,r))}\bigg\},
\end{equation}
where $\sigma\in(0,1)$.
Then a routine computation will
will show the estimates for $\|\breve{f}\|_{C^{0,\sigma}(B(y,r))}$
and $\|\breve{F}\|_{L^\infty(B(y,r))}$, and
the rigorous proofs in the following is left to the reader.
\begin{equation}\label{f:5.4}
\begin{aligned}
&~\|\breve{f}\|_{L^\infty(B(y,r))} + \|\breve{F}\|_{L^\infty(B(y,r))}
\leq C\|\mathbf{\Gamma}_0(x,\cdot)\|_{W^{2,\infty}(B(y,r))}
\leq C_k(1+\sqrt{\lambda} r)^{-k}r^{-d},\\
&~[\breve{f}]_{C^{0,\sigma}(B(y,r))}
\leq C\varepsilon^{-\sigma}\|\mathbf{\Gamma}_0(x,\cdot)\|_{W^{2,\infty}(B(y,r))}
+ C\|\mathbf{\Gamma}_0(x,\cdot)\|_{C^{2,\sigma}(B(y,r))}
\leq \frac{C_k\varepsilon^{-\sigma}}{(1+\sqrt{\lambda}r)^kr^{d}},
\end{aligned}
\end{equation}
where $C$ depends on $\mu,\kappa,\tau,\lambda,m,d,\sigma$,
and we use the estimate $\eqref{pri:3.0}$ in the last inequality of each line,
as well as the fact $r>\varepsilon$. Inserting
$\eqref{f:5.4}$ into $\eqref{f:5.3}$ leads to
\begin{equation*}
\begin{aligned}
|\nabla\Phi_\varepsilon^x(y)|
&\leq \frac{C}{r}\Bigg\{\bigg(\dashint_{B(y,r)} |\mathbf{\Gamma}_\varepsilon(x,z)-\mathbf{\Gamma}_0(x,z)|^2 dz\bigg)^{1/2}
+ \varepsilon\|\mathbf{\Gamma}_0(x,\cdot)\|_{W^{1,\infty}(B(y,r))}\Bigg\} \\
&+ \frac{C_k}{(1+\sqrt{\lambda}r)^k}
\bigg\{\varepsilon r^{-d} + \varepsilon^{1-\sigma}r^{\sigma-d}\bigg\}
\leq C\Big\{\varepsilon r^{-d} + \varepsilon^{1-\sigma}r^{\sigma-d}\Big\}
\leq C\varepsilon^{1-\sigma}r^{\sigma-d},
\end{aligned}
\end{equation*}
where the second inequality
follows from the estimate $\eqref{f:5.5}$.
This will give the second line of the stated estimate $\eqref{pri:0.7}$ by letting $\rho = 1-\sigma$,
and the first one will be achieved by the same way.
We have completed the proof.
\qed

\begin{corollary}
Assume the same conditions as in Theorem $\ref{thm:0.4}$.
Let $\mathbf{\Gamma}_{\mathcal{L}}(x,y)$
denote $\mathbf{\Gamma}_{\varepsilon}(x,y)$ in the case of $\varepsilon =1$. Then for any $|x-y|\geq 1$
and $\rho\in(0,1)$, there holds the following estimates£º
\begin{equation}\label{pri:5.9}
\begin{aligned}
|\mathbf{\Gamma}_{\mathcal{L}}(x,y)|&\leq C|x-y|^{1-d}, \\
|\nabla_x\mathbf{\Gamma}_{\mathcal{L}}(x,y)|
+ |\nabla_y \mathbf{\Gamma}_{\mathcal{L}}(x,y)|&\leq C|x-y|^{1-d-\rho},
\end{aligned}
\end{equation}
where $C$ depends on $\mu,\kappa,\tau,\lambda,m,d$ and $\rho$.
\end{corollary}

\begin{proof}
The stated estimates $\eqref{pri:5.9}$ directly follow from
the estimates $\eqref{pri:2.11}$ and $\eqref{pri:2.12}$, coupled with
$\eqref{pri:3.0}$, in which we set $\varepsilon =1$.
\end{proof}

In fact, one may have refined decay estimates.
\begin{thm}
Assume the same conditions as in Theorem $\ref{thm:2.3.1}$.
Let $\mathbf{\Gamma}_\varepsilon(x,y)$ be the fundamental solution of
$\mathcal{L}_\varepsilon$. Then for any $k>0$ there exists a constant
$C_k$ such that
\begin{equation}\label{pri:2.11}
|\mathbf{\Gamma}_\varepsilon(x,y)|
\leq \frac{C_k}{\big(1+\sqrt{\lambda}|x-y|\big)^k|x-y|^{d-2}}
\end{equation}
for any $x,y\in\mathbb{R}^d$ with $x\not=y$,
where $C_k$ depends on $\mu,\kappa,\lambda,m,d,\|A\|_{\emph{VMO}}$ and $k$. Moreover,
if the coefficients $A,V,B$ satisfy $\eqref{a:4}$, then we have
\begin{equation}\label{pri:2.12}
|\nabla_x\mathbf{\Gamma}_\varepsilon(x,y)| + |\nabla_y\mathbf{\Gamma}_\varepsilon(x,y)|
\leq \frac{\widetilde{C}_k}{\big(1+\sqrt{\lambda}|x-y|\big)^k|x-y|^{d-1}}
\end{equation}
and
\begin{equation}\label{pri:2.13}
|\nabla_x\nabla_y\mathbf{\Gamma}_\varepsilon(x,y)|
\leq \frac{\widetilde{C}_k}{\big(1+\sqrt{\lambda}|x-y|\big)^k|x-y|^{d}},
\end{equation}
where $\widetilde{C}_k$ depends on
$\mu,\kappa,\tau,\lambda,m,d$ and $k$.
\end{thm}

\begin{proof}
The main idea may be found in \cite[Theorem 1.14]{SZW26}, and
the stated estimate $\eqref{pri:2.11}$
will directly be derived from the estimates $\eqref{pri:2.3.3}$ and $\eqref{pri:2.1}$, in which
we note that $\mathcal{L}_\varepsilon(\mathbf{\Gamma}_\varepsilon(\cdot,y)) = 0$ in
$\mathbb{R}^d\setminus\{y\}$. Then the estimate $\eqref{pri:2.12}$ is based upon
$\eqref{pri:2.11}$ and
\begin{equation*}
|\nabla_x\mathbf{\Gamma}_\varepsilon(x,y)|
\leq \frac{C_\tau}{r}
\Big(\dashint_{B(x,r)}|\mathbf{\Gamma}_\varepsilon(z,y)|^2dz\Big)^{\frac{1}{2}}
\end{equation*}
which has been done in the estimate $\eqref{pri:0.3}$, where $r=|x-y|/2$, and
$C_\tau$ is independent of $r$ and $\varepsilon$. The estimate related to
$|\nabla_y\mathbf{\Gamma}_\varepsilon(x,y)|$ may be obtained by noting that
$\nabla_y\mathbf{\Gamma}_\varepsilon(x,y) = \nabla_x[\mathbf{\Gamma}_\varepsilon^*(y,x)]^t$.
Furthermore, the estimate $\eqref{pri:2.13}$ may be derived by the observation that
$\mathcal{L}_\varepsilon(\nabla_y\mathbf{\Gamma}_\varepsilon(\cdot,y)) = 0$
in $\mathbb{R}^d\setminus\{y\}$, coupled with $\eqref{pri:2.12}$, and
we have completed the proof.
\end{proof}

\section{Elliptic systems with constant coefficients}\label{sec:3}

Let $\partial/\partial n_0 = n\cdot\widehat{A}\nabla$ be
the conormal derivative related to $\mathcal{L}_0$ on $\partial\Omega$.

\begin{thm}\label{thm:3.1}
Suppose that the coefficients of $\mathcal{L}_0$ satisfy the condition $\eqref{a:5}$ with
$\lambda\geq\max\{\widehat{\lambda},\mu\}$.
Then we have the following results:
\begin{itemize}
  \item for any $g\in L^2(\partial\Omega;\mathbb{R}^m)$, there exists a unique solution
  $u_0\in C^{\infty}(\Omega;\mathbb{R}^m)$ to the Dirichlet problem
  \begin{equation}\label{pde:3.2}
(\mathbf{DH_0})\left\{
\begin{aligned}
\mathcal{L}_0(u_0) &= 0 &\quad &\emph{in}~~\Omega, \\
      u_0 &= g &\emph{n.t.~}&\emph{on} ~\partial\Omega, \\
     (u_0)^* &\in L^2(\partial\Omega), &\quad &
\end{aligned}\right.
\end{equation}
satisfying the nontangential maximal function estimate
$\|(u_0)^*\|_{L^2(\partial\Omega)} \leq C\|g\|_{L^2(\partial\Omega)};$
  \item for any $f\in L^2(\partial\Omega;\mathbb{R}^m)$,
the Neumann problem
\begin{equation}\label{pde:3.1}
(\mathbf{NH_0})\left\{
\begin{aligned}
\mathcal{L}_0(u_0) &= 0 &\quad &\emph{in}~~\Omega, \\
 \frac{\partial u_0}{\partial n_0} &= f &\emph{n.t.~}&\emph{on} ~\partial\Omega, \\
 (\nabla u_0)^* &\in L^2(\partial\Omega) &\quad &
\end{aligned}\right.
\end{equation}
has a unique solution $u_0\in C^\infty(\Omega;\mathbb{R}^m)$, and there holds the estimate
$\|(\nabla u_0)^*\|_{L^2(\partial\Omega)} \leq C\|f\|_{L^2(\partial\Omega)};$
  \item for any $g\in H^1(\partial\Omega;\mathbb{R}^m)$,
  there exists a unique solution
  $u_0$ to the regular problem
  \begin{equation}\label{pde:3.3}
(\mathbf{RH_0})\left\{
\begin{aligned}
\mathcal{L}_0(u_0) &= 0 &\quad &\emph{in}~~\Omega, \\
      u_0 &= g &\emph{n.t.~}&\emph{on} ~\partial\Omega, \\
     (\nabla u_0)^* &\in L^2(\partial\Omega), &\quad &
\end{aligned}\right.
\end{equation}
and it satisfies the estimate
$\|(u_0)^*\|_{L^2(\partial\Omega)} +
\|(\nabla u_0)^*\|_{L^2(\partial\Omega)} \leq C\|g\|_{H^1(\partial\Omega)},$
\end{itemize}
where $C$ depends only on $\mu,\kappa,\lambda,d,m$ and $\Omega$.
\end{thm}

\begin{lemma}
Let $\lambda\geq\widehat{\lambda}$, and $R>0$. Assume that $u_0$ is a solution of
$\mathcal{L}_0(u_0) = 0$ in $B(P,2R)$. Then for any integer $k>0$ and
any multi index $l$ we have the interior estimate
\begin{equation}\label{pri:3.15}
 |\nabla^l u_0(P)| \leq \frac{C_{l,k}}{(1+\sqrt{\lambda}R)^kR^{|l|}}
 \Big(\dashint_{B(P,2R)}|u_0|^2 dx\Big)^{1/2},
\end{equation}
where  $C_{l,k}$ depends on $\mu,d,m,l$ and $k$.
\end{lemma}

\begin{proof}
By the translation and dilation we may assume $P=0$ and $R=1$. Then it follows
from the Sobolev theorem and Caccioppoli's inequality $\eqref{pri:2.1}$ (only replaying $\lambda_0$ by
$\widehat{\lambda}$) that
\begin{equation}\label{f:3.19}
|\nabla^lu_0(0)| \leq \|\nabla^l u_0\|_{L^\infty(B(0,1/2))}
\leq  C\|u_0\|_{H^k(B(0,1/2))} \leq C \|u_0\|_{L^2(B(0,3/2))},
\end{equation}
where $k>|l|+d/2+1$, and $C$ depends only on $\mu,d,m$ and $l$. We mention that
$\nabla^l u_0$ for any $l$ will still be a solution and by this observation we may
repeat using Caccioppoli's inequalities in order.
Let $v(y) = u_0(Ry)$
with $y\in B(0,2)$, and in view of $\eqref{f:3.19}$ we may have
\begin{equation*}
|\nabla^lu_0(0)| \leq \frac{C_l}{R^{|l|}}\Big(\dashint_{B(0,(3/2)R)}|u_0|^2 dx\Big)^{1/2}
\leq \frac{C_l}{R^{|l|}}\Big(\dashint_{B(0,2R)}|u_0|^2 dx\Big)^{1/2}.
\end{equation*}
This together with the estimate $\eqref{f:2.1}$ gives the stated estimate
$\eqref{pri:3.15}$, and we are done.
\end{proof}

Let $\mathbf{\Gamma}_{0}(x,y) = \mathbf{\Gamma}_{0}(x-y)$
denote the fundamental solution of $\mathcal{L}_0$ with pole at $y\in\mathbb{R}^d$, and
it follows from the estimate $\eqref{pri:3.15}$ that
\begin{equation}\label{pri:3.0}
|\nabla^l\mathbf{\Gamma}_{0}(x-y)| \leq
\frac{C_{l,k}}{(1+\sqrt{\lambda}|x-y|)^k|x-y|^{d-2+|l|}}.
\end{equation}

\begin{remark}
\emph{If $m=1$, the fundamental solution may be formulated by
\begin{equation}\label{eq:3.1}
\mathbf{\Gamma}_{0}(x) = -\frac{e^{-\frac{1}{2}(\widehat{B}-\widehat{V})\widehat{A}^{-1} x}}{2^{d/2}\pi^{d/2}
\sqrt{\text{det}\widehat{A}}}\cdot\Big(\frac{x\widehat{A}^{-1}x^t}{L}\Big)
\cdot K_{d/2 -1}(\sqrt{L\cdot x\widehat{A}^{-1}x^t}),
\end{equation}
and $L= \lambda+\frac{1}{4}(\widehat{V}-\widehat{B})
\widehat{A}^{-1}(\widehat{V}-\widehat{B})^{t}$ and
$K_{d/2-1}$ is the modified Hankel function, whose details may be found
in \cite[pp.841-842]{AL} or \cite[pp.167-168]{NOPW}.
The estimate $\eqref{pri:3.0}$ may be derived from $\eqref{eq:3.1}$ straightforwardly
(see \cite[pp.843]{AL}) in such the case.}
\end{remark}

\begin{lemma}[comparing lemma I]\label{lemma:3.2}
Let $d\geq 3$ and  $l\geq 1$, and assume $\lambda\geq\max\{\widehat{\lambda},\mu\}$.
Then we have
\begin{equation}\label{pri:3.4}
\big|\nabla^l [\mathbf{\Gamma}_0(x-y) - \mathbf{\Gamma}_{\widehat{A}}(x-y)]\big|
\leq C|x-y|^{3-d-l}
\end{equation}
for any $x,y\in\mathbb{R}^d$ with
$x\not=y$,
where $C$ depends only on $\mu,\kappa,\lambda,d,m$ and $l$.
\end{lemma}

\begin{proof}
Set $\mathcal{L} = \mathcal{L}_0$ and $\widetilde{\mathcal{L}}
=-\text{div}(\widehat{A}\nabla)$ in Lemma $\ref{lemma:2.3.1}$. Then it follows
from the identity $\eqref{eq:5.1}$ that
\begin{equation*}
\begin{aligned}
&\mathbf{\Gamma}_0^{\alpha\delta}(x-y) - \mathbf{\Gamma}_{\widehat{A}}^{\alpha\delta}(x-y) \\
&\qquad\quad=\int_{\mathbb{R}^d} \Bigg\{\big(\widehat{V}_i^{\beta\gamma} - \widehat{B}_i^{\beta\gamma}\big)
\frac{\partial \mathbf{\Gamma}_{\widehat{A}}^{\gamma\delta}}{\partial z_i}(z-y)
 \mathbf{\Gamma}_{0}^{\alpha\beta}(x-z)
-\widehat{\mathbf{c}}^{\beta\gamma}
\mathbf{\Gamma}_{\widehat{A}}^{\gamma\delta}(z-y)
\mathbf{\Gamma}_0^{\alpha\beta}(x-z) \Bigg\} dz,
\end{aligned}
\end{equation*}
in which $\widehat{\mathbf{c}}= \widehat{c}+\lambda I$, and
noting the decay estimates $\eqref{pri:3.0}$ we also employ integration by parts.
Then differentiating both sides of the above equation with respect to $x$  gives
\begin{equation}
\begin{aligned}
&\nabla_{x_k}\mathbf{\Gamma}_0^{\alpha\delta}(x-y) -
\nabla_{x_k}\mathbf{\Gamma}_{\widehat{A}}^{\alpha\delta}(x-y) \\
&\qquad\quad=\int_{\mathbb{R}^d} \Bigg\{\big(\widehat{V}_i^{\beta\gamma} - \widehat{B}_i^{\beta\gamma}\big)
\frac{\partial \mathbf{\Gamma}_{\widehat{A}}^{\gamma\delta}}{\partial z_i}(z-y)
 \frac{\partial \mathbf{\Gamma}_{0}^{\alpha\beta}}{\partial x_k}(x-z)
-\widehat{\mathbf{c}}^{\beta\gamma}
\mathbf{\Gamma}_{\widehat{A}}^{\gamma\delta}(z-y)
\frac{\partial \mathbf{\Gamma}_0^{\alpha\beta}}{\partial x_k}(x-z) \Bigg\} dz,
\end{aligned}
\end{equation}
and this implies
\begin{equation}\label{f:3.8}
\begin{aligned}
\big|\nabla_{x}\mathbf{\Gamma}_0(x-y)
- \nabla_{x}\mathbf{\Gamma}_{\widehat{A}}(x-y)\big|
&\leq C\underbrace{\int_{\mathbb{R}^d} \frac{dz}{|z-y|^{d-1}|x-z|^{d-1}}}_{T_1}
+ C\underbrace{\int_{\mathbb{R}^d}
\frac{|\nabla_x\mathbf{\Gamma}_0(x-z)|dz}{|z-y|^{d-2}}}_{T_2} \\
&\leq C|x-y|^{2-d}.
\end{aligned}
\end{equation}
The computations for $T_1$ and $T_2$ could be done as follows. Let $r=|x-y|>0$ and
$Q=(x+y)/2\in\mathbb{R}^d$, then
the integral domain $\mathbb{R}^d$ decomposes into $B(x,r/4)$, $B(y,3r/4)$ and
$\mathbb{R}^d\setminus(B(x,r/4)\cup B(y,3r/4))$.
\begin{equation}
\begin{aligned}
T_1 &\leq \bigg\{\int_{B(x,r/4)}+\int_{B(y,3r/4)}
+\int_{\mathbb{R}^d\setminus(B(x,r/4)\cup B(y,3r/4))}\bigg\}
\frac{dz}{|x-z|^{d-1}|z-y|^{d-1}}\\
&\leq Cr^{1-d}\Big\{\int_0^{r/4} ds + \int_0^{3r/4} ds\Big\}
+ C\int_{r/4}^\infty\frac{ds}{s^{d-1}}\\
&\leq Cr^{2-d},
\end{aligned}
\end{equation}
and
\begin{equation}
\begin{aligned}
T_2 \leq Cr^{2-d}\int_0^{r/4} ds + Cr^{1-d}\int_0^{3r/4} sds
+ C\lambda^{-\frac{1}{2}}\int_{r/4}^\infty\frac{ds}{s^{d-1}}
\leq Cr^{2-d},
\end{aligned}
\end{equation}
where we use the geometry facts in Remark $\ref{remark:3.1.1}$.

We now turn to study the cases $l\geq 2$, and Fixed $x\in\mathbb{R}^d$,
let $G^x(y) = \mathbf{\Gamma}_{\widehat{A}}(x-y)-\mathbf{\Gamma}_0(x-y)$, and
without loss of generality, consider the following equation
\begin{equation*}
 L_0(G^x) = \text{div}(\widehat{A}\nabla\mathbf{\Gamma}_{0})
 \qquad \text{in}\quad B(y,r/2).
\end{equation*}
Since its coefficients are constant,
taking derivatives of $(l-1)$ order, we may have
\begin{equation*}
 L_0(\nabla^{l-1}G^x) = \text{div}(\widehat{A}\nabla^l\mathbf{\Gamma}_{0})
 \qquad \text{in}\quad B(y,r/2).
\end{equation*}
Note  that $|z-x|\geq (r/2)$ for any $z\in B(y,r/4)$.

Hence, it follows from an interior estimate that
\begin{equation}\label{f:3.9}
\begin{aligned}
|\nabla^l G^x(y)|
&\leq \frac{C}{r}\dashint_{B(y,r/4)}|\nabla^{l-1} G^x(z)| dz
+ Cr \bigg(\dashint_{B(y,r/4)}\big(|\nabla^l\mathbf{\Gamma}_{0}(x-z)|^q dz\bigg)^{\frac{1}{q}} \\
&\leq \frac{C}{r}\dashint_{B(y,r/4)}|\nabla^{l-1} G^x(z)| dz + Cr^{3-d-l}
\end{aligned}
\end{equation}
with $q>d$, where we employ the estimate $\eqref{pri:3.0}$ in the last inequality.
Due to the earlier estimate $\eqref{f:3.8}$,
it is not hard to see that $|\nabla^2 G^x|\leq Cr^{1-d}$ when $l=2$,
and the stated estimate $\eqref{pri:3.4}$ follows from
mathematical induction on $l$ through $\eqref{f:3.9}$. We have completed the proof.
\end{proof}

\begin{remark}\label{remark:3.1.1}
\emph{Let $x,y\in\mathbb{R}^d$ with $x\not=y$. Set $r=|x-y|$, and $Q=(x+y)/2\in\mathbb{R}^d$, and
we refer the reader to the following geometry facts:
\begin{enumerate}
  \item[(1)] $|z-y|>(3r/4)$ whenever $z\in B(x,r/4)$;
  \item[(2)] $|x-z|>(r/4)$ if $z\in B(y,3r/4)$;
  \item[(3)] $\frac{1}{3}|Q-z|<|x-z|\leq 3|Q-z|$ and $\frac{3}{5}|Q-z|<|z-y|\leq 3|Q-z|$ whenever
  $z\in\mathbb{R}^d\setminus (B(x,r/4)\cup B(y,3r/4))$, which means
  $|x-z|\approx |z-Q|$ and $|y-z|\approx |z-Q|$ in such the case;
  \item[(4)] $\mathbb{R}^d\setminus (B(x,r/4)\cup B(y,3r/4)) \subset \{z\in\mathbb{R}^d:|z-Q|>(r/4)\}$.
\end{enumerate}}
\end{remark}

\subsection{Nontangential maximal function estimates}

Let $\mathrm{M}_{\partial\Omega}$ be the Hardy-Littlewood maximal operator on $\partial\Omega$.
We define the radial maximal function $\mathcal{M}(h)$ on $\partial\Omega$ as
\begin{equation*}
\mathcal{M}(h)(Q) = \sup\big\{|h(T_r(Q))|:0<r\leq R_0/100\big\},
\end{equation*}
where $T_r:\partial\Omega\to\partial\Sigma_r$ are bi-Lipschitz maps, and
$\Sigma_r=\{x\in\Omega:\text{dist}(x,\partial\Omega)>r\}$.
We refer the reader to
\cite[Remark 2.18]{QX2} for the details.

\begin{lemma}\label{lemma:3.7}
Assume that $u_0$ satisfy $\mathcal{L}_0(u_0) = 0$ in $\Omega$, then
for any $Q\in\partial\Omega$, we have
\begin{equation}\label{pri:3.2}
 (\nabla u_0)^*(Q)
 \leq C\mathrm{M}_{\partial\Omega}(\mathcal{M}(\nabla u_0))(Q),
\end{equation}
where $C$ depends on $\mu,\kappa,d,m$ and $\Omega$.
\end{lemma}

\begin{proof}
Fix a point $Q\in\partial\Omega$, and let $\Lambda_{N_0}(Q)$ be a cone with $Q$ as the vertex.
For any $x\in \Lambda_{N_0}(Q)$, the distances between $|x-Q|$
and $\delta(x)=\text{dist}(x,\partial\Omega)$ are comparable. Let $r=\delta(x)$.
By definition, it follows from the interior estimate that
\begin{equation*}
\begin{aligned}
|\nabla u_0(x)| & \leq C\dashint_{B(x,r)}|\nabla u_0(y)| dy \\
&\leq C\dashint_{B(Q,N_0r)\cap\partial\Omega}|\mathcal{M}(\nabla u_0)| dS
\leq C\mathrm{M}_{\partial\Omega}(\mathcal{M}(\nabla u_0))(Q),
\end{aligned}
\end{equation*}
where $N_0$ is independent of $Q$. This gives the desired estimate $\eqref{pri:3.2}$, and
we are done.
\end{proof}

\begin{lemma}
Let $\Omega\subset\mathbb{R}^d$ be a bounded
Lipschitz domain, and $\mathcal{M}$ be defined
as the radical maximal function operator. Then
for any $h\in H^1(\Omega)$, we have the following estimate
\begin{equation}\label{pri:3.2.1}
\|\mathcal{M}(h)\|_{L^2(\partial\Omega)}
\leq C\|h\|_{H^1(\Omega)}
\end{equation}
where $C$ depends only on $d$ and
the character of $\Omega$.
\end{lemma}

\begin{proof}
See \cite[Lemma 2.24]{QX2}.
\end{proof}

\begin{thm}[nontangential maximal function estimate]\label{thm:3.2}
Given $g\in L^2(\partial\Omega;\mathbb{R}^m)$,
let $u_0$ be the solution of $\eqref{pde:3.1}$.
Then we have the nontangential maximal function estimate
\begin{equation}\label{pri:3.1}
\|(\nabla u_0)^*\|_{L^2(\partial\Omega)} \leq C\|g\|_{L^2(\partial\Omega)},
\end{equation}
where $C$ depends on $\mu,\kappa,\lambda,d,m$ and $\Omega$.
\end{thm}

\begin{proof}
Let $L_0 = -\text{div}(\widehat{A}\nabla)$.
By moving the lower order terms of $\mathcal{L}_0$ to the right-hand side of $\eqref{pde:3.1}$,
we may rewrite $(\mathbf{NH}_0)$ into
\begin{equation*}
 L_0(u_0) = (\widehat{V} -\widehat{B})\nabla u_0 - (\widehat{c}+\lambda I)u_0 \quad \text{in}~\Omega,
 \qquad \partial u_0/\partial n_0 = g - n\cdot \widehat{V} u_0 \quad \text{on}~\partial\Omega.
\end{equation*}
Moreover, we consider $u_0 = v+w$, and they satisfy
\begin{equation}
 (\text{i})~ L_0(v) = \breve{F} \quad \text{in}~\mathbb{R}^d,
 \qquad\quad
 (\text{ii})~ \left\{\begin{aligned}
 L_0(w) &= 0  &\quad&\text{in}~~\Omega,\\
     \frac{\partial w}{\partial n_0}  &= g - \frac{\partial v}{\partial n_0} &\quad&\text{on}~\partial\Omega,
 \end{aligned}\right.
\end{equation}
where $\breve{F} = (\widehat{V} -\widehat{B})\nabla u_0 - (\widehat{c}+\lambda I)u_0 $ in $\Omega$ and $\breve{F} = 0$
on $\mathbb{R}^d\setminus\Omega$. Note that the existence of $w$
may be found in \cite[Theorem 2.2]{GaoW}.

Let $\Gamma_{\widehat{A}}$ denote the fundamental solution of $L_0$, and then we have
$v = \Gamma_{\widehat{A}}*\breve{F}$ in $\mathbb{R}^d$. It follows from
the well known singular integral and fractional integral estimates
(see for example \cite{MGLM}) that
\begin{equation*}
\|\nabla^2 v\|_{L^2(\mathbb{R}^d)} + \|\nabla v\|_{L^2(\mathbb{R}^d)} \leq C\|u_0\|_{H^1(\Omega)}.
\end{equation*}
Hence, this together with the estimate $\eqref{pri:3.2.1}$ and the trace theorem gives
\begin{equation}\label{f:3.1}
\|\mathcal{M}(\nabla v)\|_{L^2(\partial\Omega)}
+ \|\nabla v\|_{L^2(\partial\Omega)}\leq C\|u_0\|_{H^1(\Omega)}.
\end{equation}

We now turn to study $(\text{ii})$.
In view of \cite[Theorem 2.2]{GaoW}, it is not hard to derive
\begin{equation}\label{f:3.2}
\begin{aligned}
\|\mathcal{M}(\nabla w)\|_{L^2(\partial\Omega)}
&\leq \|(\nabla w)^*\|_{L^2(\partial\Omega)}
\leq C\Big\{\|g\|_{L^2(\partial\Omega)}
+\|u_0\|_{L^2(\partial\Omega)}
+ \|\nabla v\|_{L^2(\partial\Omega)}\Big\} \\
&\leq C\Big\{\|g\|_{L^2(\partial\Omega)}
+\|u_0\|_{H^1(\Omega)}\Big\},
\end{aligned}
\end{equation}
where we also employ the trace theorem and $\eqref{f:3.1}$.
Obviously,
combining the estimates $\eqref{f:3.1}$ and $\eqref{f:3.2}$ will
lead to
\begin{equation*}
\|\mathcal{M}(\nabla u_0)\|_{L^2(\partial\Omega)}
\leq C\Big\{\|g\|_{L^2(\partial\Omega)}
+\|u_0\|_{H^1(\Omega)}\Big\}.
\end{equation*}

This together with the estimate $\eqref{pri:3.2}$ further
show the desired estimate $\eqref{pri:3.1}$,
and we have completed the
proof.
\end{proof}

Similarly, we have the following nontangential maximal function
estimate for the regular problem.

\begin{thm}\label{thm:3.4}
Given $g\in H^1(\partial\Omega;\mathbb{R}^m)$,
let $u_0$ be the solution of $\eqref{pde:3.3}$.
Then we have the nontangential maximal function estimate
\begin{equation}\label{pri:3.6}
\|(\nabla u_0)^*\|_{L^2(\partial\Omega)} \leq C\|g\|_{H^1(\partial\Omega)},
\end{equation}
where $C$ depends on $\mu,\kappa,\lambda,d,m$ and $\Omega$.
\end{thm}

\begin{proof}
The proof of this result is quite similar to that given for Theorem $\ref{thm:3.2}$
and so is omitted.
\end{proof}

\begin{lemma}
[localization]\label{lemma:3.4}
Let $\lambda\geq\max\{\widehat{\lambda},\mu\}$, and
$u_0$ be the solution to $\mathcal{L}_0(u_0) = 0$ in $\Omega$
with $(\nabla u_0)^*\in L^2(\partial\Omega)$. We assume that
$\nabla u_0$ have nontangential limits almost everywhere on $\partial\Omega$. Then we have
\begin{equation}\label{pri:3.8}
\int_{\partial\Omega} \big(|\nabla u_0|^2 + |u_0|^2\big)dS
\leq C\int_{\partial\Omega}\Big|\frac{\partial u_0}{\partial n_0}\Big|^2 dS
\end{equation}
and
\begin{equation}\label{pri:3.9}
\int_{\partial\Omega} |\nabla u_0|^2 dS
\leq C\int_{\partial\Omega}|\nabla_{\emph{tan}}u_0|^2 dS
+ C\int_{\Omega}(|\nabla u_0|^2 + |u_0|^2) dx,
\end{equation}
where $C$ depends on $\mu,\kappa,\lambda,d,m$ and $\Omega$.
\end{lemma}

\begin{proof}
In view of the estimate $\eqref{pri:3.1}$, it is clear to see that
\begin{equation*}
\|\nabla u_0\|_{L^2(\partial\Omega)}
\leq \|(\nabla u_0)^*\|_{L^2(\partial\Omega)}
\leq C\|(\partial u_0/\partial n_0)\|_{L^2(\partial\Omega)}.
\end{equation*}

Set $h\in C_0^1(\mathbb{R}^d;\mathbb{R}^d)$ such that $\big<h,n\big>\geq c>0$ on $\partial\Omega$.
By the divergence theorem, we have
\begin{equation}\label{f:3.5}
\int_{\partial\Omega}|u_0|^2 dS
\leq C\int_{\Omega} \big(|\nabla u_0||u_0| + |u_0|^2 \big)dx
\leq C\|u_0\|_{H^1(\Omega)}^2,
\end{equation}
which may be referred to as the trace theorem.
On the other hand, it follows from the estimate
$\eqref{pri:2.2.2}$ that
\begin{equation*}
 \frac{\lambda}{2}
 \|u_0\|_{H^1(\Omega)}^2
 \leq \mathrm{B}_{\mathcal{L}_0;\Omega}[u_0,u_0]
 =\int_{\partial\Omega}\frac{\partial u_0}{\partial n_0} u_0 dS
 \leq \|(\partial u_0/\partial n_0)\|_{L^2(\partial\Omega)}
 \| u_0\|_{L^2(\partial\Omega)}.
\end{equation*}
Plugging it back into $\eqref{f:3.5}$ leads to
\begin{equation*}
\| u_0\|_{L^2(\partial\Omega)} \leq C\|(\partial u_0/\partial n_0)\|_{L^2(\partial\Omega)}.
\end{equation*}

We proceed to prove the estimate $\eqref{pri:3.9}$.
The main idea is based upon \cite[Remark 3.1]{SZW12}. Let $D_r = B(P,r)\cap\Omega$ with
$P\in\partial\Omega$, and $\Delta_r = B(P,r) \cap \partial\Omega$, where $r\in[1/4,1)$.
Since $\mathcal{L}_0(u_0) = 0$ in $D_r$, it follows from the estimate $\eqref{pri:3.6}$ that
\begin{equation*}
\begin{aligned}
\int_{\partial D_r}|\nabla u_0|^2 dS
&\leq C\int_{\partial D_r}\big(|\nabla_{\text{tan}} u_0|^2 + |u_0|^2\big) dS \\
&\leq C\int_{\partial\Omega}|\nabla_{\text{tan}}u_0|^2 dS
+ C\int_{\partial D_r\setminus\Delta_1} |\nabla_{\text{tan}}u_0|^2 dS
+ C\int_{\partial D_r}|u_0|^2 dS,
\end{aligned}
\end{equation*}
where we actually employ the estimate $\eqref{f:3.5}$ in $D_r$,
and this implies
\begin{equation*}
\int_{\Delta_{1/4}}|\nabla u_0|^2 dS
\leq C\int_{\partial\Omega}|\nabla_{\text{tan}}u_0|^2 dS
+ C\int_{\partial D_r\setminus\Delta_1}|\nabla u_0|^2 dS
+ C\int_{D_r}\big(|\nabla u_0|^2+|u_0|^2\big) dx.
\end{equation*}
Integrating both sides above with respect to $r$ from $1/4$ to 1, we acquire
\begin{equation*}
\int_{\Delta_{1/4}}|\nabla u_0|^2 dS
\leq C\int_{\partial\Omega} |\nabla_{\text{tan}}u_0|^2 dS
+ C\int_{\Omega}\big( |\nabla u_0|^2 + |u_0|^2\big) dx.
\end{equation*}

The covering technique finally gives the desired estimate $\eqref{pri:3.9}$, and
the proof is complete.
\end{proof}

Recall that the notation
$\Omega_{-}=\mathbb{R}^d\setminus\Omega$ denotes
the exterior of $\Omega$.

\begin{lemma}
[localization for the exterior of domain]
\label{lemma:3.5}
Let $\lambda\geq\max\{\widehat{\lambda},\mu\}$.
Suppose that $u_0$ satisfies $\mathcal{L}_0(u_0) = 0$ in $\Omega_{-}$
with $(\nabla u_0)^*\in L^2(\partial\Omega)$, and
$\nabla u_0$ exists in the sense of nontangential convergence on $\partial\Omega$. We further assume that
$|u(x)|=O(|x|^{2-d})$ with $|\nabla u(x)|=O(|x|^{1-d})$ as $|x|\to\infty$. Then, there holds
\begin{equation}\label{pri:3.5}
\begin{aligned}
&\int_{\partial\Omega} |\nabla u_0|^2 dS
\leq C\int_{\partial\Omega}|\nabla_{\emph{tan}}u_0|^2 dS
+ C\int_{\Omega_{-}}(|\nabla u_0|^2 + |u_0|^2) dx \\
&\int_{\partial\Omega} |\nabla u_0|^2 dS
\leq C\int_{\partial\Omega}\Big|\frac{\partial u_0}{\partial n_0}\Big|^2 dS
+ C\int_{\Omega_{-}}|\nabla u_0|^2 dx
\end{aligned}
\end{equation}
and
\begin{equation}\label{pri:3.7}
 \int_{\partial\Omega}|u|^2 dS
 \leq C\int_{\partial\Omega} \Big|\frac{\partial u_0}{\partial n_0}\Big|^2dS
\end{equation}
where $C$ depends on $\mu,\kappa,\lambda,d,m$ and $\Omega$.
\end{lemma}

\begin{proof}
An argument similar to the one used in the proof of the estimate $\eqref{pri:3.8}$ shows
the first line of $\eqref{pri:3.5}$, and it will not be reproduced here.
However we will provide a proof for the second line of $\eqref{pri:3.5}$ for the sake of the completeness.
Let $D_r = B(P,r)\cap\Omega_{-}$ with $P\in\partial\Omega$, and $\Delta_r = B(P,r) \cap \partial\Omega$,
where $r\in[1/4,1)$. Since $\mathcal{L}_0(u_0) = 0$ in $D_r$,
in view of the estimate $\eqref{pri:3.1}$ we have
\begin{equation*}
\begin{aligned}
\int_{\partial D_r}|\nabla u_0|^2 dS
&\leq C\int_{\partial D_r}\Big|\frac{\partial u_0}{\partial n_0}\Big|^2 dS \\
&\leq C\int_{\partial\Omega}\Big|\frac{\partial u_0}{\partial n_0}\Big|^2 dS
+ C\int_{\partial D_r\setminus\Delta_1} |\nabla u_0|^2 dS,
\end{aligned}
\end{equation*}
and this gives
\begin{equation*}
\int_{\Delta_{1/4}}|\nabla u_0|^2 dS
\leq C\int_{\partial\Omega}\Big|\frac{\partial u_0}{\partial n_0}\Big|^2 dS
+ C\int_{\partial D_r\setminus\Delta_1} |\nabla u_0|^2 dS
\end{equation*}
Integrating both sides above with respect to $r$ from $1/4$ to 1 and then
using the covering technique, we consequently obtain
the second line of $\eqref{pri:3.5}$.

We now turn to show the estimate $\eqref{pri:3.7}$.
Since $\partial\Omega$ is compact in $\mathbb{R}^d$,
it may be covered by finite balls centered at $\partial\Omega$ and intersected by $\Omega_{-}$.
Let $D_{1/4}$ be one of them, and $\phi\in C_0^1(B(P,1/4))$ be a cut-off function such that
$\phi =1$ in $B(P,1/8)$ and $\phi=0$ outside $B(P,3/16)$ with $|\nabla\phi|\leq C$. Hence,
proceeding as in the estimate $\eqref{f:3.5}$, we obtain
\begin{equation}\label{f:3.6}
\int_{\Delta_{1/4}}|u_0|^2 dS
\leq \int_{\partial D_{1/4}}|\phi u_0|^2 dS
\leq C\int_{D_{1/4}} \big(|u_0|^2 + |\nabla u_0|^2\big) dx
\leq C\int_{\Omega_{-}} \big(|u_0|^2 + |\nabla u_0|^2\big) dx.
\end{equation}
The problem is reduced to estimate the most right-hand side above.
Let $B(0,R)\setminus\Omega\subset \Omega_{-}$, and it is not hard to see that
\begin{equation*}
\mathrm{B}_{\mathcal{L}_0;B(0,R)\setminus\Omega}[u_0,u_0]
= \int_{\partial B(0,R)}\frac{\partial u_0}{\partial n_0} u_0 dS
-\int_{\partial \Omega}\frac{\partial u_0}{\partial n_0}u_0 dS.
\end{equation*}
Since
\begin{equation*}
\Big|\int_{\partial B(0,R)}\frac{\partial u_0}{\partial n_0} u_0 dS\Big|
\leq CR^{2-d}
\end{equation*}
goes to $0$ as $R\to\infty$, we may derive
\begin{equation}\label{f:3.14}
\frac{\lambda}{2}\int_{\Omega_{-}}\big(|u_0|^2 + |\nabla u_0|^2\big) dx
\leq \mathrm{B}_{\mathcal{L}_0;\Omega_{-}}[u_0,u_0] =
-\int_{\partial \Omega}\frac{\partial u_0}{\partial n_0}u_0 dS,
\end{equation}
where we use the estimate $\eqref{pri:2.2.2}$ in the first step. This coupled with the estimate $\eqref{f:3.6}$
leads to
\begin{equation*}
\int_{\partial\Omega}|u_0|^2 dS \leq C\int_{\partial \Omega}
\Big|\frac{\partial u_0}{\partial n_0}u_0\Big| dS.
\end{equation*}
Consequently, the desired estimate $\eqref{pri:3.7}$ will be done by Cauchy's inequality, and we have
completed the whole proof.
\end{proof}

\begin{corollary}
Assume the same conditions as in Lemma $\ref{lemma:3.5}$. Then we have
\begin{equation}\label{pri:3.13}
\int_{\partial\Omega}\big(|(\nabla u_0)_{-}|^2 + |(u_0)_{-}|^2 \big)dS
\leq C\int_{\partial\Omega}\Big|\Big(\frac{\partial u_0}{\partial n_0}\Big)_{-}\Big|^2dS,
\end{equation}
where the subscript
``$-$'' indicate nontangential limits taken outside $\overline{\Omega}$,
and $C$ depends on $\mu,\kappa,\lambda,d,m$ and $\Omega$.
\end{corollary}

\begin{proof}
It follows from the estimate $\eqref{f:3.14}$ that
\begin{equation*}
\int_{\Omega_{-}} |\nabla u_0|^2 dx
\leq \int_{\partial\Omega}\Big|\frac{\partial u_0}{\partial n_0}\Big|
\big|u_0\big|dS,
\end{equation*}
and this together with the second line of $\eqref{pri:3.5}$ and $\eqref{pri:3.7}$
leads to the stated estimate $\eqref{pri:3.13}$.
\end{proof}

\subsection{Estimates for layer potentials}

Given $f\in L^p(\partial\Omega;\mathbb{R}^m)$ with $1<p<\infty$, the single layer potential is
defined by
\begin{equation}\label{def:3.1}
 \mathcal{S}_{\Theta}(f)(x) = \int_{\partial\Omega} \mathbf{\Gamma}_{\Theta}(x-y)f(y)dS(y),
\end{equation}
and the double layer potential is in the form of
\begin{equation}\label{def:3.2}
\mathcal{D}_{\Theta}(f)(x) = \int_{\partial\Omega}
\frac{\partial}{\partial n_{0}^*(y)}\big\{\mathbf{\Gamma}_{\Theta}(x-y)\big\}f(y)dS(y),
\end{equation}
where $\partial/\partial n_{0}^*(y) = n(y)\cdot \widehat{A}^*\nabla_y$.

Throughout this section, the subscript $\Theta$ may be given by $0$ or $\widehat{A}$.
If $\Theta = 0$, notation with such the subscript means that they are related to the operator
$\mathcal{L}_0$. In the case of $\Theta = \widehat{A}$, those symbols with this subscript are
associated with the homogeneous operator $L_0$.
For example, $\mathbf{\Gamma}_{0}$ and $\mathbf{\Gamma}_{\widehat{A}}$
represent the fundamental solution of $\mathcal{L}_0$ and $L_0$, respectively.
So do the above stated definitions of the single and double potential layers .

Define the truncated singular integral
\begin{equation}
\mathcal{T}_{\Theta}^\delta(f)(P) = \int_{y\in\partial\Omega\atop
|y-P|>\delta} \nabla \mathbf{\Gamma}_{\Theta}(P-y)f(y)dS(y),
\end{equation}
and then the singular integral operator and the associated maximal singular integral one may
be denoted by
\begin{equation}
\begin{aligned}
\mathcal{T}_{\Theta}(f)(P) & = \text{p.v.}\int_{\partial\Omega}
\nabla \mathbf{\Gamma}_{\Theta}(P-y)f(y)dS(y)
:= \lim_{\delta\to 0} \mathcal{T}_{\Theta}^\delta(f)(P), \\
&\qquad~\mathcal{T}_{\Theta}^*(f)(P) = \sup_{\delta>0}|T_{\Theta}^\delta(f)(P)|,
\end{aligned}
\end{equation}
respectively.

\begin{thm}\label{lemma:3.3}
Let $f\in L^p(\partial\Omega;\mathbb{R}^m)$ for $1<p<\infty$, and
assume $\lambda\geq\max\{\widehat{\lambda},\mu\}$.
Then $\mathcal{T}_{\Theta}(f)$ exists for a.e. $P\in\partial\Omega$ such that
\begin{equation}\label{pri:3.3}
\|\mathcal{T}_{\Theta}(f)\|_{L^p(\partial\Omega)} + \|\mathcal{T}_{\Theta}^*(f)\|_{L^p(\partial\Omega)}
\leq C_\Theta\|f\|_{L^p(\partial\Omega)},
\end{equation}
holds for $\Theta =0,\widehat{A}$,
in which $C_0$ depends only on $\mu,\kappa,\lambda,d,m,p$ and $\Omega$.
\end{thm}

\begin{proof}
The original idea may be found in \cite[Lemma 3.1]{SZW23}, and we provide a proof for the
sake of the completeness. Note that if we choose $\Theta = \widehat{A}$,
then the result $\eqref{pri:3.3}$ had already been established in \cite[Theorem 1.1]{GaoW}.
We now study the case of $\Theta = 0$. It is sufficient to estimate the integral
\begin{equation*}
 \Big|\int_{y\in\partial\Omega\atop
 |y-P|>\delta} \nabla \mathbf{\Gamma}_0(P-y)f(y)dS(y)\Big|,
\end{equation*}
and it will confront with two cases: (1) $\delta \geq 1/\sqrt{\lambda}$; (2) $\delta < 1/\sqrt{\lambda}$.
For (1), it follows from the estimate $\eqref{pri:3.0}$ that
\begin{equation}\label{f:3.3}
 \Big|\int_{y\in\partial\Omega\atop
 |y-P|>\delta} \nabla \mathbf{\Gamma}_0(P-y)f(y)dS(y)\Big|
 \leq C\int_{y\in\partial\Omega\atop
 |y-P|>\delta} \frac{|f(y)|}{|P-y|^{d}}dS(y) \leq C\mathrm{M}_{\partial\Omega}(f)(P).
\end{equation}
We proceed to investigate (2). In such the case, it is not hard to see that
\begin{equation}\label{f:3.4}
\begin{aligned}
 \Big|\int_{y\in\partial\Omega\atop
 |y-P|>\delta} \nabla \mathbf{\Gamma}_0(P-y)f(y)dS(y)\Big|
& \leq  \Big|\int_{y\in\partial\Omega\atop
 |y-P|\geq (1/\sqrt{\lambda})} \nabla \mathbf{\Gamma}_0(P-y)f(y)dS(y)\Big| \\
& +  \int_{y\in\partial\Omega\atop
 \delta <|y-P|< (1/\sqrt{\lambda})} \big|\nabla \mathbf{\Gamma}_{0}(P-y)
 - \nabla\mathbf{\Gamma}_{\widehat{A}}(P-y)\big||f(y)|dS(y) \\
& +  \Big|\int_{y\in\partial\Omega\atop
 \delta<|y-P|<(1/\sqrt{\lambda})} \nabla \mathbf{\Gamma}_{\widehat{A}}(P-y)f(y)dS(y)\Big| \\
&\leq 2\mathcal{T}_{\widehat{A}}^*(f)(P) + C\mathrm{M}_{\partial\Omega}(f)(P),
\end{aligned}
\end{equation}
where we use the estimate $\eqref{pri:3.4}$ in the last inequality. Combining the estimates
$\eqref{f:3.3}$ and $\eqref{f:3.4}$, we have
\begin{equation*}
 \mathcal{T}_0^*(f)(P) \leq \mathcal{T}_{\widehat{A}}^*(f)(P) + C\mathrm{M}_{\partial\Omega}(f)(P),
\end{equation*}
and this consequently implies the estimate $\eqref{pri:3.3}$. The proof is complete.
\end{proof}

\begin{lemma}\label{lemma:3.1}
Assume the same conditions as in Lemma $\ref{lemma:3.3}$.
Then for a.e. $P\in\partial\Omega$,
we have
\begin{equation}\label{Id:2.1}
\big(\nabla \mathcal{S}_{\Theta}(f)\big)_{\pm}(P)
= \pm\frac{1}{2} n(P)\mathbf{H}(n(P))f(P) + \emph{p.v.}\int_{\partial\Omega} \nabla \mathbf{\Gamma}_{\Theta}(P-y)f(y)dS(y)
\end{equation}
for $\Theta = 0, \widehat{A}$,
where $\mathbf{H}(n) = (\widehat{a}_{ij}^{\alpha\beta}n_in_j)^{-1}_{m\times m}$, and the subscripts
``$+$'' and ``$-$'' indicate nontangential limits taken inside $\Omega$ and outside $\overline{\Omega}$.
Moreover, we have
\begin{equation}\label{Id:2.2}
\Big(\frac{\partial \mathcal{S}_{\Theta}(f)}{\partial n_0}\Big)_{\pm} =
\Big(\pm\frac{1}{2}I + \mathcal{K}_{\Theta}\Big)(f)
\quad\emph{on}~\partial\Omega,
\end{equation}
where the integral operator $\mathcal{K}_{\Theta}$ is defined by
\begin{equation*}
\small
 \mathcal{K}_{\Theta}(f)(P)
 =\emph{p.v.}\int_{\partial\Omega} \frac{\partial}{\partial n_{0}(P)}
 \Big\{\mathbf{\Gamma}_{\Theta}(P-y)\Big\}f(y) dS(y),
\end{equation*}
and $\partial/\partial n_{0}(P) = n(P)\widehat{A}\nabla$.
\end{lemma}

\begin{proof}
We first mention that in the case of $\Theta = \widehat{A}$,
the identities $\eqref{Id:2.1}$ and $\eqref{Id:2.2}$
have been well known in \cite[Lemma 1.4]{GaoW}, while we focus on the case $\Theta = 0$ here.
We can employ the idea developed in
in \cite[Lemma 2.3]{SZW23} or \cite[Theorem 4.4]{SZW24} to prove our results directly. However,
it is possible to provide another one due to the constant coefficients, and the reader will realize
the benefits in the later sections although the idea is not very new in today's view.

To obtain $\eqref{Id:2.1}$, let $r = |x-P|$ where $x\in\Lambda_{N_0}^{\pm}(P)$.
For any $t>0$, we compute the following quantity
\begin{equation*}
\Big|\nabla\mathcal{S}_0(f)(x) - \nabla\mathcal{S}_{\widehat{A}}(f)(x)
- \mathcal{T}_0^{tr}(f)(P) + \mathcal{T}_{\widehat{A}}^{tr}(f)(P)\Big|,
\end{equation*}
which is controlled by
\begin{equation*}
\begin{aligned}
&\int_{y\in\partial\Omega\atop
|y-P|>tr} \big|\nabla\mathbf{\Gamma}_0(x-y)-\nabla\mathbf{\Gamma}_0(P-y)\big||f(y)|dS(y) \\
&+\int_{y\in\partial\Omega\atop
|y-P|>tr} \big|\nabla\mathbf{\Gamma}_{\widehat{A}}(x-y)-\nabla\mathbf{\Gamma}_{\widehat{A}}(P-y)\big|
|f(y)|dS(y) \\
&+\int_{y\in\partial\Omega\atop
|y-P|\leq tr} \big|\nabla\mathbf{\Gamma}_0(x-y)-\nabla\mathbf{\Gamma}_{\widehat{A}}(x-y)\big|
|f(y)|dS(y) := I_1 + I_2 + I_3.
\end{aligned}
\end{equation*}
In fact, the calculation on $I_1$ is similar to that on $I_2$, and we take $I_1$ for example.
\begin{equation*}
\begin{aligned}
I_1 \leq Cr\int_{y\in\Omega\atop
|P-y|>tr} \frac{|f(y)|dS(y)}{|P-y|^d}
\leq Cr\sum_{k=0}^{\infty}\frac{(2^{k}tr)^{-1}}{(2^ktr)^{d-1}}\int_{y\in\Omega\atop
|P-y|\leq 2^{k+1}tr}|f(y)| dS(y)
\leq \frac{C}{t}\mathrm{M}_{\partial\Omega}(f)(P),
\end{aligned}
\end{equation*}
where we use the facts
\begin{equation*}
\big|\nabla\mathbf{\Gamma}_0(x-y)-\nabla\mathbf{\Gamma}_0(P-y)\big|
\leq Cr\max_{z\in\Lambda_{N_0}^{\pm}(P)}|\nabla^2\mathbf{\Gamma}_0(z-y)|
\end{equation*}
and
\begin{equation*}
|P-y|\leq |P-z| + |z-y| \leq (N_0+1)|z-y|,
\end{equation*}
as well as the estimate $\eqref{pri:3.0}$ in the first inequality.
By the same token, it follows from the estimate $\eqref{pri:3.4}$ that
\begin{equation*}
I_3 \leq \int_{y\in\partial\Omega\atop
|P-y|\leq tr} \frac{|f(y)|dS(y)}{|x-y|^{d-2}}
\leq Ct^{d-1}r\mathrm{M}_{\partial\Omega}(f)(P),
\end{equation*}
where we employ the observation $|x-y|\geq \text{dist}(x,\partial\Omega)\geq r/N_0$. Thus we have
\begin{equation}
I_1 + I_2 + I_3
\leq C\Big\{t^{-1}+t^{d-1}r\Big\}\mathrm{M}_{\partial\Omega}(f)(P)
\leq C\sqrt[d]{r}\mathrm{M}_{\partial\Omega}(f)(P),
\end{equation}
where we choose $t=1/\sqrt[d]{r}$. Consequently, let $r\to 0$,
\begin{equation*}
\lim_{x\to P\atop x\in\Lambda_{N_0}^{\pm}(P)}
\nabla\mathcal{S}_0(f)(x) - \nabla\mathcal{S}_{\widehat{A}}(f)(x)
=
(\nabla\mathcal{S}_0(f))_{\pm}(P) - (\nabla\mathcal{S}_{\widehat{A}}(f))_{\pm}(P)
= \mathcal{T}_0(f)(P) - \mathcal{T}_{\widehat{A}}(f)(P).
\end{equation*}
This together with the known result (see \cite[Lemma 1.4]{GaoW})
\begin{equation*}
(\nabla\mathcal{S}_{\widehat{A}}(f))_{\pm}(P)
= \pm\frac{1}{2} n(P)\mathbf{H}(n(P))f(P) + \text{p.v.}\int_{\partial\Omega}
\nabla \mathbf{\Gamma}_{\widehat{A}}(P-y)f(y)dS(y)
\end{equation*}
implies the identity $\eqref{Id:2.1}$, and therefore the later one $\eqref{Id:2.2}$ immediately follows
from the definition of conormal derivative.
Up to now, we have completed the proof.
\end{proof}

\begin{lemma}\label{lemma:3.8}
Let $f\in L^p(\partial\Omega;\mathbb{R}^m)$ with $1<p<\infty$. Then  we have
\begin{equation}\label{Id:2.3}
\big(\mathcal{D}_{0}(f)\big)_{\pm} =
\Big(\mp\frac{1}{2}I + \mathcal{K}_{0}^*\Big)(f)
\quad\emph{on}~\partial\Omega,
\end{equation}
where $\mathcal{K}_{0}^*$ is the dual operator of $\mathcal{K}_0$, defined by
\begin{equation*}
\mathcal{K}_{0}^*(f)(P)
= \emph{p.v.}\int_{\partial\Omega}
\frac{\partial}{\partial n_0^*(y)}\Big\{\mathbf{\Gamma}_0(P,y)\Big\}f(y)dS(y)
\end{equation*}
for a.e. $P\in\partial\Omega$.
\end{lemma}

\begin{proof}
Recalling the
definition of the double layer potential, we observe that
\begin{equation*}
\mathcal{D}_0(f)(x) = -\nabla \mathcal{S}_0(n\widehat{A}^*f)(x)
\end{equation*}
where $f\in L^p(\partial\Omega;\mathbb{R}^d)$ and $x\in\mathbb{R}^d\setminus\partial\Omega$.
This together with the identity $\eqref{Id:2.1}$
leads  to
\begin{equation*}
\begin{aligned}
\big[\mathcal{D}_0(f)\big]_{\pm}(P)
&= - \big[\nabla\mathcal{S}_0(n\widehat{A}^*f)\big]_{\pm}(P)\\
&=  \mp (1/2)I(f)(P)
+ \text{p.v.}\int_{\partial\Omega} n(y)\widehat{A}^*\nabla_y\mathbf{\Gamma}_0(P-y)f(y)dS(y)
\end{aligned}
\end{equation*}
for a.e. $P\in\partial\Omega$,
where we use the fact that $\nabla_P\mathbf{\Gamma}_0(P-y)
=- \nabla_y\mathbf{\Gamma}_0(P-y)$.

Due to the estimates $\eqref{pri:3.3}$
it is not hard to infer that $\mathcal{K}_0^*(f)\in L^p(\partial\Omega;\mathbb{R}^m)$ with
$1<p<\infty$. Also, for any $g\in L^{p^\prime}(\partial\Omega;\mathbb{R}^m)$ with $1/p+1/p^\prime=1$ we have
\begin{equation}\label{eq:3.3}
\int_{\partial\Omega}\mathcal{K}_0^*(f)(P)g(P)dS(P)
= \int_{\partial\Omega}f(y)\mathcal{K}_0(g)(y)dS(y),
\end{equation}
which reveals that $\mathcal{K}_0^*$ is the dual operator of $\mathcal{K}_0$,
and this ends the proof.
\end{proof}

\begin{lemma}\label{lemma:3.6}
Let $1<p<\infty$, and the operators $\mathcal{K}_{0}$ and $\mathcal{K}_{\widehat{A}}$ be given as
in Lemma $\ref{lemma:3.1}$. Then for any $f\in L^p(\partial\Omega;\mathbb{R}^m)$
we have
\begin{equation}\label{pri:3.10}
\big\|(\mathcal{K}_{\widehat{A}}-\mathcal{K}_{0})(f)\big\|_{W^{1,p}(\partial\Omega)}
\leq C\|f\|_{L^p(\partial\Omega)}
\end{equation}
and so the operator $\mathcal{K}_{\widehat{A}}-\mathcal{K}_{0}$ is
compact on $L^p(\partial\Omega;\mathbb{R}^m)$, where $C$ depends on $\mu,\kappa,\lambda,d,m,p$ and
$\Omega$.
\end{lemma}

\begin{proof}
The proof is straightforward. Let $f\in L^p(\partial\Omega;\mathbb{R}^m)$ with
$1<p<\infty$, and $P\in\partial\Omega$.
In view of the estimates $\eqref{pri:3.0}$ and $\eqref{pri:3.4}$ we obtain
\begin{equation}\label{f:3.7}
\big|(\mathcal{K}_0- \mathcal{K}_{\widehat{A}})(f)(P)\big|
\leq C\int_{\partial\Omega}\frac{|f(y)|}{|P-y|^{d-2}}dS(y) := I_1(f)(P),
\end{equation}
where $I_1(f)$ denotes the Riesz potential of order $1$ on $\partial\Omega$.
Thus by the fractional integral estimates we have
\begin{equation}\label{f:3.12}
\|(\mathcal{K}_0- \mathcal{K}_{\widehat{A}})(f)\|_{L^q(\partial\Omega)}
\leq C\|f\|_{L^p(\partial\Omega)},
\end{equation}
where $1/q=1/p-1/(d-1)$. Since $W^{1,p}(\partial\Omega)\subset
\subset L^q(\partial\Omega)\subset L^p(\partial\Omega)$,
it suffices to verify that
\begin{equation}\label{f:3.13}
\|\nabla_{\text{tan}}(\mathcal{K}_0- \mathcal{K}_{\widehat{A}})(f)\|_{L^p(\partial\Omega)}
\leq C\|f\|_{L^p(\partial\Omega)},
\end{equation}
which is equivalent to showing that
\begin{equation}\label{f:3.10}
|\nabla \tilde{G}(x-y)|
\leq \frac{C}{|x-y|^{d-1}},
\end{equation}
where $\tilde{G}(x) = \nabla \mathbf{\Gamma}_0(x)-\nabla \mathbf{\Gamma}_{\widehat{A}}(x)$, and
\begin{equation}\label{f:3.11}
\big|\nabla \tilde{G}(x-y)- \nabla \tilde{G}(z-y)\big| \leq C\frac{|x-z|}{|x-y|^{d}}
\quad \text{if}\quad |x-y|>2|x-z|.
\end{equation}
The estimate $\eqref{f:3.10}$ has already been established in Lemma $\ref{lemma:3.2}$, while
the estimate $\eqref{f:3.11}$ is also based upon it. By mean-value theorem,
\begin{equation*}
\begin{aligned}
\big|\nabla \tilde{G}(x-y)- \nabla \tilde{G}(z-y)\big|
&\leq \sup_{t\in[0,1]}\big|\nabla^2 \tilde{G}(tz+(1-t)x-y)\big||x-z|\\
&\leq C\sup_{t\in[0,1]}\frac{|x-z|}{|tz+(1-t)x-y|^d}\\
&\leq C\frac{|x-z|}{|x-y|^{d}}
\end{aligned}
\end{equation*}
where we employ the estimate $\eqref{pri:3.4}$ in the second step, and in the last one we use the
fact that $|tz+(1-t)x-y|>(|x-y|/2)$ due to the condition $|x-y|>2|x-z|$.
Hence, combining the estimates $\eqref{f:3.12}$ and $\eqref{f:3.13}$ leads to
the desired estimate $\eqref{pri:3.10}$. We ends the proof by mention that
the stated estimate $\eqref{f:3.11}$ is referred to as the H\"ormander condition.
\end{proof}

\begin{lemma}\label{lemma:3.9}
Given $f\in L^p(\partial\Omega;\mathbb{R}^m)$ with $1<p<\infty$,
let $w_0=\mathcal{D}_{0}(f)$ be the double layer potential. Then we have
\begin{equation}\label{pri:3.14}
\|(w_0)^*\|_{L^p(\partial\Omega)}
\leq C\|f\|_{L^p(\partial\Omega)},
\end{equation}
where $C$ depends on $\mu,\tau,\kappa,\lambda,m,d,p$ and $\Omega$.
\end{lemma}

\begin{proof}
Let $v_0=\mathcal{D}_{\widehat{A}}(f)$ denote the single layer potential
associated with the operator $L_0=-\text{div}(\widehat{A}\nabla)$, and
it follows from \cite[Theorem 1.1]{GaoW} that $(v_0)^*(P)$ exists for
a.e. $P\in\partial\Omega$ and there holds
\begin{equation}\label{f:3.18}
\|(v_0)^*\|_{L^p(\partial\Omega)}\leq C\|f\|_{L^p(\partial\Omega)}
\end{equation}
for any $f\in L^p(\partial\Omega;\mathbb{R}^m)$ with $1<p<\infty$.

Fixed $P\in\partial\Omega$, let $r=|x-P|$, where $x\in\Lambda_{N_0}^{\pm}(P)$.
The idea is that we manage to use $(v_0)^*(P)$
to control the behavior of $w_0(x)$ when $x$ is close to $P$, which is
actually inspired by \cite[Theorem 3.5]{SZW24}.

In the case of $r\geq 1/\sqrt{\lambda}$, we have
\begin{equation}\label{f:3.15}
\begin{aligned}
\big|w_0(x)\big|
&\leq C\int_{y\in\partial\Omega\atop |y-P|\leq 2r}
\frac{|f(y)|}{|x-y|^{d-1}}dS(y)
+ C\int_{y\in\partial\Omega\atop |y-P|>2r}
\big|\nabla_y\mathbf{\Gamma}_{0}(x-y)
f(y)\big|dS(y)\\
&\leq  \frac{C}{r^{d-1}}
\int_{y\in\partial\Omega\atop |y-P|\leq 2r}|f(y)|dS(y)
+ \frac{C}{\sqrt{\lambda}}\int_{y\in\partial\Omega\atop |y-P|>2r}
\frac{|f(y)|}{|y-P|^{d}} dS(y)\\
& \leq C\mathrm{M}_{\partial\Omega}(f)(P).
\end{aligned}
\end{equation}
In the second inequality, we mainly employ the facts that
$|x-y|\geq r/N_0$ in the case of $|y-P|\leq 2r$, and
$|x-y|\geq |y-P|-r \geq 1/\sqrt{\lambda}$ whenever $|y-P|>2r$, as well as the estimate
$\eqref{pri:3.0}$.

We now turn to the case of $r < 1/\sqrt{\lambda}$. In such the case, an similar argument
leads to
\begin{equation*}
\begin{aligned}
\big|w_0(x)\big|
&\leq C\underbrace{\int_{y\in\partial\Omega\atop |y-P|<r}
\frac{|f(y)|}{|x-y|^{d-1}}dS(y)}_{I_1}
+ \underbrace{\Big|\int_{y\in\partial\Omega\atop r<|y-P|<(1/\sqrt{\lambda})}
\frac{\partial}{\partial n_0(y)}
\Big\{\mathbf{\Gamma}_{0}(x-y)
\Big\}f(y)dS(y)\Big|}_{I_2}\\
&
+ \frac{C}{\sqrt{\lambda}}\underbrace{\int_{y\in\partial\Omega\atop |y-P|>(1/\sqrt{\lambda})}\frac{|f(y)|}{|x-y|^{d}}dS(y)}_{I_3}
\end{aligned}
\end{equation*}
where we use the decay estimates $\eqref{pri:3.0}$ in the inequality, and by
the analogous computations as in $\eqref{f:3.15}$ it is not hard to derive that
\begin{equation}\label{f:3.16}
I_1 + I_3  \leq C\mathrm{M}_{\partial\Omega}(f)(P),
\end{equation}
and we proceed to estimate $I_2$ by using $v_\varepsilon$.
\begin{equation}\label{f:3.17}
\begin{aligned}
I_2 &\leq C\int_{y\in\partial\Omega\atop
r<|y-P|<1}\big|\nabla_y\mathbf{\Gamma}_{0}(x-y)-
\nabla_y\mathbf{\Gamma}_{\widehat{A}}(x-y)\big||f(y)|dS(y) \\
&+ \bigg|\int_{y\in\partial\Omega\atop
r<|y-P|<(1/\sqrt{\lambda})}\frac{\partial}{\partial n_0(y)}
\Big\{\mathbf{\Gamma}_{\widehat{A}}(x-y)\Big\}f(y) dS(y)\bigg| \\
&\leq C\int_{y\in\partial\Omega\atop
r<|y-P|<(1/\sqrt{\lambda})} \frac{|f(y)|}{|P-y|^{d-2}}dS(y)
+ (v_0)^*(P) + C\mathrm{M}_{\partial\Omega}(f)(P)\\
&\leq C\mathrm{M}_{\partial\Omega}(f)(P)
+ (v_0)^*(P)
\end{aligned}
\end{equation}
where we use the estimate $\eqref{pri:3.4}$ in the second inequality.

Hence, combining the estimates $\eqref{f:3.15}$, $\eqref{f:3.16}$ and
$\eqref{f:3.17}$, we consequently derived that for any
$x\in \Lambda_{N_0}^{\pm}(P)$ there holds
\begin{equation*}
|w_0(x)| \leq C\mathrm{M}_{\partial\Omega}(f)(P)
+ (v_0)^*(P),
\end{equation*}
which together with $\eqref{f:3.18}$ implies the desired estimate
$\eqref{pri:3.14}$ and we have completed the proof.
\end{proof}

\subsection{Solvability of $L^2$ Dirichlet, Neumann, and regular problems}

\begin{thm}\label{thm:3.3}
Let the singular integral operator $\mathcal{K}_0$ be
defined in Lemma $\ref{lemma:3.1}$.
Then the operators $\pm(1/2)I+\mathcal{K}_0$ are isomorphisms on $L^2(\partial\Omega;\mathbb{R}^m)$ and
there holds
\begin{equation}\label{pri:3.11}
\big\|f\big\|_{L^2(\partial\Omega)}\leq C\big\|\big(\pm(1/2)I+\mathcal{K}_0\big)(f)\big\|_{L^2(\partial\Omega)}
\end{equation}
for any $f\in L^2(\partial\Omega;\mathbb{R}^m)$.  Moreover,
let $\mathcal{K}_0^*$ be the dual operator of $\mathcal{K}_0$, given in Lemma $\ref{lemma:3.8}$,
and then the operators $\pm(1/2)I+\mathcal{K}_0^*$ are also invertible on
$L^2(\partial\Omega;\mathbb{R}^m)$, satisfying the estimate
\begin{equation}\label{pri:3.12}
\big\|f\big\|_{L^2(\partial\Omega)}
\leq C\big\|\big(\mp(1/2)I+\mathcal{K}_0^*\big)(f)\big\|_{L^2(\partial\Omega)},
\end{equation}
where $C$ depends on $\mu,\kappa,\lambda,d,m$ and $\Omega$.
\end{thm}

\begin{proof}
We first address the estimate $\eqref{pri:3.11}$. According to the identity $\eqref{Id:2.2}$,
we have the following jump relation
\begin{equation*}
f = \Big(\frac{\partial \mathcal{S}_0(f)}{\partial n_0}\Big)_{+} -
\Big(\frac{\partial \mathcal{S}_0(f)}{\partial n_0}\Big)_{-}
\end{equation*}
for any $f\in L^2(\partial\Omega;\mathbb{R}^m)$. Hence, the problem is reduced to show
\begin{equation*}
\Big\|\Big(\frac{\partial \mathcal{S}_0(f)}{\partial n_0}\Big)_{\mp}\Big\|_{L^2(\partial\Omega)}
\leq C\Big\|\Big(\frac{\partial \mathcal{S}_0(f)}{\partial n_0}\Big)_{\pm}\Big\|_{L^2(\partial\Omega)}.
\end{equation*}

Let $u_0 = \mathcal{S}_0(f)$. It is clear to see that $\mathcal{L}_0(u_0) = 0$ in
$\mathbb{R}^d\setminus\partial\Omega$. In view of Lemmas $\ref{lemma:3.3}$ and $\ref{lemma:3.1}$,
we have $(\nabla u_0)^*\in L^2(\partial\Omega)$ and $\nabla u_0$ exists on $\partial\Omega$
in the sense of nontangential convergence.
For any $x\in\mathbb{R}^d\setminus\Omega$ with $\text{dist}(x,\partial\Omega)>R$, we have
\begin{equation*}
  |u_0(x)|\leq C\int_{\partial\Omega}\frac{|f(y)|}{|x-y|^{d-2}}dS(y)
  \leq CR^{2-d}\|f\|_{L^2(\partial\Omega)},
\end{equation*}
and $|\nabla u(x)|\leq CR^{1-d}\|f\|_{L^2(\partial\Omega)}$ by the interior estimate, which means
$|u_0(x)|+|x||\nabla u_0(x)| = O(|x|^{2-d})$ as $x\to\infty$. We now have verified all the
conditions in Lemmas $\ref{lemma:3.4}$ and $\ref{lemma:3.5}$. Thus, there holds
\begin{equation*}
\begin{aligned}
\Big\|\Big(\frac{\partial u_0}{\partial n_0}\Big)_{\mp}\Big\|_{L^2(\partial\Omega)}^2
&\leq C\big\|(\nabla_{\tan}u_0)_{\mp}\big\|_{L^2(\partial\Omega)}^2
+ C\int_{\Omega_{\mp}}\big(|\nabla u_0|^2 + |u_0|^2\big) dx \\
&\leq C\big\|(\nabla_{\tan}u_0)_{\pm}\big\|_{L^2(\partial\Omega)}^2
+ C\int_{\partial\Omega}
\Big|\Big(\frac{\partial u_0}{\partial n_0}\Big)_{\mp}\Big|\big|(u_0)_{\pm}\big| dx \\
&\leq C\big\|(\nabla u_0)_{\pm}\big\|_{L^2(\partial\Omega)}^2
+ C\delta\Big\|\Big(\frac{\partial u_0}{\partial n_0}\Big)_{\mp}\Big\|_{L^2(\partial\Omega)}^2
+ C_\delta\|(u_0)_{\pm}\|_{L^2(\partial\Omega)}^2,
\end{aligned}
\end{equation*}
where we use the estimates $\eqref{pri:3.9}$ and $\eqref{pri:3.5}$ in the first step, and
the facts $(\nabla_{\text{tan}}u_0)_{\mp} = (\nabla_{\text{tan}}u_0)_{\pm}$ and
$(u_0)_{\mp} = (u_0)_{\pm}$ in the second one, as well as Young's inequality in the last one.
Then we may choose $\delta\in(0,1)$ such that $C\delta = 1/2$, and
this implies
\begin{equation*}
\Big\|\Big(\frac{\partial u_0}{\partial n_0}\Big)_{\mp}\Big\|_{L^2(\partial\Omega)}^2
\leq C\Big\{\big\|(\nabla u_0)_{\pm}\big\|_{L^2(\partial\Omega)}^2
+ \|(u_0)_{\pm}\|_{L^2(\partial\Omega)}^2\Big\}
\leq C\Big\|\Big(\frac{\partial u_0}{\partial n_0}\Big)_{\pm}\Big\|_{L^2(\partial\Omega)}^2
\end{equation*}
where we employ the estimates $\eqref{pri:3.8}$ and $\eqref{pri:3.13}$ in the last inequality. Up to now,
we have proved the stated estimate $\eqref{pri:3.11}$.

Since the operators $\pm(1/2)I+\mathcal{K}_0$ is injective by $\eqref{pri:3.11}$,
to verify they are isomorphisms on $L^2(\partial\Omega;\mathbb{R}^m)$ is equivalent
to showing $\pm(1/2)I+\mathcal{K}_0:
L^2(\partial\Omega;\mathbb{R}^m)\to L^2(\partial\Omega;\mathbb{R}^m)$ are Fredholm operators with index zero.
Recall that $\pm(1/2)I + \mathcal{K}_{\widehat{A}}$ are Fredholm operators with index zero by
\cite[Lemma 2.1]{GaoW}, and so are $\pm(1/2)I + \mathcal{K}_{0}$,
since we have known $\mathcal{K}_0 - \mathcal{K}_{\widehat{A}}$ is compact
on $L^2(\partial\Omega;\mathbb{R}^m)$ in Lemma $\ref{lemma:3.6}$.

Finally, the estimate $\eqref{pri:3.12}$ will be derived
by a duality method. For any $h\in L^2(\partial\Omega;\mathbb{R}^m)$, there exists
$g\in L^2(\partial\Omega;\mathbb{R}^m)$ such that $h = (\frac{1}{2}I+\mathcal{K}_0)(g)$,
and then it follows from the estimate $\eqref{pri:3.11}$ that
$\|g\|_{L^2(\partial\Omega)}
= \|(\frac{1}{2}I+\mathcal{K}_0)^{-1}(h)\|_{L^2(\partial\Omega)}\leq C\|h\|_{L^2(\partial\Omega)}$.
For any $f\in L^2(\partial\Omega;\mathbb{R}^m)$, it follows from  $\eqref{eq:3.3}$ that
\begin{equation*}
\int_{\partial\Omega}fhdS = \int_{\partial\Omega} (\frac{1}{2}I+\mathcal{K}_0)(g)f dS
= \int_{\partial\Omega} (\frac{1}{2}I+\mathcal{K}_0^*)(f)g dS,
\end{equation*}
which implies $\|f\|_{L^2(\partial\Omega)}
\leq C\|(\frac{1}{2}I+\mathcal{K}_0^*)(f)\|_{L^2(\partial\Omega)}$, and using the same procedure
leads to the estimate $\|f\|_{L^2(\partial\Omega)}
\leq C\|(-\frac{1}{2}I+\mathcal{K}_0^*)(f)\|_{L^2(\partial\Omega)}$. Obviously, the operators
$\pm(1/2)+\mathcal{K}_0^*$ are invertible on $L^2(\partial\Omega;\mathbb{R}^m)$, and we have completed
the whole proof.
\end{proof}

\noindent\textbf{Proof of Theorem $\ref{thm:3.1}$.}
We now proceed to establish the existence for
the Dirichlet problem $(\mathbf{DH}_0)$ with
the given data $g\in L^2(\partial\Omega;\mathbb{R}^m)$.
From Theorem $\ref{thm:3.3}$ we know that
the operator $(-1/2)I+\mathcal{K}^*_{0}$ is invertible.
So by
$\eqref{pri:3.12}$ one may have
$\|((-1/2)I+\mathcal{K}^*_{0})^{-1}\|_{L^2(\partial\Omega)\to L^2(\partial\Omega)}
\leq C$. Then the double layer potential
$w_0=\mathcal{D}_{\mathcal{L}}\big((-\frac{1}{2}I+\mathcal{K}^*_{0})^{-1}(g)\big)$
satisfies $\mathcal{L}(w_0) = 0$ in $\Omega$, and
$\|(w_0)^*\|_{L^2(\partial\Omega)}\leq C\|g\|_{L^2(\partial\Omega)}$
due to the estimate $\eqref{pri:3.14}$.
For the Neumann problem $(\mathbf{NH}_0)$, the estimate
$\|(\nabla u_0)^*\|_{L^2(\partial\Omega)}
\leq C\|g\|_{L^2(\partial\Omega)}$ is given by Theorem $\ref{thm:3.2}$.
Let $\phi\in L^2(\partial\Omega;\mathbb{R}^m)$, and it is not hard to see
that the single layer potential $u_0=\mathcal{S}_{0}(\phi)$ satisfies
$\mathcal{L}_0(u_0) =0$ in $\Omega$. According to
Theorem $\ref{thm:3.3}$,
the trace operator $(1/2)I+\mathcal{K}_{0}$ is invertible on
$L^2(\partial\Omega;\mathbb{R}^m)$, and then
for any $f\in L^2(\partial\Omega;\mathbb{R}^m)$ the expression
$u_0 = \mathcal{S}_{0}
\big((\frac{1}{2}I+\mathcal{K}_{\mathcal{L}})^{-1}(f)\big)$ gives
a solution of $(\mathbf{NH}_0)$.
Concerned with the regular problem $(\mathbf{RH}_0)$,
the existence and the uniqueness
may be known from the Dirichlet problem, and the estimate
$\|(\nabla u_0)^*\|_{L^2(\partial\Omega)}
\leq C\|g\|_{H^1(\partial\Omega)}$ has been shown in Theorem
$\ref{lemma:3.4}$. We ends the proof by mention that the uniqueness
follows from the same arguments stated in the later proof of
Theorem $\ref{thm:4.2}$.
\qed

\section{Well-posed properties in small scales}\label{sec:4}

We may assume $\mathrm{Y}= [-\frac{1}{2},\frac{1}{2}]$ by a translation,
and the following assumption will be convenient for our later discussion, while
it may be removed in the end of the section.
\begin{equation}\label{a:5.1}
\left\{\begin{aligned}
&A,V,B\in C^1(\mathrm{Y}\setminus\partial\Omega),\\
&|\nabla A(x)|+|\nabla V(x)|+|\nabla B(x)|\leq C\big[\text{dist}(x,\partial\Omega)\big]^{\tau_0-1}
\text{~for~any~} x\in \mathrm{Y}\setminus\partial\Omega,
\end{aligned}\right.
\end{equation}
where $\tau_0\in(0,1)$.

Recall that the notation $\mathcal{L}$ denotes
the elliptic operator $\mathcal{L}_\varepsilon$ in the case of $\varepsilon =1$.

\begin{thm}\label{thm:4.2}
Let $\Omega\subset\mathbb{R}^d$ be a bounded Lipschitz domain with $R_0\leq 1/4$.
Suppose that the coefficients of $\mathcal{L}$ satisfy
$\eqref{a:1}$, $\eqref{a:3}$ and $\eqref{a:4}$ with
$\lambda\geq\max\{\lambda_0,\mu\}$.
Assume that the coefficients $A,V,B$ satisfy
the additional condition $\eqref{a:5.1}$.
Then we have the following results:
\begin{itemize}
  \item[$(1)$] for any $g\in L^2(\partial\Omega;\mathbb{R}^m)$,
  there exists a unique solution
  $u\in C^1(\Omega;\mathbb{R}^m)$ to the Dirichlet problem $(\mathbf{DH}_1)$
such that
$\|(u)^*\|_{L^2(\partial\Omega)} \leq C\|g\|_{L^2(\partial\Omega)}$, provided
the coefficients $V,B$ additionally satisfy $\|V-B\|_{L^\infty(\partial\Omega)}\leq \epsilon_0$,
where $\epsilon_0>0$ is sufficiently small;
  \item[$(2)$] for any $g\in L^2(\partial\Omega;\mathbb{R}^m)$, there exists a
unique solution $u\in C^1(\Omega;\mathbb{R}^m)$ to the Neumann problem $(\mathbf{NH}_1)$
and one may have the estimate
$\|(\nabla u)^*\|_{L^2(\partial\Omega)} \leq C\|g\|_{L^2(\partial\Omega)};$
  \item[$(3)$] for any $g\in H^1(\partial\Omega;\mathbb{R}^m)$,
  there exists a unique solution
  $u\in C^1(\Omega;\mathbb{R}^m)$ to the regular problem $(\mathbf{RH}_1)$
and there holds the estimate
$\|(\nabla u)^*\|_{L^2(\partial\Omega)} \leq C\|g\|_{H^1(\partial\Omega)}$.
\end{itemize}
Here the constant $C$ depends only on $\mu,\kappa,\lambda,d,m,\tau,\tau_0$ and $\Omega$.
\end{thm}

\subsection{Rellich estimates}

\begin{lemma}
Suppose that the coefficients of $\mathcal{L}$ satisfies $\eqref{a:1}$ and $\eqref{a:3}$
with $A\in \emph{VMO}(\mathbb{R}^d)$.
Let $u$ be the solution of $\mathcal{L}(u) = 0$ in $\Omega$.
Then we have the following estimate
\begin{equation}\label{pri:4.1}
 (u)^*(Q) \leq C\mathrm{M}_{\partial\Omega}(\mathcal{M}(u))(Q)
\end{equation}
for any $Q\in\partial\Omega$,
where $C$ depends on $\mu,\kappa,\lambda,m,d$ and $\|A\|_{\emph{VMO}}$.
\end{lemma}

\begin{proof}
The estimate $\eqref{pri:4.1}$ is based upon the interior estimate
\begin{equation*}
 |u(x)|\leq C\Big(\dashint_{B(x,r)}|u|^2\Big)^{1/2},
\end{equation*}
where $r=\text{dist}(x,\partial\Omega)$, and the remainder of the argument is
analogous to that in Lemma $\ref{lemma:3.9}$.
\end{proof}

\begin{lemma}\label{lemma:4.1}
Suppose that the coefficients of $\mathcal{L}$ satisfy $\eqref{a:1}$, $\eqref{a:3}$ and
$\eqref{a:5.1}$ with $A^* = A$.
Assume that $u_{\pm}$ are the solutions to
$\mathcal{L}(u_{\pm}) = 0$ in $\Omega_{\pm}$ with $(\nabla u_{\pm})^*\in L^2(\partial\Omega)$. Let
$\mathbf{h}$ be a $C^1$ vector field on $\mathbb{R}^d$ such that
$\emph{supp}(\mathbf{h})\subset\{x:\emph{dist}(x,\partial\Omega)<c R_0\}$,
and $\big<\mathbf{h},n\big>\geq c>0$ on $\partial\Omega$ with
$|\nabla \mathbf{h}|\leq C/R_0$ in $\mathbb{R}^d$.
Then we have two Rellich type identities
\begin{equation}\label{eq:4.1}
\begin{aligned}
\int_{\partial\Omega} \big<\mathbf{h},n\big>a_{ij}^{\alpha\beta}\frac{\partial u^\beta}{\partial x_j}
\frac{\partial u^\alpha}{\partial x_i} dS
&= 2\int_{\partial\Omega}\mathbf{h}_k\frac{\partial u^\alpha}{\partial x_k}
\Big(\frac{\partial u}{\partial\nu_{\mathcal{L}}}\Big)^{\alpha}dS
-  2\int_{\partial\Omega}\mathbf{h}_k\frac{\partial u^\alpha}{\partial x_k}
n_i V_i^{\alpha\beta}u^\beta dS\\
&+ \int_{\Omega_{\pm}}\emph{div}\big(\mathbf{h}\big)a_{ij}^{\alpha\beta}
\frac{\partial u^\beta}{\partial x_j}
\frac{\partial u^\alpha}{\partial x_i} dx
+ \int_{\Omega_{\pm}} \mathbf{h}_k
\frac{\partial}{\partial x_k}\big(a_{ij}^{\alpha\beta}\big)\frac{\partial u^\beta}{\partial x_j}
\frac{\partial u^\alpha}{\partial x_i}dx\\
&  -2 \int_{\Omega_{\pm}} \frac{\partial \mathbf{h}_k}{\partial x_i}a_{ij}^{\alpha\beta}
\frac{\partial u^\beta}{\partial x_j}\frac{\partial u^\alpha}{\partial x_k}
+ 2 \int_{\Omega_{\pm}} \mathbf{h}_k
\frac{\partial u^\alpha}{\partial x_k} \emph{div}\big(V^{\alpha\beta}\big)u^\beta dx \\
&+ 2\int_{\Omega_{\pm}} \mathbf{h}_k \frac{\partial u^\alpha}{\partial x_k} F^\alpha dx,
\end{aligned}
\end{equation}
where $\partial/\partial\nu_{\mathcal{L}} = n\cdot(A\nabla + V)$, and
\begin{equation}\label{eq:4.2}
\begin{aligned}
\int_{\partial\Omega} \big<\mathbf{h},n\big>a_{ij}^{\alpha\beta}\frac{\partial u^\beta}{\partial x_j}
\frac{\partial u^\alpha}{\partial x_i} dS
&= 2\int_{\partial\Omega}\mathbf{h}_ka_{ij}^{\alpha\beta}\frac{\partial u^\beta}{\partial x_j}
\nabla_{\emph{tan}}u^\alpha dS
- 2\int_{\Omega_{\pm}} \mathbf{h}_k \frac{\partial u^\alpha}{\partial x_k} F^\alpha dx\\
&- \int_{\Omega_{\pm}}\emph{div}\big(\mathbf{h}\big)a_{ij}^{\alpha\beta}
\frac{\partial u^\beta}{\partial x_j}
\frac{\partial u^\alpha}{\partial x_i} dx
- \int_{\Omega_{\pm}} \mathbf{h}_k
\frac{\partial}{\partial x_k}\big(a_{ij}^{\alpha\beta}\big)\frac{\partial u^\beta}{\partial x_j}
\frac{\partial u^\alpha}{\partial x_i}dx\\
&  +2 \int_{\Omega_{\pm}} \frac{\partial \mathbf{h}_k}{\partial x_i}a_{ij}^{\alpha\beta}
\frac{\partial u^\beta}{\partial x_j}\frac{\partial u^\alpha}{\partial x_k}
- 2 \int_{\Omega_{\pm}} \mathbf{h}_k
\frac{\partial u^\alpha}{\partial x_k} \emph{div}\big(V^{\alpha\beta}\big)u^\beta dx,
\end{aligned}
\end{equation}
where $\nabla_{\emph{tan}}=n_k\frac{\partial}{\partial x_i}-n_i\frac{\partial}{\partial x_k}$, and
$F^\alpha =\big( V^{\alpha\beta}_i - B_i^{\alpha\beta}\big)\frac{\partial u^\beta}{\partial x_i}
 + c^{\alpha\beta} u^\beta + \lambda u^\alpha.$
\end{lemma}

\begin{proof}
The calculation is standard, and we refer the reader to \cite[Lemma 6.4]{SZW24}.
\end{proof}

\begin{lemma}\label{lemma:4.6}
Assume the same conditions as in Lemma $\ref{lemma:4.1}$, and then we have
\begin{equation}\label{pri:4.5}
\begin{aligned}
\int_{\partial\Omega}|\nabla u|^2 dS
&\leq C\theta^{\tau_0-1}\int_{\partial\Omega}\Big|\frac{\partial u}{\partial \nu_{\mathcal{L}}}\Big|^2dS
+ C\theta^{\tau_0}\int_{\partial\Omega}|(\nabla u)^*|^2dS\\
\int_{\partial\Omega}|\nabla u|^2 dS
&\leq C\int_{\partial\Omega}|\nabla_{\emph{tan}}u|^2dS
+ C\theta^{\tau_0}\int_{\partial\Omega}|(\nabla u)^*|^2dS
+ C\theta^{2\tau_0-2}\int_{\partial\Omega}|u|^2 dS.
\end{aligned}
\end{equation}
and
\begin{equation}\label{pri:4.6}
\int_{\partial\Omega}|u|^2 dS
\leq C\int_{\partial\Omega}\Big|\frac{\partial u}{\partial \nu_{\mathcal{L}}}\Big|^2dS,
\end{equation}
where $C$ depends on $\mu,\kappa,\lambda,m,d,\tau_0$ and $R_0$.
\end{lemma}

\begin{proof}
The main idea may be found in \cite[Lemma 6.6]{SZW24}, and we provide a proof for the sake of
the completeness. We first address the first line of $\eqref{pri:4.5}$, and it follows from the
identity $\eqref{eq:4.1}$ that
\begin{equation*}
\int_{\partial\Omega}|\nabla u|^2 dS
\leq C\int_{\partial\Omega}\Big|\frac{\partial u}{\partial\nu_{\mathcal{L}}}\Big|^2dS
+ C\|u\|_{H^1(\Omega)}^2
+ C\underbrace{\int_{\Omega}\big(|\nabla A|+|\nabla V|\big)|\nabla u|^2 dx}_{T_1}
+ C\underbrace{\int_{\Omega}|\nabla V||u|^2 dx}_{T_2}
\end{equation*}
where we use Young's inequality and the estimate $\eqref{f:3.5}$. To complete the proof,
set $\Sigma_{\theta} = \{x\in\Omega:\text{dist}(x,\partial\Omega)>\theta\}$.
Since the condition $|\nabla A(x)|+|\nabla V(x)|\leq C[\text{dist}(x,\partial\Omega)]^{\tau_0-1}$ for
any $x\in\Omega$, we have
\begin{equation}\label{f:4.9}
\begin{aligned}
T_1 &\leq C\int_{\Omega\setminus\Sigma_{\theta}}
[\text{dist}(x,\partial\Omega)]^{\tau_0-1} |\nabla u|^2 dx
+ C\theta^{\tau_0-1}\int_{\Sigma_{\theta}}|\nabla u|^2 dx \\
&\leq C\theta^{\tau_0}\int_{\partial\Omega}|(\nabla u)^*|^2 dS
+ C\theta^{\tau_0-1}\int_\Omega|\nabla u|^2 dx,
\end{aligned}
\end{equation}
and similarly,
\begin{equation}\label{f:4.10}
T_2 \leq
 C\theta^{\tau_0}\int_{\partial\Omega}|(u)^*|^2 dS
+ C\theta^{\tau_0-1}\int_\Omega|u|^2 dx.
\end{equation}

Combining the above estimates for $T_1$ and $T_2$ we obtain
\begin{equation}\label{f:4.7}
\int_{\partial\Omega}|\nabla u|^2 dS
\leq C\int_{\partial\Omega}\Big|\frac{\partial u}{\partial\nu_{\mathcal{L}}}\Big|^2dS
+ C(1+\theta^{\tau_0-1})\|u\|_{H^1(\Omega)}^2
+ C\theta^{\tau_0}\int_{\partial\Omega}\Big(|(\nabla u)^*|^2 + |(u)^*|^2\Big) dS.
\end{equation}
In fact, the integral of $|(u)^*|^2$ over $\partial\Omega$
may be controlled by the second term in the right-hand side of
the above inequality  since we have the following estimates
\begin{equation}\label{f:4.8}
 \|(u)^*\|_{L^2(\partial\Omega)}\leq C\|\mathcal{M}(u)\|_{L^2(\partial\Omega)}
 \leq C\|u\|_{H^1(\Omega)}
\end{equation}
where we employ the estimate $\eqref{pri:4.1}$ in the first inequality
and the second one follows from $\eqref{pri:3.2.1}$. Thus, the problem is reduced to estimate the quantity $\|u\|_{H^1(\Omega)}$.
Recall $\mathcal{L}(u) = 0$ in $\Omega$, and integrating by parts we derive that
\begin{equation}\label{eq:4.7}
 \mathrm{B}_{\mathcal{L};\Omega}[u,u] = \int_{\partial\Omega}\frac{\partial u}{\partial \nu_{\mathcal{L}}}udS
\end{equation}
which implies the estimate
\begin{equation}
\begin{aligned}
\frac{\lambda}{2}\|u\|_{H^1(\Omega)}^2 \leq
C_\delta\int_{\partial\Omega}\Big|\frac{\partial u}{\partial\nu_{\mathcal{L}}}
\Big|^2 dS
+ \delta \int_{\partial\Omega} |u|^2 dS
\end{aligned}
\end{equation}
where we use the estimate $\eqref{pri:2.0.2}$ and Young's inequality with $\delta$. Thus this together with
$\eqref{f:3.5}$ gives the stated estimate $\eqref{pri:4.6}$, where we choose $\delta>0$ such that
$C\delta = 1/2$. Also, we may have
\begin{equation*}
\|u\|_{H^1(\Omega)}^2
\leq C\int_{\partial\Omega}\Big|\frac{\partial u}{\partial\nu_{\mathcal{L}}}
\Big|^2 dS,
\end{equation*}
by inserting this estimate to $\eqref{f:4.7}$ and $\eqref{f:4.8}$, we obtain
the first line of the estimate $\eqref{pri:4.5}$.

We now proceed to prove the second line of $\eqref{pri:4.5}$. Similarly, by $\eqref{eq:4.2}$ we may have
\begin{equation*}
\begin{aligned}
\int_{\partial\Omega}|\nabla u|^2 dS
&\leq C\int_{\partial\Omega}|\nabla_{\text{tan}}u|dS
+ C\|u\|_{H^1(\Omega)}^2
+ C\int_{\Omega}\big(|\nabla A|+|\nabla V|\big)|\nabla u|^2 dx
+ C\int_{\Omega}|\nabla V||u|^2 dx \\
& \leq C\int_{\partial\Omega}|\nabla_{\text{tan}}u|^2dS
+ C(1+\theta^{\tau_0-1})\|u\|_{H^1(\Omega)}^2
+ C\theta^{\tau_0}\int_{\partial\Omega}|(\nabla u)^*|^2dS,
\end{aligned}
\end{equation*}
where we use the estimates $\eqref{f:4.9}$,$\eqref{f:4.10}$ and $\eqref{f:4.8}$ in the
last inequality. Thus the desired estimate is based upon
\begin{equation}
\theta^{\tau_0-1}\|u\|_{H^1(\Omega)}^2
\leq C\delta\int_{\partial\Omega}|\nabla u|^2 dS
+ C_\delta\theta^{2\tau_0-2}\int_{\partial\Omega}|u|^2 dS,
\end{equation}
where we employ Young's inequality with $\delta$, and this ends the whole proof.
\end{proof}

\begin{lemma}\label{lemma:4.5}
Suppose that $u$ satisfies $\mathcal{L}(u) = 0$ in $\Omega_{-}$
with $(\nabla u)^*\in L^2(\partial\Omega)$, and
$\nabla u$ exists in the sense of nontangential convergence on $\partial\Omega$.
We further assume that
$|u(x)|=O(|x|^{2-d})$ with $|\nabla u(x)|=O(|x|^{1-d})$ as $|x|\to\infty$.
Then we have
\begin{equation}\label{pri:4.7}
\begin{aligned}
&\int_{\partial\Omega} |(\nabla u)_{-}|^2 dS
\leq C\theta^{\tau_0-1}\int_{\partial\Omega}\Big|\Big(\frac{\partial u}{\partial \nu_{\mathcal{L}}}\Big)_{-}\Big|^2 dS
+ C\theta^{\tau_0}\int_{\partial\Omega}|(\nabla u)^*|^2dS\\
&\int_{\partial\Omega} |(\nabla u)_{-}|^2 dS
\leq C\int_{\partial\Omega}|(\nabla_{\emph{tan}}u)_{-}|^2 dS
+ C\theta^{\tau_0}\int_{\partial\Omega}|(\nabla u)^*|^2dS
+ C\theta^{2\tau_0-2}\int_{\partial\Omega}|u_{-}|^2 dS,
\end{aligned}
\end{equation}
and there holds
\begin{equation}\label{pri:4.8}
 \int_{\partial\Omega}|u_{-}|^2 dS
 \leq C\int_{\partial\Omega} \Big|\Big(\frac{\partial u}{\partial \nu_{\mathcal{L}}}\Big)_{-}\Big|^2dS
\end{equation}
where $C$ depends on $\mu,\kappa,\lambda,d,m,\tau_0$ and $\Omega$.
\end{lemma}

\begin{proof}
An argument similar to the one used in the proof of the estimate $\eqref{pri:4.5}$ shows
the estimates $\eqref{pri:4.7}$, and it will not be reproduced here.
We want to point out that the condition
$|u(x)|=O(|x|^{2-d})$ with $|\nabla u(x)|=O(|x|^{1-d})$ as $|x|\to\infty$ guarantees
the truth of the estimate $\eqref{pri:4.8}$. The reader may find the related details in
the proof of Lemma $\ref{lemma:3.5}$.
\end{proof}

\subsection{Comparability between fundamental solutions}

If we fix the coefficients of $\mathcal{L}$ at the point $x\in\mathbb{R}^d$,
it turns to be an operator with constant coefficients whose
fundamental solution is denoted by
$\mathbf{E}(\cdot,\cdot;x)$. For a function $F = F(x,y,z)$, we use the notation
\begin{equation*}
\nabla_1 F(x,y,z) = \nabla_x F(x,y,z)
\quad\text{and}\quad
\nabla_2 F(x,y,z) = \nabla_y F(x,y,z)
\end{equation*}
(see \cite[pp.7]{SZW24}), and this notation will be used throughout.

\begin{lemma}\label{lemma:4.2}
Suppose that the coefficients of $\mathcal{L}$ satisfy $\eqref{a:1}$ and $\eqref{a:3}$
with $\lambda\geq\max\{\lambda_0,\mu\}$. Assume that $A,V,B$ satisfy $\eqref{a:2}$ and $\eqref{a:4}$.
Then we have
\begin{equation}\label{pri:4.3}
\begin{aligned}
|\nabla_1\mathbf{\Gamma}_{\mathcal{L}}(x,y)-\nabla_1 \mathbf{E}(x,y;x) |
&\leq C|x-y|^{1-d+\tau}, \\
|\nabla_1\mathbf{\Gamma}_{\mathcal{L}}(x,y)-\nabla_1 \mathbf{E}(x,y;y) |
&\leq C|x-y|^{1-d+\tau},
\end{aligned}
\end{equation}
and
\begin{equation}\label{pri:4.4}
\begin{aligned}
|\nabla_2\mathbf{\Gamma}_{\mathcal{L}}(x,y)-\nabla_2 \mathbf{E}(x,y;y) |
&\leq C|x-y|^{1-d+\tau}, \\
|\nabla_2\mathbf{\Gamma}_{\mathcal{L}}(x,y)-\nabla_2 \mathbf{E}(x,y;x) |
&\leq C|x-y|^{1-d+\tau}
\end{aligned}
\end{equation}
for any $x,y\in\mathbb{R}^d$ with $0<|x-y|\leq 1$,
where $C$ depends on $\mu,\kappa,\lambda,d,m,\tau$.
\end{lemma}

\begin{proof}
By suitable modification to the proof of \cite[Lemma 2.2]{SZW24}, it follows from
$\eqref{eq:5.1}$ that
\begin{equation}\label{f:4.2}
\begin{aligned}
\big|\nabla_1\mathbf{\Gamma}_{\mathcal{L}}(x,y) -
\nabla_1\mathbf{\Gamma}_{\widetilde{\mathcal{L}}}(x,y)\big|
& \leq C\int_{\mathbb{R}^d} \frac{|\widetilde{A}(z)-A(z)|}{|x-z|^d|z-y|^{d-1}}dz \\
&+C\int_{\mathbb{R}^d} \frac{|\widetilde{V}(z)-V(z)|}{|x-z|^d|z-y|^{d-2}}dz \\
&+C\int_{\mathbb{R}^d} \frac{|\widetilde{B}(z)-B(z)|}{|x-z|^{d-1}|z-y|^{d-1}}dz\\
&+C\int_{\mathbb{R}^d} |\widetilde{c}(z)-c(z)||\mathbf{\Gamma}_{\widetilde{\mathcal{L}}}(z,y)|
|\nabla_1\mathbf{\Gamma}_{\mathcal{L}}(x,z)|dz.
\end{aligned}
\end{equation}
To obtain the stated estimate $\eqref{pri:4.3}$, we fix $x\in\mathbb{R}^d$ and let
$\widetilde{A}, \widetilde{B}, \widetilde{V},\widetilde{c}$ be valued at this point $x$.
In such the case, we replace $\mathbf{\Gamma}_{\widetilde{\mathcal{L}}}(\cdot,y)$
by $\mathbf{E}(\cdot,y;x)$.
Hence, the problem
is reduced to estimate the following quantities
\begin{equation}\label{f:4.1}
\underbrace{\int_{\mathbb{R}^d}\frac{dz}{|x-z|^{d-\tau}|z-y|^{d-1}}}_{I_1}
+ \underbrace{\int_{\mathbb{R}^d}\frac{dz}{|x-z|^{d-\tau}|z-y|^{d-2}}}_{I_2}
+ \underbrace{\int_{\mathbb{R}^d}\frac{dz}{|x-z|^{d-1-\tau}|z-y|^{d-1}}}_{I_3}
\end{equation}
and
\begin{equation}\label{f:4.2.2}
I_4=\int_{\mathbb{R}^d}|\nabla_1\mathbf{\Gamma}_{\mathcal{L}}(x,z)||\mathbf{E}(z,y;x)|dz.
\end{equation}
For $\eqref{f:4.1}$, it is clear to see that the three integrals in $\eqref{f:4.1}$
own a similar form and so we only
address the first one in details while the other two will follow the same computations.
In fact, the following calculation has been used for $\eqref{f:3.8}$.
Let $r=|x-y|$, and $Q=(x+y)/2\in\mathbb{R}^d$.
\begin{equation}\label{f:4.3}
\begin{aligned}
I_1&\leq \bigg\{\int_{B(x,r/4)}+\int_{B(y,3r/4)}
+\int_{\mathbb{R}^d\setminus(B(x,r/4)\cup B(y,3r/4))}\bigg\}
\frac{dz}{|x-z|^{d-\tau}|z-y|^{d-1}}\\
&\leq Cr^{1-d}\int_0^{r/4} \frac{ds}{s^{1-\tau}} + Cr^{\tau-d}\int_0^{3r/4} ds
+ C\int_{r/4}^\infty\frac{ds}{s^{d-\tau}}\\
&\leq Cr^{1-d+\tau},
\end{aligned}
\end{equation}
where we refer the reader to some geometry facts in Remark $\ref{remark:3.1.1}$.
By the same token, it is not hard to derive that
\begin{equation}\label{f:4.4}
I_2\leq Cr^{2-d+\tau},\qquad I_3 \leq Cr^{2-d+\tau}.
\end{equation}

We now turn to estimate $I_4$. By the estimate
$\eqref{pri:3.0}$ we may have $|\mathbf{E}(z,y;x)|\leq C\lambda^{-\frac{1}{2}}|z-y|^{1-d}$
for $|z-y|\geq 1/\sqrt{\lambda}$, and
$|\mathbf{E}(z,y;x)|\leq C|z-y|^{2-d}$ for any $x,z\in\mathbb{R}^d$.
Let $r^* = \max\{1/\sqrt{\lambda},2r\}$.
Proceeding as in the estimates for $I_1$, we obtain
\begin{equation*}
\begin{aligned}
I_4
&\leq \int_{B(Q,r^*)}\frac{dz}{|x-z|^{d-1}|y-z|^{d-2}}
+ \int_{\mathbb{R}^d\setminus B(Q,r^*)}\frac{dz}{|x-z|^{d-1}|y-z|^{d-1}} \\
&\leq Cr^{2-d}\int_0^{r/4} ds + Cr^{1-d}\int_0^{3r/4}sds
+ C\int_{r/4}^{r^*}\frac{ds}{s^{d-2}} + C\lambda^{-\frac{1}{2}}\int_{r^*}^\infty \frac{ds}{s^{d-1}}.
\end{aligned}
\end{equation*}
In the case of $r>1/(2\sqrt{\lambda})$, we have
\begin{equation}\label{f:4.5}
I_4\leq C\left\{\begin{aligned}
&1,&\quad& d=3;\\
&r^{3-d},&\quad& d\geq 4.
\end{aligned}\right.
\end{equation}
For the case $r\leq 1/(2\sqrt{\lambda})$, it is not hard to see
\begin{equation}\label{f:4.2.1}
I_4\leq C\left\{\begin{aligned}
&\ln(4/(\sqrt{\lambda}r)),&\quad& d=3;\\
&r^{3-d},&\quad& d\geq 4,
\end{aligned}\right.
\end{equation}
where $C$ depends on $\mu,\kappa,\tau,\lambda,m,d$. Up to now, we have established
\begin{equation}\label{f:4.6}
|\nabla_1\mathbf{\Gamma}_{\mathcal{L}}(x,y)-\nabla_1 \mathbf{E}(x,y;x) |
\leq Cr^{1-d+\tau},
\end{equation}
from the estimates $\eqref{f:4.2}$, $\eqref{f:4.3}$ $\eqref{f:4.4}$, $\eqref{f:4.5}$ and
$\eqref{f:4.2.1}$  with $r\in(0,1]$.

The second line of the stated estimate $\eqref{pri:4.3}$ is following from $\eqref{f:4.6}$
and $\eqref{pri:4.12}$,
while the desired estimate $\eqref{pri:4.4}$ will be proved by the same argument, and it suffices to
fix $y$ and let the coefficients of $\widetilde{\mathcal{L}}$ be frozen at $y$. We have completed the
whole proof.
\end{proof}

\begin{lemma}\label{lemma:4.3}
Suppose that the coefficients of $\mathcal{L}$ and $\widetilde{\mathcal{L}}$
satisfy $\eqref{a:1},\eqref{a:2}$, $\eqref{a:3}$ and $\eqref{a:4}$
with $\lambda = \tilde{\lambda}$. Let
\begin{equation*}
\begin{aligned}
\vartheta_1
&= \max\Big\{\|A-\widetilde{A}\|_{L^\infty(\mathbb{R}^d)},
\|V-\widetilde{V}\|_{L^\infty(\mathbb{R}^d)},
\|B-\widetilde{B}\|_{L^\infty(\mathbb{R}^d)},
\|c-\widetilde{c}\|_{L^\infty(\mathbb{R}^d)}\Big\},\\
\vartheta_2
&=\max\Big\{\|A-\widetilde{A}\|_{C^{0,\tau}(\mathbb{R}^d)},
\|V-\widetilde{V}\|_{C^{0,\tau}(\mathbb{R}^d)},
\|B-\widetilde{B}\|_{L^\infty(\mathbb{R}^d)},
\|c-\widetilde{c}\|_{L^\infty(\mathbb{R}^d)}\Big\}.
\end{aligned}
\end{equation*}
Then
we have
\begin{equation}\label{pri:4.9}
\big|\mathbf{\Gamma}_{\mathcal{L}}(x,y)
-\mathbf{\Gamma}_{\widetilde{\mathcal{L}}}(x,y)\big|
\leq C\vartheta_1
|x-y|^{2-d},
\end{equation}
and
\begin{equation}\label{pri:4.10}
\begin{aligned}
\big|\nabla_x\mathbf{\Gamma}_{\mathcal{L}}(x,y) -
\nabla_x\mathbf{\Gamma}_{\widetilde{\mathcal{L}}}(x,y)\big|
&\leq C\vartheta_2|x-y|^{1-d},\\
\big|\nabla_x\nabla_y\mathbf{\Gamma}_{\mathcal{L}}(x,y) -
\nabla_x\nabla_y\mathbf{\Gamma}_{\widetilde{\mathcal{L}}}(x,y)\big|
&\leq C\vartheta_2|x-y|^{-d}
\end{aligned}
\end{equation}
for any $x,y\in\mathbb{R}^d$ with $0<|x-y|\leq 1$, where $C$ depends on $\mu,\kappa,\lambda,d,m$ and $\tau$.
\end{lemma}

\begin{proof}
The main ideas may be found in \cite[Lemma 2.6]{SZW24}, and we provide a proof
for the sake of the completeness.
In view of the identity $\eqref{eq:5.1}$, it is not hard to see that
\begin{equation}\label{f:4.11}
\begin{aligned}
\big|\mathbf{\Gamma}_{\mathcal{L}}(x,y)
-\mathbf{\Gamma}_{\widetilde{\mathcal{L}}}(x,y)\big|
&\leq C\|\widetilde{A}-A\|_{L^\infty(\mathbb{R}^d)}
\int_{\mathbb{R}^d}\frac{dz}{|x-z|^{d-1}|z-y|^{d-1}}\\
&+C\|\widetilde{V}-V\|_{L^\infty(\mathbb{R}^d)}\int_{\mathbb{R}^d}
|\nabla\mathbf{\Gamma}_{\mathcal{L}}(x,z)|
|\mathbf{\Gamma}_{\widetilde{\mathcal{L}}}(z,y)|dz\\
&+C\|\widetilde{B}-B\|_{L^\infty(\mathbb{R}^d)}\int_{\mathbb{R}^d}
|\mathbf{\Gamma}_{\mathcal{L}}(x,z)|
|\nabla\mathbf{\Gamma}_{\widetilde{\mathcal{L}}}(z,y)|dz\\
&+C\|\widetilde{c}-c\|_{L^\infty(\mathbb{R}^d)}\int_{\mathbb{R}^d}
|\mathbf{\Gamma}_{\mathcal{L}}(x,z)|
|\mathbf{\Gamma}_{\widetilde{\mathcal{L}}}(z,y)|dz
\end{aligned}
\end{equation}
Then we will show the right-hand side of $\eqref{f:4.11}$ term by term, and
the computations are quite similar to those given earlier for Lemma $\ref{lemma:4.2}$. Let
$r=|x-y|$, and $Q=(x+y)/2\in\mathbb{R}^d$.
\begin{equation}\label{f:4.12}
\int_{\mathbb{R}^d}\frac{dz}{|x-z|^{d-1}|z-y|^{d-1}}
\leq Cr^{1-d}\Big\{\int_0^{\frac{r}{4}}+\int_0^{\frac{3r}{4}}\Big\}ds
+ \int_{\frac{r}{4}}^\infty\frac{ds}{s^{d-1}}\leq Cr^{2-d},
\end{equation}
where the reader may refer to some geometry facts in Remark $\ref{remark:3.1.1}$.
The calculations for
the second line and the third line are similar to those given for $\eqref{f:4.2.2}$,
and we take the second line for example. Let $r^* = \max\{1/\sqrt{\lambda},2r\}$.
By $\eqref{pri:2.11}$ and $\eqref{pri:2.12}$, it is not hard to see that
\begin{equation*}
\begin{aligned}
\int_{\mathbb{R}^d}
|\nabla\mathbf{\Gamma}_{\mathcal{L}}(x,z)|
|\mathbf{\Gamma}_{\widetilde{\mathcal{L}}}(z,y)|dz
&\leq Cr^{2-d}\int_0^{\frac{r}{4}}ds + Cr^{1-d}\int_0^{\frac{3r}{4}}sds
+ C\int_{\frac{r}{4}}^{r^*} \frac{ds}{s^{d-2}}
+ C\int_{r^*}^\infty\frac{ds}{s^{d-1}}\\
&\leq C\left\{\begin{aligned}
&1 + \ln(4r^*/r) + 1/{r^*} &\quad& d=3,\\
&r^{2-d}+ (r^*)^{2-d} &\quad& d\geq 4.
\end{aligned}\right.
\end{aligned}
\end{equation*}
Thus by noting
that $r\in(0,1]$ and an analogous computation to the third line of $\eqref{f:4.11}$ we may derive
\begin{equation}\label{f:4.13}
\int_{\mathbb{R}^d}
|\nabla\mathbf{\Gamma}_{\mathcal{L}}(x,z)|
|\mathbf{\Gamma}_{\widetilde{\mathcal{L}}}(z,y)|dz
+ \int_{\mathbb{R}^d}
|\mathbf{\Gamma}_{\mathcal{L}}(x,z)|
|\nabla\mathbf{\Gamma}_{\widetilde{\mathcal{L}}}(z,y)|dz
\leq Cr^{2-d}.
\end{equation}

We now proceed
to investigate the last line of $\eqref{f:4.11}$, and
\begin{equation*}
\begin{aligned}
\int_{\mathbb{R}^d}|\mathbf{\Gamma}_{\mathcal{L}}(x,z)|
|\mathbf{\Gamma}_{\widetilde{\mathcal{L}}}(z,y)|dz
&\leq Cr^{2-d}\Big\{\int_0^{\frac{r}{4}}+\int_0^{\frac{3r}{4}}\Big\}sds +
C\int_{\frac{r}{4}}^{r^*}\frac{ds}{s^{d-3}} + C\int_{r^*}^\infty\frac{ds}{s^{d-1}}\\
&\leq C\left\{\begin{aligned}
&r+ r^* + 1/r^{*} &\quad& d=3,\\
&1+ \ln(4r^*/r) + (r^*)^{-2} &\quad& d=4,\\
&r^{2-d}+ (r^*)^{2-d}&\quad& d\geq 5,
\end{aligned}\right.
\end{aligned}
\end{equation*}
and this will lead to
\begin{equation}\label{f:4.14}
\begin{aligned}
\int_{\mathbb{R}^d}|\mathbf{\Gamma}_{\mathcal{L}}(x,z)|
|\mathbf{\Gamma}_{\widetilde{\mathcal{L}}}(z,y)|dz
\leq Cr^{2-d}
\end{aligned}
\end{equation}
for $0<r\leq 1$.
Plugging the estimates $\eqref{f:4.12}$, $\eqref{f:4.13}$ and $\eqref{f:4.14}$ back into
the estimate $\eqref{f:4.11}$ we obtain the desired estimate $\eqref{pri:4.9}$.

Then we continue to show the first line of $\eqref{pri:4.10}$. Let
$v^y(z) = \mathbf{\Gamma}_{\mathcal{L}}(z,y)-\mathbf{\Gamma}_{\widetilde{\mathcal{L}}}(z,y)$ in
$B=B(x,r/2)$, and
\begin{equation*}
\mathcal{L}(v^y) = -\mathcal{L}(\mathbf{\Gamma}_{\widetilde{\mathcal{L}}}(\cdot,y))
= \big(\widetilde{\mathcal{L}}-\mathcal{L}\big)(\mathbf{\Gamma}_{\widetilde{\mathcal{L}}}(\cdot,y))
\qquad\text{in}\quad B(x,r/2).
\end{equation*}
Thus by the interior estimate $\eqref{pri:0.5}$ there holds
\begin{equation*}
\begin{aligned}
\big|\nabla v^y(x)\big|
\leq Cr^{-1}\|v^y\|_{L^\infty(B)}
&+ Cr^\tau \big[(\widetilde{A}-A)\nabla\mathbf{\Gamma}_{\widetilde{\mathcal{L}}}(\cdot,y)
+ (\widetilde{V}-V)\mathbf{\Gamma}_{\widetilde{\mathcal{L}}}(\cdot,y)\big]_{C^{0,\tau}(B)}\\
&+ C\big\|(\widetilde{A}-A)\nabla\mathbf{\Gamma}_{\widetilde{\mathcal{L}}}(\cdot,y)
+ (\widetilde{V}-V)\mathbf{\Gamma}_{\widetilde{\mathcal{L}}}(\cdot,y)\big\|_{L^\infty(B)}\\
&+ Cr\big\|(\widetilde{B}-B)\nabla\mathbf{\Gamma}_{\widetilde{\mathcal{L}}}(\cdot,y)
+ (\widetilde{c}-c)\mathbf{\Gamma}_{\widetilde{\mathcal{L}}}(\cdot,y)\big\|_{L^\infty(B)}.
\end{aligned}
\end{equation*}
Since it is known that $\|\nabla\mathbf{\Gamma}_{\widetilde{\mathcal{L}}}(\cdot,y)\|_{C^{0,\tau}}\leq Cr^{1-d-\tau}$
and $\|\mathbf{\Gamma}_{\widetilde{\mathcal{L}}}(\cdot,y)\|_{C^{0,\tau}}\leq Cr^{2-d-\tau}$ from interior
Schauder estimate $\eqref{f:0.3}$,
the above estimate together with $\eqref{pri:4.9}$ actually leads to
\begin{equation*}
|\nabla v^y(x)|\leq C\vartheta_2 r^{1-d},
\end{equation*}
which is exactly the first line of $\eqref{pri:4.10}$.
By the same method stated in the proof of \cite[Lemma 2.6]{SZW24},
the second line of the desired estimate $\eqref{pri:4.10}$ will be established without any real
difficulty and so we do not reproduce here. The whole proof is complete.
\end{proof}

\begin{corollary}
Assume the same conditions as in Lemma $\ref{lemma:4.3}$.
Fix all the coefficients of $\mathcal{L}$ and $\widetilde{\mathcal{L}}$ at a point $x$, and
let $\mathbf{E}(\cdot,0;x)$ and $\widetilde{\mathbf{E}}(\cdot,0;x)$ be two related fundamental solutions,
respectively.
Then for any integer $l\geq 0$ there holds
\begin{equation}\label{pri:4.11}
\big|\nabla^l \mathbf{E}(z,0;x)
- \nabla^l \widetilde{\mathbf{E}}(z,0;x)\big|\leq C\vartheta_1|z|^{2-d-l}
\end{equation}
for any $z\in\mathbb{R}^d$ with $0<|z|\leq 1$, where $C$ depends on $\mu,\kappa,\lambda,d,m$ and $l$.
Moreover,
if $\widetilde{\mathcal{L}}=\mathcal{L}$ and
its coefficients are evaluated at $y$, then we have
\begin{equation}\label{pri:4.12}
\big|\nabla \mathbf{E}(x-y,0;x)
- \nabla \mathbf{E}(x-y,0;y)\big|\leq C|x-y|^{1-d+\tau}
\end{equation}
for any $x,y\in\mathbb{R}^d$ with $0<|x-y|\leq1$,
where $C$ depends on $\mu,\kappa,\lambda,d,m$ and $\tau$.
\end{corollary}

\begin{proof}
In the case $l=0$, the estimate $\eqref{pri:4.11}$ may directly follow from $\eqref{pri:4.9}$.
Since $\mathbf{E}$ and $\widetilde{\mathbf{E}}$ are related to the elliptic operators with constant coefficients,
there holds the translation invariant property, which means
$\mathbf{E}(z,y;x)=\mathbf{E}(z-y,0;x)=\mathbf{E}(y-z,0;x)=\mathbf{E}(y,z;x)$. Thus by some
manipulations as we did in Lemma $\ref{lemma:3.2}$
it is not hard to derive the stated estimate $\eqref{pri:4.11}$ for the case $l\geq 1$, while the
estimate $\eqref{pri:4.12}$ directly follows
from the definition of $\vartheta_1$ and the estimate $\eqref{pri:4.11}$, and this ends the proof.
\end{proof}

We borrow the notation from \cite{SZW24}, and define
\begin{equation}\label{eq:4.4}
\Pi_{\mathcal{L}}^i(x,y) = \nabla_i\mathbf{\Gamma}_{\mathcal{L}}(x,y) - \nabla_i \mathbf{E}(x,y;x)
\end{equation}
for $i=1,2$.

\begin{lemma}\label{lemma:4.4}
Suppose that the coefficients of $\mathcal{L}$ and $\bar{\mathcal{L}}$ satisfy
$\eqref{a:1}$, $\eqref{a:2}$, $\eqref{a:3}$ and $\eqref{a:4}$. Let $\vartheta_1,\vartheta_2$ be given in
Lemma $\ref{lemma:4.3}$. Then we obtain
\begin{equation}\label{pri:4.14}
\begin{aligned}
\big|\Pi_{\mathcal{L}}^1(x,y)-\Pi_{\bar{\mathcal{L}}}^1(x,y)\big|
&\leq C\vartheta_2|x-y|^{1-d+\tau}\\
\big|\Pi_{\mathcal{L}}^2(x,y)-\Pi_{\bar{\mathcal{L}}}^2(x,y)\big|
&\leq C\vartheta_2|x-y|^{1-d+\tau}
\end{aligned}
\end{equation}
for any $x,y\in\mathbb{R}^d$ with $0<|x-y|< 1/4$,
where $C$ depends on $\mu,\kappa,\tau,\lambda,m,d$.
\end{lemma}

\begin{proof}
The main idea may be found in the proof of \cite[Lemma 2.7]{SZW24}. Although the lower order terms
in $\mathcal{L}$ and $\bar{\mathcal{L}}$ do not cause any real difficulty,
we still provide a proof for the sake of the completeness.

Let $\mathbf{\Gamma}_{\mathcal{L}}(\cdot,y)$ and
$\mathbf{\Gamma}_{\widetilde{\mathcal{L}}}(\cdot,y)$ be the fundamental solutions of
$\mathcal{L}$ and $\widetilde{\mathcal{L}}$ and it is fine to assume $\widetilde{\lambda}=\lambda$.
Set $r=|x-y|<1/4$, and
$\Omega = B(x,3/4)$.
In view of $\eqref{eq:2.3.1}$, we have
\begin{equation*}
\begin{aligned}
\mathbf{\Gamma}_{\mathcal{L}}^{\alpha\delta}(x,y)
- \mathbf{\Gamma}_{\widetilde{\mathcal{L}}}^{\alpha\delta}(x,y)
&=\underbrace{\mathrm{B}_{\widetilde{\mathcal{L}};\Omega}
\big[\mathbf{\Gamma}_{\widetilde{\mathcal{L}}}^{\cdot\delta}(\cdot,y),
\mathbf{\Gamma}_{\mathcal{L}}^{\alpha\cdot}(x,\cdot)\big]
- \mathrm{B}_{\mathcal{L}^*;\Omega}
\big[\mathbf{\Gamma}_{\mathcal{L}}^{\alpha\cdot}(x,\cdot),
\mathbf{\Gamma}_{\widetilde{\mathcal{L}}}^{\cdot\delta}(\cdot,y)\big]}_{T_1(x,y)} \\
&+ \underbrace{\mathrm{B}_{\widetilde{\mathcal{L}};\mathbb{R}^d\setminus\Omega}
\big[\mathbf{\Gamma}_{\widetilde{\mathcal{L}}}^{\cdot\delta}(\cdot,y),
\mathbf{\Gamma}_{\mathcal{L}}^{\alpha\cdot}(x,\cdot)\big]
- \mathrm{B}_{\mathcal{L}^*;\mathbb{R}^d\setminus\Omega}
\big[\mathbf{\Gamma}_{\mathcal{L}}^{\alpha\cdot}(x,\cdot),
\mathbf{\Gamma}_{\widetilde{\mathcal{L}}}^{\cdot\delta}(\cdot,y)\big]}_{T_2(x,y)},
\end{aligned}
\end{equation*}
and it is not hard to see that
\begin{equation*}
\begin{aligned}
T_2(x,y) &= -\int_{\partial\Omega}
\frac{\partial\mathbf{\Gamma}_{\widetilde{\mathcal{L}}}}{\partial \nu_{\widetilde{\mathcal{L}}}}(z,y)
\mathbf{\Gamma}_{\mathcal{L}}(x,z)dS(z)
+\int_{\partial\Omega}
\frac{\partial\mathbf{\Gamma}_{\mathcal{L}}}{\partial \nu_{\mathcal{L}^*}}(x,z)
\mathbf{\Gamma}_{\widetilde{\mathcal{L}}}(z,y)dS(z)\\
& = -\int_{\partial\Omega}n(z)\widetilde{A}(z)
\nabla_z\mathbf{\Gamma}_{\widetilde{\mathcal{L}}}(z,y)
\mathbf{\Gamma}_{\mathcal{L}}(x,z) dS(z)
+\int_{\partial\Omega}n(z){A}(z)\nabla_z\mathbf{\Gamma}_{{\mathcal{L}}}(x,z)
\mathbf{\Gamma}_{\widetilde{\mathcal{L}}}(z,y)dS(z)\\
&\qquad-\int_{\partial\Omega}n(z)\widetilde{V}(z)
\mathbf{\Gamma}_{\widetilde{\mathcal{L}}}(z,y)\mathbf{\Gamma}_{\mathcal{L}}(x,z)dS(z)
+\int_{\partial\Omega}n(z){B}(z)
\mathbf{\Gamma}_{\mathcal{L}}(x,z)\mathbf{\Gamma}_{\widetilde{\mathcal{L}}}(z,y)dS(z),
\end{aligned}
\end{equation*}
where we employ the decay estimates $\eqref{pri:5.9}$ and integration by parts. Also, we observe
that
\begin{equation*}
\begin{aligned}
T_1(x,y)&= \int_{\Omega}\Big[\widetilde{A}(z)-A(z)\Big]
\nabla_z\mathbf{\Gamma}_{\mathcal{L}}(x,z)
\nabla_z\mathbf{\Gamma}_{\widetilde{\mathcal{L}}}(z,y)dz
+\int_{\Omega}\Big[\widetilde{B}(z)-B(z)\Big]
\mathbf{\Gamma}_{\mathcal{L}}(x,z)
\nabla_z\mathbf{\Gamma}_{\widetilde{\mathcal{L}}}(z,y)dz\\
&+\int_{\Omega}\Big[\widetilde{V}(z)-V(z)\Big]
\nabla_z\mathbf{\Gamma}_{\mathcal{L}}(x,z)
\mathbf{\Gamma}_{\widetilde{\mathcal{L}}}(z,y) dz
+\int_{\Omega}\Big[\widetilde{c}(z)-c(z)\Big]
\mathbf{\Gamma}_{\mathcal{L}}(x,z)
\mathbf{\Gamma}_{\widetilde{\mathcal{L}}}(z,y)dz
\end{aligned}
\end{equation*}
in terms of $\eqref{eq:5.1}$. Thus we have
\begin{equation*}
\nabla_1\mathbf{\Gamma}_{\mathcal{L}}(x,y)
- \nabla_1\mathbf{\Gamma}_{\widetilde{\mathcal{L}}}(x,y)
= \nabla_x T_1(x,y) + \nabla_x T_2(x,y),
\end{equation*}
and then by setting $\mathbf{\Gamma}_{\widetilde{\mathcal{L}}}(\cdot,y)=\mathbf{E}(\cdot,y;x)$
there holds
\begin{equation}\label{f:4.17}
\small
\begin{aligned}
\Pi_{\mathcal{L}}^1(x,y)
&= \int_{\Omega}\Big[A(x)-A(z)\Big]
\nabla_z\nabla_x\mathbf{\Gamma}_{\mathcal{L}}(x,z)
\nabla_z\mathbf{E}(z,y;x)dz
+\int_{\Omega}\Big[B(x)-B(z)\Big]
\nabla_x\mathbf{\Gamma}_{\mathcal{L}}(x,z)
\nabla_z\mathbf{E}(z,y;x)dz\\
&+\int_{\Omega}\Big[V(x)-V(z)\Big]
\nabla_x\nabla_z\mathbf{\Gamma}_{\mathcal{L}}(x,z)
\mathbf{E}(z,y;x) dz
+\int_{\Omega}\Big[c(x)-c(z)\Big]
\nabla_x\mathbf{\Gamma}_{\mathcal{L}}(x,z)
\mathbf{E}(z,y;x)dz\\
& -\int_{\partial\Omega}n(z)A(x)\nabla_z\mathbf{E}(z,y;x)
\nabla_x\mathbf{\Gamma}_{\mathcal{L}}(x,z)dS(z)
+\int_{\partial\Omega}n(z)A(z)\nabla_x\nabla_z
\mathbf{\Gamma}_{\mathcal{L}}(x,z)\mathbf{E}(z,y;x)dS(z)\\
&-\int_{\partial\Omega}n(z)V(x)\mathbf{E}(z,y;x)
\nabla_x\mathbf{\Gamma}_{\mathcal{L}}(x,z)dS(z)
+\int_{\partial\Omega}n(z)B(z)
\nabla_x\mathbf{\Gamma}_{\mathcal{L}}(x,z)\mathbf{E}(z,y;x)dS(z).
\end{aligned}
\end{equation}

The remainder task is to estimate the quantity
$\Pi_{\mathcal{L}}^i(x,y)-\Pi_{\bar{\mathcal{L}}}^i(x,y)$
for $i=1,2$.
In terms of the right-hand side of $\eqref{f:4.17}$, the core idea is based upon the following
algebra fact
\begin{equation}\label{f:4.18}
ABC - \bar{A}\bar{B}\bar{C} = (A-\bar{A})BC + \bar{A}(B-\bar{B})C + \bar{A}\bar{B}(C-\bar{C}).
\end{equation}
According to $\eqref{f:4.18}$ the full formula on $\Pi_{\mathcal{L}}(x,y)-\Pi_{\bar{\mathcal{L}}}(x,y)$
will be too long to be given in the paper.
Taking into account both conciseness and details of the proof,
we offer some examples to show how to carry out $\eqref{f:4.18}$ on solid integrals and surface integrals
in $\Pi_{\mathcal{L}}^i(x,y)-\Pi_{\bar{\mathcal{L}}}^i(x,y)$ for $i=1,2$.

The first one is
\begin{equation*}
\small
\begin{aligned}
\Pi_{\mathcal{L}}^1(x,y)-\Pi_{\bar{\mathcal{L}}}^1(x,y)
&=\int_{\Omega}\Big[A(x)-A(z)\Big]
\nabla_z\nabla_x\mathbf{\Gamma}_{\mathcal{L}}(x,z)
\nabla_z\mathbf{E}(z,y;x)dz &~&\\
&- \int_{\Omega}\Big[\bar{A}(x)-\bar{A}(z)\Big]
\nabla_z\nabla_x\mathbf{\Gamma}_{\bar{\mathcal{L}}}(x,z)
\nabla_z\bar{\mathbf{E}}(z,y;x)dz + \text{other~terms}  &~&\\
&= \int_{\Omega}\Big[A(x)-A(z)-\bar{A}(x)+\bar{A}(z)\Big]
\nabla_z\nabla_x\mathbf{\Gamma}_{\mathcal{L}}(x,z)
\nabla_z\mathbf{E}(z,y;x)dz  &:=I_1 &\\
&+ \int_{\Omega}\Big[\bar{A}(x)-\bar{A}(z)\Big]\Big[
\nabla_z\nabla_x\mathbf{\Gamma}_{\mathcal{L}}(x,z)
-\nabla_z\nabla_x\mathbf{\Gamma}_{\bar{\mathcal{L}}}(x,z)\Big]
\nabla_z\mathbf{E}(z,y;x)dz  &:=I_2 &\\
&+ \int_{\Omega}\Big[\bar{A}(x)-\bar{A}(z)\Big]
\Big[\nabla_z\mathbf{E}(z,y;x) - \nabla_z\bar{\mathbf{E}}(z,y;x)
\Big]\nabla_z\nabla_x\mathbf{\Gamma}_{\bar{\mathcal{L}}}(x,z)dz
&:=I_3&\\
& +\text{other~terms,} &~&
\end{aligned}
\end{equation*}
and it is not hard to see that
\begin{equation}\label{f:4.19}
\begin{aligned}
|I_1|&\leq C\|A-\bar{A}\|_{C^{0,\tau}(\mathbb{R}^d)}
\int_{\Omega}\frac{dz}{|x-z|^{d-\tau}|z-y|^{d-1}} \\
&\leq C\|A-\bar{A}\|_{C^{0,\tau}(\mathbb{R}^d)}
\Bigg\{r^{1-d}\int_0^{\frac{r}{t}}\frac{ds}{s^{1-\tau}}
+r^{\tau-d}\int_0^{\frac{3r}{4}}ds + \int_{\frac{r}{4}}^{1}\frac{ds}{s^{d-\tau}}\Bigg\}
\leq C\vartheta_2r^{1-d+\tau},
\end{aligned}
\end{equation}
where we also use $\eqref{pri:2.13}$ and $\eqref{pri:3.0}$ in the first inequality. A similar computation
will give
\begin{equation}\label{f:4.20}
|I_2| \leq C\vartheta_2\int_{\Omega}\frac{dz}{|x-z|^{d-\tau}|z-y|^{d-1}}
\leq C\vartheta_2 r^{1-d+\tau}
\end{equation}
where we use the estimates $\eqref{pri:4.10}$ and $\eqref{pri:3.0}$ in the first inequality, and
\begin{equation}\label{f:4.21}
|I_3| \leq C\vartheta_1\int_{\Omega}\frac{dz}{|x-z|^{d-\tau}|z-y|^{d-1}}
\leq C\vartheta_1 r^{1-d+\tau}
\end{equation}
by using $\eqref{pri:2.13}$ and $\eqref{pri:4.11}$ in the same place.

The second example is related to the computations on the lower order terms of $\mathcal{L}$ and
$\bar{\mathcal{L}}$, and we will find that the results may be
controlled by those from the leading terms. See
\begin{equation*}
\small
\begin{aligned}
\Pi_{\mathcal{L}}^1(x,y)-\Pi_{\bar{\mathcal{L}}}^1(x,y)
&= \int_{\Omega}\Big[V(x)-V(z)-\bar{V}(x)+\bar{V}(z)\Big]
\nabla_z\nabla_x\mathbf{\Gamma}_{\mathcal{L}}(x,z)
\mathbf{E}(z,y;x)dz  &:=I_4 &\\
&+ \int_{\Omega}\Big[\bar{V}(x)-\bar{V}(z)\Big]\Big[
\nabla_z\nabla_x\mathbf{\Gamma}_{\mathcal{L}}(x,z)
-\nabla_z\nabla_x\mathbf{\Gamma}_{\bar{\mathcal{L}}}(x,z)\Big]
\mathbf{E}(z,y;x)dz  &:=I_5 &\\
&+ \int_{\Omega}\Big[\bar{V}(x)-\bar{V}(z)\Big]
\Big[\mathbf{E}(z,y;x) - \bar{\mathbf{E}}(z,y;x)
\Big]\nabla_z\nabla_x\mathbf{\Gamma}_{\bar{\mathcal{L}}}(x,z)dz
&:=I_6&\\
& +\text{other~terms,} &~&
\end{aligned}
\end{equation*}
and from the similar calculations as those in $\eqref{f:4.19}$,
$\eqref{f:4.20}$ and $\eqref{f:4.21}$, it follows that
\begin{equation*}
|I_4|+|I_5|+|I_6|\leq C(\vartheta_2+\vartheta_1) r^{2-d+\tau}
\leq C\vartheta_2 r^{1-d+\tau}
\end{equation*}
in terms of the estimates $\eqref{pri:2.13}$, $\eqref{pri:3.0}$, $\eqref{pri:4.10}$ and $\eqref{pri:4.11}$, as well as
the facts $r\in(0,1)$ and $\vartheta_2\geq\vartheta_1$.

Another example is shown the computations related to the zero
order terms of $\mathcal{L}$ and $\bar{\mathcal{L}}$, and we will see that
\begin{equation*}
\small
\begin{aligned}
\Pi_{\mathcal{L}}^1(x,y)-\Pi_{\bar{\mathcal{L}}}^1(x,y)
&= \int_{\Omega}\Big[c(x)-c(z)-\bar{c}(x)+\bar{c}(z)\Big]
\nabla_x\mathbf{\Gamma}_{\mathcal{L}}(x,z)
\mathbf{E}(z,y;x)dz  &:=I_7 &\\
&+ \int_{\Omega}\Big[\bar{c}(x)-\bar{c}(z)\Big]\Big[
\nabla_x\mathbf{\Gamma}_{\mathcal{L}}(x,z)
-\nabla_x\mathbf{\Gamma}_{\bar{\mathcal{L}}}(x,z)\Big]
\mathbf{E}(z,y;x)dz  &:=I_8 &\\
&+ \int_{\Omega}\Big[\bar{c}(x)-\bar{c}(z)\Big]
\Big[\mathbf{E}(z,y;x) - \bar{\mathbf{E}}(z,y;x)
\Big]\nabla_x\mathbf{\Gamma}_{\bar{\mathcal{L}}}(x,z)dz
&:=I_9&\\
& +\text{other~terms.} &~&
\end{aligned}
\end{equation*}
Thus, in view of $\eqref{pri:2.12}$, $\eqref{pri:3.0}$, $\eqref{pri:4.10}$ and $\eqref{pri:4.11}$
we obtain that
\begin{equation*}
\begin{aligned}
|I_7|+|I_8|+|I_9|
&\leq C\Big\{\|c-\bar{c}\|_{L^\infty(\mathbb{R}^d)}
+\vartheta_2 + \vartheta_1\Big\}\int_{\Omega}\frac{dz}{|x-z|^{d-1}|y-z|^{d-2}}\\
&\leq C\vartheta_2\left\{\begin{aligned}
&\ln(1/r) &~& d=3,\\
&r^{3-d}  &~& d\geq 4
\end{aligned}\right.\\
&\leq C\vartheta_2 r^{1-d+\tau},
\end{aligned}
\end{equation*}
where the second step follows from a similar manipulation to that used for $\eqref{f:4.2.2}$, and
the last inequality is due to the assumption $0<r<1/4$.

We now turn to study the surface integrals in $\Pi_{\mathcal{L}}^i(x,y)-\Pi_{\bar{\mathcal{L}}}^i(x,y)$
with $i=1,2$.
Before proceeding further, we note that $|y-z|>(1/2)$ for any
$z\in\partial\Omega$ since $0<|y-x|<1/4$.
The last example is
\begin{equation*}
\begin{aligned}
\Pi_{\mathcal{L}}^1(x,y)-\Pi_{\bar{\mathcal{L}}}^1(x,y)
&= \int_{\partial\Omega}n(z)A(z)\nabla_z\nabla_x\mathbf{\Gamma}_{\mathcal{L}}(x,z)
\mathbf{E}(z,y;x)dS(z) &~& \\
& -\int_{\partial\Omega}n(z)\bar{A}(z)\nabla_z\nabla_x\mathbf{\Gamma}_{\bar{\mathcal{L}}}(x,z)
\bar{\mathbf{E}}(z,y;x)dS(z) + \text{other~terms} &~&\\
& =\int_{\partial\Omega}n(z)\big[A(z)-\bar{A}(z)]\nabla_z\nabla_x\mathbf{\Gamma}_{\mathcal{L}}(x,z)
\mathbf{E}(z,y;x)dS(z) &:=I_{10}&\\
&+\int_{\partial\Omega}n(z)\bar{A}(z)\big[\nabla_z\nabla_x\mathbf{\Gamma}_{\mathcal{L}}(x,z)
-\nabla_z\nabla_x\mathbf{\Gamma}_{\bar{\mathcal{L}}}(x,z)\big]\mathbf{E}(z,y;x)dS(z)
&:=I_{11}&\\
&+\int_{\partial\Omega}n(z)\bar{A}(z)\nabla_z\nabla_x\mathbf{\Gamma}_{\bar{\mathcal{L}}}(x,z)
\big[\mathbf{E}(z,y;x) - \bar{\mathbf{E}}(z,y;x)\big]dS(z) &:=I_{12}&\\
&+\text{other terms,} &~&
\end{aligned}
\end{equation*}
and it is not hard to see that
\begin{equation*}
|I_{10}|+|I_{11}|+|I_{12}|
\leq C\Big\{\|A-\bar{A}\|_{L^\infty(\mathbb{R}^d)}+\vartheta_2 + \vartheta_1\Big\}
\int_{\partial\Omega}\frac{dS(z)}{|x-z|^d|z-y|^{d-2}}\leq C\vartheta_2,
\end{equation*}
where we employ the estimates $\eqref{pri:2.13}$, $\eqref{pri:3.0}$, $\eqref{pri:4.10}$ and
$\eqref{pri:4.11}$ in the first inequality. Similarly, the other surface integrals are also controlled
by $C\vartheta_2$. In sum, we may derive
\begin{equation*}
\big|\Pi_{\mathcal{L}}^1(x,y)-\Pi_{\bar{\mathcal{L}}}^1(x,y)\big|
\leq C\vartheta_2(1+r^{1-d+\tau})\leq C\vartheta_2r^{1-d+\tau}
\end{equation*}
by the fact that $0<r<1/4$. Similarly, we can obtain the second line of the stated estimate
$\eqref{pri:4.14}$, and this consequently ends the proof.
\end{proof}

Fix all the coefficients of $\mathcal{L}$ and $\bar{\mathcal{L}}$ at a point $x$, and
the related fundamental solutions are denoted by $\mathbf{E}(\cdot,0;x)$ and $\bar{\mathbf{E}}(\cdot,0;x)$,
respectively. Define $\mathbf{E}_A(\cdot,0;x)$ as the principal part of $\mathbf{E}(\cdot,0;x)$, which
is related to the operator $L=-\text{div}(A(x)\nabla)$, and the principal part of
$\mathbf{E}_{\bar{A}}(\cdot,0;x)$ is represented by $\mathbf{E}_{\bar{A}}(\cdot,0;x)$.
We now introduce the new notation as follows:
\begin{equation}\label{eq:4.3}
\begin{aligned}
\mathbf{R}(\cdot,0;x) &= \mathbf{E}(\cdot,0;x) - \mathbf{E}_{A}(\cdot,0;x),\\
\bar{\mathbf{R}}(\cdot,0;x) &= \bar{\mathbf{E}}(\cdot,0;x) - \mathbf{E}_{\bar{A}}(\cdot,0;x),
\end{aligned}
\end{equation}
which represent the corresponding lower order terms of $\mathbf{E}(\cdot,0;x)$ and $\bar{\mathbf{E}}(\cdot,0;x)$,
respectively.

\begin{lemma}
Assume the same conditions as in Lemma $\ref{lemma:4.4}$.
Let $\mathbf{R}(\cdot,0;x)$ and $\bar{\mathbf{R}}(\cdot,0;x)$ be given in $\eqref{eq:4.3}$.
Then we have
\begin{equation}\label{pri:4.13}
\big|\nabla\mathbf{R}(z,0;x) - \nabla\bar{\mathbf{R}}(z,0;x)\big|\leq C\vartheta_1|z|^{2-d}
\end{equation}
for any $z\in\mathbb{R}^d$ with $0<|z|<1/4$, where $C$ depends on $\mu,\kappa,\lambda,m,d$.
\end{lemma}

\begin{proof}
The main idea of the proof is inspired by Lemma $\ref{lemma:4.4}$. If ignored
the quantity $\vartheta_1$ in the estimate, we would derive the stated estimate $\eqref{pri:4.13}$
in terms of Lemma $\ref{lemma:3.2}$ at once. However $\vartheta_1$ is important in our later purposes.

Let $r=|x-y|<1/4$ and $\Omega = B(x,3/4)$.
Observing the identity $\eqref{f:4.17}$, we replace $\mathbf{\Gamma}_{\mathcal{L}}(\cdot,\cdot)$
by $\mathbf{E}_{A}(\cdot,\cdot;x)$, which means $B(z)=V(z)=c(z)=0$ in $\eqref{f:4.17}$. Due to
$A(z)=A(x)$ for all $z\in\mathbb{R}^d$ in our cases, it follows from $\eqref{f:4.17}$ and
integration by parts that
\begin{equation}
\small
\begin{aligned}
&\nabla_x\mathbf{R}(x,y;x)
=\big(B(x)-V(x)\big)\int_{\Omega}
\nabla_x\mathbf{E}_{A}(x,z;x)
\nabla_z\mathbf{E}(z,y;x)dz
-\mathbf{c}(x)\int_{\Omega}
\nabla_x\mathbf{E}_{A}(x,z;x)
\mathbf{E}(z,y;x)dz\\
& - \int_{\partial\Omega}n(z)A(x)\nabla_x\nabla_z\mathbf{E}_{A}(x,z;x)
\mathbf{E}(z,y;x)dS(z)
+\int_{\partial\Omega}n(z)A(x)
\nabla_z\mathbf{E}(z,y;x)\nabla_x\mathbf{E}_A(x,z;x)dS(z)
\end{aligned}
\end{equation}
where $\mathbf{c}(x) = c(x)+\lambda I$. Thus we have
\begin{equation*}
\begin{aligned}
\nabla_x\mathbf{R}(x,y;x)-\nabla_x\bar{\mathbf{R}}(x,y;x)
&=\big(B(x)-V(x)\big)\int_{\Omega}
\nabla_x\mathbf{E}_{A}(x,z;x)
\nabla_z\mathbf{E}(z,y;x)dz &:=I_1&\\
&-\big(\bar{B}(x)-\bar{V}(x)\big)\int_{\Omega}
\nabla_x\mathbf{E}_{\bar{A}}(x,z;x)
\nabla_z\bar{\mathbf{E}}(z,y;x)dz &:=I_2&\\
& -c(x)\int_{\Omega}
\nabla_x\mathbf{E}_{A}(x,z;x)
\mathbf{E}(z,y;x)dz &:= I_3&\\
&+\bar{c}(x)\int_{\Omega}
\nabla_x\mathbf{E}_{\bar{A}}(x,z;x)
\bar{\mathbf{E}}(z,y;x)dz &:= I_4&\\
&+ \text{other~terms}
\end{aligned}
\end{equation*}
and then proceeding as in the proof of Lemma $\ref{lemma:4.4}$ and using the identity $\eqref{f:4.18}$
it is not hard to derive
\begin{equation*}
|I_1 + I_2|\leq C\vartheta_1 r^{2-d}
\quad\text{and}\quad
|I_3+I_4|\leq C\vartheta_1\left\{\begin{aligned}
&\ln(1/r) &~& d=3,\\
&r^{3-d}  &~& d\geq4,
\end{aligned}\right.
\end{equation*}
where we use the estimates $\eqref{pri:4.11}$ and $\eqref{pri:3.0}$ in the inequalities.
The remainder of the argument is analogous to that in Lemma $\ref{lemma:4.4}$, and we consequently obtain
\begin{equation*}
|\nabla_x\mathbf{R}(x,y;x)-\nabla_x\bar{\mathbf{R}}(x,y;x)|
\leq C\vartheta_1|x-y|^{2-d}
\end{equation*}
by noting the fact that $0<|x-y|<1/4$.
Let $z=x-y$ and a translation argument leads to the stated estimate $\eqref{pri:4.13}$.
We have completed the proof.
\end{proof}

\subsection{Estimates for layer potentials}

\begin{definition}\label{def:4.2}
Let $\mathbf{\Gamma}_{\mathcal{L}}$ and
$\mathbf{\Gamma}_{A}$ represent the fundamental solutions of
$\mathcal{L}$ and $L = -\emph{div}(A\nabla)$, respectively.
Set $P\in\partial\Omega$. Define the truncated singular integral operators
\begin{equation}
\begin{aligned}
\mathcal{T}_{\Theta}^{1,\delta}(f)(P) &= \int_{y\in\partial\Omega\atop
|y-P|>\delta} \nabla_1 \mathbf{\Gamma}_{\Theta}(P,y)f(y)dS(y),\\
\mathcal{T}_{\Theta}^{2,\delta}(f)(P) &= \int_{y\in\partial\Omega\atop
|y-P|>\delta} \nabla_2 \mathbf{\Gamma}_{\Theta}(P,y)f(y)dS(y),
\end{aligned}
\end{equation}
and then the singular integral operators and the associated maximal singular integral ones may
be denoted by
\begin{equation}\label{eq:5.2}
\begin{aligned}
\mathcal{T}_{\Theta}^{1}(f)(P) & = \text{p.v.}\int_{\partial\Omega}
\nabla_1 \mathbf{\Gamma}_{\Theta}(P,y)f(y)dS(y)
:= \lim_{\delta\to 0} \mathcal{T}_{\Theta}^{1,\delta}(f)(P), \\
\mathcal{T}_{\Theta}^{2}(f)(P) & = \text{p.v.}\int_{\partial\Omega}
\nabla_2 \mathbf{\Gamma}_{\Theta}(P,y)f(y)dS(y)
:= \lim_{\delta\to 0} \mathcal{T}_{\Theta}^{2,\delta}(f)(P), \\
&\qquad~\mathcal{T}_{\Theta}^{1,*}(f)(P) = \sup_{\delta>0}|T_{\Theta}^{1,\delta}(f)(P)|,\\
&\qquad~\mathcal{T}_{\Theta}^{2,*}(f)(P) = \sup_{\delta>0}|T_{\Theta}^{2,\delta}(f)(P)|,\\
\end{aligned}
\end{equation}
where the subscript $\Theta$ will be given by $\mathcal{L}$ or $A$ in this section,
and by $\varepsilon$ or $A_\varepsilon$ in the next section.
\end{definition}

\begin{lemma}[comparing lemma II]\label{lemma:4.7}
Suppose that the coefficients of $\mathcal{L}$ satisfy $\eqref{a:1}$ and $\eqref{a:3}$
with $\lambda\geq\max\{\lambda_0,\mu\}$. Assume that
$A,V,B$ satisfy $\eqref{a:2}$ and $\eqref{a:4}$.
Then there holds
\begin{equation}\label{pri:5.8}
\begin{aligned}
|\nabla_1\mathbf{\Gamma}_{\mathcal{L}}(x,y)-\nabla_1\mathbf{\Gamma}_{A}(x,y)|
&\leq C |x-y|^{1-d+\tau},\\
|\nabla_2\mathbf{\Gamma}_{\mathcal{L}}(x,y)-\nabla_2\mathbf{\Gamma}_{A}(x,y)|
&\leq C |x-y|^{1-d+\tau}
\end{aligned}
\end{equation}
for any $x,y\in\mathbb{R}^d$ with $0<|x-y|\leq 1$, where
$\mathbf{\Gamma}_A(\cdot,y)$ denotes the principal part of
$\mathbf{\Gamma}_{\mathcal{L}}(\cdot,y)$, and $C$ depends on $\mu,\kappa,\lambda,m,d$ and
$\tau$.
\end{lemma}

\begin{proof}
Although the results
are quite similar to those shown in Lemma $\ref{lemma:3.2}$,
we can not use the arguments developed there except
$V\in C^1(\mathbb{R}^d;\mathbb{R}^{m^2})$. That makes us go back to the method of
frozen coefficients, which has been developed in Lemma $\ref{lemma:4.2}$.
Fix the coefficient of $L = -\text{div}(A\nabla)$ at $x$, and let
$\mathbf{E}_A(x,y;x)$ be the corresponding fundamental solution.
Thus it is not hard
to observe that
\begin{equation*}
\begin{aligned}
\nabla_1\mathbf{\Gamma}_{\mathcal{L}}(x,y) - \nabla_1\mathbf{\Gamma}_{A}(x,y)
&= \nabla_1\mathbf{\Gamma}_{\mathcal{L}}(x,y) - \nabla_1\mathbf{E}(x,y;x) \\
&+ \nabla_1\mathbf{E}(x,y;x) - \nabla_1\mathbf{E}_A(x,y;x)
+ \nabla_1\mathbf{E}_A(x,y;x) - \nabla_1\mathbf{\Gamma}_{A}(x,y),
\end{aligned}
\end{equation*}
and the first line of
$\eqref{pri:5.8}$
follows from the estimates $\eqref{pri:4.3}$, $\eqref{pri:3.4}$ and \cite[Lemma 2.2]{SZW24}.
By the same token, we can derive the second one and the proof ends here.
\end{proof}

\begin{thm}\label{thm:5.2}
Suppose that the coefficients $\mathcal{L}$ satisfy the conditions
$\eqref{a:1}-\eqref{a:4}$ with $\lambda\geq\max\{\lambda_0,\mu\}$.
Let $f\in L^p(\partial\Omega;\mathbb{R}^m)$ for $1<p<\infty$.
Then $\mathcal{T}_{\Theta}(f)$ exists for a.e. $P\in\partial\Omega$ and
\begin{equation}\label{pri:5.10}
\begin{aligned}
\|\mathcal{T}_{\Theta}^{1}(f)\|_{L^p(\partial\Omega)}
+ \|\mathcal{T}_{\Theta}^{1,*}(f)\|_{L^p(\partial\Omega)} &\leq C\|f\|_{L^p(\partial\Omega)},\\
\|\mathcal{T}_{\Theta}^{2}(f)\|_{L^p(\partial\Omega)}
+ \|\mathcal{T}_{\Theta}^{2,*}(f)\|_{L^p(\partial\Omega)} &\leq C\|f\|_{L^p(\partial\Omega)},
\end{aligned}
\end{equation}
hold for $\Theta = A,\mathcal{L}$,
where $C$ depends only on $\mu,\kappa,d,m,p$ and $\Omega$.
\end{thm}

\begin{proof}
The proof is quite similar to that used in Lemma $\ref{lemma:3.3}$
and the original idea may be found in \cite[Lemma 3.1]{SZW23}, and we provide a proof for the
sake of the completeness. Note that if we choose $\Theta = A$,
then the result $\eqref{pri:5.10}$ had already been established in \cite[Theorem 3.1]{SZW24}.
We now study the case of $\Theta = \mathcal{L}$. It is sufficient to estimate the integral
\begin{equation*}
 \Big|\int_{y\in\partial\Omega\atop
 |y-P|>\delta} \nabla_1 \mathbf{\Gamma}_{\mathcal{L}}(P,y)f(y)dS(y)\Big|,
\end{equation*}
and it will confront with two cases: (1) $\delta > 1$; (2) $\delta \leq 1$.
In the case of (1), it follows from the estimate $\eqref{pri:5.9}$ that
\begin{equation}\label{f:5.9}
 \Big|\int_{y\in\partial\Omega\atop
 |y-P|>\delta} \nabla_1 \mathbf{\Gamma}_{\mathcal{L}}(P,y)f(y)dS(y)\Big|
 \leq C\int_{y\in\partial\Omega\atop
 |y-P|>\delta} \frac{|f(y)|}{|P-y|^{d-1+\rho}}dS(y) \leq C\mathrm{M}_{\partial\Omega}(f)(P)
\end{equation}
with $0<\rho<1$.
We proceed to investigate the case (2). In such case, it is not hard to see that
\begin{equation}\label{f:5.10}
\begin{aligned}
 \Big|\int_{y\in\partial\Omega\atop
 |y-P|>\delta} \nabla_1 \mathbf{\Gamma}_{\mathcal{L}}(P,y)f(y)dS(y)\Big|
& \leq  \Big|\int_{y\in\partial\Omega\atop
 |y-P|\geq 1} \nabla \mathbf{\Gamma}_{\mathcal{L}}(P,y)f(y)dS(y)\Big| \\
& +  \int_{y\in\partial\Omega\atop
 \delta <|y-P|< 1} \big|\nabla_1 \mathbf{\Gamma}_{\mathcal{L}}(P,y)
 - \nabla_1\mathbf{\Gamma}_{A}(P,y)\big||f(y)|dS(y) \\
& +  \Big|\int_{y\in\partial\Omega\atop
 \delta<|y-P|<1} \nabla_1 \mathbf{\Gamma}_{A}(P,y)f(y)dS(y)\Big| \\
&\leq 2\mathcal{T}_{A}^{1,*}(f)(P) + C\mathrm{M}_{\partial\Omega}(f)(P),
\end{aligned}
\end{equation}
where we use the estimate $\eqref{pri:5.8}$ in the last inequality. Combining the estimates
$\eqref{f:5.9}$ and $\eqref{f:5.10}$, we have
\begin{equation*}
 \mathcal{T}_{\mathcal{L}}^{1,*}(f)(P)
 \leq 2\mathcal{T}_{A}^{1,*}(f)(P) + C\mathrm{M}_{\partial\Omega}(f)(P),
\end{equation*}
and this together with \cite[Theorem 3.1]{SZW24}
finally leads to
\begin{equation}
\|\mathcal{T}_{\mathcal{L}}^{1}(f)\|_{L^p(\partial\Omega)} +
\|\mathcal{T}_{\mathcal{L}}^{1,*}(f)\|_{L^p(\partial\Omega)} \leq C\|f\|_{L^p(\partial\Omega)}.
\end{equation}
The second line of $\eqref{pri:5.10}$ may be proved by the same way, and
we do not reproduce here. The proof has been completed.
\end{proof}

\begin{thm}
Let $\Omega\subset\mathbb{R}^d$ be a Lipschitz domain with $\emph{diam}(\Omega)\leq 1/4$.
Suppose that the coefficients of $\mathcal{L}$ and $\bar{\mathcal{L}}$ satisfy
$\eqref{a:1}$, $\eqref{a:3}$ and $\eqref{a:4}$ with $\bar{\lambda} = \lambda$
satisfying $\lambda\geq\max\{\lambda_0,\mu\}$. Let $\mathcal{T}_{\mathcal{L}}^{1},
\mathcal{T}_{\bar{\mathcal{L}}}^{1}, \mathcal{T}_{\mathcal{L}}^{2}$ and
$\mathcal{T}_{\bar{\mathcal{L}}}^{2}$ be defined in $\eqref{eq:5.2}$. Then for any
$1<p<\infty$ we have
\begin{equation}\label{pri:5.11}
\begin{aligned}
\|\mathcal{T}_{\mathcal{L}}^{1}(f)
-\mathcal{T}_{\bar{\mathcal{L}}}^{1}(f)\|_{L^p(\partial\Omega)}
&\leq C\vartheta_2\|f\|_{L^p(\partial\Omega)},\\
\|\mathcal{T}_{\mathcal{L}}^{2}(f)-
\mathcal{T}_{\bar{\mathcal{L}}}^{2}(f)\|_{L^p(\partial\Omega)}
&\leq C\vartheta_2\|f\|_{L^p(\partial\Omega)},
\end{aligned}
\end{equation}
where $C$ depends on $\mu,\kappa,\tau,\lambda,m,d$ and the character of $\Omega$.
\end{thm}

\begin{remark}\label{remark:4.1}
\emph{If $\text{diam}(\Omega) = R_0> 1/4$, the constant $C$ will additionally depend on $R_0$.}
\end{remark}

\begin{proof}
The main idea may be found in \cite[Theorem 3.4]{SZW24}, and we provide a proof
for the sake of the completeness. In view of the notation given in $\eqref{eq:4.4}$ and $\eqref{eq:4.3}$,
it is not hard to see that
\begin{equation*}
\begin{aligned}
\nabla_1\mathbf{\Gamma}_{\mathcal{L}}(x,y)
-\nabla_1\mathbf{\Gamma}_{\bar{\mathcal{L}}}(x,y)
&= \Pi_{\mathcal{L}}^1(x,y) - \Pi_{\bar{\mathcal{L}}}^1(x,y)
+ \nabla_1\mathbf{E}(x-y,0;x) - \nabla_1\bar{\mathbf{E}}(x-y,0;x) \\
&= \Pi_{\mathcal{L}}^1(x,y) - \Pi_{\bar{\mathcal{L}}}^1(x,y)
+ \nabla_1\mathbf{E}_{A}(x-y,0;x) - \nabla_1\mathbf{E}_{\bar{A}}(x-y,0;x)\\
&+ \nabla_1\mathbf{R}(x-y,0;x) - \nabla_1\bar{\mathbf{R}}(x-y,0;x).
\end{aligned}
\end{equation*}

With abuse of notation the integral operators with the kernels
$\Pi_{\mathcal{L}}^1(x,y) - \Pi_{\bar{\mathcal{L}}}^1(x,y)$ and
$\nabla_1\mathbf{R}(x-y,0;x) - \nabla_1\bar{\mathbf{R}}(x-y,0;x)$ are denoted by
$\Pi_{\mathcal{L}}^1-\Pi_{\bar{\mathcal{L}}}^1$
and $\nabla_1\mathbf{R} - \nabla_1\bar{\mathbf{R}}$, respectively.
Then it follows from the estimate $\eqref{pri:4.14}$ that
\begin{equation}\label{f:5.18}
\|\big(\Pi_{\mathcal{L}}^1 - \Pi_{\bar{\mathcal{L}}}^1\big)(f)\|_{L^p(\partial\Omega)}
\leq C\vartheta_2\|f\|_{L^p(\partial\Omega)},
\end{equation}
while on account of $\eqref{pri:4.13}$ we have
\begin{equation}\label{f:5.19}
\|\big(\nabla_1\mathbf{R} - \nabla_1\bar{\mathbf{R}}\big)(f)\|_{L^p(\partial\Omega)}
\leq C\vartheta_1\|f\|_{L^p(\partial\Omega)}.
\end{equation}
Moreover, in terms of Lemma $\ref{lemma:3.6}$ it is known that the integral operator
$\nabla_1\mathbf{R}- \nabla_1\bar{\mathbf{R}}$ is compact on $L^p(\partial\Omega)$ for
$1<p<\infty$.

For the singular integral operator with kernel
$\nabla_1\mathbf{E}_{A}(x-y,0;x)-\nabla_1\mathbf{E}_{A}(x-y,0;x)$ on $L^p(\partial\Omega)$, it has
been known in \cite[Theorem 3.4]{SZW24} that
the $L^p$-bounds may be given by $C\|A-\bar{A}\|_{L^\infty(\mathbb{R}^d)}$ and we also refer the
reader to \cite[Proposition 1.2]{MMMT} for more details. Thus this together with
$\eqref{f:5.18}$ and $\eqref{f:5.19}$ finally leads to
the first line of the desired estimate $\eqref{pri:5.11}$, and the second line will be established
by the same token.
We have completed the proof.
\end{proof}

\begin{definition}\label{def:4.1}
Given $f\in L^p(\partial\Omega;\mathbb{R}^m)$ with $1<p<\infty$, the single layer potential is
defined by
\begin{equation}\label{eq:4.5}
 \mathcal{S}_{\Theta}(f)(x) = \int_{\partial\Omega} \mathbf{\Gamma}_{\Theta}(x,y)f(y)dS(y),
\end{equation}
and the double layer potential is in the form of
\begin{equation}\label{eq:4.6}
\mathcal{D}_{\Theta}(f)(x) = \int_{\partial\Omega}
\frac{\partial}{\partial\nu_{\Theta}^*(y)}
\mathbf{\Gamma}_{\Theta}(x,y)f(y)dS(y).
\end{equation}
Here the subscript $\Theta$ may be fixed by $\mathcal{L}$ or $A$,
which indicates what kind of the operator that
the single or double layer potentials are associated with, and
we may have the notation
\begin{equation}\label{eq:4.9}
\begin{aligned}
\partial/\partial\nu_{A}^*(y)
&= n(y)A^*(y)\nabla_2 \\
\partial/\partial\nu_{\mathcal{L}}^*(y)
&= n(y)\cdot\big[A^*(y)\nabla_2 + V^*(y)\big]
\end{aligned}
\end{equation}
for a.e. $y\in\partial\Omega$ in the double layer potentials $\eqref{eq:4.6}$.
\end{definition}

\begin{thm}\label{thm:5.1}
Suppose that the coefficients of $\mathcal{L}$ satisfy the conditions
$\eqref{a:1}$, $\eqref{a:2}$, $\eqref{a:3}$, and $\eqref{a:4}$ with
$\lambda\geq\max\{\lambda_0,\mu\}$.
Let $f\in L^p(\partial\Omega;\mathbb{R}^m)$ with $1<p<\infty$. Then for a.e. $P\in\partial\Omega$,
there holds
\begin{equation}\label{Id:5.1}
\big(\nabla \mathcal{S}_{\Theta}(f)\big)_{\pm}(P)
= \pm\frac{1}{2} n(P)\mathbf{H}(n(P))f(P)
+ \emph{p.v.}\int_{\partial\Omega} \nabla_1
\mathbf{\Gamma}_{\Theta}(P,y)f(y)dS(y)
\end{equation}
for $\Theta = A, \mathcal{L}$, respectively,
where $\mathbf{H}(n) = (a_{ij}^{\alpha\beta}n_in_j)^{-1}_{m\times m}$.
Moreover, we have
\begin{equation}\label{Id:5.2}
\Big(\frac{\partial \mathcal{S}_{\mathcal{L}}(f)}{\partial \nu_{\mathcal{L}}}\Big)_{\pm} =
\Big(\pm\frac{1}{2}I + \mathcal{K}_{\mathcal{L}}\Big)(f)
\quad\emph{on}~\partial\Omega,
\end{equation}
where the integral operator $\mathcal{K}_{\mathcal{L}}$ is defined by
\begin{equation*}
\small
 \mathcal{K}_{\mathcal{L}}(f)(P)
 =\emph{p.v.}\int_{\partial\Omega} \frac{\partial}{\partial\nu_{\mathcal{L}}(P)}
 \Big\{\mathbf{\Gamma}_{\mathcal{L}}(P,y)\Big\}f(y) dS(y).
\end{equation*}
and the conormal derivatives are given by
$ \partial/\partial \nu_{\mathcal{L}}(P) = n(P)A(P)\nabla_1 + n(P)V(P)$.
\end{thm}

\begin{proof}
We first mention that in the case of $\Theta = A$,
the identities $\eqref{Id:5.1}$ and $\eqref{Id:5.2}$
have been well known in \cite[Theorem 4.4]{SZW24}, while we focus on the case $\Theta = \mathcal{L}$ here.
The main idea may be found in \cite[Lemma 2.3]{SZW23} or \cite[Theorem 4.4]{SZW24}, and we provide
a proof for the sake of completeness.

In fact, the key ingredient is to show
\begin{equation}\label{f:5.17}
\big[\nabla \mathcal{S}_{\mathcal{L}}(f)\big]_{\pm} -
\big[\nabla \mathcal{S}_{A}(f)\big]_{\pm}
= \text{p.v.}\int_{\partial\Omega} \big[\nabla_1 \mathbf{\Gamma}_{\mathcal{L}}(\cdot,y)
-\nabla_1 \mathbf{\Gamma}_{A}(\cdot,y)
\big]f(y)dS(y)
\quad \text{a.e.~on~}\partial\Omega,
\end{equation}
which will be completed by Lebesgue's dominated convergence theorem. Let $P\in\partial\Omega$ be
a Lebesgue point of $f$, and let $x_n\in \Lambda^{\pm}_{N_0}(P)$
such that $x_n\to P$ as $n\to\infty$. Thus we may write
\begin{equation*}
 F_n(y) = \big[\nabla_1\mathbf{\Gamma}_{\mathcal{L}}(x_n,y) -
 \nabla_1 \mathbf{\Gamma}_{A}(x_n,y)\big]f(y)
\quad\text{and}\quad
 F_0(y) = \big[\nabla_1\mathbf{\Gamma}_{\mathcal{L}}(P,y) -
 \nabla_1 \mathbf{\Gamma}_{A}(P,y)\big]f(y)
\end{equation*}
and its clear to see that $F_n \to F_0$ a.e. on $\partial\Omega$. Moreover,
for any $x_n\in \Lambda_{N_0}^\pm(P)$ we obtain the following fact
\begin{equation}\label{f:5.16}
|P-y|\leq |P-x_n| + |x_n-y| \leq (N_0 + 1)|x_n-y|,
\end{equation}
and this implies
\begin{equation*}
\int_{\partial\Omega}|F_n|dS\leq
\int_{y\in\partial\Omega\atop
|P-y|\leq 1}\frac{|f(y)|}{|P-y|^{d-1-\tau}}dS(y) +
C\int_{\partial\Omega
}|f(y)|dS(y)
\leq C\mathrm{M}_{\partial\Omega}(f)(P) + C\|f\|_{L^p(\partial\Omega)}<\infty,
\end{equation*}
where $C$ is independent of $n$, and we employ the estimate $\eqref{pri:5.8}$ coupled with
$\eqref{f:5.16}$ in the first inequality. Thus the identity $\eqref{Id:5.1}$ for the case
$\Theta =\mathcal{L}$ will follows from $\eqref{f:5.17}$ and \cite[Theorem 4.4]{SZW24} immediately.
By $\eqref{Id:5.1}$ and the definition of
conormal derivatives of $\mathcal{L}$, it is not hard to see
the identity $\eqref{Id:5.2}$, and this ends the whole proof.
\end{proof}

\begin{thm}\label{thm:4.4}
Assume the same conditions as in Theorem $\ref{thm:5.1}$ with an additional symmetry condition
$A^* = A$.
Let $f\in L^p(\partial\Omega;\mathbb{R}^m)$ with $1<p<\infty$. Then we have
\begin{equation}\label{eq:4.8}
\big(\mathcal{D}_{\Theta}(f)\big)_{\pm} =
\Big(\mp\frac{1}{2}I + \mathcal{K}_{\Theta^*}\Big)(f)
\quad\emph{on}~\partial\Omega
\end{equation}
for $\Theta = \mathcal{L},A$,
where $\mathcal{K}_{\Theta^*}$ is given by
\begin{equation*}
\mathcal{K}_{\Theta^*}(f)(P)
= \emph{p.v.}\int_{\partial\Omega}
\frac{\partial}{\partial \nu_\Theta^*(y)}
\Big\{\mathbf{\Gamma}_\Theta(P,y)\Big\}f(y)dS(y)
\end{equation*}
for a.e. $P\in\partial\Omega$, and the notation $\partial/\partial\nu_{\Theta}^*$
is shown in $\eqref{eq:4.9}$
for $\Theta = A,\mathcal{L}$, respectively.
Moreover, if we define the following operator by
\begin{equation}\label{eq:5.4}
T(f) = \int_{\partial\Omega}n(y)\cdot\big(B^*(y)-V^*(y)\big)
\mathbf{\Gamma}_{\mathcal{L}}(x,y)f(y)dS(y),
\end{equation}
and then $\mathcal{K}_{\mathcal{L}^*}+T$ is the dual operator of
$\mathcal{K}_{\mathcal{L}}$, denoted by $\mathcal{K}_{\mathcal{L}}^*$.
\end{thm}

\begin{proof}
The case of $\Theta = A$ has been shown in \cite[Theorem 4.6]{SZW24}, and
by a similar argument we may derive the equation $\eqref{eq:4.8}$ for
$\Theta = \mathcal{L}$. Since $\nabla_x\mathbf{E}(x-y,0,x)
=-\nabla_y\mathbf{E}(x-y,0,x)$ it is not hard to see that
\begin{equation}
\Big|\nabla_2\mathbf{\Gamma}_{\mathcal{L}}(x,y)
+\nabla_1\mathbf{\Gamma}_{\mathcal{L}}(x,y)\Big|
\leq C|x-y|^{1-d+\tau},
\end{equation}
according to the estimates $\eqref{pri:4.3}$ and $\eqref{pri:4.4}$.
Let $w =\mathcal{D}_{\mathcal{L}}(f)$.
Thus
using the same procedure as in the proof of Theorem $\ref{thm:5.1}$ it
follows from the Lebesgue's dominated convergence theorem that
\begin{equation}
\begin{aligned}
w_{\pm}(P) &= -\nabla_x\mathcal{S}_{\mathcal{L}}(nA^*f)(P)
+ \int_{\partial\Omega}
n(y)A^*(y)\Big[\nabla_2\mathbf{\Gamma}_{\mathcal{L}}(P,y)
+ \nabla_1\mathbf{\Gamma}_{\mathcal{L}}(P,y)\Big] f(y)dS(y)\\
&+ \int_{\partial\Omega}n(y)V^*(y)\mathbf{\Gamma}_{\mathcal{L}}(P,y)f(y)dS(y)
\end{aligned}
\end{equation}
for a.e. $P\in\partial\Omega$.
In view of $\eqref{Id:5.1}$ and $A^*=A$ we have the trace formula $\eqref{eq:4.8}$.

From the estimates $\eqref{pri:5.10}$
it is not hard to infer that $\mathcal{K}_{\mathcal{L}^*}(f)\in L^p(\partial\Omega;\mathbb{R}^m)$ with
$1<p<\infty$. Also, for any
$g\in L^{p^\prime}(\partial\Omega;\mathbb{R}^m)$ with $1/p+1/p^\prime=1$ we have
\begin{equation}\label{eq:3.3}
\int_{\partial\Omega}\big(\mathcal{K}_{\mathcal{L}^*}+T\big)(f)(P)g(P)dS(P)
= \int_{\partial\Omega}f(y)\mathcal{K}_{\mathcal{L}}(g)(y)dS(y),
\end{equation}
which implies that
$\mathcal{K}_{\mathcal{L}^*}+T$ is the dual operator of
$\mathcal{K}_{\mathcal{L}}$,
and the proof is complete.
\end{proof}

\begin{thm}\label{thm:4.5}
Given $f\in L^p(\partial\Omega;\mathbb{R}^m)$ with $1<p<\infty$,
let $u = \mathcal{S}_{\mathcal{L}}(f)$ be the
single layer potential, and $w=\mathcal{D}_{\mathcal{L}}(f)$ be the double layer potential. Then we have
\begin{equation}\label{pri:5.13}
\|(\nabla u)^*\|_{L^p(\partial\Omega)}+\|(w)^*\|_{L^p(\partial\Omega)}
\leq C\|f\|_{L^p(\partial\Omega)},
\end{equation}
where $C$ depends on $\mu,\tau,\kappa,\lambda,m,d,p$ and $\Omega$.
\end{thm}

\begin{proof}
The main idea may be found in \cite[Theorem 3.5]{SZW24}, and we provide a proof for the sake of
completeness.
Let $P\in\partial\Omega$, and $x\in\mathbb{R}^d\setminus\partial\Omega$ such
that $x\in\Lambda_{N_0}^{\pm}(P)$. Set $r=|x-P|$.
In the case of $r\geq (1/2)$, it follows from the estimate $\eqref{pri:5.9}$ that
\begin{equation}\label{f:5.22}
\begin{aligned}
|\nabla u(x)|
&\leq C\int_{y\in\partial\Omega\atop
|y-P|\leq 4r}\frac{|f(y)|}{|x-y|^{d-1}}dS(y)
+ C\int_{y\in\partial\Omega\atop
|y-P|> 4r}\frac{|f(y)|}{|x-y|^{d-1+\rho}}dS(y)\\
&\leq C\mathrm{M}_{\partial\Omega}(f)(P),
\end{aligned}
\end{equation}
where we use the facts that $|x-y|>(r/N_0)$ if $|y-P|\leq 4r$, and
$|x-y|>3r$ if $|y-P|>4r$.

Then we study the case of $0<r<1/2$. In such the case,  there holds
\begin{equation*}
\begin{aligned}
\big|\nabla u(x)\big|
&\leq C\underbrace{\int_{y\in\partial\Omega\atop |y-P|<r}\frac{|f(y)|}{|x-y|^{d-1}}dS(y)}_{I_1}
+ \underbrace{\Big|\int_{y\in\partial\Omega\atop r<|y-P|<2}\nabla_1\mathbf{\Gamma}_{\mathcal{L}}(x,y)
f(y)dS(y)\Big|}_{I_2}\\
&+ C\underbrace{\int_{y\in\partial\Omega\atop |y-P|\geq 2}
\frac{|f(y)|}{|x-y|^{d-1+\rho}}dS(y)}_{I_3}
\end{aligned}
\end{equation*}
where $\rho\in(0,1)$, and we use the decay estimate $\eqref{pri:5.9}$ in $I_3$.
A similar computation as that given for $\eqref{f:5.22}$ will lead to
\begin{equation}\label{f:5.24}
I_1 + I_3 \leq C\mathrm{M}_{\partial\Omega}(f)(P).
\end{equation}
We now turn to study $I_2$. Since one may have
\begin{equation}\label{f:5.3.1}
\int_{y\in\partial\Omega\atop
(1/2)\leq |y-P|<2} \big|\nabla_1\mathbf{\Gamma}_{\mathcal{L}}(x,y)
f(y)\big|dS(y) \leq C\mathrm{M}_{\partial\Omega}(f)(P)
\end{equation}
via a simple geometry fact $\eqref{f:5.16}$,
it suffices to estimate the quantity
\begin{equation*}
\Big|\int_{y\in\partial\Omega\atop r<|y-P|<1/2}\nabla_1\mathbf{\Gamma}_{\mathcal{L}}(x,y)
f(y)dS(y)\Big|,
\end{equation*}
denoted by $I_2^\prime$, and
\begin{equation}\label{f:5.25}
\begin{aligned}
I_2^\prime &\leq C\int_{y\in\partial\Omega\atop
r<|y-P|<1/2} \frac{|f(y)|}{|x-y|^{d-1-\tau}}dS(y)
+ \bigg|\int_{y\in\partial\Omega\atop
r<|y-P|<1/2}\nabla_1 \mathbf{E}(x,y;y)f(y) dS(y)\bigg| \\
&\leq C\int_{y\in\partial\Omega\atop
r<|y-P|<1/2} \frac{|f(y)|}{|P-y|^{d-1-\tau}}dS(y)
+ Cr\int_{y\in\partial\Omega\atop
|y-P|>r}\frac{|f(y)|}{|P-y|^d}dS(y)\\
&+\bigg|\int_{y\in\partial\Omega\atop
r<|y-P|<1/2}\nabla_1 \mathbf{E}(P,y;P)f(y) dS(y)\bigg|\\
&\leq C\mathrm{M}_{\partial\Omega}(f)(P)
+2\sup_{\rho>0}\bigg|\int_{y\in\partial\Omega\atop
|y-P|>\rho}\nabla_1 \mathbf{E}(P,y;P)f(y) dS(y)\bigg|,
\end{aligned}
\end{equation}
where we use the estimate $\eqref{pri:4.3}$ in the first inequality. In the second one follows from
and the estimates $\eqref{pri:4.12}$, $\eqref{pri:3.0}$ coupled with $\eqref{f:5.16}$,
in which we also note the identity
$\mathbf{E}(x,y;y)=
\mathbf{E}(x,y;y)- \mathbf{E}(P,y;y)
+ \mathbf{E}(P,y;y)
-\mathbf{E}(P,y;P)+ \mathbf{E}(P,y;P)$. Here we mention that the constant $C$ in the above estimates does not depend on the
location of $x$. Hence, collecting the estimates $\eqref{f:5.22}$,
$\eqref{f:5.24}$, $\eqref{f:5.3.1}$, and $\eqref{f:5.25}$, we consequently derived that for any
$x\in \Lambda_{N_0}^{\pm}(P)$ there holds
\begin{equation*}
|\nabla u(x)| \leq C\mathrm{M}_{\partial\Omega}(f)(P)
+2\sup_{\rho>0}\bigg|\int_{y\in\partial\Omega\atop
|y-P|>\rho}\nabla_1 \mathbf{E}(P,y;P)f(y) dS(y)\bigg|,
\end{equation*}
which together with $\eqref{pri:5.10}$ implies the desired estimate $\eqref{pri:5.13}$. Its second line
may be derived in the same manner, and we have completed the proof.
\end{proof}

\subsection{Invertibility properties of layer potentials}
\begin{thm}\label{thm:4.1}
Let $\Omega\subset\mathbb{R}^d$ be a bounded Lipschitz domain with $\emph{diam}(\Omega)\leq (1/4)$.
Suppose that the coefficients of $\mathcal{L}$ satisfy $\eqref{a:1}$, $\eqref{a:3}$, $\eqref{a:4}$ and
$\eqref{a:5.1}$ with $\lambda\geq\max\{\lambda_0,\mu\}$.
Then the trace operators
$\pm(1/2)I+\mathcal{K}_{\mathcal{L}}:L^2(\partial\Omega;\mathbb{R}^m)\to
L^2(\partial\Omega;\mathbb{R}^m)$ are invertible, and there hold
\begin{equation}\label{pri:5.12}
\big\|f\big\|_{L^2(\partial\Omega)}
\leq C\big\|\big(\pm(1/2)I+\mathcal{K}_{\mathcal{L}}\big)(f)\big\|_{L^2(\partial\Omega)}
\end{equation}
for any $f\in L^2(\partial\Omega;\mathbb{R}^m)$,
where $\mathcal{K}_{\mathcal{L}}$ is defined in Theorem $\ref{thm:5.1}$,
and $C$ depends on
$\mu,\kappa,\lambda,m,d,\tau,\tau_0$ and $\Omega$.
\end{thm}

\begin{proof}
The main idea is to use a so-called continuity argument to establish the invertibility, which has
been well developed in \cite[Lemma 5.7]{SZW24} and \cite[Theorem 3.2]{GZS1}. To achieve our goal,
we first address the estimate $\eqref{pri:5.12}$.
Let $u = \mathcal{S}_{\mathcal{L}}(f)$. It is clear to see that $\mathcal{L}(u) = 0$ in
$\mathbb{R}^d\setminus\partial\Omega$. In view of Theorem $\ref{thm:4.5}$,
we have $(\nabla u)^*\in L^2(\partial\Omega)$ and $\nabla u$ exists on $\partial\Omega$
in the sense of nontangential convergence.
For any $x\in\mathbb{R}^d\setminus\Omega$ with $\text{dist}(x,\partial\Omega)>R$, we have
\begin{equation*}
  |u(x)|\leq C\int_{\partial\Omega}\frac{|f(y)|}{|x-y|^{d-2}}dS(y)
  \leq CR^{2-d}\|f\|_{L^2(\partial\Omega)},
\end{equation*}
and $|\nabla u(x)|\leq CR^{1-d}\|f\|_{L^2(\partial\Omega)}$, which means
$|u(x)|+|x||\nabla u(x)| = O(|x|^{2-d})$ as $x\to\infty$. We now have verified all the
conditions in Lemmas $\ref{lemma:4.5}$ and $\ref{lemma:4.6}$. In view of the identity $\eqref{eq:5.2}$
we obtain the jump relationship
\begin{equation}\label{f:5.21}
 f = \Big(\frac{\partial u}{\partial\nu_{\mathcal{L}}}\Big)_{+} -
 \Big(\frac{\partial u}{\partial\nu_{\mathcal{L}}}\Big)_{-}.
\end{equation}
Thus the stated estimate $\eqref{pri:5.12}$ may be reduced to
\begin{equation}\label{f:5.20}
\int_{\partial\Omega}\Big|\Big(\frac{\partial u}{\partial\nu_{\mathcal{L}}}\Big)_{\pm}\Big|^2dS
\leq C(\theta_1^{2\tau_0-2}+\theta_2^{\tau-1})
\int_{\partial\Omega}\Big|\Big(\frac{\partial u}{\partial\nu_{\mathcal{L}}}\Big)_{\mp}\Big|^2dS
+ C(\theta_1^{\tau_0} + \theta_2^{\tau_0})\int_{\partial\Omega}|(\nabla u)^*|^2 dS.
\end{equation}

Assume the claim $\eqref{f:5.20}$ for a moment, and it follows from $\eqref{f:5.21}$ that
\begin{equation*}
\|f\|_{L^2(\partial\Omega)}^2
\leq C(1+\theta_1^{2\tau_0-2}+\theta_2^{\tau_0-1})
\Big\|\Big(\frac{\partial u}{\partial\nu_{\mathcal{L}}}\Big)_{\pm}\Big\|_{L^2(\partial\Omega)}^2
+ C(\theta_1^{\tau_0} + \theta_2^{\tau_0})\|f\|_{L^2(\partial\Omega)}^2,
\end{equation*}
where we also use the estimate $\eqref{pri:5.13}$ in the inequality.
By choosing $\theta_1,\theta_2\in(0,1)$ such that $C(\theta_1^{\tau_0} + \theta_2^{\tau_0})=1/2$, we will obtain
\begin{equation*}
\|f\|_{L^2(\partial\Omega)}
\leq C\Big\|\Big(\frac{\partial u}{\partial\nu_{\mathcal{L}}}\Big)_{\pm}\Big\|_{L^2(\partial\Omega)},
\end{equation*}
and this gives the estimate $\eqref{pri:5.12}$.

We now turn to show the estimate $\eqref{f:5.20}$, it follows from Lemmas $\ref{lemma:4.6}$ and
$\ref{lemma:4.5}$ that
\begin{equation*}
\begin{aligned}
\Big\|\Big(\frac{\partial u}{\partial\nu_{\mathcal{L}}}\Big)_{\pm}\Big\|_{L^2(\partial\Omega)}^2
&\leq C\big\|(\nabla_{\text{tan}} u)_{\pm}\big\|_{L^2(\partial\Omega)}^2
+ C\theta_1^{\tau_0}\int_{\partial\Omega}|(\nabla u)^*|^2 dS
+ C\theta_1^{2\tau_0-2}\int_{\partial\Omega}|u_{\pm}|^2 dS\\
&\leq C\big\|(\nabla_{\text{tan}} u)_{\mp}\big\|_{L^2(\partial\Omega)}^2
+ C\theta_1^{\tau_0}\int_{\partial\Omega}|(\nabla u)^*|^2 dS
+ C\theta_1^{2\tau_0-2}\int_{\partial\Omega}|u_{\mp}|^2 dS \\
&\leq C\theta_2^{\tau_0-1}
\Big\|\Big(\frac{\partial u}{\partial\nu_{\mathcal{L}}}\Big)_{\mp}\Big\|_{L^2(\partial\Omega)}^2
+ C(\theta_1^{\tau_0} + \theta_2^{\tau_0})\int_{\partial\Omega}|(\nabla u)^*|^2 dS
+ C\theta_1^{2\tau_0-2}\int_{\partial\Omega}|u_{\mp}|^2 dS\\
&\leq C(\theta_1^{2\tau_0-2}+\theta_2^{\tau_0-1})
\Big\|\Big(\frac{\partial u}{\partial\nu_{\mathcal{L}}}\Big)_{\mp}\Big\|_{L^2(\partial\Omega)}^2
+ C(\theta_1^{\tau_0} + \theta_2^{\tau_0})\int_{\partial\Omega}|(\nabla u)^*|^2 dS,
\end{aligned}
\end{equation*}
where we use the fact $(\nabla_{\text{tan}}u)_{+}=(\nabla_{\text{tan}}u)_{-}$ and
$u_{+}=u_{-}$ on $\partial\Omega$ in the second step, and the last one follows from
the estimates $\eqref{pri:4.6}$ and $\eqref{pri:4.8}$.

We are ready to prove the invertibility of $\pm\frac{1}{2}I+\mathcal{K}_{\mathcal{L}}$
on $L^2(\partial\Omega;\mathbb{R}^m)$. By $\eqref{pri:5.12}$,
it suffices to show that $\pm\frac{1}{2}I+\mathcal{K}_{\mathcal{L}}:L^2(\partial\Omega;\mathbb{R}^m)
\to L^2(\partial\Omega;\mathbb{R}^m)$ are onto. Fixing the coefficients of $\mathcal{L}$ at some point
$x_0\in\mathbb{R}^d$ produces a new operator with constant coefficients and we denote it by
$\mathcal{L}_{x_0}$. Let
\begin{equation}\label{eq:5.3}
\mathcal{L}^t = t\mathcal{L} + (1-t)\mathcal{L}_{x_0},
\end{equation}
where $t\in [0,1]$, and it is not hard to verify that the coefficients of $\mathcal{L}^t$ are
still satisfy $\eqref{a:1}$, $\eqref{a:2}$, $\eqref{a:3}$
and $\eqref{a:5.1}$ with $\lambda\geq \max\{\lambda_0,\mu\}$. Thus there also hold
\begin{equation*}
\|f\|_{L^2(\partial\Omega)}\leq
C\big\|\big(\pm(1/2)I+\mathcal{K}_{\mathcal{L}^t}\big)(f)\big\|_{L^2(\partial\Omega)}
\end{equation*}
and the constant $C$ is independent of $t$. On the other hand, it is not hard to see that
$\big\{\pm\frac{1}{2}I+\mathcal{K}_{\mathcal{L}^t}:t\in[0,1]\big\}$ are continuous families of
bounded operators on $L^2(\partial\Omega;\mathbb{R}^m)$ since we have the estimate
\begin{equation}\label{f:5.26}
\|\mathcal{K}_{\mathcal{L}^{t_1}}-\mathcal{K}_{\mathcal{L}^{t_2}}\|_{L^2(\partial\Omega)\to
L^2(\partial\Omega)}
\leq C\vartheta_2\leq C|t_1-t_2|
\end{equation}
in terms of the estimate $\eqref{pri:5.11}$. Hence by the continuity method
the invertibility of $\pm\frac{1}{2}I +\mathcal{K}_{\mathcal{L}_{x_0}}$ implies our desired result at once,
and we have completed the proof.
\end{proof}

\begin{thm}\label{thm:4.3}
Assume the same conditions as in Theorem $\ref{thm:4.1}$.
Given $f\in L^2(\partial\Omega;\mathbb{R}^m)$, let $u=\mathcal{S}_{\mathcal{L}}(f)$ be
the single layer potential associated with $\mathcal{L}$. Then the operator
$\mathcal{S}_{\mathcal{L}}:L^2(\partial\Omega;\mathbb{R}^m)\to H^1(\partial\Omega;\mathbb{R}^m)$
is invertible and
\begin{equation}\label{pri:4.15}
 \|f\|_{L^2(\partial\Omega)}
 \leq C\|\mathcal{S}_{\mathcal{L}}(f)\|_{H^1(\partial\Omega)}
\end{equation}
where $C$ depends on $\mu,\kappa,\lambda,m,d,\tau,\tau_0$ and $\Omega$.
\end{thm}

\begin{proof}
We first address the estimate $\eqref{pri:4.15}$. It follows from the jump relationship
$\eqref{f:5.21}$ that
\begin{equation*}
\begin{aligned}
\big\|f\big\|_{L^2(\partial\Omega)}
&\leq \Big\|\Big(\frac{\partial u}{\partial\nu_{\mathcal{L}}}\Big)_{+}\Big\|_{L^2(\partial\Omega)}
+ \Big\|\Big(\frac{\partial u}{\partial\nu_{\mathcal{L}}}\Big)_{-}\Big\|_{L^2(\partial\Omega)}\\
&\leq C\big\|\nabla_{\text{tan}}u\big\|_{L^2(\partial\Omega)}
+ C\theta_1^{\tau_0/2}\big\|(\nabla u)^*\big\|_{L^2(\partial\Omega)}
+ C\theta_1^{\tau_0-1}\big\|u\big\|_{L^2(\partial\Omega)}\\
&\leq C\big\|\mathcal{S}_{\mathcal{L}}(f)\big\|_{H^1(\partial\Omega)}
+ C\theta_1^{\tau_0/2}\big\|f\big\|_{L^2(\partial\Omega)}
\end{aligned}
\end{equation*}
where we use the Lemmas $\ref{lemma:4.5}$ and $\ref{lemma:4.6}$ in the second inequality, and
the estimate $\eqref{pri:5.13}$ in the last one.
By choosing $\theta_1\in(0,1)$ such that $C\theta_1^{\tau_0/2} = 1/2$ we may derive the stated estimate
$\eqref{pri:4.15}$.

Then proceeding as in $\eqref{eq:5.3}$ we can construct the operator $\mathcal{L}^t$ for $t\in[0,1]$, and
$\mathcal{S}_{\mathcal{L}^t}$ denotes the corresponding single layer potential operator. The
invertibility of $\mathcal{S}_{\mathcal{L}^t}$ is based upon a continuity argument, which require that
the estimate
$\|f\|_{L^2(\partial\Omega)}
\leq C\|\mathcal{S}_{\mathcal{L}^t}(f)\|_{H^1(\partial\Omega)}$
is independent of $t$, and
$\|\mathcal{S}_{\mathcal{L}^{t_1}}
-\mathcal{S}_{\mathcal{L}^{t_2}}\|_{L^2(\partial\Omega)\to H^1(\partial\Omega)}
\leq C|t_1-t_2|$. Clearly, they can be derived in the same way as in the proof of
Theorem $\ref{thm:4.1}$, and we are done.
\end{proof}

\begin{thm}\label{thm:4.6}
Let $\epsilon_0>0$ be sufficiently small.
Assume the same conditions as in Theorem $\ref{thm:4.1}$.
If the coefficients $V,B$ additionally satisfy $\|V-B\|_{L^\infty(\partial\Omega)}\leq \epsilon_0$,
then the trace operators
$\pm(1/2)I+\mathcal{K}_{\mathcal{L}^*}:L^2(\partial\Omega;\mathbb{R}^m)\to
L^2(\partial\Omega;\mathbb{R}^m)$ are invertible, and there hold
\begin{equation}\label{pri:4.16}
\big\|f\big\|_{L^2(\partial\Omega)}
\leq C\big\|\big(\pm(1/2)I+\mathcal{K}_{\mathcal{L}^*}\big)(f)\big\|_{L^2(\partial\Omega)}
\end{equation}
for any $f\in L^2(\partial\Omega;\mathbb{R}^m)$,
where the operators $\mathcal{K}_{\mathcal{L}^*}$ is defined in Theorem $\ref{thm:4.4}$,
and $C$ depends on
$\mu,\kappa,\lambda,m,d,\tau,\tau_0$ and $\Omega$.
\end{thm}

\begin{proof}
Let $\mathcal{K}_{\mathcal{L}}^*$ be the dual operator of
$\mathcal{K}_{\mathcal{L}}$, and $f\in L^2(\partial\Omega;\mathbb{R}^m)$.
Then it follows from the estimate
$\eqref{pri:5.12}$ that
\begin{equation}\label{f:4.22}
\big\|f\big\|_{L^2(\partial\Omega)}
\leq C\big\|\big(\pm(1/2)I+\mathcal{K}_{\mathcal{L}}^*\big)(f)
\big\|_{L^2(\partial\Omega)}.
\end{equation}
Recalling that $\mathcal{K}_{\mathcal{L}}^* = \mathcal{K}_{\mathcal{L}^*}+T$, where
the operator $T$ is defined by $\eqref{eq:5.4}$,
we can arrive at
\begin{equation*}
\pm(1/2)I + \mathcal{K}_{\mathcal{L}^*}
= \pm(1/2)I + \mathcal{K}_{\mathcal{L}}^* - T
= \big(\pm(1/2)I + \mathcal{K}_{\mathcal{L}}^*\big)
\Big[I-\big(\pm(1/2)I + \mathcal{K}_{\mathcal{L}}^*\big)^{-1}T\Big].
\end{equation*}
On the one hand, the estimate $\eqref{f:4.22}$ coupled with the condition
$\|V-B\|_{L^\infty(\partial\Omega)}\leq \epsilon_0$ leads to
\begin{equation*}
\big\|\big(\pm(1/2)I + \mathcal{K}_{\mathcal{L}}^*\big)^{-1}T
\big\|_{L^2(\partial\Omega)\to L^2(\partial\Omega)}
\leq C\epsilon_0 \leq 1/2,
\end{equation*}
provided $\epsilon_0>0$ is sufficiently small, where
we also employ the estimate $\eqref{pri:4.15}$ in the first inequality.
This together with estimate $\eqref{f:4.22}$ gives
the stated estimate $\eqref{pri:4.16}$. On the other hand,
due to Theorem $\ref{thm:4.3}$
the operator $T:L^2(\partial\Omega;\mathbb{R}^d)
\to L^2(\partial\Omega;\mathbb{R}^d)$ is compact. Hence, it is not hard to see
that the
trace operators $\pm(1/2)I + \mathcal{K}_{\mathcal{L}^*}$ are invertible, and
we have completed the proof.
\end{proof}

\noindent\textbf{Proof of Theorem $\ref{thm:4.2}$.}
We first establish the existences for the Neumann problem
$(\mathbf{NH}_1)$ and the regular problem $(\mathbf{RH}_1)$.
Let $f\in L^2(\partial\Omega;\mathbb{R}^m)$, and it is not hard to see
that the single layer potential $u=\mathcal{S}_{\mathcal{L}}(f)$ satisfies
$\mathcal{L}(u) =0$ in $\mathbb{R}^d\setminus\partial\Omega$. According to
Theorem $\ref{thm:4.1}$,
the trace operator $(1/2)I+\mathcal{K}_{\mathcal{L}}$ is invertible on
$L^2(\partial\Omega;\mathbb{R}^m)$, and then
for any $g\in L^2(\partial\Omega;\mathbb{R}^m)$ the expression
$u = \mathcal{S}_{\mathcal{L}}
\big((\frac{1}{2}I+\mathcal{K}_{\mathcal{L}})^{-1}(g)\big)$ gives
a solution of $(\mathbf{NH}_1)$. Furthermore, it follows from the estimates
$\eqref{pri:5.12}$ and $\eqref{pri:5.13}$ that
$\|(\nabla u)^*\|_{L^2(\partial\Omega)} \leq C\|g\|_{L^2(\partial\Omega)}$.
For the regular problem $(\mathbf{RH}_1)$, the existence is based upon the invertibility
of $\mathcal{S}_{\mathcal{L}}:L^2(\partial\Omega;\mathbb{R}^m)\to H^1(\partial\Omega;\mathbb{R}^m)$
by Theorem $\ref{thm:4.3}$. For any $g\in H^1(\partial\Omega;\mathbb{R}^m)$, there exists
$f\in L^2(\partial\Omega;\mathbb{R}^m)$ such that
$\mathcal{S}_{\mathcal{L}}(f) = g$ on $\partial\Omega$, and one may consider
$u$ to be the solution of the Neumann problem with the boundary data $f\in L^2(\partial\Omega;\mathbb{R}^m)$.
Due to the estimate $\eqref{pri:4.15}$ we may derive
$\|(\nabla u)^*\|_{L^2(\partial\Omega)}
\leq C\|f\|_{L^2(\partial\Omega)}
\leq C\|g\|_{H^1(\partial\Omega)}$.
We now proceed to establish the existence for the Dirichlet problem $(\mathbf{DH}_1)$ with
the given data $g\in L^2(\partial\Omega;\mathbb{R}^m)$. From the estimate
$\eqref{pri:4.16}$ one may have
$\|((-1/2)I+\mathcal{K}_{\mathcal{L}^*})^{-1}\|_{L^2(\partial\Omega)\to L^2(\partial\Omega)}
\leq C$. Then the double layer potential
$w=\mathcal{D}_{\mathcal{L}}\big((-\frac{1}{2}I+\mathcal{K}_{\mathcal{L}^*})^{-1}(g)\big)$
satisfies $\mathcal{L}(w) = 0$ in $\Omega$, and
$\|(u)^*\|_{L^2(\partial\Omega)}\leq C\|g\|_{L^2(\partial\Omega)}$
due to the estimate $\eqref{pri:5.13}$.

Clearly, the uniqueness for the Neumann problem
$(\mathbf{NH}_1)$ is based upon the equality $\eqref{eq:4.7}$, while
the uniqueness for the regular problem $(\mathbf{RH}_1)$ may also follow from
the equality $\eqref{eq:4.7}$, or from that of the Dirichlet problem. So
we now turn to verify the uniqueness for
the Dirichlet problem $(\mathbf{DH}_1)$. To do so, suppose that
$\mathcal{L}(w) = 0$ in $\Omega$ with $(w)^*\in L^2(\partial\Omega)$ and
$w=0$ on $\partial\Omega$.

Set $\Sigma_{r}=\{x\in\Omega:\text{dist}(x,\partial\Omega)>r\}$.
Let $\psi_r\in C^1_0(\Omega)$ be a cut-off function
such that $\psi_r = 1$ in $\Sigma_{2r}$ and $\psi_r = 0$ outside $\Sigma_{r}$
with $|\nabla\psi_r|\leq Cr^{-1}$. Thus, it is not hard to derive that
\begin{equation*}
\mathcal{L}(\psi_r w)
= -\text{div}(A\nabla\psi_r w) - A\nabla w\nabla\psi_r - V\nabla\psi_r w
+ B\nabla\psi_r w
\quad\text{in}~\Omega.
\end{equation*}
Let $\mathcal{G}(x,y)$ denote the
Green's function associated with $\mathcal{L}$
(the existence and decay estimates may be found in \cite{QXS}), and for any $x\in\Sigma_{5r}$
we have
\begin{equation*}
\begin{aligned}
w(x) &= \int_{\Omega} \nabla_y\mathcal{G}(x,y)A(y)\nabla\psi_r(y)w(y)dy
-\int_{\Omega}\mathcal{G}(x,y) A(y)\nabla\psi_r(y)\nabla w(y) dy \\
&+\int_{\Omega}\mathcal{G}(x,y)\big[B(y)-V(y)\big]\nabla\psi_r(y)w(y) dy
\end{aligned}
\end{equation*}
and this gives
\begin{equation}\label{f:5.27}
\begin{aligned}
|w(x)|
&\leq \frac{C}{r}
\int_{\Sigma_r\setminus\Sigma_{2r}}\Big(|\nabla_y\mathcal{G}(x,y)|
+ |\mathcal{G}(x,y)|\Big)|w(y)|dy
+ \frac{C}{r}\int_{\Sigma_r\setminus\Sigma_{2r}}
|\mathcal{G}(x,y)||\nabla w(y)|dy\\
&\leq  Cr^{-1}[\delta(x)]^{1-d}
\int_{\Sigma_r\setminus\Sigma_{2r}}|w(y)|dy
+ \frac{C}{r}
\Big(\int_{\Sigma_r\setminus\Sigma_{2r}}|\nabla\mathcal{G}(x,y)|^2
dy\Big)^{\frac{1}{2}}
\Big(\int_{\Omega\setminus\Sigma_{4r}}|w(y)|^2dy\Big)^{\frac{1}{2}}\\
&\leq C[\delta(x)]^{1-d}\Big(\int_{\partial\Omega}
|\mathcal{M}_{5r}(w)|^2 dS\Big)^{1/2},
\end{aligned}
\end{equation}
where we use the fact that $|x-y|>(\delta(x)/2)$ for any
$y\in\Sigma_r\setminus\Sigma_{2r}$, as well as Poincar\'e's inequality
coupled with Caccioppoli's inequality $\eqref{pri:2.6}$, in the second step, and in the last
one we employ the co-area
formula and the definition of the radical maximal function. Since
$\mathcal{M}_{5r}(u)(P)\to 0$ for a.e. $P\in\partial\Omega$ as
$r\to 0$, and the fact $\mathcal{M}_{5r}(u)\leq (u)^*$ on $\partial\Omega$
with $(u)^*\in L^2(\partial\Omega)$, it follows from the Lebesgue dominated
theorem that the right-hand side of $\eqref{f:5.27}$ will converge to
zero as $r\to 0$. This completes the whole proof.
\qed

\subsection{Improvements}
In the following context,
we plan to get rid of  the condition $\eqref{a:5.1}$
in Theorem $\ref{thm:4.2}$.
The methods has originally been developed by Kenig and
Shen in \cite{SZW24}, which is referred to
as a three-step approximate argument.
Recall the notation $\Sigma_r = \{x\in\Omega:\text{dist}(x,\partial\Omega)>r\}$.

\begin{thm}\label{thm:5.3}
Let $\Omega\subset\mathbb{R}^d$ be a bounded Lipschitz
domain with $\emph{diam}(\Omega)\leq 1/4$ and $0\in\Omega$.
Let the coefficients of $\mathcal{L}$ satisfy
$\eqref{a:1}$, $\eqref{a:2}$, $\eqref{a:3}$ and $\eqref{a:4}$ with
$A^*=A$, as well as $\lambda\geq\max\{\lambda_0,\mu\}$.
Then one may construct a new operator $\widetilde{\mathcal{L}}$ such that
its coefficients satisfy similar conditions as those given for $\mathcal{L}$, and
there holds
\begin{equation}\label{a:5.2}
\widetilde{A} = A, \quad
\widetilde{V} = V, \quad \text{in}~\Omega\setminus\Sigma_{c_0R_0}
\end{equation}
for some small $c_0>0$.
Moreover, the new corresponding trace operators admit
\begin{equation}\label{pri:5.15}
\begin{aligned}
\big\|\big(\pm(1/2)I+\mathcal{K}_{\widetilde{\mathcal{L}}}
\big)^{-1}\big\|_{L^2(\partial\Omega)\to L^2(\partial\Omega)}
\leq C\quad\text{and}\quad
\big\|(\mathcal{S}_{\widetilde{\mathcal{L}}})^{-1}
\big\|_{H^1(\partial\Omega)\to L^2(\partial\Omega)}
\leq C,
\end{aligned}
\end{equation}
where $C$ depends on $\mu,\kappa,\tau,\lambda,m,d$ and $\Omega$.
\end{thm}

\begin{lemma}\label{lemma:5.3}
Assume the same conditions as those in Theorem $\ref{thm:5.3}$. Then there exist
$\bar{A}$, $\bar{V}$ and $\bar{B}$ satisfying
$\eqref{a:1}$, $\eqref{a:2}$
$\eqref{a:3}$ and $\eqref{a:5.1}$ such that $\bar{A}=A$, $\bar{V}=V$ and
$\bar{B}=B$ on $\partial\Omega$. Moreover, we have $(\bar{A})^* = \bar{A}$.
\end{lemma}

\begin{proof}
The proof on the coefficient $A$ has already been given in \cite[Lemma 7.1]{SZW24}, and
this lemma can be proved in the same way. Here we like to take $V$ as an example
to introduce the arguments developed in \cite[Lemma 7.1]{SZW24} to the reader.
By periodicity we may construct $\bar{V}_i^{\alpha\beta}$ as follows:
\begin{equation*}
(1)\left\{\begin{aligned}
\Delta \bar{V}_i^{\alpha\beta} &= 0 &~&\text{in}~\Omega,\\
 \bar{V}_i^{\alpha\beta} &= V_i^{\alpha\beta} &~&\text{on}~\partial\Omega,
\end{aligned}\right.
\qquad\text{and}\qquad
(2)\left\{\begin{aligned}
\Delta \bar{V}_i^{\alpha\beta} &= 0 &~&\text{in}~\mathrm{Y}\setminus\Omega,\\
 \bar{V}_i^{\alpha\beta} &= V_i^{\alpha\beta} &~&\text{on}~\partial\Omega, \\
 \bar{V}_i^{\alpha\beta} &= 1 &~&\text{on}~\partial\mathrm{Y}.
\end{aligned}\right.
\end{equation*}
It is clear to see that the boundary condition $\bar{V}_i^{\alpha\beta} = 1$ may
guarantee the extension of $\bar{V}$ to $\mathbb{R}^d$ in a periodic way. Thus
$\bar{V}$ actually satisfies the condition $\eqref{a:2}$. On account of
the maximum principle it is known that
$\|\bar{V}\|_{L^\infty(\mathrm{Y})}\leq \|V\|_{L^\infty(\partial\Omega)}\leq \kappa_1$, which means
that $\bar{V}$ shares the same condition $\eqref{a:3}$. Furthermore,
by a global H\"older estimate (see for example \cite[Theorem 1.2]{QXS}),
we have $\max\big\{\|\bar{V}\|_{C^{0,\tau_0}(\Omega)},
\|\bar{V}\|_{C^{0,\tau_0}(\mathrm{Y}\setminus\Omega)}\big\}\leq C\kappa$ where
$\tau_0\in(0,\tau]$. Let $\delta(x) =\text{dist}(x,\partial\Omega)$, and
$x^\prime\in\partial\Omega$ be the point such that $\delta(x)=|x-x^\prime|$.  Then it follows from a interior Lipschitz estimate that
\begin{equation*}
|\nabla V(x)| \leq \frac{C}{\delta(x)}\dashint_{B(x,\delta(x))}|V(y)-V(x^\prime)| dy
\leq C\kappa\big[\delta(x)\big]^{\tau_0-1}
\end{equation*}
and this verified the condition $\eqref{a:5.1}$. Then $\bar{B}$ follows from a similar construction like
$(1)$ and $(2)$, and it will be proved to
satisfy the conditions $\eqref{a:2}$, $\eqref{a:3}$ and $\eqref{a:5.1}$. Up to now, we have completed the proof.
\end{proof}

\noindent\textbf{Proof of Theorem $\ref{thm:5.3}$.}
The main idea may be found in \cite[pp.34-36]{SZW24}, and we provide a proof for the sake of
the completeness. Let $\psi\in C_0^1(-\frac{1}{2},\frac{1}{2})$ be a cut-off function such that
$\psi = 1$ in $(-\frac{1}{4},\frac{1}{4})$, $\psi = 0$ outside $(-\frac{3}{7},\frac{3}{7})$.
Let $\bar{A},\bar{V}$ be given as in Lemma $\ref{lemma:5.3}$, and
we may define
\begin{equation*}
\begin{aligned}
 A^t(x) &= \psi\Big(\frac{\delta(x)}{t}\Big)A(x) + \Big[1-\psi\Big(\frac{\delta(x)}{t}\Big)\Big]
 \bar{A}(x),\\
 V^t(x) &= \psi\Big(\frac{\delta(x)}{t}\Big)V(x) + \Big[1-\psi\Big(\frac{\delta(x)}{t}\Big)\Big]
 \bar{V}(x)
\end{aligned}
\end{equation*}
for any $x\in\mathrm{Y}$, where $t\in(0,1/10)$ and $\delta(x)=\text{dist}(x,\partial\Omega)$. Then
the coefficients $A^t,V^t$ may be extended to $\mathbb{R}^d$ in a periodicity way. Also, they
will satisfy the conditions $\eqref{a:1}$,$\eqref{a:3}$ and $\eqref{a:4}$, by which we can
construct a new operator
\begin{equation}
\mathcal{L}^t = -\text{div}(A^t\nabla + V^t) + B\nabla + c + \lambda I.
\end{equation}
Let the notation $\mathbf{\Gamma}_{\mathcal{L}^t}$ be the fundamental solution of $\mathcal{L}^t$, which
may define the corresponding trace operator, denoted by $\mathcal{K}_{\mathcal{L}^t}$. Thus one has the
following identity
\begin{equation}
\begin{aligned}
  (\pm1/2)I + \mathcal{K}_{\mathcal{L}^t}
  &= (\pm1/2)I + \mathcal{K}_{\bar{\mathcal{L}}}
  + \mathcal{K}_{\mathcal{L}^t} - \mathcal{K}_{\bar{\mathcal{L}}}\\
  &= \big((\pm1/2)I + \mathcal{K}_{\bar{\mathcal{L}}}\big)
  \Big[I - \underbrace{\big((\pm1/2)I + \mathcal{K}_{\bar{\mathcal{L}}}\big)^{-1}
  \big(\mathcal{K}_{\bar{\mathcal{L}}}-\mathcal{K}_{\mathcal{L}^t}\big)}_{T}\Big].
\end{aligned}
\end{equation}
We claim that the problem is reduced to
\begin{equation}\label{f:5.34}
\big\|\mathcal{K}_{\mathcal{L}^t} - \mathcal{K}_{\bar{\mathcal{L}}}
\big\|_{L^2(\partial\Omega)\to L^2(\partial\Omega)}
\leq C\vartheta_2
\leq Ct^{\tau/2},
\end{equation}
where $t\in(0,1/10)$ and we actually employ the estimate $\eqref{pri:5.11}$ in
the first inequality,
and the above concrete computations are as similar as those in
\cite[Lemma 7.2]{SZW24}. So we do not reproduce it here.
In view of $\eqref{f:5.34}$ we may derive
\begin{equation}\label{f:5.35}
\|T\|_{L^2(\partial\Omega)\to L^2(\partial\Omega)}\leq 1/2
\end{equation}
by choosing a suitable $t>0$, where we also use the estimate $\eqref{pri:5.12}$ in the inequality.
Moreover, there holds
\begin{equation*}
 \frac{1}{2}\big\|f\big\|_{L^2(\partial\Omega)}
 \leq \big\|(I-T)(f)\big\|_{L^2(\partial\Omega)}
 \leq C\big\|\big((\pm1/2)I + \mathcal{K}_{\mathcal{L}^t}\big)(f)\big\|_{L^2(\partial\Omega)}
\end{equation*}
for any $f\in L^2(\partial\Omega;\mathbb{R}^m)$, and the invertibility of the trace operators
$(\pm1/2)I + \mathcal{K}_{\mathcal{L}^t}$ is based upon the estimate $\eqref{f:5.35}$ and
the invertibility of $(\pm1/2)I + \mathcal{K}_{\bar{\mathcal{L}}}$. This gives
\begin{equation}
\big\|\big((\pm1/2)I + \mathcal{K}_{\mathcal{L}^t}\big)^{-1}
\big\|_{L^2(\partial\Omega)\to L^2(\partial\Omega)}\leq 2C,
\end{equation}
which is exactly the first estimate in $\eqref{pri:5.15}$.

Similarly, to show the second one in $\eqref{pri:5.15}$ may be reduced to estimate
\begin{equation}\label{f:5.36}
\begin{aligned}
 \|(\mathcal{S}_{\mathcal{L}^{t}}-\mathcal{S}_{\bar{\mathcal{L}}})(f)\|_{H^1(\partial\Omega)}
& \leq \|(\mathcal{T}_{\mathcal{L}^t}^{1}-\mathcal{T}_{\bar{\mathcal{L}}}^{1})(f)\|_{L^2(\partial\Omega)}
 + \|(\mathcal{S}_{\mathcal{L}^{t}}-\mathcal{S}_{\bar{\mathcal{L}}})(f)\|_{L^2(\partial\Omega)} \\
& \leq C\big\{\vartheta_2+\vartheta_1\big\}\|f\|_{L^2(\partial\Omega)}
\leq C\big\{t^{\tau/2}+t^\tau\big\}\|f\|_{L^2(\partial\Omega)}
\leq Ct^{\tau/2}\|f\|_{L^2(\partial\Omega)}
\end{aligned}
\end{equation}
where we use the estimates $\eqref{pri:4.9}$ and $\eqref{pri:5.11}$ in the second inequality,
and the third one follows from \cite[Lemma 7.2]{SZW24}, and the parameter $t\in(0,1/10)$ will be chosen later.
Then
\begin{equation*}
\mathcal{S}_{\mathcal{L}^t}
= \mathcal{S}_{\bar{\mathcal{L}}} + \mathcal{S}_{\mathcal{L}^t} - \mathcal{S}_{\bar{\mathcal{L}}}
=\mathcal{S}_{\bar{\mathcal{L}}}
\Big[I-\mathcal{S}_{\bar{\mathcal{L}}}^{-1}
\big(\mathcal{S}_{\bar{\mathcal{L}}}-\mathcal{S}_{\mathcal{L}^t}\big)\Big]
\end{equation*}
will lead to
\begin{equation}\label{f:5.37}
\big\|f\big\|_{L^2(\partial\Omega)}
\leq C\|\mathcal{S}_{\mathcal{L}^t}(f)\|_{H^1(\partial\Omega)}
+\big\|\mathcal{S}_{\bar{\mathcal{L}}}^{-1}
\big(\mathcal{S}_{\mathcal{L}^t} - \mathcal{S}_{\bar{\mathcal{L}}}\big)(f)\big\|_{L^2(\partial\Omega)}
\end{equation}
since the single layer potential $\mathcal{S}_{\bar{\mathcal{L}}}$ satisfy the estimate
$\eqref{pri:4.15}$. Thus it follows from the estimates $\eqref{pri:4.15}$ and $\eqref{f:5.36}$ that
\begin{equation*}
\big\|\mathcal{S}_{\bar{\mathcal{L}}}^{-1}
\big(\mathcal{S}_{\mathcal{L}^t} - \mathcal{S}_{\bar{\mathcal{L}}}\big)(f)\big\|_{L^2(\partial\Omega)}
\leq C\big\|\big(\mathcal{S}_{\mathcal{L}^t}
- \mathcal{S}_{\bar{\mathcal{L}}}\big)(f)\big\|_{H^1(\partial\Omega)}
\leq Ct^{\tau/2}\|f\|_{L^2(\partial\Omega)}.
\end{equation*}
Plugging the above estimate back into $\eqref{f:5.37}$ will give the desired estimate by choosing
small $t>0$ such that $Ct^{\tau/2}=1/2$. Also, the above estimate implies the invertibility of
$\mathcal{S}_{\mathcal{L}^t}$. Finally, fixed the small parameter $t$,
let the notation $\widetilde{\mathcal{L}}$ represent $\mathcal{L}^t$ and the constructed coefficients will
satisfy the condition $\eqref{a:5.2}$. We have completed the whole proof.
\qed

\begin{thm}
Suppose that the coefficients of $\mathcal{L}$ and
$\bar{\mathcal{L}}$ satisfy $\eqref{a:1}$, $\eqref{a:2}$, $\eqref{a:3}$ and
$\eqref{a:4}$ with $A^*=A$ and $\bar{A}^*=\bar{A}$. Also, we assume that
the coefficients of $\mathcal{L}$ and
$\bar{\mathcal{L}}$ agree on the condition $\eqref{a:5.2}$.
Let $u$ be the solution of
$\mathcal{L}(u) = 0$ in $\Omega$ with $(\nabla u)^*\in L^2(\partial\Omega)$. Then we have
\begin{equation}\label{pri:5.14}
\begin{aligned}
\|\nabla u\|_{L^2(\partial\Omega)}
& \leq C\Big\|\frac{\partial u}{\partial\nu}\Big\|_{L^2(\partial\Omega)} \\
\|\nabla u\|_{L^2(\partial\Omega)}
& \leq C\|\nabla_{\emph{tan}}u\|_{L^2(\partial\Omega)}
+ C\|u\|_{L^2(\partial\Omega)}
\end{aligned}
\end{equation}
where $C$ depends on $\mu,\kappa,\lambda,\tau,m,d$ and $\Omega$.
\end{thm}

\begin{proof}
Let $\psi\in C_0^1(\mathbb{R}^d)$ be a cut-off function such that
$\psi= 1$ in
$\big\{x\in\mathbb{R}^d:\text{dist}(x,\partial\Omega)\leq (r_0/30)\big\}$
and $\psi = 0$ outside
$\{x\in\mathbb{R}^d:\text{dist}(x,\partial\Omega)> (r_0/20)\}$ with
$|\nabla\psi|\leq Cr_0^{-1}$. Set $\tilde{u}=\psi u$, and it is not hard to
derive
\begin{equation}\label{f:5.28}
\bar{\mathcal{L}}(\tilde{u}) = -\text{div}(\bar{A}\nabla\psi u)
-\bar{A}\nabla\psi\nabla u + (\bar{B}-\bar{V})\nabla\psi u
\qquad\text{in}~\Omega,
\end{equation}
and $\tilde{u} = u$ on $\partial\Omega$,
where we use the condition $\eqref{a:5.2}$. By using it again we have
\begin{equation*}
\frac{\partial \tilde{u}}{\partial\bar{\nu}}
= \frac{\partial u}{\partial\nu}
\quad\text{on}~\partial\Omega,
\quad\text{where}~\partial/\partial\bar{\nu}
= n\cdot(\bar{A}\nabla + \bar{V}).
\end{equation*}
For the ease of the statement, we denote the right-hand side of
$\eqref{f:5.28}$ by $F$. We mention that $F$ is supported in $\Omega$ and
may be zero-extended to $\mathbb{R}^d$, still denoted by itself.
To estimate the first line of $\eqref{pri:5.14}$
we start from the following equations:
$\bar{\mathcal{L}}(\tilde{u}) = F$ in $\Omega$ with
$\partial\tilde{u}/\partial\nu_{\bar{\mathcal{L}}} = \partial u/\partial\nu_{\mathcal{L}}$ on $\partial\Omega$, which may be divided by
\begin{equation*}
(1)~\bar{\mathcal{L}}(v) = F
\quad\text{in}~\mathbb{R}^d,
\qquad\text{and}\quad
(2)~\left\{\begin{aligned}
\bar{\mathcal{L}}(w) &= 0
&~&\text{in}~\Omega,\\
\frac{\partial w}{\partial\nu_{\bar{\mathcal{L}}}}&=
\frac{\partial u}{\partial\nu_{\mathcal{L}}}
-\frac{\partial v}{\partial\nu_{\mathcal{L}}} &~&\text{on}~\partial\Omega.
\end{aligned}\right.
\end{equation*}
For $(1)$, let $\bar{\mathbf{\Gamma}}(x,y)$ be the fundamental solution associated with
$\bar{\mathcal{L}}$. Thus one may have
\begin{equation*}
\begin{aligned}
 v(x) &= \int_{\Omega}\nabla_y\bar{\mathbf{\Gamma}}(x,y)A(y)\nabla\psi(y)u(y)dy
 -\int_{\Omega}\bar{\mathbf{\Gamma}}(x,y)A(y)\nabla\psi(y)\nabla u(y)dy \\
 &-\int_{\Omega}\bar{\mathbf{\Gamma}}(x,y)\big[B(y)-V(y)\big]
 \nabla\psi(y)u(y)dy
\end{aligned}
\end{equation*}
for any $x\in\mathbb{R}^d$,
where we use the fact $A\nabla\psi u = 0$ on $\partial\Omega$. Then
for any $x\in\Omega\setminus\Sigma_{r_0/60}$ there holds
\begin{equation}\label{f:5.29}
|\nabla v(x)|^2\leq Cr_0^{-d-1}\int_\Omega
\big(|\nabla u|^2 + |u|^2\big)dy
\leq C\Big\|\frac{\partial u}{\partial\nu_{\mathcal{L}}}\Big\|_{L^2(\partial\Omega)}
\|u\|_{L^2(\partial\Omega)},
\end{equation}
where we use the identity $\eqref{eq:4.7}$.
On the other hand, it follows from
an interior $L^p$ estimate that there holds
$\|u\|_{W^{1,p}(\Sigma_{r_0/80})}\leq C_p\|u\|_{H^1(\Omega)}$ for any
$p\geq 2$. Also, by observing $\eqref{f:5.29}$
we may assume $|\nabla v(x)|=O(|x|^{-d-1})$ as $|x|\to\infty$ and so
the energy estimate will leads to
$\|v\|_{H^1(\mathbb{R}^d)}\leq C\|u\|_{H^1(\Omega)}$. Thus
the above two estimates together with an interior Lipschitz estimate show
\begin{equation}\label{f:5.30}
\|\nabla v\|_{L^\infty(\Sigma_{r_0/60})}\leq
C\Big\{\|v\|_{L^2(\Omega)}+\|u\|_{W^{1,q}(\Sigma_{r_0/80})}\Big\}
\leq C\Big\|\frac{\partial u}{\partial\nu_{\mathcal{L}}}
\Big\|_{L^2(\partial\Omega)}^{1/2}
\|u\|_{L^2(\partial\Omega)}^{1/2}
\end{equation}
where $q>d$, as well as
\begin{equation}\label{f:5.31}
\|v\|_{L^2(\partial\Omega)}\leq C\|v\|_{H^1(\Omega)}
\leq C\Big\|\frac{\partial u}{\partial\nu_{\mathcal{L}}}
\Big\|_{L^2(\partial\Omega)}^{1/2}
\|u\|_{L^2(\partial\Omega)}^{1/2}
\leq C\theta\|\nabla u\|_{L^2(\partial\Omega)}
+C_\theta\|u\|_{L^2(\partial\Omega)}
\end{equation}
where we use Young's inequality with small $\theta$.

Then the estimates $\eqref{f:5.29}$ and
$\eqref{f:5.30}$ may define the nontangential maximal function
$(\nabla v)^*$  on $\partial\Omega$, and give
\begin{equation}\label{f:5.32}
\|\nabla v\|_{L^2(\partial\Omega)}
\leq \|(\nabla v)^*\|_{L^2(\partial\Omega)}
\leq C\Big\|\frac{\partial u}{\partial\nu_{\mathcal{L}}}
\Big\|_{L^2(\partial\Omega)}^{1/2}
\|u\|_{L^2(\partial\Omega)}^{1/2}
\end{equation}


Since $(\nabla u)^*\in L^2(\partial\Omega)$, one may obtain
$(\nabla w)^*\in L^2(\partial\Omega)$. Hence the equation $(2)$ can be
considered as a Neumann problem,
and it follows from Theorem $\ref{thm:4.2}$ that
\begin{equation}\label{f:5.33}
\begin{aligned}
\|\nabla w\|_{L^2(\partial\Omega)}
&\leq C\Big\{\Big\|\frac{\partial u}{\partial\nu_{\mathcal{L}}}\Big\|_{L^2(\partial\Omega)}
+ \|\nabla v\|_{L^2(\partial\Omega)}\Big\} \\
&\leq
C\Big\{\Big\|\frac{\partial u}{\partial\nu_{\mathcal{L}}}\Big\|_{L^2(\partial\Omega)}
+ \Big\|\frac{\partial u}{\partial\nu_{\mathcal{L}}}\Big\|_{L^2(\partial\Omega)}^{1/2}
\|u\|_{L^2(\partial\Omega)}^{1/2}\Big\}
\leq C\Big\|\frac{\partial u}{\partial\nu_{\mathcal{L}}}\Big\|_{L^2(\partial\Omega)}
\end{aligned}
\end{equation}
where we also use the estimate $\eqref{pri:4.6}$ in the last inequality.
Thus recalling that $u = w+v$ near $\partial\Omega$, we have
\begin{equation*}
\|\nabla u\|_{L^2(\partial\Omega)}
\leq \|\nabla w\|_{L^2(\partial\Omega)} + \|\nabla v\|_{L^2(\partial\Omega)}
\leq  C\Big\|\frac{\partial u}{\partial\nu_{\mathcal{L}}}\Big\|_{L^2(\partial\Omega)}
\end{equation*}
where we employ the estimates $\eqref{f:5.33}$ and $\eqref{f:5.32}$ coupled with
$\eqref{pri:4.6}$ in the last inequality, and this gives the first line of
$\eqref{pri:5.14}$.

Similarly, to obtain the second line of $\eqref{pri:5.14}$
the following equations:$\bar{\mathcal{L}}(\tilde{u}) = F$ in $\Omega$ with
$\tilde{u} = u$ on $\partial\Omega$ may be divided by
\begin{equation*}
(1)~\bar{\mathcal{L}}(v) = F
\quad\text{in}~\mathbb{R}^d,
\qquad\text{and}\quad
(3)~\left\{\begin{aligned}
\bar{\mathcal{L}}(w) &= 0
&~&\text{in}~\Omega,\\
w&= u-v &~&\text{on}~\partial\Omega.
\end{aligned}\right.
\end{equation*}

Since $(\nabla w)^*\in L^2(\partial\Omega)$, the equation $(3)$ can be
regarded as a regular problem,
and it follows from Theorem $\ref{thm:4.2}$ that
\begin{equation}
\begin{aligned}
 \|(\nabla w)^*\|_{L^2(\partial\Omega)}
& \leq C\|w\|_{H^1(\partial\Omega)}
 \leq C\Big\{\|u\|_{H^1(\partial\Omega)}
 +\|v\|_{H^1(\partial\Omega)}\Big\}\\
& \leq C\|\nabla_{\text{tan}}u\|_{L^2(\partial\Omega)}
+ C\theta\|\nabla u\|_{L^2(\partial\Omega)}
+C_\theta\|u\|_{L^2(\partial\Omega)},
\end{aligned}
\end{equation}
where we employ the estimates $\eqref{f:5.31}$ and
$\eqref{f:5.32}$ in the second inequality. Then it is not hard to see
\begin{equation*}
\begin{aligned}
\|\nabla u\|_{L^2(\partial\Omega)}
&\leq \|\nabla w\|_{L^2(\partial\Omega)}
+ \|\nabla v\|_{L^2(\partial\Omega)}\\
&\leq C\|\nabla_{\text{tan}}u\|_{L^2(\partial\Omega)}
+ C\theta\|\nabla u\|_{L^2(\partial\Omega)}
+C_\theta\|u\|_{L^2(\partial\Omega)},
\end{aligned}
\end{equation*}
and this yields the second line of $\eqref{pri:5.14}$ by
choosing $\theta\in(0,1)$ such that $C\theta=1/2$. We have completed the whole proof.
\end{proof}

\section{$L^2$ boundary value problems in full scales}\label{sec:5}

\subsection{Rellich estimates for large scales}

\begin{thm}\label{thm:6.1}
Suppose that the coefficients of $\mathcal{L}_\varepsilon$ satisfy
$\eqref{a:1}$, $\eqref{a:2}$ and $\eqref{a:3}$ with $A^*=A$. For given $g\in H^1(\partial\Omega;\mathbb{R}^m)$,
let $u_\varepsilon\in H^1(\Omega;\mathbb{R}^m)$
be a weak solution to the Dirichlet problem $\eqref{pde:1.1}$. Then for any $\varepsilon\leq r<R_0$
we have
\begin{equation}\label{pri:6.3}
\bigg\{\frac{1}{r}\int_{\Omega\setminus\Sigma_{r}} \big(|\nabla u_\varepsilon|^2
+|u_\varepsilon|^2\big) dx\bigg\}^{1/2}
\leq C \|g\|_{H^1(\partial\Omega)},
\end{equation}
where $C$ depends only on $\mu,\kappa,\lambda,m,d$ and $\Omega$.
\end{thm}

\begin{proof}
It suffices to prove the estimate $\eqref{pri:6.3}$ in the case of $r=\varepsilon$, and
this result
will be derived from \cite[Theorem 4.1]{QX2} without any real difficulty.
\end{proof}

\begin{thm}\label{thm:6.2}
Suppose that the coefficients of $\mathcal{L}_\varepsilon$ satisfy
$\eqref{a:1}$, $\eqref{a:2}$ and $\eqref{a:3}$ with $A^*=A$. Let $u_\varepsilon\in H^1(\Omega;\mathbb{R}^m)$
be a weak solution to the Neumann problem $\eqref{pde:1.2}$ with
$g\in L^2(\partial\Omega;\mathbb{R}^d)$. Then there holds
\begin{equation}\label{pri:6.2}
\bigg\{\frac{1}{r}\int_{\Omega\setminus\Sigma_r}\big(|\nabla u_\varepsilon|^2 + |u_\varepsilon|^2\big) dx
\bigg\}^{1/2}
\leq C\|g\|_{L^2(\partial\Omega)}
\end{equation}
for any $\varepsilon\leq r<R_0$, where $C$ depends only on $\mu,\kappa,\lambda,m,d$ and $\Omega$.
\end{thm}

\begin{proof}
See \cite[Theorem 5.10]{QXS1}.
\end{proof}

\begin{lemma}\label{lemma:6.2}
Let $\Omega$ be bounded Lipschitz domain. Suppose that the coefficients of
$\mathcal{L}_\varepsilon$ satisfy $\eqref{a:1}$, $\eqref{a:2}$, $\eqref{a:3}$ and
$\eqref{a:4}$ with $A=A^*$. Let $u_\varepsilon\in C^{1}(\Omega;\mathbb{R}^m)$ be a solution of
$\mathcal{L}_\varepsilon(u_\varepsilon) = 0$ in $\Omega$ with
$(\nabla u_\varepsilon)^*\in L^2(\partial\Omega)$. Then we have
\begin{equation}\label{pri:6.1}
\begin{aligned}
\int_{\partial\Omega}|\nabla u_\varepsilon|^2 dS
&\leq C\int_{\partial\Omega}\Big|\frac{\partial u_\varepsilon}{\partial\nu_\varepsilon}\Big|^{2}dS, \\
\int_{\partial\Omega}|\nabla u_\varepsilon|^2 dS
&\leq C\int_{\partial\Omega}|\nabla_{\emph{tan}}u_\varepsilon|^{2}dS
+ C\int_{\partial\Omega}|u_\varepsilon|^{2}dS,
\end{aligned}
\end{equation}
where $C$ depends on $\mu,\kappa,\lambda,\tau,m,d$ and $\Omega$.
\end{lemma}

\begin{proof}
The estimate $\eqref{pri:6.1}$ is based upon Theorems
$\ref{thm:6.1}$ and $\ref{thm:6.2}$ and the main idea may be found in
\cite[Remark 3.3]{SZW12} or \cite[Theorem 3.7]{GZS1}.
Let $D_r = B(P,r)\cap\Omega$ with
$P\in\partial\Omega$, and $\Delta_r = B(P,r) \cap \partial\Omega$,
where $r\in[\varepsilon/4,\varepsilon)$. Since $\mathcal{L}_\varepsilon(u_\varepsilon) = 0$ in $D_r$,
from the estimate $\eqref{pri:5.14}$ it follows that
\begin{equation*}
\begin{aligned}
\int_{\Delta_{\varepsilon/4}}|\nabla u_\varepsilon|^2 dS\leq
\int_{\partial D_r}|\nabla u_\varepsilon|^2 dS
&\leq C\int_{\partial D_r}\Big|\frac{\partial u_\varepsilon}{\partial \nu_\varepsilon}\Big|^2 dS \\
&\leq C\int_{\Delta_{\varepsilon}}\Big|\frac{\partial u_\varepsilon}{\partial \nu_\varepsilon}\Big|^2 dS
+ C\int_{\partial D_r\setminus\Delta_\varepsilon} |\nabla u_\varepsilon|^2 dS,
\end{aligned}
\end{equation*}
and this shows
\begin{equation*}
\int_{\Delta_{\varepsilon/4}}|\nabla u_\varepsilon|^2 dS\leq
C\int_{\Delta_{\varepsilon}}\Big|\frac{\partial u_\varepsilon}{\partial \nu_\varepsilon}\Big|^2 dS
+\frac{C}{\varepsilon}\int_{D_{\varepsilon}}|\nabla u_\varepsilon|^2 dx
\end{equation*}
by integrating $r$ from $\varepsilon/4$ to $\varepsilon$. Then using a covering argument leads to
the following estimate
\begin{equation*}
\int_{\partial\Omega}|\nabla u_\varepsilon|^2 dS\leq
C\int_{\partial\Omega}\Big|\frac{\partial u_\varepsilon}{\partial \nu_\varepsilon}\Big|^2 dS
+\frac{C}{\varepsilon}\int_{\Omega\setminus\Sigma_\varepsilon}|\nabla u_\varepsilon|^2 dx.
\end{equation*}
Hence, this together with $\eqref{pri:6.2}$ gives the first line of the stated estimate $\eqref{pri:6.1}$.

To get the second one of $\eqref{pri:6.1}$, we may similarly derive
\begin{equation*}
\int_{\Delta_{\varepsilon/4}}|\nabla u_\varepsilon|^2 dS\leq
C\int_{\Delta_{\varepsilon}}\big(|\nabla_{\text{tan}}u_\varepsilon|^2
+|u_\varepsilon|^2\big) dS
+\frac{C}{\varepsilon}\int_{D_{\varepsilon}}
\big(|\nabla_{\text{tan}}u_\varepsilon|^2
+|u_\varepsilon|^2\big) dx
\end{equation*}
by using the estimate $\eqref{pri:5.14}$ and the same covering argument gives
\begin{equation*}
\begin{aligned}
\int_{\partial\Omega}|\nabla u_\varepsilon|^2 dS
&\leq C\int_{\partial\Omega}\big(|\nabla_{\text{tan}}u_\varepsilon|^2
+|u_\varepsilon|^2\big) dS
+\frac{C}{\varepsilon}\int_{\Omega\setminus\Sigma_\varepsilon}
\big(|\nabla_{\text{tan}}u_\varepsilon|^2
+|u_\varepsilon|^2\big) dx \\
&\leq  C\int_{\partial\Omega}\big(|\nabla_{\text{tan}}u_\varepsilon|^2
+|u_\varepsilon|^2\big) dS,
\end{aligned}
\end{equation*}
where we use the estimate $\eqref{pri:6.3}$ in the second inequality, and this ends the proof.
\end{proof}

\begin{lemma}\label{lemma:6.3}
Suppose that the coefficients of
$\mathcal{L}_\varepsilon$ satisfy $\eqref{a:1}$, $\eqref{a:2}$, $\eqref{a:3}$ and
$\eqref{a:4}$ with $\lambda\geq\max\{\lambda_0,\mu\}$ and $A=A^*$.
Let $u_\varepsilon$ satisfy $\mathcal{L}_\varepsilon(u_\varepsilon) = 0$ in $\Omega_{-}$
with $(\nabla u_\varepsilon)^*\in L^2(\partial\Omega)$, and
$\nabla u_\varepsilon$ exists in the sense of nontangential convergence on $\partial\Omega$.
We further assume that
$|u_\varepsilon(x)|=O(|x|^{2-d})$ with $|\nabla u_\varepsilon(x)|=O(|x|^{1-d})$ as $|x|\to\infty$.
Then, there holds
\begin{equation}\label{pri:6.6}
\begin{aligned}
&\int_{\partial\Omega} |\nabla u_\varepsilon|^2 dS
\leq C\int_{\partial\Omega}\Big|\frac{\partial u_\varepsilon}{\partial \nu_{\varepsilon}}\Big|^2 dS
+ C\int_{\Omega_{-}}|\nabla u_\varepsilon|^2 dx\\
&\int_{\partial\Omega} |\nabla u_\varepsilon|^2 dS
\leq C\int_{\partial\Omega}|\nabla_{\emph{tan}}u_\varepsilon|^2 dS
+ C\int_{\Omega_{-}}(|\nabla u_\varepsilon|^2 + |u_\varepsilon|^2) dx
\end{aligned}
\end{equation}
and
\begin{equation}\label{pri:6.7}
 \int_{\partial\Omega}|u_\varepsilon|^2 dS
 \leq C\int_{\partial\Omega}
 \Big|\frac{\partial u_\varepsilon}{\partial \nu_{\varepsilon}}\Big|^2dS
\end{equation}
where $C$ depends on $\mu,\kappa,\lambda,d,m$ and $\Omega$.
\end{lemma}

\begin{proof}
The main idea may be found in \cite[Theorem 2.4]{GZS1}, and we provide a proof for the sake of the completeness.
Let $Q\in\partial\Omega$, and $r=r_0/8$.
Since $\mathcal{L}_\varepsilon(u_\varepsilon) = 0$ in $B(Q,r)\cap\Omega_{-}$, it follows from
the first line of $\eqref{pri:6.1}$ that
\begin{equation*}
\int_{\Delta_{r/2}}|\nabla u_\varepsilon|^2 dS
\leq \int_{\Delta_{r}}\Big|\frac{\partial u_\varepsilon}{\partial \nu_\varepsilon}\Big|^2 dS
+ C\int_{\Omega_{-}}|\nabla u_\varepsilon|^2 dx
\end{equation*}
and this implies the first line of $\eqref{pri:6.6}$ by a covering argument.
Then using the same idea that approached for $\eqref{f:3.6}$ we may have
\begin{equation}\label{f:6.5}
 \int_{\partial\Omega} |u_\varepsilon|^2 dS
 \leq C\int_{\Omega_{-}}\big(|\nabla u_\varepsilon|^2 + |u_\varepsilon|^2\big)dx
\end{equation}
where $C$ depends on $d,m$ and $r_0$. Due to the decay conditions, it is not hard to observe that
\begin{equation}\label{f:6.6}
 \frac{\lambda}{2}\int_{\Omega_{-}}\big(|\nabla u_\varepsilon|^2 + |u_\varepsilon|^2\big)dx
 \leq \int_{\partial\Omega}\Big|\frac{\partial u_\varepsilon}{\partial\nu_\varepsilon} u_\varepsilon\Big|dS
\end{equation}
where we use the estimate $\eqref{pri:2.0.2}$ with the condition $\lambda\geq\max\{\lambda_0,\mu\}$,
as well as the divergence theorem. Combining the estimates $\eqref{f:6.5}$ and $\eqref{f:6.6}$
leads to the stated estimate $\eqref{pri:6.7}$ by Young's inequality.

We now proceed to show the second line of $\eqref{pri:6.6}$.
On account of the estimate $\eqref{pri:6.1}$, one may derive that
\begin{equation*}
\begin{aligned}
\int_{\Delta_{r/2}} |\nabla u_\varepsilon|^2 dS
& \leq C\int_{\Delta_r}\big(|\nabla_{\text{tan}} u_\varepsilon|^2+|u_\varepsilon|^2\big)dS
+ C\int_{\Omega_{-}}\big(|\nabla u_\varepsilon|^2 + |u_\varepsilon|^2\big)dx \\
& \leq C\int_{\partial\Omega} |\nabla_{\text{tan}} u_\varepsilon|^2 dS
+ C\int_{\Omega_{-}}\big(|\nabla u_\varepsilon|^2 + |u_\varepsilon|^2\big) dx
\end{aligned}
\end{equation*}
where we use the estimate $\eqref{f:6.5}$ in the second inequality,
and this gives the second line of $\eqref{pri:6.6}$. We have completed the proof.
\end{proof}

\subsection{Methods of layer potentials}

Let $\mathbf{\Gamma}_{\varepsilon}$ and $\mathbf{\Gamma}_{A_\varepsilon}$
denote the fundamental solutions related to $\mathcal{L}_\varepsilon$ and
$L_\varepsilon = -\text{div}(A(x/\varepsilon)\nabla)$, respectively. Then one may have
a comparing lemma in the small scale.

\begin{lemma}[comparing lemma III]\label{lemma:5.2}
Suppose that the coefficients of $\mathcal{L}_\varepsilon$ satisfy $\eqref{a:1}$ and $\eqref{a:3}$
with $\lambda\geq\max\{\lambda_0,\mu\}$. Assume that
$A,V,B$ satisfy $\eqref{a:2}$ and $\eqref{a:4}$.
Then there holds
\begin{equation}\label{pri:6.5}
\begin{aligned}
|\nabla_1\mathbf{\Gamma}_{\varepsilon}(x,y)-\nabla_1\mathbf{\Gamma}_{A_\varepsilon}(x,y)|
&\leq C\varepsilon^{-\tau} |x-y|^{1-d+\tau},\\
|\nabla_2\mathbf{\Gamma}_{\varepsilon}(x,y)-\nabla_2\mathbf{\Gamma}_{A_\varepsilon}(x,y)|
&\leq C \varepsilon^{-\tau}|x-y|^{1-d+\tau}
\end{aligned}
\end{equation}
for any $x,y\in\mathbb{R}^d$ with $0<|x-y|\leq\varepsilon\leq 1$,
where $C$ depends only on $\mu,\kappa,\lambda,m,d$ and $\tau$.
\end{lemma}

\begin{proof}
The proof is based upon a rescaling observation, and
it will be divided into two steps.

Step 1.  Consider a family of the new operators
\begin{equation*}
 \mathcal{L}^\delta
 = -\text{div}\big(A\nabla + \delta V\big) + \delta B\nabla + \delta^2 (c+\lambda I)
 \qquad \text{for}~\delta\in(0,1],
\end{equation*}
and let $\mathbf{\Gamma}_{\mathcal{L}^\delta}(\cdot,y)$ be the fundamental solution of
$\mathcal{L}^\delta$, satisfying
$\mathcal{L}^\delta\big(\mathbf{\Gamma}_{\mathcal{L}^\delta}(\cdot,y)\big) = 0$
in $\mathbb{R}^d\setminus\{y\}$.
The existence of $\mathbf{\Gamma}_{\mathcal{L}^\delta}(\cdot,y)$ may be
found in Theorem $\ref{thm:2.3.1}$. Set $\delta =\varepsilon$, and then we claim that
\begin{equation}\label{f:6.7}
 \mathbf{\Gamma}_{\varepsilon}(x,y)
 =\varepsilon^{2-d}\mathbf{\Gamma}_{\mathcal{L}^\varepsilon}(x/\varepsilon,y/\varepsilon),
 \qquad x\not=y, \quad 0<\varepsilon\leq 1.
\end{equation}

Let $v^\varepsilon(x,y)
= \varepsilon^{2-d}\mathbf{\Gamma}_{\mathcal{L}^\varepsilon}(x/\varepsilon,y/\varepsilon)$,
and it suffices to prove $v^\varepsilon(x,y) = \mathbf{\Gamma}_{\varepsilon}(x,y)$.
For any $\phi\in C_0^\infty(\mathbb{R}^d;\mathbb{R}^m)$, we obtain
\begin{equation}\label{f:6.8}
\begin{aligned}
\mathrm{B}_{\mathcal{L}_\varepsilon;\mathbb{R}^d}[v^\varepsilon(\cdot,y),\phi]
&= \int_{\mathbb{R}^d}
\Big\{A(x^\prime)\nabla_1\mathbf{\Gamma}_{\mathcal{L}^\varepsilon}(x^\prime,y^\prime)
+\varepsilon V(x^\prime)\mathbf{\Gamma}_{\mathcal{L}^\varepsilon}(x^\prime,y^\prime)\Big\}
\nabla_{x^\prime}\phi(\varepsilon x^\prime) dx^\prime \\
& + \int_{\mathbb{R}^d} \Big\{\varepsilon B(x^\prime)
\nabla_1\mathbf{\Gamma}_{\mathcal{L}^\varepsilon}(x^\prime,y^\prime)
+\varepsilon^2 \mathbf{c}(x^\prime)\mathbf{\Gamma}_{\mathcal{L}^\varepsilon}(x^\prime,y^\prime)
\Big\} \phi(\varepsilon x^\prime) dx^\prime
= \phi(y),
\end{aligned}
\end{equation}
where $\mathbf{c} = c+\lambda I$ and $x^\prime = x/\varepsilon$, and we employ the fact that
\begin{equation*}
 \int_{\mathbb{R}^d} \delta_{y/\varepsilon}(x^\prime)\phi(\varepsilon x^\prime) dx^\prime
 = \phi(y)\qquad \forall \phi\in C_0^\infty(\mathbb{R}^d),
\end{equation*}
and $\delta_{y/\varepsilon}$ is the Delta function with pole at $y/\varepsilon$.
Hence, by the uniqueness of $v^\varepsilon(\cdot,y)$ in $\eqref{f:6.8}$,
we conclude that $v^\varepsilon(x,y) = \mathbf{\Gamma}_{\varepsilon}(x,y)$.

Then, in terms of the proof of Lemma $\ref{lemma:4.7}$ we claim that
\begin{equation}\label{f:6.9}
\big|\nabla_1\mathbf{\Gamma}_{\mathcal{L}^\varepsilon}(x,y)-
\nabla_1\mathbf{\Gamma}_{A}(x,y)\big|\leq C|x-y|^{1-d+\tau}
\end{equation}
for any $x,y\in\mathbb{R}^d$ with $0<|x-y|\leq 1$, where $\varepsilon\in(0,1]$, and $C$ is
independent of $\varepsilon$.

Assume $\eqref{f:6.9}$ holds for a moment. This together with $\eqref{f:6.7}$ gives
\begin{equation*}
\begin{aligned}
\big|\nabla_1\mathbf{\Gamma}_{\varepsilon}(x,y)-\nabla_1\mathbf{\Gamma}_{A_\varepsilon}(x,y)\big|
&=\varepsilon^{1-d}
\big|\nabla_1\mathbf{\Gamma}_{\mathcal{L}^\varepsilon}(x/\varepsilon,y/\varepsilon)
-\nabla_1\mathbf{\Gamma}_{A}(x/\varepsilon,y/\varepsilon)\big| \\
&\leq C\varepsilon^{-\tau}|x-y|^{1-d+\tau}
\end{aligned}
\end{equation*}
for $0<|x-y|\leq \varepsilon$, where we also use the fact
$\mathbf{\Gamma}_{A_\varepsilon}(x,y)
=\varepsilon^{2-d}
\mathbf{\Gamma}_{A}(x/\varepsilon,y/\varepsilon)$ (see \cite[(4.23)]{SZW24}).
This will lead to the stated estimates $\eqref{pri:6.5}$.

Step 2. We will show the estimate $\eqref{f:6.9}$. For any $\varepsilon\in(0,1]$,
fix all the coefficients of $\mathcal{L}^\varepsilon$ at a point $x\in\mathbb{R}^d$ and then let
the notation $\mathbf{E}^\varepsilon(x,y;x)$ be the corresponding fundamental solution.
In view of the proof of $\eqref{pri:4.3}$, it is not
hard to obtain
\begin{equation}\label{f:6.10}
 \big|\nabla_1\mathbf{\Gamma}_{\mathcal{L}^\varepsilon}(x,y)-\nabla_1\mathbf{E}^\varepsilon(x,y;x)\big|
 \leq C|x-y|^{1-d+\tau}
\end{equation}
for any $0<|x-y|\leq 1$,
where the constant $C$ is independent of $\varepsilon$. Meanwhile,
recalling the proof of the estimate $\eqref{pri:3.4}$, one may similarly obtain
\begin{equation}\label{f:6.11}
\big|\nabla_1\mathbf{E}^\varepsilon(x,y;x)-\nabla_1\mathbf{E}_{A}(x,y;x)\big|
\leq C|x-y|^{2-d},
\end{equation}
where $C$ is also independent of $\varepsilon$. Thus, collecting the estimates
$\eqref{f:6.10}$, $\eqref{f:6.11}$ and \cite[Lemma 2.2]{SZW24}, we consequently
get the desired estimate $\eqref{f:6.9}$ by a triangle inequality.
\end{proof}

Recalling Definition $\ref{def:4.2}$,
we have the $L^p$ bounds for singular integral operators on Lipschitz surfaces.

\begin{thm}\label{thm:6.10}
Let $f\in L^p(\partial\Omega;\mathbb{R}^m)$ for $1<p<\infty$.
Then $\mathcal{T}_{\Theta}(f)$ exists for a.e. $P\in\partial\Omega$ and
\begin{equation}\label{pri:5.2.1}
\begin{aligned}
\|\mathcal{T}_{\Theta}^{1}(f)\|_{L^p(\partial\Omega)}
+ \|\mathcal{T}_{\Theta}^{1,*}(f)\|_{L^p(\partial\Omega)} &\leq C\|f\|_{L^p(\partial\Omega)},\\
\|\mathcal{T}_{\Theta}^{2}(f)\|_{L^p(\partial\Omega)}
+ \|\mathcal{T}_{\Theta}^{2,*}(f)\|_{L^p(\partial\Omega)} &\leq C\|f\|_{L^p(\partial\Omega)},
\end{aligned}
\end{equation}
hold for $\Theta = A_\varepsilon,\varepsilon$,
where $C$ depends only on $\mu,\kappa,\tau,\lambda,d,m,p$ and $\Omega$.
\end{thm}

\begin{proof}
We still borrow some ideas from \cite[Theorem 3.1]{SZW24}, and due to a rescaling argument, it
directly implies the stated result in the case of $\Theta =A_\varepsilon$.
Our now task is to show the estimate $\eqref{pri:5.2.1}$ for $\Theta =\varepsilon$.
It is sufficient to estimate the integral
\begin{equation*}
 \Big|\int_{y\in\partial\Omega\atop
 |y-P|>\delta} \nabla_1 \mathbf{\Gamma}_{\varepsilon}(P,y)f(y)dS(y)\Big|,
\end{equation*}
and we will study two cases: (1) $\delta >\varepsilon$; (2) $\delta \leq \varepsilon$.

For (1). We apply the estimate $\eqref{pri:0.7}$ to deriving
\begin{equation*}
\begin{aligned}
 \Big|\int_{y\in\partial\Omega\atop
 |y-P|>\delta} \nabla_1 \mathbf{\Gamma}_{\varepsilon}(P,y)f(y)dS(y)\Big|
&\leq C\varepsilon^\rho \int_{y\in\partial\Omega\atop
 |y-P|>\delta} \frac{|f(y)|}{|y-P|^{d-1+\rho}}dS(y) \\
&\leq C\Big|\int_{y\in\partial\Omega\atop
 |y-P|>\delta} \nabla_1 \mathbf{\Gamma}_{0}(P-y)f(y)dS(y)\Big|
 + C\int_{\partial\Omega}\frac{|f(y)|}{|y-P|^{d-2}}dS(y) \\
&\leq C\mathrm{M}_{\partial\Omega}(f)(P)
+ C\mathcal{T}_{0}^*(f)(P)
+ C\int_{\partial\Omega}\frac{|f(y)|}{|y-P|^{d-2}}dS(y).
\end{aligned}
\end{equation*}

For (2). We use the estimate $\eqref{pri:6.5}$ to obtain
\begin{equation*}
\begin{aligned}
 \Big|\int_{y\in\partial\Omega\atop
 |y-P|>\delta} \nabla_1 \mathbf{\Gamma}_{\varepsilon}(P,y)f(y)dS(y)\Big|
&\leq C\varepsilon^{-\tau}\int_{y\in\partial\Omega\atop
 \delta<|y-P|\leq\varepsilon} \frac{|f(y)|}{|y-P|^{d-1-\tau}}dS(y) \\
& + 2\mathcal{T}^{1,*}_{A_\varepsilon}(f)(P)
+\Big|\int_{y\in\partial\Omega\atop
 |y-P|>\varepsilon} \nabla_1 \mathbf{\Gamma}_{\varepsilon}(P,y)f(y)dS(y)\Big| \\
 \leq C\mathrm{M}_{\partial\Omega}(f)(P)
&+ 2\mathcal{T}^{1,*}_{A_\varepsilon}(f)(P)
+ C\mathcal{T}_{0}^*(f)(P)
+ C\int_{\partial\Omega}\frac{|f(y)|}{|y-P|^{d-2}}dS(y).
\end{aligned}
\end{equation*}
Obviously, combining two cases we have the first line of the estimate $\eqref{pri:5.2.1}$
by $\eqref{pri:3.3}$ and the fractional integral estimates (see \cite{MGLM}). The second line
follows by the same token and we have completed the proof.
\end{proof}

Recalling Definition $\ref{def:4.1}$,
the single and double layer potentials
$\mathcal{S}_{\Theta}(f)$ and $\mathcal{D}_{\Theta}(f)$ are given in the form of
$\eqref{eq:4.5}$ and $\eqref{eq:4.6}$, respectively. In this section,
their subscript $\Theta$ will be chosen as $A_\varepsilon$ or $\varepsilon$
to indicate what kind of the operator that the single
or double layer potentials are associated with.
Note that the conormal derivative of $u_\varepsilon$ will be given in the same way
as those in $\eqref{eq:4.9}$, and we write $\partial/\partial \nu_{\varepsilon}
= \partial/\partial \nu_{\mathcal{L}_\varepsilon}$
for the ease of the statement.

\begin{thm}\label{thm:6.3}
Suppose that the coefficients of $\mathcal{L}_\varepsilon$ satisfy the conditions
$\eqref{a:1}$, $\eqref{a:2}$, $\eqref{a:3}$, and $\eqref{a:4}$ with
$\lambda\geq\max\{\lambda_0,\mu\}$ and $A=A^*$.
Let $f\in L^p(\partial\Omega;\mathbb{R}^m)$ with $1<p<\infty$. Then for a.e. $P\in\partial\Omega$,
there holds
\begin{equation}\label{Id:6.1}
\big(\nabla \mathcal{S}_{\varepsilon}(f)\big)_{\pm}(P)
= \pm\frac{1}{2} n(P)\mathbf{H}_\varepsilon(n(P))f(P)
+ \emph{p.v.}\int_{\partial\Omega} \nabla_1
\mathbf{\Gamma}_{\varepsilon}(P,y)f(y)dS(y),
\end{equation}
where $\mathbf{H}_\varepsilon(n) = (a_{ij}^{\alpha\beta}(x/\varepsilon)n_in_j)^{-1}_{m\times m}$.
Moreover, we have
\begin{equation}\label{Id:6.2}
\begin{aligned}
\Big(\frac{\partial \mathcal{S}_{\varepsilon}(f)}{\partial \nu_{\varepsilon}}\Big)_{\pm}
& = \Big(\pm\frac{1}{2}I + \mathcal{K}_{\varepsilon}\Big)(f)
\quad\emph{on}~\partial\Omega, \\
\big(\mathcal{D}_{\varepsilon}(f)\big)_{\pm}
& = \Big(\mp\frac{1}{2}I + \mathcal{K}_{\mathcal{L}_\varepsilon^*}\Big)(f)
\quad\emph{on}~\partial\Omega,
\end{aligned}
\end{equation}
where the integral operators $\mathcal{K}_{\varepsilon}$ and $\mathcal{K}_{\mathcal{L}_\varepsilon^*}$
are defined by
\begin{equation*}
\small
\begin{aligned}
 \mathcal{K}_{\varepsilon}(f)(P)
& =\emph{p.v.}\int_{\partial\Omega} \frac{\partial}{\partial\nu_{\varepsilon}(P)}
 \Big\{\mathbf{\Gamma}_{\varepsilon}(P,y)\Big\}f(y) dS(y) \\
\mathcal{K}_{\mathcal{L}_\varepsilon^*}(f)(P)
&= \emph{p.v.}\int_{\partial\Omega}
\frac{\partial}{\partial \nu_{\varepsilon}^*(y)}
\Big\{\mathbf{\Gamma}_{\varepsilon}(P,y)\Big\}f(y)dS(y)
\end{aligned}
\end{equation*}
and the conormal derivatives are given by
$ \partial/\partial \nu_{\varepsilon}(P) = n(P)A(P/\varepsilon)\nabla_1 + n(P)V(P/\varepsilon)$ and
$\partial/\partial\nu_{\varepsilon}^*(y)
= n(y)\cdot\big[A^*(y/\varepsilon)\nabla_2 + V^*(y/\varepsilon)\big]$, respectively.
\end{thm}

\begin{proof}
Since the estimates $\eqref{pri:6.5}$ are established in Lemma $\ref{lemma:5.2}$,
we may accomplish the proof
by the analogous arguments to those given for Theorems $\ref{thm:5.1}$, $\ref{thm:4.4}$.
Thus we do not reproduced here.
\end{proof}

\begin{remark}
\emph{If we define the following operator
\begin{equation*}
T_\varepsilon(f) = \int_{\partial\Omega}n(y)\cdot\big(B^*(y/\varepsilon)-V^*(y/\varepsilon)\big)
\mathbf{\Gamma}_{\varepsilon}(x,y)f(y)dS(y),
\end{equation*}
then
$ \mathcal{K}_{\varepsilon}^* =
\mathcal{K}_{\mathcal{L}_\varepsilon^*}+T_\varepsilon$ will be the dual operator of
$\mathcal{K}_{\varepsilon}$ (see the proof of Theorem $\ref{thm:4.4}$).}
\end{remark}

\begin{lemma}\label{lemma:6.1}
Given $f\in L^p(\partial\Omega;\mathbb{R}^m)$ with $1<p<\infty$,
let $u_\varepsilon = \mathcal{S}_{\varepsilon}(f)$ be the
single layer potential, and $w_\varepsilon=\mathcal{D}_{\varepsilon}(f)$ be the double layer potential. Then we have
\begin{equation}\label{pri:6.4}
\|(\nabla u_\varepsilon)^*\|_{L^p(\partial\Omega)}+\|(w_\varepsilon)^*\|_{L^p(\partial\Omega)}
\leq C\|f\|_{L^p(\partial\Omega)},
\end{equation}
where $C$ depends on $\mu,\tau,\kappa,\lambda,m,d,p$ and $\Omega$.
\end{lemma}

\begin{proof}
In order to simply the proof,
we manage to use the known result \cite[Theorem 3.5]{SZW24} as much as possible.
To do so, let $v_\varepsilon=\mathcal{D}_{A_\varepsilon}(f)$ denote the double layer potential
associated with the operator $L_\varepsilon=-\text{div}(A(x/\varepsilon)\nabla)$, and it is
known that $(v_\varepsilon)^*(P)$ exists for a.e. $P\in\partial\Omega$ and there holds
\begin{equation}\label{f:6.1}
\|(v_\varepsilon)^*\|_{L^p(\partial\Omega)}\leq C\|f\|_{L^p(\partial\Omega)}
\end{equation}
for any $f\in L^p(\partial\Omega;\mathbb{R}^m)$, and we refer the reader to \cite[Theorem 3.5]{SZW24}.
Moreover, its proof also reveals the following estimate
\begin{equation}\label{f:5.3.2}
\Big|\int_{y\in\partial\Omega\atop
r<|y-P|<\varepsilon/2}\frac{\partial}{\partial\nu_{A_\varepsilon}^*(y)}
\Big\{\mathbf{\Gamma}_{A_\varepsilon}(x,y)\Big\}f(y)dS(y)\Big|
\leq (v_\varepsilon)^*(P) + C\mathrm{M}_{\partial\Omega}(f)(P) + C\mathcal{T}^{*}_{\widehat{A}}(g)(P)
\end{equation}
for any $x\in \Lambda_{N_0}^{\pm}(P)$ with $P\in\partial\Omega$,
in which $|g(y)|\leq C|f(y)|$ on $\partial\Omega$ and, $C$ is independent of $\varepsilon$.

Fixed $P\in\partial\Omega$, let $r=|x-P|$, where $x\in\Lambda_{N_0}^{\pm}(P)$.
Recalling the definition of double layer potential associated with $\mathcal{L}_\varepsilon$, we have
\begin{equation*}
w_\varepsilon(x) = \underbrace{\int_{\partial\Omega} n(y)A^*(y/\varepsilon)
\nabla_2\mathbf{\Gamma}_{\varepsilon}(x,y)f(y)dS(y)}_{w_{\varepsilon,1}(x)}
+  \underbrace{\int_{\partial\Omega} n(y)V^*(y/\varepsilon)
\mathbf{\Gamma}_{\varepsilon}(x,y)f(y)dS(y)}_{w_{\varepsilon,2}(x)}.
\end{equation*}
Then it is clear to see that
\begin{equation*}
 |w_{\varepsilon,2}(x)|\leq \int_{\partial\Omega}\frac{|f(y)|}{|x-y|^{d-2}}dS(y)
 \leq C\int_{\partial\Omega}\frac{|f(y)|}{|y-P|^{d-2}}dS(y),
\end{equation*}
where we use a simple geometry fact $\eqref{f:5.16}$, and this implies
\begin{equation}\label{f:5.3.3}
 (w_{\varepsilon,2})^*(P) \leq C\int_{\partial\Omega}\frac{|f(y)|}{|y-P|^{d-2}}dS(y).
\end{equation}

Thus the main task is to estimate the quantity $(w_{\varepsilon,1})^*(P)$, and we divide the proof into
two parts: the first one handles the case $r\geq \varepsilon/2$, while
the second one is dedicated to the case $r<\varepsilon/2$. Let
\begin{equation*}
 \tilde{f}(y) = n(y)A^*(y/\varepsilon)f(y) \quad \text{with}~~|\tilde{f}(y)|\leq C|f(y)|,
\end{equation*}
where $C$ depends only on $\mu,d,m$ and the character of $\Omega$.

In the case of $r\geq \varepsilon/2$, using the estimate $\eqref{pri:0.7}$ leads to
\begin{equation*}
\begin{aligned}
|w_{\varepsilon,1}(x)|
&\leq \Big|\int_{y\in\partial\Omega\atop |y-P|>4r}
\nabla_2\mathbf{\Gamma}_{\varepsilon}(x,y)\tilde{f}(y)dS(y)\Big|
+ \Big|\int_{y\in\partial\Omega\atop |y-P|\leq 4r}
\nabla_2\mathbf{\Gamma}_{\varepsilon}(x,y)\tilde{f}(y)dS(y)\Big| \\
&\leq C\varepsilon^\rho\int_{y\in\partial\Omega\atop|y-P|>4r}
\frac{|\tilde{f}(y)|}{|x-y|^{d-1+\rho}}dS(y)
+ \Big|\int_{y\in\partial\Omega\atop|y-P|>4r}\frac{\partial}{\partial n_0^{*}(y)}
\Big\{\mathbf{\Gamma}_0(x-y)\Big\}\tilde{g}(y)dS(y)\Big| \\
& + \Big|\int_{y\in\partial\Omega\atop|y-P|>4r}\mathbf{\Gamma}_0(x-y)
\nabla_z\chi_0^*(z)\tilde{f}(y)dS(y)\Big| + C\mathrm{M}_{\partial\Omega}(f)(P),
\end{aligned}
\end{equation*}
where $z=y/\varepsilon$
and $\tilde{g}(y) = (n(y)\widehat{A^*})^{-1}n(y)A^*(y)\big[I+\nabla_z\vec{\chi}(z)^*\big]f(y)$ with
$\vec{\chi} = (\chi_1,\cdots,\chi_d)$.
An analogous computation used in Lemma $\ref{lemma:3.9}$ will lead to
\begin{equation} \label{f:6.2}
|w_{\varepsilon,1}(x)|
\leq C\mathrm{M}_{\partial\Omega}(f)(P) + (w_0)^*(P) + C\int_{\partial\Omega}
\frac{|f(y)|}{|y-P|^{d-2}} dS(y)
\end{equation}
where $w_0 = \mathcal{D}_0(\tilde{g})$ (recall the definition in $\eqref{def:3.2}$).

We now handle the case of $0<r\leq \varepsilon/2$. In such the case, we observe that
\begin{equation*}
\begin{aligned}
\big|w_{\varepsilon,1}(x)\big|
&\leq C\underbrace{\int_{y\in\partial\Omega\atop |y-P|<r}\frac{|f(y)|}{|x-y|^{d-1}}dS(y)}_{I_1}
+ \underbrace{\Big|\int_{y\in\partial\Omega\atop r<|y-P|\leq (\varepsilon/2)}
\nabla_2\mathbf{\Gamma}_{\varepsilon}(x,y)
\tilde{f}(y)dS(y)\Big|}_{I_2}\\
& + \underbrace{\Big|\int_{y\in\partial\Omega\atop (\varepsilon/2)<|y-P|\leq 2\varepsilon}
\nabla_2\mathbf{\Gamma}_{\varepsilon}(x,y)
\tilde{f}(y)dS(y)\Big|}_{I_3}
+ \underbrace{\Big|\int_{y\in\partial\Omega\atop |y-P|> 2\varepsilon}
\nabla_2\mathbf{\Gamma}_{\varepsilon}(x,y)
\tilde{f}(y)dS(y)\Big|}_{I_4}
\end{aligned}
\end{equation*}
where we use the decay estimates $\eqref{pri:2.12}$ in the inequality,
and in fact the term $I_4$ has already been developed in previous step. Thus,
\begin{equation}\label{f:6.3}
I_1 + I_3 + I_4  \leq C\mathrm{M}_{\partial\Omega}(f)(P) + (w_0)^*(P) + C\int_{\partial\Omega}
\frac{|f(y)|}{|y-P|^{d-2}} dS(y).
\end{equation}
We proceed to estimate $I_2$ by using Lemma $\ref{lemma:5.2}$ and the estimate $\eqref{f:5.3.2}$,
and we obtain
\begin{equation}\label{f:6.4}
\begin{aligned}
I_2
&\leq C\varepsilon^{-\tau}\int_{y\in\partial\Omega\atop
r<|y-P|<\varepsilon/2} \frac{|f(y)|}{|P-y|^{d-1-\tau}}dS(y)
+\Big|\int_{y\in\partial\Omega\atop
r<|y-P|<\varepsilon/2}\frac{\partial}{\partial\nu_{A_\varepsilon}^*(y)}
\Big\{\mathbf{\Gamma}_{A_\varepsilon}(x,y)\Big\}f(y)dS(y)\Big|\\
&\leq C\mathrm{M}_{\partial\Omega}(f)(P) + (v_\varepsilon)^*(P)
+ C\mathcal{T}^{*}_{\widehat{A}}(g)(P).
\end{aligned}
\end{equation}

Hence, combining the estimates $\eqref{f:5.3.3}$, $\eqref{f:6.2}$, $\eqref{f:6.3}$ and
$\eqref{f:6.4}$, we consequently derived that for any
$x\in \Lambda_{N_0}^{\pm}(P)$ there holds
\begin{equation*}
|w_\varepsilon(x)| \leq C\mathrm{M}_{\partial\Omega}(f)(P) + (w_0)^*(P)
+(v_\varepsilon)^*(P)+ C\mathcal{T}^{*}_{\widehat{A}}(g)(P)
+ C\int_{\partial\Omega}
\frac{|f(y)|}{|y-P|^{d-2}} dS(y),
\end{equation*}
which together with $\eqref{f:6.1}$, $\eqref{pri:3.14}$, $\eqref{pri:3.3}$ and
fractional integral estimates (see \cite{MGLM}) implies
\begin{equation*}
\|(w_\varepsilon)^*\|_{L^p(\partial\Omega)} \leq C\|f\|_{L^p(\partial\Omega)}.
\end{equation*}
This is one part of the estimates $\eqref{pri:6.4}$, and we plan to end the proof here
since the corresponding part for the single layer potential $u_\varepsilon$ will
be accomplished by the same procedure without any real difficulty.
\end{proof}

\begin{thm}\label{thm:6.4}
Let $\Omega\subset\mathbb{R}^d$ be a bounded Lipschitz domain.
Suppose that the coefficients of $\mathcal{L}_\varepsilon$
satisfy $\eqref{a:1}$, $\eqref{a:2}$, $\eqref{a:3}$, and $\eqref{a:4}$
with $\lambda\geq\max\{\lambda_0,\mu\}$.
Then the trace operators
$\pm(1/2)I+\mathcal{K}_{\varepsilon}:L^2(\partial\Omega;\mathbb{R}^m)\to
L^2(\partial\Omega;\mathbb{R}^m)$ are invertible,
where $\mathcal{K}_{\mathcal{L}_\varepsilon}$ is defined in Theorem $\ref{thm:6.3}$. Moreover, there holds
the following estimates
\begin{equation}\label{pri:6.8}
\big\|f\big\|_{L^2(\partial\Omega)}
\leq C\big\|\big(\pm(1/2)I+\mathcal{K}_{\varepsilon}\big)(f)\big\|_{L^2(\partial\Omega)}
\end{equation}
and
\begin{equation}\label{pri:6.9}
 \|f\|_{L^2(\partial\Omega)}
 \leq C\|\mathcal{S}_{\varepsilon}(f)\|_{H^1(\partial\Omega)}
\end{equation}
for any $f\in L^2(\partial\Omega;\mathbb{R}^m)$,
where $C$ depends on
$\mu,\kappa,\lambda,m,d,\tau$ and $\Omega$.
\end{thm}

\begin{proof}
The main ideas have been used in Theorems $\ref{thm:3.3}$ and $\ref{thm:4.1}$, originally
from \cite[Lemma 5.7]{SZW24} and \cite[Theorem 3.2]{GZS1}.
Let $u_\varepsilon = \mathcal{S}_{\varepsilon}(f)$, and it is not hard to see that
$\mathcal{L}_\varepsilon (u_\varepsilon) = 0$ in $\mathbb{R}^d\setminus\partial\Omega$.
Then in view of Lemma $\ref{lemma:6.1}$ it is known that
$(\nabla u_\varepsilon)^*\in L^2(\partial\Omega)$. Thus due to
the jump relation
\begin{equation*}
f = \Big(\frac{\partial S_{\varepsilon}(f)}{\partial\nu_\varepsilon}\Big)_{+}
- \Big(\frac{\partial S_{\varepsilon}(f)}{\partial\nu_\varepsilon}\Big)_{-}
\end{equation*}
from $\eqref{Id:6.2}$, and the results in Lemmas $\ref{lemma:6.2}$ and $\ref{lemma:6.3}$,
the same calculation to that given for Theorem $\ref{thm:3.3}$ leads to
the following estimate
\begin{equation*}
\Big\|\Big(\frac{\partial \mathcal{S}_{\varepsilon}(f)}{\partial \nu_\varepsilon}\Big)_{\mp}
\Big\|_{L^2(\partial\Omega)}
\leq C\Big\|\Big(\frac{\partial \mathcal{S}_{\varepsilon}(f)}{\partial
\nu_\varepsilon}\Big)_{\pm}\Big\|_{L^2(\partial\Omega)}.
\end{equation*}
This yields the stated estimate $\eqref{pri:6.8}$.

The invertibility of $\pm\frac{1}{2}I+\mathcal{K}_{\varepsilon}$
on $L^2(\partial\Omega;\mathbb{R}^m)$ is based upon a continuity argument. We mention that the proof of this part
is quite similar to that given for $\eqref{thm:4.1}$.
Assume the same operator $\mathcal{L}_{x_0}$ as in $\eqref{eq:5.4}$, and let
\begin{equation}
 \mathcal{L}_\varepsilon^\theta = \theta\mathcal{L}_\varepsilon + (1-\theta)\mathcal{L}_{x_0}.
\end{equation}
Note that the coefficients of $\mathcal{L}_\varepsilon^\theta$ satisfy $\eqref{a:1}$,
$\eqref{a:2}$, $\eqref{a:3}$ and $\eqref{a:4}$ with $\lambda\geq\max\{\mu,\lambda_0\}$.
It follows from the estimate $\eqref{pri:6.8}$ that
\begin{equation*}
\|f\|_{L^2(\partial\Omega)} \leq C\big\|(\pm1/2)I+\mathcal{K}_\varepsilon^\theta\big\|_{L^2(\partial\Omega)}
\end{equation*}
where $C$ is independent of $\theta$. Also, we may obtain
\begin{equation*}
\|\mathcal{K}_{\varepsilon}^{\theta_1}-\mathcal{K}_{\varepsilon}^{\theta_2}\|_{L^2(\partial\Omega)\to
L^2(\partial\Omega)}
\leq C_\varepsilon|\theta_1-\theta_2|
\end{equation*}
from the estimate $\eqref{pri:5.11}$ and Remark $\ref{remark:4.1}$,
where $C_\varepsilon$ depending on $\varepsilon$ is also acceptable, and this
implies that
$\big\{\pm\frac{1}{2}I+\mathcal{K}_{\varepsilon}^\theta:t\in[0,1]\big\}$
are continuous families of bounded operators on $L^2(\partial\Omega;\mathbb{R}^m)$.
Since $(\pm1/2+\mathcal{K}_\varepsilon^\theta)$ is invertible in the case of $\theta =0$
(see Theorem $\ref{thm:3.3}$), we conclude that it will be invertible for any $\theta\in[0,1]$
due to the estimate $\eqref{pri:6.8}$.

Finally, the estimate $\eqref{pri:6.9}$ could be proved by the same token and we refer the reader
to Theorem $\ref{thm:4.3}$ for the details. We have completed the proof.
\end{proof}

\begin{thm}\label{thm:6.5}
Let $\epsilon_0>0$ be sufficiently small.
Assume the same conditions as in Theorem $\ref{thm:6.4}$.
If the coefficients $V,B$ additionally satisfy $\|V-B\|_{L^\infty(\mathbb{R}^d)}\leq \epsilon_0$,
then the trace operators
$\pm(1/2)I+\mathcal{K}_{\mathcal{L}_\varepsilon^*}:L^2(\partial\Omega;\mathbb{R}^m)\to
L^2(\partial\Omega;\mathbb{R}^m)$ are invertible,
where the operators $\mathcal{K}_{\mathcal{L}^*_\varepsilon}$ is defined in Theorem $\ref{thm:6.3}$.
Also, there exists a constant
$C>0$, depending on $\mu,\kappa,\lambda,m,d,\tau$ and $\Omega$, such that
for any $g\in L^2(\partial\Omega;\mathbb{R}^m)$,
\begin{equation}
\big\|g\big\|_{L^2(\partial\Omega)}
\leq C\big\|\big(\pm(1/2)I+\mathcal{K}_{\mathcal{L}^*_\varepsilon}\big)(g)\big\|_{L^2(\partial\Omega)}.
\end{equation}
\end{thm}

\begin{proof}
The proof is based upon a duality argument (see Theorem $\ref{thm:4.6}$),
and is not particularly difficult but will not be reproduced here.
\end{proof}

\subsection{Estimates for square functions}

\begin{thm}\label{thm:5.5}
Suppose that the same conditions as in Theorem $\ref{thm:1.1}$.
Given $g\in L^2(\partial\Omega;\mathbb{R}^m)$, let
$u_\varepsilon$  be the solution to
$\mathcal{L}_\varepsilon(u_\varepsilon)= 0$
in $\Omega$ and $u_\varepsilon = g$ n.t. on $\partial\Omega$
with $(u_\varepsilon)^*\in L^2(\partial\Omega)$.
Then we have
\begin{equation}\label{pri:5.16}
\int_\Omega |\nabla u_\varepsilon(x)|^2 \delta(x)dx
\leq C\int_{\partial\Omega}|g|^2dS,
\end{equation}
where $C$ depends on
$\mu,\tau,\kappa,\lambda,m,d$ and $\Omega$.
\end{thm}

Our proof is based upon the related square function estimates
for homogeneous elliptic operators,
which has been new developed in \cite{GZS1,SZW2}.
This releases us from reusing the double layer potential
representation coupled with a complicated T(b)-theorem argument
(see \cite[Theorem 1.1]{MMMT}). For the ease of the statement, recall
the notation $L_\varepsilon = -\text{div}(A(x/\varepsilon)\nabla)$ and it
will appear in the following.

\begin{thm}\label{thm:5.4}
Suppose that $A$ satisfies $\eqref{a:1},\eqref{a:2}$ and $\eqref{a:4}$. Let
$u_\varepsilon$  be the solution to
$L_\varepsilon(u_\varepsilon)= F_0 + \emph{div}(F_1)$
in $\Omega$ and $u_\varepsilon = 0$ on $\partial\Omega$, where
$F_0\in L^2(\Omega;\mathbb{R}^m)$ and $F_1\in L^2(\Omega;\mathbb{R}^{md})$.
Then there holds
\begin{equation}\label{pri:5.17}
\int_\Omega |\nabla u_\varepsilon(x)|^2[\delta(x)]^{\sigma_2}dx
\leq C\int_\Omega|F_0|^2[\delta(x)]^{\sigma_1+2} dx
+C\int_\Omega|F_1|^2[\delta(x)]^{\sigma_1} dx
\end{equation}
for any $0\leq \sigma_1<\sigma_2\leq 1$, where $C$ depends on
$\mu,\tau,\kappa,m,d,\sigma_1,\sigma_2$ and $\Omega$.
\end{thm}

\begin{proof}
The proof may be found in \cite[Theorem 4.3]{GZS1}.
\end{proof}

\begin{thm}
Assume the same conditions as in Theorem $\ref{thm:5.4}$. Let
$f\in L^2(\partial\Omega;\mathbb{R}^m)$ and $u_\varepsilon$ be
the unique solution of the $L^2$ Dirichlet problem:
$L_\varepsilon(u_\varepsilon) = 0$ in $\Omega$,
$u_\varepsilon = f$ n.t. on $\partial\Omega$
with $(u_\varepsilon)^*\in L^2(\partial\Omega)$. Then we have
\begin{equation}\label{pri:5.18}
\Big(\int_{\Omega}|\nabla u_\varepsilon(x)|^2\delta(x)dx
\Big)^{1/2}
\leq C\|f\|_{L^2(\partial\Omega)},
\end{equation}
where $C$ depends on $\mu,\kappa,\tau,m,d$ and $\Omega$.
\end{thm}

\begin{proof}
The proof may be found in \cite[Theorem 2.1]{SZW2}.
\end{proof}

\noindent\textbf{Proof of Theorem $\ref{thm:5.5}$.} We may
decompose $u_\varepsilon = v_\varepsilon + w_\varepsilon$ in $\Omega$
due to the linearity, where
\begin{equation*}
(1)\left\{\begin{aligned}
L_\varepsilon(v_\varepsilon) &= \text{div}(F_1) + F_0
&~&\text{in}~\Omega,\\
v_\varepsilon &= 0
&~&\text{on}~\partial\Omega,
\end{aligned}\right.
\qquad\quad
(2)\left\{\begin{aligned}
L_\varepsilon(w_\varepsilon) &= 0
&~&\text{in}~\Omega,\\
w_\varepsilon &= u_\varepsilon
&~&\text{on}~\partial\Omega,
\end{aligned}\right.
\end{equation*}
in which $F_0 = -B(x/\varepsilon)\nabla u_\varepsilon
-\big(c(x/\varepsilon)+\lambda I\big)u_\varepsilon$ and
$F_1 = V(x/\varepsilon)u_\varepsilon$. Hence, in terms of (1), it
follows from the estimate $\eqref{pri:5.17}$ that
\begin{equation}\label{f:5.38}
\int_{\Omega}|\nabla v_\varepsilon|^2 \delta(x) dx
\leq C\int_{\Omega}|\nabla u_\varepsilon|^2[\delta(x)]^{2+\sigma_1}dx
+ C\int_{\Omega} |u_\varepsilon|^2 [\delta(x)]^{\sigma_1}dx
\leq C\int_{\partial\Omega}|(u_\varepsilon)^*|^2dS.
\end{equation}
In the last inequality, we employ the interior estimate
\begin{equation*}
 |\nabla u_\varepsilon(x)|^2[\delta(x)]^{2+\sigma_1}
 \leq C[\delta(x)]^{\sigma_1}\dashint_{B(x,\delta(x)/2)}|u_\varepsilon|^2 dx,
\end{equation*}
which implies
\begin{equation*}
\int_{\Omega\setminus\Sigma_{c_0}}
|\nabla u_\varepsilon(x)|^2[\delta(x)]^{2+\sigma_1}dx
\leq C\int_{\partial\Omega}|(u_\varepsilon)^*|^2 dS\int_0^{c_0}
r^{\sigma_1}dr
\leq C\int_{\partial\Omega}|(u_\varepsilon)^*|^2 dS
\end{equation*}
by co-area formula, and
\begin{equation*}
\int_{\Sigma_{c_0}}
|\nabla u_\varepsilon(x)|^2[\delta(x)]^{2+\sigma_1}dx
\leq C\Big(\int_{\Omega}|u_\varepsilon|^{\frac{2d}{d-1}} dx\Big)^{\frac{d-1}{d}}
\leq C\int_{\partial\Omega}|(u_\varepsilon)^*|^{2} dS,
\end{equation*}
where we use $\|w\|_{L^{\frac{2d}{d-1}}(\Omega)}
\leq C\|w\|_{L^2(\partial\Omega)}$ (see for example \cite[Remark 9.3]{SZW12})
in the last step.

For (2), in view of the estimate $\eqref{pri:5.18}$ we obtain
\begin{equation*}
\int_{\Omega}|\nabla w_\varepsilon|^2\delta(x)dx
\leq C\int_{\partial\Omega} |u_\varepsilon|^2 dS,
\end{equation*}
and this together with $\eqref{f:5.38}$ leads to
\begin{equation*}
\int_{\Omega} |\nabla u_\varepsilon|^2 \delta(x)dx
\leq C\int_{\partial\Omega}
|(u_\varepsilon)^*|^2 dS
\leq C\int_{\partial\Omega}
|g|^2 dS
\end{equation*}
where we use the estimate $\eqref{pri:1.1}$ in the last inequality, and
the proof is complete.
\qed

Now, we are ready to show the well-posedness of the $L^2$ Dirichlet, Neumann, and regularity problems.

\noindent\textbf{Proof of Theorem $\ref{thm:1.1}$.}
For any $g\in L^2(\partial\Omega;\mathbb{R}^m)$, it is known by Theorem $\ref{thm:6.5}$ that
$(-\frac{1}{2}I+\mathcal{K}_{\mathcal{L}^*_\varepsilon})^{-1}(g)\in L^2(\partial\Omega;\mathbb{R}^m)$ and
\begin{equation}\label{f:6.12}
\big\|\big(-(1/2)I+\mathcal{K}_{\mathcal{L}^*_\varepsilon}\big)^{-1}(g)\big\|_{L^2(\partial\Omega)}
\leq C\|g\|_{L^2(\partial\Omega)}.
\end{equation}
Let
$u_\varepsilon=\mathcal{D}_{\mathcal{L}_\varepsilon}
\big((-\frac{1}{2}I+\mathcal{K}_{\mathcal{L}^*_\varepsilon})^{-1}(g)\big)$ be the double layer potential,
which is such that $\mathcal{L}_\varepsilon(u_\varepsilon) = 0$ in $\Omega$
and $u_\varepsilon$ n.t. on $\partial\Omega$ with $(u_\varepsilon)^*\in L^2(\partial\Omega)$
(see Lemma $\ref{lemma:6.1}$).
On account of $\eqref{f:6.12}$ one may clearly derive the stated estimate $\eqref{pri:1.1}$. Also,
the square function estimate $\eqref{pri:1.4}$ has been shown in Theorem $\ref{thm:5.5}$. Finally,
$u_\varepsilon\in H^{1/2}(\Omega)$ follows from the estimate $\eqref{pri:1.4}$ by the real interpolation
(see \cite[P.181-182]{JK}). The uniqueness of the solution will be verified
by a similar way as in the proof of Theorem $\ref{thm:4.2}$. We have completed the proof.
\qed

\noindent\textbf{Proof of Theorem $\ref{thm:1.2}$.}
For any $g\in L^2(\partial\Omega;\mathbb{R}^d)$, it follows from Theorem $\ref{thm:6.4}$ that
$(\frac{1}{2}I+\mathcal{K}_{\varepsilon})^{-1}(g)\in L^2(\partial\Omega;\mathbb{R}^m)$
and $\|(\frac{1}{2}I+\mathcal{K}_{\varepsilon})^{-1}(g)\|_{L^2(\partial\Omega)}
\leq C\|g\|_{L^2(\partial\Omega)}$. Thus, it is fine to define a single layer potential
$u_\varepsilon = \mathcal{S}_{\varepsilon}
((\frac{1}{2}I+\mathcal{K}_{\varepsilon})^{-1}(g))$, and not hard to see that
$u_\varepsilon$ satisfies $(\mathbf{NH_\varepsilon})$. Due to the estimate $\eqref{pri:6.4}$ we have
$\|(\nabla u_\varepsilon)^*\|_{L^2(\partial\Omega)}\leq C\|g\|_{L^2(\partial\Omega)}$ which is exactly
the estimate $\eqref{pri:1.2}$. The uniqueness
of the solution is based upon the equality
\begin{equation}\label{eq:6.1}
\mathrm{B}_{\mathcal{L}_\varepsilon;\Omega}[u_\varepsilon,u_\varepsilon]
 = \int_{\partial\Omega}\frac{\partial u_\varepsilon}{\partial \nu_{\varepsilon}}u_\varepsilon dS
\end{equation}
as we have pointed out in the proof of Theorem $\ref{thm:4.2}$. This ends the proof.

\noindent\textbf{Proof of Theorem $\ref{thm:1.3}$.}
The existence is due to the invertibility
of $\mathcal{S}_{\varepsilon}:L^2(\partial\Omega;\mathbb{R}^m)\to H^1(\partial\Omega;\mathbb{R}^m)$
in Theorem $\ref{thm:6.4}$.
For any $g\in H^1(\partial\Omega;\mathbb{R}^m)$, there exists
$f\in L^2(\partial\Omega;\mathbb{R}^m)$ such that
$\mathcal{S}_{\varepsilon}(f) = g$ on $\partial\Omega$, and one may consider
$u_\varepsilon$ to be the solution of the Neumann problem with the boundary data $f\in L^2(\partial\Omega;\mathbb{R}^m)$.
In view of the estimate $\eqref{pri:6.9}$ coupled with $\eqref{pri:1.2}$ we may derive
$\|(\nabla u)^*\|_{L^2(\partial\Omega)}
\leq C\|f\|_{L^2(\partial\Omega)}
\leq C\|g\|_{H^1(\partial\Omega)}$. Also, the uniqueness may be derived from $\eqref{eq:6.1}$.
The remainder of the proof is to show
$\|(u_\varepsilon)^*\|_{L^2(\partial\Omega)}\leq C\|g\|_{L^2(\partial\Omega)}$, and it will be
achieved from
\begin{equation*}
 \|(u_\varepsilon)^*\|_{L^2(\partial\Omega)}
 \leq C  \|\mathcal{M}(u_\varepsilon)\|_{L^2(\partial\Omega)}
 \leq C \|u_\varepsilon\|_{H^1(\Omega\setminus\Sigma_{c_0})}\leq C\|g\|_{L^2(\partial\Omega)},
\end{equation*}
where we use the fact that $(u_\varepsilon)^*(Q)
\leq C\mathrm{M}_{\partial\Omega}(\mathcal{M}(u_\varepsilon))(Q)$ for any $Q\in\partial\Omega$
in the first inequality, as well as, the estimates $\eqref{pri:3.2.1}$ and $\eqref{pri:6.3}$.
We have completed the proof.
\qed

\begin{center}
\textbf{Acknowledgements}
\end{center}

The authors thank Prof. Zhongwei Shen for very helpful discussions
regarding this work when he visited Peking University and Lanzhou University.
The first author was supported by the China Postdoctoral Science Foundation (Grant No. 2017M620490).
The second author was supported by the National Natural Science Foundation of China (Grant No. 11471147).
The third author was supported by the National Natural Science Foundation of China (Grant No. 11571020).

\end{document}